%% file: periods1.tex
\documentclass[10pt]{amsart}

\input{preamble-arxiv}

\input{macros-IP}

\begin{document}

\title[]{Periods of quaternionic Shimura varieties. I.}
\author{Atsushi Ichino}
\author{Kartik Prasanna}

\begin{abstract}
We study ``quadratic periods'' on quaternionic Shimura varieties and formulate an integral refinement of Shimura's conjecture regarding Petersson inner products of automorphic forms that are related by the Jacquet-Langlands correspondence. The main result is that this integral refinement is implied by another conjecture (Conjecture D below) regarding integrality of theta lifts between certain quaternionic unitary groups.
\end{abstract}

\maketitle

\tableofcontents

\input{introduction1}

\input{quat-shim-var}

\input{un-qu-groups}

\input{theta-correspondence}

\input{rallis-B-general}

\input{schwartz}

\input{rallis-B-computation}
\input{main-conj}

\newpage

\appendix
\settocdepth{section}
\input{hodge-polarization}

\input{weil-GSp}

\input{spl-B}

\input{spl-double-B}

\input{references}
\end{document}

%% file: preamble-arxiv.tex
\usepackage{amsmath}
\usepackage{amssymb}
\usepackage{amscd}
\usepackage{mathrsfs}
\usepackage{amsthm}
\usepackage{verbatim}
\usepackage{bbm}
\usepackage[mathscr]{eucal}
\usepackage[all]{xy}
\usepackage{tocvsec2}
\usepackage{xr-hyper}
\usepackage[hmargin=3cm,vmargin=3.5cm]{geometry}
\usepackage{hyperref}
\usepackage{color}

\usepackage{enumitem}
\setlist[enumerate]{leftmargin=*}
\setlist[itemize]{labelindent=\parindent, leftmargin=*}

\numberwithin{equation}{section}

%% file: macros-IP.tex
\theoremstyle{plain}
\newtheorem{thm}{Theorem}[section]
\newtheorem{lem}[thm]{Lemma}
\newtheorem{prop}[thm]{Proposition}

\theoremstyle{definition}
\newtheorem{defn}[thm]{Definition}
\theoremstyle{remark}
\newtheorem{rem}[thm]{Remark}
\newtheorem{conjecture}[thm]{Conjecture}
\newtheorem{example}[thm]{Example}

\theoremstyle{plain}
\newtheorem{mconjecture}{Conjecture}

\theoremstyle{plain}
\newtheorem{mtheorem}{Theorem}

\theoremstyle{definition}
\newtheorem{mremark}{Remark}

\makeatletter
\def\varddots{\mathinner{\mkern1mu
    \raise\p@\hbox{.}\mkern2mu\raise4\p@\hbox{.}\mkern2mu
    \raise7\p@\vbox{\kern7\p@\hbox{.}}\mkern1mu}}
\makeatother

\def\Gal{\operatorname{Gal}}
\def\Ind{\operatorname{Ind}}
\def\Im{\operatorname{Im}}
\def\pr{\operatorname{pr}}

\def\Res{\operatorname{Res}}
\def\vol{\operatorname{vol}}
\def\tr{\operatorname{tr}}
\def\Hom{\operatorname{Hom}}

\def\ord{\operatorname{ord}}
\def\Aut{\operatorname{Aut}}
\def\End{\operatorname{End}}

\def\disc{\mathrm{disc}}

\def\Filt{\operatorname{Filt}}
\def\Rep{\operatorname{Rep}}
\def\Gr{\operatorname{Gr}}
\def\Triv{\operatorname{Triv}}
\def\Frob{\operatorname{Frob}}
\def\Eis{\operatorname{Eis}}
\def\eval{\operatorname{eval}}
\def\rank{\operatorname{rank}}
\def\im{\operatorname{im}}

\def\Lift{\operatorname{Lift}}
\def\Ad{\operatorname{Ad}}
\def\Lie{\operatorname{Lie}}

\def\sgn{\operatorname{sgn}}

\def\Kudla{\operatorname{Kudla}}

\def\GL{\mathrm{GL}}
\def\GU{\mathrm{GU}}

\def\U{\mathrm{U}}

\def\G{\mathrm{G}}
\def\Sp{\mathrm{Sp}}
\def\GSp{\mathrm{GSp}}
\def\Mp{\mathrm{Mp}}

\def\SO{\mathrm{SO}}
\def\GO{\mathrm{GO}}
\def\SL{\mathrm{SL}}

\def\GMp{\mathrm{GMp}}
\def\Sym{\mathrm{Sym}}

\def\Tam{\mathrm{Tam}}
\def\Sh{\mathrm{Sh}}
\def\fin{\mathrm{fin}}
\def\M{\mathrm{M}}
\def\N{\mathrm{N}}
\def\et{\mathrm{et}}
\def\id{\mathrm{id}}
\def\ad{\mathrm{ad}}
\def\DS{\mathrm{DS}}
\def\St{\mathrm{St}}

\def\s{\mathrm{s}}
\def\ps{\mathrm{ps}}

\def\t{\mathbf{t}}

\def\fd{\mathfrak{d}}
\def\fD{\mathfrak{D}}
\def\fg{\mathfrak{g}}

\def\ii{\mathfrak{i}}

\def\fh{\mathfrak{h}}

\def\jj{\mathfrak{j}}
\def\ff{\mathfrak{f}}
\def\fk{\mathfrak{k}}
\def\fl{\mathfrak{l}}
\def\fm{\mathfrak{m}}

\def\fN{\mathfrak{N}}

\def\o{\mathfrak{o}}
\def\p{\mathfrak{p}}
\def\fp{\mathfrak{p}}

\def\fq{\mathfrak{q}}
\def\fQ{\mathfrak{Q}}
\def\fR{\mathfrak{R}}
\def\fS{\mathfrak{S}}

\def\fX{\mathfrak{X}}

\def\cA{\mathcal{A}}

\def\Gc{\mathcal{G}}
\def\cG{\mathcal{G}}
\def\cH{\mathcal{H}}

\def\II{\mathcal{I}}

\def\cK{\mathcal{K}}
\def\cKt{\tilde{\mathcal{K}}}
\def\K{\mathcal{K}}

\def\cL{\mathcal{L}}
\def\cM{\mathcal{M}}

\def\cO{\mathcal{O}}
\def\cP{\mathcal{P}}

\def\SS{\mathcal{S}}
\def\cS{\mathcal{S}}

\def\cV{\mathcal{V}}
\def\Wc{\mathcal{W}}

\def\A{\mathbb{A}}
\def\C{\mathbb{C}}
\def\Ff{\mathbb{F}}
\def\H{\mathbb{H}}
\def\I{\mathbb{I}}
\def\P{\mathbb{P}}
\def\Q{\mathbb{Q}}
\def\Qbar{\overline{\Q}}

\def\R{\mathbb{R}}
\def\Ss{\mathbb{S}}

\def\X{\mathbb{X}}
\def\Y{\mathbb{Y}}
\def\Z{\mathbb{Z}}

\def\V{\mathbb{V}}

\def\1{\mathbf{1}}

\def\a{\mathbf{a}}
\def\b{\mathbf{b}}
\def\d{\mathbf{d}}
\def\e{\mathbf{e}}

\def\g{\mathbf{g}}
\def\h{\mathbf{h}}
\def\i{\mathbf{i}}
\def\j{\mathbf{j}}

\def\m{\mathbf{m}}
\def\n{\mathbf{n}}
\def\v{\mathbf{v}}
\def\w{\mathbf{w}}
\def\x{\mathbf{x}}
\def\y{\mathbf{y}}
\def\0{\mathbf{0}}

\def\F{\mathbf{F}}

\def\GG{\mathbf{G}}
\def\J{\mathbf{J}}
\def\KK{\mathbf{K}}

\def\VV{\mathbf{V}}
\def\WW{\mathbf{W}}

\def\As{\mathscr{A}}

\def\uo{\underline{\omega}}
\def\uk{\underline{k}}

\def\ba{\boldsymbol{\alpha}}
\def\bb{\boldsymbol{\beta}}
\def\bv{\boldsymbol{\varphi}}

\def\ve{\varepsilon}

\def\Mu{\boldsymbol{\mu}}

\def\lambdat{\tilde{\lambda}}
\def\Rt{\tilde{R}}
\def\rhob{\bar{\rho}}
\def\rar{\rightarrow}

\def\(({( \! (}
\def\)){) \! )}
\def\llangle{\langle \hspace{-1mm} \langle}
\def\rrangle{\rangle \hspace{-1mm} \rangle}

\SelectTips{cm}{}
\newdir^{ (}{{}*!/-5pt/@^{(}}

\newcommand{\mat}[4]{\begin{pmatrix} #1 & #2 \\ #3 & #4 \end{pmatrix}}
\newcommand{\smat}[4]{\left( \begin{smallmatrix} #1 & #2 \\ #3 & #4 \end{smallmatrix} \right)}

\makeatletter
\def\pmodx#1{\allowbreak\ifinner\mkern8mu\else\mkern18mu\fi
 ({\fam\z@ mod}^\times\,#1)}
\makeatother

\newcounter{defcounter}
\newenvironment{intro-equation}{%
\addtocounter{equation}{-1}
\refstepcounter{defcounter}

\begin{equation}}
{\end{equation}}

%% file: introduction1.tex
\section*{Introduction}
\label{sec:intro}

In this paper and its sequels \cite{periods2}, \cite{periods3}, we 
 study periods of automorphic forms on quaternionic Shimura varieties.
Specifically, the periods that we focus on are the Petersson inner products of Hilbert modular forms 
and of their Jacquet--Langlands lifts to quaternionic Shimura varieties. 
This subject was pioneered by Shimura who proved many 
results on algebraicity of ratios of Petersson inner products and made a precise general 
conjecture (\cite{shimura-amj} Conjecture 5.10) that predicts a large number of algebraic relations in the $\Qbar$-algebra generated 
by such periods. Shimura's conjecture was proved by Harris \cite{harris-jams}
under a technical hypothesis on the local components of the corresponding 
automorphic representation. 
This hypothesis was relaxed partly by Yoshida \cite{yoshida-amj}, 
who also used these period relations to prove a refined conjecture of Shimura (\cite{shimura-amj} Conj. 5.12, \cite{shimura-inven88} Conj 9.3) on the factorization of 
Petersson inner products into {\it fundamental periods} up to algebraic factors. 
In later papers \cite{harris-coh1} \cite{harris-coh2}, Harris has considered
the question of generalizing such period relations to the setting of unitary Shimura 
varieties. Specialized to the case of hermitian spaces of dimension two, these latter results provide more precise
information about the fields of rationality of quadratic period ratios of quaternionic modular forms.

In this series of papers, we will study the corresponding {\it integrality} questions. 
The simplest interesting case is the period ratio 
\[
\frac{\langle f,f \rangle }{\langle g,g \rangle }
\]
where $f$ 
is a usual modular form of (even) weight $2k$ (for $\GL_{2}$ over $\Q$) and trivial 
central character, and $g$ is its lift to a Shimura curve 
corresponding to an indefinite quaternion algebra also over $\Q$. The forms $f$ and $g$ 
 here are assumed to be newforms and to be suitably integrally normalized. 
 In this case, there is a very precise rationality result due to Harris and Kudla \cite{hk-triple} which asserts that
the ratio above lies in the field generated by the Hecke eigenvalues of $f$. 
As for the more refined integrality question, what is known is the following:

\begin{enumerate}

\item In the special case when the weight $2k$ equals $2$ and 
$f$ corresponds to an isogeny class of elliptic curves,  it can be shown (see \cite{p-a}, \S 2.2.1) that the period ratio above equals (up to Eisenstein primes) an explicit product of Tamagawa numbers,
  which in turn are related 
  to {\it level-lowering} congruences satisfied by the form $f$. This 
  suggests that such period ratios contain rather deep arithmetic information. The 
  proof in this case follows from combining {\it three} geometric ingredients: 
\begin{itemize} 
\item The work of Ribet on level-lowering 
 \cite{ribet} and its extension due to Ribet and Takahashi \cite{r-t}
  which depend on a study of the geometry of Shimura curves, especially a 
  description of their bad reduction and of the component groups of the 
  N\'{e}ron models of their Jacobians. 
    
  \item The Tate 
  conjecture for products of curves over number fields, which was proved by Faltings \cite{faltings},
  and which implies that modular elliptic curves 
are equipped with a uniformization map $X \rightarrow E$, with $X$ being a Shimura curve.  
  
 \item A study of the {\it Manin constant} for the map $X \rightarrow E$, following 
 \cite{mazur-rat-iso}, \cite{edixhoven}, \cite{abbes-ullmo}.
    
\end{itemize}
  
\item In the more general case of weight $2k >2$, such geometric arguments 
  are not available. The main obstruction is that the Tate conjecture 
  is unknown for products of Kuga--Sato varieties. 
  Instead, one may try to use purely automorphic techniques. 
  This is the strategy employed in 
 \cite{p-a}, where we showed (using the theta correspondence and results from Iwasawa theory) that for $f$ and $g$ of arbitrary even weight, the ratio $\langle f,f \rangle/ \langle g,g \rangle$ is 
 integral outside of an explicit finite set of small primes, and further 
 that it is always divisible by primes at which the form $f$ satisfies 
 certain level-lowering  congruences. The converse divisibility and the 
 more precise relation to Tamagawa numbers is also 
 expected to hold in general, but seems harder to prove. This is 
 one problem that we hope to eventually address by the methods of this paper.

\end{enumerate} 

 Let us now elaborate a bit on the relation of this problem to the Tate conjecture. 
 As described above, 
the case of weight two forms for $\GL_2$ over $\Q$  is
 relatively simple since one knows by Faltings that the Tate conjecture holds for a product of curves.
  This implies that there exists an algebraic cycle on the product
  $X_1 \times X_2$, where $X_1$ and $X_2$ are modular and Shimura curves respectively, 
  that at the level of cohomology, identifies the $f$ and $g$-isotypic components of the 
  ``motives'' $H^1(X_1)$ and $H^1(X_2)$ respectively. The rationality of the period ratio $\langle f,f \rangle/\langle g,g \rangle$ 
  is then a simple consequence of the fact that such a cycle induces an isomorphism of the {\it Hodge--de Rham} structures
  \cite{harris-motives}
  attached to $f$ and $g$. 
 For forms of higher weight, the Jacquet--Langlands correspondence
 can similarly be used to produce Tate classes on a product $W_1 \times W_2$ where 
 $W_1$ and $W_2$ are Kuga--Sato varieties fibered over $X_1$ and $X_2$ respectively. 
 However, we are very far from understanding the Tate (or even Hodge) conjecture in 
 this case. The case of Hilbert modular forms considered in this paper is even harder: 
 in the simplest setting, namely for forms of {\it parallel weight two} and trivial central character, 
the Tate conjecture predicts the existence of algebraic cycles on products of the form $X \times (X_1 \times X_2)$,
 where $X$, $X_1$ and $X_2$ are suitably chosen quaternionic Shimura varieties such that 
 $\dim (X) = \dim (X_1) + \dim (X_2)$. Again, these cycles should induce isomorphisms of Hodge--de Rham structures
 $H^* (X)_\Pi \simeq H^*(X_1)_\Pi \otimes H^*(X_2)_\Pi$ that in turn should imply the predicted period relations up to rationality. (Here the subscript $\Pi$ denotes the 
 $\Pi_f$-isotypic component for a fixed automorphic representation $\Pi=\Pi_\infty \otimes \Pi_f$.)
 This 
 point of view - at least the factorization of Hodge structures - 
 occurs explicitly in the work of Oda (\cite{oda}, \cite{oda-nagoya}).
  It is worth remarking here that the Tate and Hodge conjectures are only expected to hold 
rationally in general and not integrally, and thus by themselves do not predict any statements about
integrality of period ratios. Nevertheless, the discussion above suggests that 
in the setting of arithmetic automorphic forms on Shimura varieties, such integral relations do hold 
and that their proofs lie much deeper than those of the corresponding rational relations.

With this background, we will outline the main results of this paper. 
Let $F$ be a totally real number field with $[F:\Q]=d$, ring of integers $\cO_F$,
class number $h_F$ and discriminant $D_F$.
Let 
 $\Pi=\otimes_v \Pi_v$ be
an irreducible cuspidal automorphic representation of $\GL_2 (\A_F)$ 
of weight 
$(\uk,r)=(k_1, \ldots, k_d,r)$, conductor $\fN$ and central character $\xi_\Pi$. We assume that $k_1 \equiv k_2 \equiv \cdots \equiv k_d \equiv r \pmod 2$ and all the $k_i \ge 1$. These 
are thus the automorphic representations associated with classical Hilbert modular forms. (Note that we allow 
forms of partial or parallel weight one.)

For simplicity we will assume that 
at all finite places $v$ where $\Pi_v$ is ramified, it is 
either a special representation with square-free conductor (i.e., an unramified twist of the Steinberg representation)
or a ramified principal series representation $\Ind (\chi_1 \otimes \chi_2)$ with 
$\chi_1$ unramified and $\chi_2$ ramified. 
We can thus factor the conductor $\fN$ of $\Pi$ as
\[
\fN= \fN_{\s} \cdot \fN_{\ps}
\]
where $\fN_\s$ is the (square-free) product of the conductors at places where $\Pi_v$ is special 
and $\fN_{\ps}$ is the product of the conductors at places where $\Pi_v$ is ramified principal series.  

Let $K_\Pi$ be the number field generated by the Hecke eigenvalues of $\Pi$ and $\cO_{K_\Pi}$ the ring of integers of $K_\Pi$.
We set $N_\Pi :=\N \fN$, $k_\Pi  := \max k_i$ and 
\[
 N(\Pi)  := 2\cdot h_F \cdot D_F \cdot N_\Pi \cdot k_\Pi!, \qquad  R  := \cO_{\Qbar} [1/N(\Pi)].
\]

Let $\Sigma_F$ denote the set of all places of $F$ and $\Sigma_\infty$ and $\Sigma_\fin$ the subsets of infinite and finite places respectively. 
Let $\Sigma_\Pi$  be the set of places $v$ of $F$ at which $\Pi_v$ is discrete series. 
Thus, $\Sigma_\Pi$ equals the union of $\Sigma_{\Pi,\infty}$ and $\Sigma_{\Pi,\fin}$, where  
\begin{align*}
\Sigma_{\Pi,\infty} &:= \Sigma_\Pi \cap \Sigma_\infty = \{ v \in \Sigma_\infty: \ k_v \ge 2\},  \\
\Sigma_{\Pi,\fin} &:= \Sigma_\Pi \cap \Sigma_\fin = \{ v \in \Sigma_\fin:  \  \ord_v(\fN_s) >0 \}.
\end{align*}
For any quaternion algebra $B$ over $F$, let $\Sigma_B$ denote the set of places of $F$ at which $B$ is ramified. Also set
\begin{align*}
 \Sigma_{B,\infty} &:=\Sigma_B \cap \Sigma_\infty, \\
 \Sigma_{B,\fin} &:= \Sigma_B \cap \Sigma_\fin.
\end{align*}

Henceforth we suppose that $\Sigma_B \subset \Sigma_\Pi$, so that by Jacquet--Langlands \cite{jl}, $\Pi$ transfers  to an automorphic representation $\Pi_B$ of $B^\times (\A)$. 
To such a pair $(B,\Pi)$, we will attach in Sec. \ref{ssec:defn-period-invariant} below a canonical quadratic period invariant 
\[
q_B(\Pi) \in \C^\times / R^\times.
\]
This period invariant is essentially 
(i.e., up to some factors coming from normalizations of measures) equal to 
the Petersson inner product of a normalized eigenform $f_B$ in $\Pi_B$. 
Here we use the assumption that $\fN_\s$ is square-free to first fix $f_B$ up to a scalar. 
The scalar is then fixed by requiring that $f_B$ correspond to an integrally normalized section of 
a suitable automorphic vector bundle on the Shimura variety associated with the algebraic group $B^\times$.

The goal of this 
paper and its sequels is to study the relations between the invariants $q_B(\Pi)$ for fixed $\Pi$ as $B$ varies 
over all quaternion algebras in $\Sigma_\Pi$. 
The following conjecture is a more precise version of \cite{p-AIM} Conjecture 4.2
and provides an integral refinement of Shimura's conjecture on algebraic period relations.
 The reader may
consult Sec. 4 of \cite{p-AIM} for a discussion of the motivation behind this formulation. 
To state the conjecture, 
let $L(s,\Pi,\ad)$ denote the adjoint (finite) $L$-function attached to $\Pi$ 
and let $\Lambda (s,\Pi, \ad)$ denote the corresponding 
completed $L$-function that includes the $\Gamma$-factors at the infinite places. 
Let us recall the following invariant of $\Pi$, which 
has played a crucial role in the study of congruences of modular forms (see \cite{hida-cong1}, \cite{hida-cong2}, \cite{wiles-fermat}):

\begin{intro-equation}
\label{eqn:defn-of-Lambda-Pi}
\Lambda (\Pi) := \Lambda(1, \Pi, \ad).
\end{intro-equation}

\begin{mconjecture}
\label{conj:oldintro}
There exists a function 
\[
c(\Pi): \Sigma_\Pi \rar \C^\times/R^\times, \quad v\mapsto c_v(\Pi),
\]
such that:
\begin{enumerate}
\item $c_v (\Pi)$ lies in $R$ (mod $R^\times$) if $v$ is a finite place, and

\item for all $B$ with $\Sigma_B \subseteq \Sigma_\Pi$, we have 
\[
q_B (\Pi) =  \frac{\Lambda (\Pi) }{\prod_{v \in \Sigma_B} c_v (\Pi) } \quad (\text{in } \C^\times/R^\times).
\]
\end{enumerate}
\end{mconjecture}

\begin{mremark}
\label{rem:nature-of-cv}
It is easy to see that if it exists, the function $c(\Pi)$ is uniquely determined as long as $|\Sigma_\Pi|\ge 3$. 
Also, as the notation suggests, the invariants $c_v (\Pi)$ are {\it not} invariants of the local representation $\Pi_v$ but rather 
are really invariants of the global representation $\Pi$. 
\end{mremark}

\begin{mremark}
The conjecture should be viewed as describing {\it period relations} between the 
quadratic periods $q_B (\Pi)$ as $B$ varies. Indeed, the number of $B$ 
with $\Sigma_B \subseteq \Sigma_\Pi$ is $2^{|\Sigma_\Pi|-1}$ but the conjecture predicts that the corresponding invariants $q_B(\Pi)$
can all be described using only $|\Sigma_\Pi|+1$ invariants, 
namely the $L$-value $\Lambda (\Pi)$ and 
the additional invariants $c_v (\Pi)$, which are $|\Sigma_\Pi|$ in number. 
\end{mremark}

\begin{mremark}
\label{rem:conj-in-split-case}
For $B=\M_2(F)$, the conjecture simply predicts that
\[
q_{\M_2(F)} (\Pi) = \Lambda(\Pi) \quad \text{in} \ \C^\times/R^\times.
\]
This piece of the conjecture is known to be true. Indeed, it 
follows from the fact that the integral normalization of $f_B$ in the 
split case coincides with the $q$-expansion normalization (see \cite{deligne-ribet}, \S 5), 
combined with the well known relation between the Petersson inner product 
of a Whittaker normalized form $f\in \Pi$ and the value of the 
adjoint $L$-function at $s=1$. (See Prop. \ref{prop:<f,f>} for instance.)
\end{mremark}

It is natural to ask for an independent description of the invariants $c_v (\Pi)$. 
Before discussing this, we recall the notion of {\it Eisenstein primes} for $\Pi$. 
To any finite place $\lambda$ of $K_\Pi$, one can associate (by 
\cite{shimura-book}, \cite{deligne-galois}, \cite{deligne-serre}, \cite{rog-tun},
\cite{carayol-ens}, \cite{wiles-inv}, \cite{taylor-inv}, \cite{jarvis-crelle}; see also \cite{bl-rog})
an irreducible two dimensional Galois representation
\[
\rho_{\Pi,\lambda}: \Gal (\Qbar /F) \rightarrow \GL_2 (K_{\Pi,\lambda})
\]
that is characterized up to isomorphism by the requirement that
\[
\tr \rho_{\Pi,\lambda} (\Frob_v) = a_v(\Pi)
\]
for any finite place $v$ of $F$ that is prime to $\fN \cdot \N \lambda$, with 
$a_v (\Pi)$ being the eigenvalue of the Hecke operator $T_v$ acting on a new-vector
in $\Pi_v$. 
Choose a model for $\rho_{\Pi,\lambda}$ that takes values in $\GL_2 (\cO_{K_\Pi,\lambda})$ and denote by $\rhob_{\Pi,\lambda}$ the semisimplification of the mod $\lambda$ reduction of $\rho_{\Pi,\lambda}$.
The isomorphism class of $\rhob_{\Pi,\lambda}$ is independent of the choice of model of $\rho_{\Pi,\lambda}$. 
Let $\F_\lambda = \cO_{K_\Pi}/\lambda$ be the residue field at $\lambda$. 
The prime $\lambda$ is said to be Eisenstein for $\Pi$ if 
\[
\rhob_{\Pi,\lambda}: \Gal (\Qbar/F) \rightarrow \GL_2 (\F_\lambda)
\]
is (absolutely) reducible, and non-Eisenstein otherwise. 

Let $N(\Pi)_{\Eis}$ be the product of the $\N \lambda$ as $\lambda$ varies over 
all the Eisenstein primes for $\Pi$. (There are only finitely many such.) Let $\Rt$ denote the ring
\[
\Rt := R [1/N(\Pi)_{\Eis}] = \cO_{\Qbar} [1/ N(\Pi) N(\Pi)_{\Eis}].
\]
The following conjecture characterizes the invariants $c_v(\Pi)$ for finite places $v$ up to Eisenstein primes, 
relating them to {\it level-lowering congruences} for $\Pi$. 
(It is obviously conditional on the truth of Conjecture \ref{conj:oldintro}.)

\begin{mconjecture}
\label{conj:cv-finite-intro}
Suppose that $v$ belongs to $\Sigma_{\Pi,\fin}$.  Let $L \supseteq K_\Pi $ be a number field containing (a representative of) $c_v (\Pi)$ and let $\lambdat$ be a 
finite place of $L$ such that $(\lambdat, N(\Pi))=1$. 
Let $\lambda$ be the place of $K_\Pi$ under $\lambdat$ and suppose that $\Pi$ is not Eisenstein at $\lambda$. Then $v_{\lambdat}(c_v(\Pi))$ equals the largest integer $n$ such that 
$\rho_{\Pi, \lambda} \bmod \lambda^{n}$ is {\it unramified} at $v$.
\end{mconjecture}

At the infinite places $v$, one might hope to have similarly a description of the 
invariants $c_v (\Pi)$ purely in terms of the compatible system $\rho_{\Pi,\lambda}$ of two-dimensional Galois 
representations attached to $\Pi$. In principle,
to such a system one should be able to attach a {\it motive} defined over $F$, 
and the $c_v (\Pi)$ should be related to periods of this motive 
taken with respect to suitable integral structures 
on the de Rham and Betti realizations. In practice, the only case 
in which one can make an unconditional definition is when 
$\Pi$ satisfies the following conditions:
\begin{enumerate}[label=(\alph*),ref=\alph*]
\item \label{cond:pi-of-weight-2} $\Pi$ is of {\it parallel} weight $2$, that is $\uk = (2, \ldots, 2)$.
\item \label{cond:d-odd-or-ds} Either $d(=[F:\Q])$ is odd or $\Sigma_{\Pi,\fin}$ is nonempty.
\end{enumerate}
If $\Pi$ satisfies both \eqref{cond:pi-of-weight-2} and \eqref{cond:d-odd-or-ds} above, it is known (using \cite{carayol-ens}) that one can associate 
to $\Pi$ an abelian variety $A$ over $F$ (or more precisely, an isogeny class 
of abelian varieties) 
such that 
\begin{itemize}
\item $\dim (A) = [K_\Pi:\Q]$;
\item $\End_F (A) \otimes \Q \supset K_\Pi$;
\item $A$ has good reduction outside $\fN$;
\item For any prime $\lambda$ of $K_\Pi$ lying over a rational prime $\ell$, the representation $\rho_{\Pi,\lambda}$ is 
isomorphic to the representation of $\Gal(\Qbar/F)$ on $H^1_{\et} (A_{\Qbar}, \Q_\ell) \otimes_{K_{\Pi} \otimes \Q_\ell} K_{\Pi,\lambda}$. 
\end{itemize}
We may pick in the isogeny class above an abelian variety $A$ such that $\End_F (A) \supset \cO_{K_\Pi}$. 
Then one can make a precise conjecture for $c_v (\Pi)$ for $v\in \Sigma_\infty $ in terms of the periods of $A$. 
Here, we will be content to state this conjecture in the case $K_\Pi = \Q$, 
namely when $A$ is an elliptic curve over $F$. Let $\cA$ denote the N\'{e}ron model of $A$ over 
$R_F:=\cO_F [1/N(\Pi)]$. Then $\cL:=H^0 (\cA, \Omega^1_{\cA/R_F})$ is an invertible $R_F$-module. 
This module can be trivialized by 
picking a large enough number field $K\supseteq F $ and 
extending scalars to the ring $R_K:= \cO_K [1/N(\Pi)]$. 
Pick a generator $\omega$ for $\cL \otimes_{R_F} R_K$ viewed as an $R_K$-module. 
Let $v'$ be any archimedean place of $K$ 
extending $v$, and denote by $\sigma_v: F \rightarrow \R$ the real
embedding of $F$ corresponding to $v$. 
The class of the integral
\[
\frac{1}{(2\pi i)^2} \int_{A\otimes_{\sigma_{v}} \C} \omega_{v'} \wedge \bar{\omega}_{v'}
 \]
 in $\C^\times /\Rt^\times$ can be checked to be independent of the choices above.

\begin{mconjecture}
\label{conj:cv-infinite-intro}
Suppose that $K_\Pi = \Q$ so that $A$ is an elliptic curve. Then
\[
c_v (\Pi) =  \frac{1}{(2\pi i)^2} \int_{A\otimes_{\sigma_{v}} \C} \omega_{v'} \wedge \bar{\omega}_{v'}   \quad \text{ in } \ \C^\times / \Rt^\times.
\]
\end{mconjecture}

\begin{mremark}
One expects that the invariants $c_v(\Pi)$ are transcendental for any infinite place $v$. 
Note that if $A$ is the base change of an elliptic curve defined over a smaller 
totally real field $F'$ (in which case $\Pi$ is 
the base change of a Hilbert modular form for $F'$), then there are obvious algebraic relations between the $c_v (\Pi)$. 
It would be interesting to formulate a converse to this: namely,
can one give a criterion for $\Pi$ to be a base change purely in terms of the $\Qbar$-algebra generated by the invariants $c_v (\Pi)$?
 \end{mremark}
 
 \begin{mremark}
 It would also be interesting to formulate the conjectures above without inverting $N(\Pi)$. 
 There are lots of obvious difficulties with primes 
 that are small with respect to the weight as well as with integral models at primes of bad reduction.  
 In \cite{periods3}, we will extend the conjectures above in the case $F=\Q$ to include primes of bad reduction
 at which the local component of the automorphic representation $\Pi$ is {\it ramified principal series}.
 The only Shimura varieties 
 that occur then are Shimura curves and those associated with definite quaternion algebras over $\Q$.
 The geometric difficulties with primes of bad reduction can be dealt with in this case ``by hand". 
 \end{mremark}

The goal of this first paper is to reformulate Conjecture \ref{conj:oldintro} in terms of a new conjecture (Conjecture \ref{conj:newintro} below)
on the arithmetic properties of a theta lift between quaternionic Shimura varieties. 
This reformulation has many advantages 
since the arithmetic of theta lifts can be studied via a range of automorphic techniques 
including the Rallis inner product formula and period integrals along tori. 
Moreover, the constructions involved seem 
to be useful in attacking several other related problems involving algebraic cycles. We 
will briefly discuss two such applications below. 

Now we outline the main construction. Let $B_1$, $B_2$ and $B$ be 
three quaternion algebras in $\Sigma_\Pi$ such that $B=B_1\cdot B_2$ in the Brauer group of $F$. 
There is then, up to isometry, a unique skew-hermitian $B$-space $(V, \langle \cdot, \cdot \rangle)$ such that
\[
\GU_B (V)^0 \simeq (B_1^\times \times B_2^\times)/F^\times.
\]
Here $\GU_B (V)^0$ denotes the identity component of the group of 
quaternionic unitary similitudes of $V$. For computational purposes, we will need an explicit construction 
of such a space $V$. For this, we pick 
a CM extension $E/F$ with 
\[
E=F+F\i, \quad \i^2 = u \in F^\times,
\]
such that $E$ embeds in $B_1$ and $B_2$. Fix 
embeddings 
\[
E \hookrightarrow B_1, \quad E \hookrightarrow B_2
\]
and write
\[
B_1 = E+ E\j_1, \quad B_2 = E+ E\j_2,
\]
where $\j_1^2 = J_1$ and $\j_2^2 = J_2$ lie in $F$. Then there is an embedding of $E$ in $B$ such that 
\[
B= E+E\j, \quad \j^2= J,
\]
where $J=J_1 J_2$. Let $V=B_1 \otimes_E B_2$, viewed as a right $E$-vector space. In Chapter \ref{sec:uqu} below, we show that 
$V$ can naturally be equipped with a right $B$-action extending the action of $E$ as well as a 
$B$-skew Hermitian form  $\langle \cdot, \cdot \rangle$ such that the quaternionic unitary similitude group $\GU_B (V)^0$
has the form above. 

Let $W$ be a one-dimensional $B$-space equipped with the standard $B$-hermitian form so that 
\[
\GU_B (W) = B^\times.
\] 
We wish to study the theta lift
\[
\Theta: \As(\GU_B(W)) \longrightarrow \As (\GU_B(V)^0 ),
\]
where $\As$ denotes the space of cuspidal automorphic forms. 
The pair $(\U_B(W), \U_B (V))$ is an example of a classical reductive dual pair. For our applications
we need to work with the corresponding similitude groups. In order to construct 
the theta lift, one needs to first construct (local) splittings of the metaplectic cover over the subgroup
\[
\{ (g,h) \in \GU_B(V)^0  \times \GU_B (W): \nu(g) = \nu (h) \},
\]
that satisfy the product formula. (Here $\nu$ denotes the similitude character.) 
For quaternionic unitary similitude groups, this does not seem to be covered in the existing literature. This problem is handled in the appendices under the assumption that $u$, $J_1$ and $J_2$ are chosen such that 
for every finite place $v$ of $F$, at least one of $u$, $J_1$, $J_2$ an $J$ is locally a square. (See Remark \ref{rem:splittings} below.)

The splittings being chosen, the 
correspondence $\Theta$ above can be defined and studied. 
For any quaternion algebra $B'$ with $\Sigma_{B'} \subseteq \Sigma_\Pi$, we let $\pi_{B'}$ denote the {\it unitary} representation associated with $\Pi_{B'}$.
Thus
\[
\pi_{B'}=  \Pi_{B'} \otimes \| \nu_{B'} \|^{-r/2},
\]
where $\nu_{B'}$ denotes the reduced norm on $B'$. 
In Chapter \ref{sec:rallis-general}, we  prove the following theorem regarding $\Theta$ (in the case $B\neq \M_2(F)$) which gives an explicit realization of the Jacquet--Langlands correspondence, extending the work of Shimizu \cite{shimizu-theta}.
\begin{mtheorem}
\label{thm:JLintro}
\[
\Theta (\pi_B) = \pi_{B_1} \boxtimes \pi_{B_2}.
\]
\end{mtheorem}

\medskip

\begin{mremark}
Up to this point in the paper, we make no restrictions on $F$ or $\Pi$. 
However from Chapter \ref{sec:schwartz} onwards (and thus in the rest of the introduction), we assume for simplicity the following:
\begin{itemize}
\item $\fN$ is prime to $2\fD_{F/\Q}$, where $\fD_{F/\Q}$ denotes the different of $F/\Q$.
\end{itemize}
These assumptions simplify some of the local computations in Chapters \ref{sec:schwartz} and \ref{sec:rallis-explicit},
and could be relaxed with more work. 
\end{mremark}

While Theorem \ref{thm:JLintro} is an abstract representation theoretic statement, for our purposes we need to study a 
more explicit theta lift. 
The Weil representation used to define the theta lift above is realized on a certain Schwartz space $\SS (\X)$.  
In Chapter \ref{sec:schwartz}, we pick an explicit canonical Schwartz function $\varphi \in \SS (\X) $ with the property that 
$\theta_\varphi (f_B) $ is a scalar multiple of $f_{B_1} \boxtimes f_{B_2}$. Thus 
\begin{intro-equation}
\label{eqn:alphaintro}
\theta_\varphi (f_B) = \alpha(B_1,B_2) \cdot (f_{B_1} \times f_{B_2}),
\end{intro-equation}
for some scalar $\alpha(B_1,B_2) \in \C^\times$.
The scalar $\alpha (B_1,B_2)$ depends not just on $B_1$ and $B_2$ but also on the other choices made above. However,
we will omit these other dependencies in the notation.

That $\alpha(B_1,B_2)$ is nonzero follows from the following explicit version of the Rallis inner product formula, proved
in Chapter \ref{sec:rallis-explicit}. (The assumption $B\neq \M_2(F)$ in the statement below is 
made since the proof in the case $B=\M_2(F)$ would be somewhat different. See for 
instance \S 6 of \cite{gi}. Note that this case corresponds to the original setting of Shimuzu \cite{shimizu-theta}, and 
is not needed in this paper.)

\begin{mtheorem}
\label{thm:rallisintro} 
Suppose $B_1\neq B_2$, or equivalently, $B\neq \M_2(F)$. Then
\[
|\alpha (B_1,B_2) |^2 \cdot \langle f_{B_1}, f_{B_1} \rangle \cdot \langle f_{B_2}, f_{B_2} \rangle = C \cdot \langle f,f \rangle \cdot \langle f_B, f_B \rangle,
\]
where $C$ is an explicit constant (see Thm. \ref{thm:rallis-B-explicit}) and 
$f$ is a Whittaker normalized form in $\Pi$ (as in Remark \ref{rem:conj-in-split-case}).
\end{mtheorem}

The arithmetic properties of $\alpha(B_1,B_2)$ are of key importance. 
As such, the choice of measures needed to define the invariants $q_B(\Pi)$ 
requires us to work with a slight modification of $\alpha(B_1,B_2)$, denoted 
$\ba (B_1,B_2)$, as described in \S \ref{ssec:mainconjecture}.
We are especially interested in questions of integrality of $\ba (B_1,B_2)$ for which we may work one prime at a time. 
Thus we fix a prime $\ell$ not dividing $N(\Pi)$ and then choose all the data (for example, $E$, $J_1$, $J_2$, $\varphi$) to be suitably adapted to $\ell$. The choices are described in detail in Sec. \ref{ssec:choices}. 
Finally, we come to main conjecture of this paper, which
is motivated by combining Theorem \ref{thm:rallisintro} with Conj. \ref{conj:oldintro}.

\begin{mconjecture}
\label{conj:newintro}
Suppose that $B_1\neq B_2$ and $\Sigma_{B_1} \cap \Sigma_{B_2} \cap \Sigma_\infty = \varnothing$, that is $B_1$ and $B_2$ have no infinite places of ramification in common. Then
\begin{enumerate}
\item $\ba(B_1,B_2)$ lies in $\Qbar^\times$.
\item $\ba(B_1,B_2)$ is integral at all primes above $\ell$. 
\item If in addition $B_1$ and $B_2$ have no finite places of ramification in common, then $\ba(B_1,B_2)$ is a unit at 
all primes above $\ell$. 
\end{enumerate}
\end{mconjecture}

While not immediately apparent, 
Conjecture \ref{conj:newintro} implies Conjecture \ref{conj:oldintro}.
 Indeed, in Sec. \ref{ssec:mainconjecture}, we show the following.

\begin{mtheorem}
Suppose that Conjecture \ref{conj:newintro} is true for all $\ell$ in some set of primes $\Xi$. 
Then  Conjecture \ref{conj:oldintro} holds with $R$ replaced by $R[1/\ell : \ell \not \in \Xi]$.
Consequently, if Conjecture \ref{conj:newintro} is true for all $\ell \nmid N(\Pi)$, then 
Conjecture \ref{conj:oldintro} is true. 
\end{mtheorem}

At this point, the reader may feel a bit underwhelmed since all we seem to have done is 
reformulate Conjecture \ref{conj:oldintro} in terms of another conjecture that is not visibly easier. 
However, we believe that Conjecture \ref{conj:newintro} provides the correct perspective to attack 
these fine integrality questions about period ratios, for several reasons. 
Firstly, it does not require an a priori definition of the invariants $c_v$. 
Second, it fits into the philosophy that {\it theta lifts have excellent arithmetic properties}
and is amenable to attack by automorphic methods of various kinds. Lastly, 
it is usually a very hard problem (in Iwasawa theory, say) to prove {\it divisibilities}; on 
the other hand, if a quantity is expected to be a unit, then this might be easier to show, for instance using congruences.
Part (iii) of Conjecture \ref{conj:newintro}, which states that $\ba (B_1,B_2)$ is often a unit, has hidden in it a large number of divisibilities
that would be very hard to show directly, but that might be more accessible 
when approached in this way. This is the approach 
taken in the sequels \cite{periods2} and \cite{periods3} where we 
study Conjecture \ref{conj:newintro} and give various applications to periods.

As mentioned earlier, the constructions discussed above also have 
concrete applications to problems about algebraic cycles. 
We mention two articles in progress 
that rely crucially on this paper:

\begin{itemize} 
\item In \cite{gz}, we study the Bloch--Beilinson conjecture for Rankin--Selberg $L$-functions $L(f_E, \chi,s)$,
where $f$ is a modular form of weight $k$ and $\chi$ is a Hecke character of an imaginary quadratic field $E$ of infinity 
type $(k',0)$ with $k' \ge k$. The simplest case is when $(k,k')=(2,2)$. In this case we give 
an explicit construction of cycles corresponding to the vanishing of the $L$-function
and prove a relation between the $p$-adic logarithms of such cycles and values of $p$-adic $L$-functions. 
(All previous constructions of cycles for such $L$-functions (\cite{gross-zagier},
\cite{nekovar}, \cite{bdp1}) only work in the case $k>k'$.) The key input from this paper is 
the embedding 
\[
\GU_B(V)^0 \rightarrow \GU_E(V)
\]
which provides a morphism of Shimura varieties that can be used to construct the relevant cycle. 

\smallskip

\item In \cite{hodge-classes} we consider the Tate conjecture for products 
$X_{1} \times X_{2}$ where $X_1$ and $X_2$ are the Shimura varieties 
associated with two quaternion algebras $B_1$ and $B_2$ over 
a totally real field $F$ that have identical ramification at the infinite places of $F$. 
As explained earlier, the Jacquet--Langlands correspondence gives rise to natural Tate classes 
on $X_1 \times X_2$ and the Tate conjecture predicts the existence of algebraic cycles on the product
giving rise to these Tate classes. While we cannot as yet show the existence of such cycles, we are able to 
at least give an unconditional construction of the corresponding Hodge classes. Moreover, these Hodge classes 
are constructed not by comparing periods but rather by finding a morphism
\[
X_1 \times X_2 \rightarrow X
\]
into an auxiliary Shimura variety $X$ and constructing Hodge classes on $X$ that 
restrict nontrivially to $X_1 \times X_2$. Thus we reduce the Tate conjecture on
$X_1 \times X_2$ to the Hodge conjecture on $X$ which should in principle be an easier problem.
The relation with the current paper is that $X_1\times X_2$ and $X$ may be viewed as the Shimura varieties associated with 
certain skew-hermitian $B$ spaces, with $B=B_1\cdot B_2$. 
\end{itemize}

Finally, we give a brief outline of the contents of each chapter. 
In Chapter \ref{sec:qvs} we recall the theory of automorphic vector bundles on quaternionic
Shimura varieties and define the canonical quadratic period invariants $q_B (\Pi)$. 
In Chapter \ref{sec:uqu}, we give the key constructions involving quaternionic skew-hermitian 
forms. Chapter \ref{sec:theta} discusses the general theory of the theta correspondence 
as well as the special case of quaternionic dual pairs, while 
Chapter \ref{sec:rallis-general} establishes the general form of the
Rallis inner product formula in our situation and proves that the 
theta lift we are considering agrees with the Jacquet--Langlands correspondence. 
In Chapter \ref{sec:schwartz}, we pick explicit Schwartz functions, which 
are then used in Chapter \ref{sec:rallis-explicit} to compute the precise form of the 
Rallis inner product formula in our setting. In Chapter \ref{sec:main-conj} we 
first discuss all the choices involved in 
formulating the main conjecture, Conjecture \ref{conj:newintro} above, and then 
show that it implies Conjecture \ref{conj:oldintro}. 
Appendix \ref{sec:hodge} is strictly not necessary but is useful in motivating 
some constructions in Chapter \ref{sec:qvs}. 
The results from Appendix  \ref{sec:weil-gsp} on metaplectic covers of symplectic {\it similitude} groups
are used in the computations in Appendix \ref{sec:spl-B}.
Appendices 
\ref{sec:spl-B} and \ref{sec:spl-double-B} are invoked in Chapters \ref{sec:theta},
\ref{sec:rallis-general} and \ref{sec:schwartz}, and contain the construction
of the relevant splittings, on which more is said in the remark below. 

\begin{mremark}
\label{rem:splittings}
The problem of constructing the required splittings and checking various compatibilities involving them turns out to be rather nontrivial and 
occupies the lengthy Appendices \ref{sec:spl-B} and \ref{sec:spl-double-B}. For {\it 
isometry groups}, these can be handled using the doubling method as in Kudla \cite[\S 4]{kudla-splitting}.
This gives a collection of splittings (one for each place $v$)  
\[
s_{\Kudla,v}: \U_B(V) (F_v) \times \U_B(W) (F_v) \rightarrow \C^{(1)}
\] 
that satisfy the {\it product formula}:
\[
\prod_v s_{\Kudla,v} (\gamma) =1
\]
for $\gamma \in \U_B(V) (F) \times \U_B (W) (F)$. 
The problem is really to extend these splittings to the groups
\[
\{ (g,h) \in \GU_B(V)^0 (F_v) \times \GU_B(W)(F_v): \nu(g) = \nu (h) \}
\]
in such a way that they still satisfy the product formula. 
A similar problem for the dual pairs consisting of the unitary similitude groups of a hermitian $E$-space $\VV$ and 
a skew-hermitian $E$-space $\WW$ can be solved using the fact that
$\VV \otimes_E \WW $ can be considered as 
a {\it skew-hermitian} $E$-space,
and the group
\[
\{ (g,h) \in \GU_E(\VV)  \times \GU_E (\WW): \nu(g) = \nu (h) \}
\]
(almost) embeds in $\U_E (\VV \otimes_E \WW)$.  This fails when working with $B$-spaces 
since $B$ is non-commutative and the tensor product construction is not available. 
To circumvent this problem, we first construct by hand, splittings
\[
s_v: \{ (g,h) \in \GU_B(V)^0 (F_v) \times \GU_B(W)(F_v): \nu(g) = \nu (h) \} \rightarrow \C^{(1)}
\]
in Appendix \ref{sec:spl-B} and 
check that they satisfy several natural properties including the product formula (Proposition \ref{prop:spl-B-prod}).
This suffices to construct the theta lift $\Theta$. 
In order to prove the Rallis inner product formula, we need 
to check a further compatibility between our   
 splittings $s_v$  and the splittings $s_{\Kudla,v}$, in the 
 context of the doubling method. 
This is accomplished in Lemma \ref{lem:spl-double-B-compare} in Appendix \ref{sec:spl-double-B}. 
\end{mremark}

\noindent {\bf Acknowledgments:} During the preparation 
of this article, A.I. was partially supported by JSPS KAKENHI Grant Numbers 22740021, 26287003,
and K.P. was partially supported by NSF grants DMS 0854900, DMS 1160720, a grant from the 
Simons Foundation (\# 305784) and by a Margaret and Herman Sokol Faculty Award from the University of Michigan.

%% file: quat-shim-var.tex
\section{Quaternionic Shimura varieties}
\label{sec:qvs}

\subsection{Shimura varieties}
\label{ssec:shim-var}

\subsubsection{Shimura varieties and canonical models}

We recall quickly the general theory of Shimura varieties and their canonical models \cite{deligne-shimura}. Let $\Ss:=\Res_{\C/\R} \GG_m$ denote the Deligne torus. There is an equivalence of categories 
\[
\R \text{-Hodge structures } \leftrightarrow \R  \text{-vector spaces with an algebraic action of } \Ss,
\]
described as follows. Suppose that $V$ is an $\R$-vector space equipped with a pure Hodge structure of weight $n$. Thus we have a decomposition of $V_\C:=V\otimes_\R \C$:
$$V_\C = \bigoplus_{p+q=n} V^{pq},$$
where $V^{pq}= \overline{V^{qp}}$. Define an action $h$ of $\C^\times$ on $V_\C$ by 
$$ h(z)v = z^{-p} \bar{z}^{-q} v \quad \text{for} \quad v \in V^{pq}.$$ 
Since $h(z)$ commutes with complex conjugation, it is obtained by extension of scalars from an automorphism of $V$ defined over $\R$. This gives a map on real points $h: \Ss (\R) =\C^\times \rightarrow \GL(V) (\R)$, that comes from an algebraic map $\Ss \rightarrow \GL(V)$. 

A Shimura datum is a pair $(G,X)$ consisting of a reductive algebraic group $G$ over $\Q$ 
and a $G(\R)$-conjugacy class $X$ of homomorphisms $h:\Ss \rar G_\R$ satisfying the following conditions:
\begin{enumerate}
\item For $h$ in $X$, the Hodge structure on the Lie algebra $\fg$ of $G_\R$ given by $\Ad \circ h$ is of type $(0,0) + (-1,1) + (1,-1)$. (In particular, the restriction of such an $h$ to $\GG_{m,\R} \subset \Ss$ is trivial.)
\item For $h$ in $X$, $\Ad \ h(i)$ is a Cartan involution on $G_{\R}^{\ad}$, where $G^{\ad}$ is the adjoint group of $G$.  
\item $G^\ad$ has no factor defined over $\Q$ whose real points form a compact group. 
\end{enumerate}

These conditions imply that $X$ has the natural structure of a disjoint union of Hermitian 
symmetric domains. 
The group $G(\R)$ acts on $X$ on the left by
\[
(g \cdot h)(z) = g\cdot h(z) \cdot g^{-1}.
\]
To agree with our geometric intuition, we will sometimes write $\tau_h$ (or simply $\tau$) 
for $h$ in $X$. 

Let $\A$ and $\A_f$ denote respectively the ring of ad\`{e}les and finite ad\`{e}les of $\Q$.
Let $\K$ be an open compact subgroup of $G(\A_f)$. The Shimura variety associated to $(G,X,\K)$ is the quotient
\[
\Sh_\K (G,X) =  G(\Q) \backslash X \times G(\A_f)/ \K.
\]
For $\cK$ small enough, this has the natural structure of a smooth variety over $\C$. The inverse limit
\[
\Sh (G,X) = \underleftarrow{\lim}_\K \ \Sh_\K (G,X)
\]
is a pro-algebraic variety that has a canonical model over a number field $E(G,X)$, the reflex field of the Shimura datum $(G,X)$. In particular, each $\Sh_\K (G,X)$ has a canonical model over $E(G,X)$. 

We recall the definition of $E(G,X)$.
This field is defined to be the field of definition of the conjugacy class of co-characters
\[
\mu_h: \GG_{m,\C} \rightarrow \Ss_\C \rar G_\C,
\]
where the first map is $z\mapsto (z,1)$ and the second is the one induced by $h$. 
Let us say more precisely what this means.
For any subfield $k$ of $\C$ , let 
$\cM (k)$ denote the set of $G(k)$-conjugacy classes of 
homomorphisms $\GG_{m,k} \rightarrow G_k$. Then 
the inclusion $\Qbar \hookrightarrow \C$ gives a bijection 
between $\cM(\Qbar)$ and $\cM (\C)$. This gives a 
natural action of $\Gal(\Qbar/\Q)$ on $\cM(\C)$. The reflex 
field $E(G,X)$ is then the fixed field of the subgroup
\[
\{ \sigma\in \Gal(\Qbar / \Q): \ \sigma M_X = M_X\}
\]
where $M_X$ is the conjugacy class of $\mu_h$ for any $h\in X$.

\subsubsection{Automorphic vector bundles} 
We recall the basics of the theory of automorphic vector bundles following \cite{harris-avb1}, \cite{harris-avb2},
\cite{gar-har}. 
First, to any $\mu : \GG_{m, \C} \rar G_\C$ as above one can associate 
a filtration $\Filt(\mu)$ of $\Rep_\C (G_\C)$. This is the functor which 
assigns to every complex representation $(V,\rho)$ of $G_\C$ the 
filtered vector space $(V,F_\mu^{\cdot})$ where $F_\mu^{\cdot}$ is the filtration on $V$
corresponding to $\rho \circ \mu$; that is, $F_\mu^{p} V =  \oplus_{i\ge p} V_\mu^{i}$, where 
$V_\mu^{i}$ is the subspace of $V$ on which $\GG_m (\C)$ acts via $z \mapsto z^i $.
In particular, one obtains a filtration on $\fg_\C$ via the adjoint representation of $G(\C)$. 
Let $P_\mu$ be the subgroup of $G_\C$ that preserves the filtration $F_\mu^\cdot$ 
in every representation $(V,\rho)$. Then $P_\mu$ is a 
parabolic subgroup of $G_\C$ that contains the image of $\mu$ and has Lie algebra $F_\mu^0 \fg_\C$. 
The unipotent radical $R_u P_\mu$ of $P_\mu$ has Lie algebra $F_\mu^1 \fg_\C$ and is the subgroup that 
acts as the identity on $\Gr_\mu^\cdot (V)$ in every representation $(V,\rho)$. 
The centralizer $Z(\mu)$ of $\mu$ in $G_\C$ is a Levi subgroup of $P_\mu$, isomorphic to $P_\mu/ R_u P_\mu$.
Thus the composite map
\[
\bar{\mu}: \GG_{m,\C} \rightarrow P_\mu \rightarrow P_\mu/ R_u P_\mu
\]
is a central homomorphism.
Then $\Filt (\mu)$ equals $\Filt (\mu')$ if and only if $P_\mu=P_{\mu'}$ and $\bar{\mu}=\bar{\mu'}$. 

Let $\check{X}$ denote the compact dual Hermitian symmetric space to $X$. 
As a set, it may be defined as the set of filtrations of 
$\Rep_\C (G_\C)$  that are $G(\C)$-conjugate to $\Filt(\mu_h)$. Equivalently, it may be described 
as the set of equivalence classes $[(P,\mu)]$ of pairs where 
$P$ is a parabolic in $G_\C$ and $\mu: \GG_{m, \C} \rightarrow P$ is a 
co-character such that $(P,\mu)$ is $G(\C)$-conjugate to $(P_{\mu_h},\mu_h)$ for 
some (and therefore every) $h\in X$. Here we say that $(P,\mu)$ is equivalent to 
$(P',\mu')$ if $P=P'$ and $\bar{\mu} = \bar{\mu'}$. Note that 
if $(P,\mu)$ is conjugate to $(P,\mu')$, then $\bar{\mu} = \bar{\mu'}$. 
Indeed, if $g^{-1} (P,\mu) g = (P,\mu')$, then $g\in N_{G_\C} (P)=P$. Write $g=\ell u $,
with $\ell \in Z(\mu)$ and $u\in R_u P$, we see that
\[
\mu'= g^{-1} \mu g = u^{-1} \mu u,
\]
so that $\bar{\mu'}=\bar{\mu}$ as claimed. Thus in a given 
conjugacy class of pairs $(P,\mu)$, the homomorphism $\bar{\mu}$
is determined entirely by $P$. Conversely, for any pair $(P,\mu)$ in the conjugacy 
class of $(P_{\mu_h},\mu_h)$, the parabolic $P$ must equal $P_\mu$ so that 
$\mu$ determines $P$. It follows from this discussion that the natural map
\[
G(\C) \times \check{X} \rightarrow \check{X}, \quad (g, [(P,\mu)]) \mapsto [g(P,\mu)g^{-1}]
\]
makes $\check{X}$ into a homogeneous space for $G(\C)$ and the choice of any basepoint 
$[(P,\mu)]$ gives a bijection
$G(\C)/P \simeq \check{X}$. Further, there is 
a unique way to make $\check{X}$ into a complex algebraic variety such that this map is an isomorphism of complex varieties for any choice of base point. Moreover, the map
\[ 
\xi: M_X \rightarrow \check{X}, \quad \mu \mapsto [(P_\mu, \mu)]
\]
is surjective and $\check{X}$ has the natural structure of 
a variety over $E(G,X)$ such that the map $\xi$ is $\Aut(\C/ E(G,X))$-equivariant. 
When we wish to emphasize the rational structure of $\check{X}$, 
we will write $\check{X}_\C$ instead of $\check{X}$.

There is a natural embedding (the Borel embedding)
\[
\beta: X \hookrightarrow \check{X}, \quad h \mapsto [(P_h, \mu_h)],
\]
where henceforth we write $P_h$ for $P_{\mu_h}$. 
Let $\check{\cV}$ be a $G_\C$-vector bundle on $\check{X}$. The action of $G(\C)$ on $\check{X}$ extends the $G(\R)$ action on $X$. Thus $\cV:= \check{\cV} |_{X}$
is a $G(\R)$-vector bundle on $X$. 
For an open compact subgroup $\cK$ of $G(\A_f)$, define 
\[
\cV_\cK  = G(\Q) \backslash \cV \times G(\A_f)/\cK,
\]
which we view as fibered over $\Sh_{\cK} (G,X)$. 
In order that this define a vector bundle on $\Sh_{\cK} (G,X)$, we need to assume that $\check{\cV}$ satisfies the following condition:
\begin{equation}
\label{eqn:avb-condition}
\text{The action of $G_\C$ on $\check{\cV}$ factors through $G^c_\C$.}
\end{equation}
Here $G^c = G/Z_s$, where $Z_s$ is the largest subtorus of the center $Z_G$ of $G$ that is split over $\R$ but that 
has no subtorus split over $\Q$. Assuming \eqref{eqn:avb-condition}, for sufficiently small $\cK$, 
$\cV_\cK$ is a vector bundle on $\Sh_{\cK}(G,X)$. If $\check{\cV}$ is defined over $E\supseteq E(G,X)$, then $\cV_\cK $ has a canonical model 
over $E$ as well.

\begin{rem}
The reader may keep in mind the following example which occurs in this paper. Let $G=\Res_{F/\Q} \GL_2$,
with $F$ a totally real field. Then $Z_G=\Res_{F/\Q} \GG_m$ and $Z_s= \ker(\N_{F/\Q} : Z_G \rar \GG_m).$
\end{rem}

We now recall the relation between sections of the bundle $\cV_\cK$ and 
automorphic forms on $G(\A)$. This requires the choice of a base point 
 $h\in X$. 
Let $K_h$ be the stabilizer in $G(\R)$ of $h$. Let $\fk_h$ denote the Lie algebra of $K_h$ and consider the decomposition of $\fg_\C$ with respect to the action of $\Ad \circ h$:
\[
\fg_\C = \fp_h^+ \oplus \fk_{h,\C} \oplus \fp_h^-.
\]
Here $\fp_h^+ =\fg_\C^{-1,1}$, $\fp_h^- =\fg_\C^{1,-1}$ and $\fk_{h,\C} = \fg_\C^{0,0}$ for the Hodge decomposition on $\fg_\C$ induced by
$\Ad \circ h$. 
Thus $\fp_h^\pm$ correspond to the holomorphic and antiholomorphic tangent spaces of $X$ at $h$. 
Then $P_h$ is the parabolic subgroup of $G(\C)$ with Lie algebra $\fk_{h,\C} \oplus \fp_h^-$. The
choice of $h$ gives identifications $X = G(\R)/K_h$,  $\check{X} = G(\C)/P_h$ and the Borel embedding is given by the natural map
\[
G(\R)/K_h \hookrightarrow G(\C)/ P_h.
\]

Let $\cV_h$ denote the fiber of $\cV$ at $h$; equivalently this is the fiber of the bundle
$\check{\cV}$ at $\beta(h) \in \check{X}$. 
This comes equipped with a natural action of $K_h$, denoted $\rho_{\cV_h}$. 
Let 
$\ve_h$ denote the map
\[
G(\A) \rar \Sh_\cK (G,X) = G(\Q) \backslash X \times G(\A_f) / \cK, \quad g =(g_\infty, g_f) \mapsto [(g_\infty (h), g_f)].
\]
Then there is a canonical isomorphism
\[
\ve_h^* (\cV_\cK) \simeq G(\A) \times \cV_h,
\]
via which sections of $\cV_\cK$ can be identified with suitable functions from $G(\A)$ into $\cV_h$. 
This gives a canonical injective map
\[
\Lift_h:   \Gamma \left(\Sh_{\cK} (G,X), \cV\right) \rar C^\infty (G(\Q)\backslash G(\A)/\cK, \cV_h)
\]
whose image is the subspace $A(G,\cK, \cV, h)$ consisting of 
$\varphi \in C^\infty (G(\Q)\backslash G(\A)/\cK, \cV_h) $ satisfying:
\begin{enumerate}
\item $\varphi(gk)= \rho_{\cV_h} (k)^{-1} \varphi (g)$, for $g\in G(\A)$ and $k \in K_h$;
\item $Y\cdot \varphi=0$ for all $Y\in \fp_h^-$;
\item $\varphi$ is slowly increasing, $K_h$-finite and $Z(\fg_\C)$-finite, where $Z(\fg_\C)$ is the 
center of the universal enveloping algebra of $\fg_\C$. 
\end{enumerate}

Let us make explicit the map $\Lift_h$. 
Fix some $\tau =\tau_h \in X$ and let $s$ be a section of 
$\cV_\cK$. 
For any $g_f \in G(\A_f)$, there is a canonical identification
\[
\cV_{\tau} \simeq \cV_{\cK, [\tau,g_f]}
\]
where $[\tau,g_f]$ denotes the class of $(\tau,g_f)$ in $\Sh_{\cK} (G,X)$. 
Let $g=(g_\infty,g_f) \in G(\A)=G(\R) \times G(\A_f)$.
The section $s$ gives 
an element $s([g_\infty \tau, g_f]) \in  \cV_{g_\infty \tau}$. However,
the element $g_\infty$ induces an isomorphism
\[
t_{g_\infty}: \cV_\tau \simeq \cV _{g_\infty \tau}.
\]
The map $\Lift_h (s): G(\A) \rightarrow \cV_\tau$ is then defined by sending 
\[
g \mapsto  t_{g_\infty}^{-1} s([g_\infty \tau, g_f]).
\]

\begin{rem}
The subgroup $P_h$ of $G_\C$ acts on the fiber $\check{\cV}_{\beta(h)}$ at the point 
$\beta(h)$. This gives an equivalence of categories
\[
\text{ $G_\C$-vector bundles on $\check{X}$} \longleftrightarrow \text{complex representations of $P_h$}.
\]
The functor in the opposite direction sends a representation $(V,\rho)$ of $P_h$ to the 
vector bundle 
\[
G_\C \times_\rho V = (G_\C \times V )/\{ (g p,v) \sim (g, \rho(p) v), \ p\in P_h \},
\]
which fibers over $G_\C/P_h$ in the obvious way. Sections of this vector bundle can 
be identified with functions 
\[
f: G(\C) \rightarrow V, \quad f(gp) = \rho(p)^{-1} f(g).
\]

\end{rem}

\begin{example} 
\label{eg:gl2}
This example will serve to normalize our conventions. Let $G=\GL_{2,\Q}$ and $X$ the $G(\R)$-conjugacy class containing
\[
h_0: \Ss \rightarrow G_\R, \quad a+bi \mapsto \begin{pmatrix} a & b \\ -b & a \end{pmatrix}.
\]
Identify $X$ with $\fh^\pm$, the union of the upper and lower half planes; $h_0$ is identified with the point $i$. Then $E(G,X)=\Q$ and $\check{X} \simeq \P^1_\Q = G/P$, where 
$P$
is the Borel subgroup (of upper triangular matrices) stabilizing $\infty$, for the standard action of $G$ on $\P^1$. 
We will fix the isomorphism $\check{X} \simeq \P^1_\Q$ such that the map 
\[
\fh^\pm = X \stackrel{\beta}{\hookrightarrow} \check{X}_\C
\]
is the identity map. 
For $k\equiv r \bmod 2$, let $\check{\cV}_{k,r}$ be the homogeneous $G_\C$-bundle on $\check{X}_\C$ corresponding to the character
\[
\chi_{k,r}: P_\C \rightarrow \C^\times,  \quad 
\begin{pmatrix} a & * \\ & d \end{pmatrix} \mapsto a^k \det(\cdot)^{\frac{r-k}{2}} \]
of $P_\C$. Note that abstractly $\check{\cV}_{k,r} \simeq \cO(-k)$ though the $G_\C$-action depends 
on $r$ as well. 
For any $h\in X$, we write $\rho_{k,r}$ for the corresponding representation of $K_h$.
The representation $\rho_{k,r}$ of $K_{h_0}= \R^\times \cdot \SO_2 (\R)$ is the character given by 
\[
z \cdot \kappa_\theta  \mapsto z^{-r} e^{-ik\theta}, \quad \kappa_\theta=\begin{pmatrix} \cos \theta & \sin \theta \\ 
- \sin \theta & \cos \theta \end{pmatrix}.
\]
For more general $h$, the character $\rho_{k,r}$ is given by composing the above character with 
the isomorphism $K_h \simeq K_{h_0}$ given by $x \mapsto \alpha^{-1} x \alpha$ for any $\alpha \in G(\R)$ such that $\alpha h_0 \alpha^{-1} = h$. 
The corresponding automorphic line bundle $\cV_{k,r,\cK}$ is defined over $\Q$ and is the usual bundle of modular forms of weight $k$ and level $\cK$. 
We can make this more explicit as follows. 

The {\it connected} hermitian space $\fh^+$ carries a natural family of (polarized) elliptic curves, 
the fiber over $\tau \in \fh^+$ being the elliptic curve $A_\tau = \C/( \Z \tau + \Z)$. Let 
$\uo$ be the sheaf of relative one-forms; it is a line bundle on $\fh^+$ and there is 
a canonical isomorphism $\beta^{*} \check{\cV}_{k,r}|_{\fh^+} \simeq \uo^k$. 
This gives a canonical trivialization $\Triv_h: \cV_{h} \simeq \C$ for all $h\in \fh^+$, 
namely the map sending $dz^{\otimes k}$ to $1$, where $z$ is the coordinate on $\C= \Lie(A_\tau)$. 
Thus any section $\varphi $ of $\cV_{k,r,\cK}$ on $\Sh_{\cK} (G,X)$ gives rise (via $\Triv_h \circ \Lift_h$) to a function 
\[
\varphi_h: \GL_2 (\A) \rightarrow \C, \quad h\in \fh^+,
\]
such that $\varphi_h ( g \kappa ) = \rho_{k,r} (\kappa)^{-1} \varphi_h (g)$ for all $\kappa \in K_h$. In 
particular, for $z \cdot \kappa_\theta \in K_{h_0}$, we have 
\[
\varphi_{h_0} (g \cdot z \cdot \kappa_\theta) = \varphi_{h_0} (g) \cdot z^r \cdot e^{i k \theta}.
\]
Finally, there is a unique modular form $f$ of weight $k$ on $\fh^+$ such that for all $h\in \fh^+$, we have 
\[
\varphi_h (g) = f(g_\infty (\tau_h)) j(g_\infty,\tau_h)^{-k} \det(g)^{\frac{r-k}{2}},
\]
where $g= g_\Q ( g_\cK g_\infty)$ with $g_\Q\in G(\Q)$, $g_\cK \in \cK$ and $g_\infty \in G (\R)^+$. (Here 
$G(\R)^+$ denotes the topological identity component of $G(\R)$.)

\end{example}

\begin{example}
Let $G=B^\times$, where $B$ is a non-split indefinite quaternion algebra over $\Q$. 
Then $E(G,X)=\Q$ and $\check{X}$ is a form of $\P^1$; in fact it is a Severi--Brauer variety 
associated to the class of $B$ in the Brauer group of $\Q$. The variety $\check{X}$ (over $\C$) carries the  line bundles
$\cO(k)$ but only for $k$ even do these descend to line bundles over $\Q$. 
Indeed, the canonical bundle on $\check{X}$ has degree $-2$, so $\cO(-2)$ descends. On the other hand,
$\cO(1)$ does not descend since if it did,  by Riemann--Roch it would admit a section whose zero locus is a rational point. 
Nevertheless, for any $\sigma \in \Aut(\C/\Q)$, the line bundle $\cL:=\cO(1)$ on $\check{X}_\C$ satisfies 
$\sigma^* \cL \simeq \cL$, so its {\it field of definition} is $\Q$. 
\end{example}

\subsubsection{Integral models}
\label{sec:int-models}
We assume in this section that the Shimura variety $(G,X)$ 
is of abelian type. 
Let $\cO$ denote the ring of integers of $E(G,X)$ and 
$\lambda \mid \ell$ a prime of $\cO$. 
We assume that we are given a reductive group $\cG_0$ over 
$\Z_{(\ell)}$ such that $\cG_{\Q} = G$. 
Let $\cG= \cG_{0,\Z_\ell}$ and $\cK_\ell = \cG (\Z_\ell)$. 
Then $\cK_\ell$ is a hyperspecial (maximal compact) subgroup of $G(\Q_\ell)$. 
Suppose that $\cK$ is an 
open compact subgroup of $G(\A_f)$ 
of the form $\cK_\ell \cdot \cK^\ell$, with  $\cK_\ell$   
 as above and $\cK^\ell $ a subgroup of  $G(\A_f^\ell)$, where $\A_f^\ell$ denotes 
the finite ideles whose component at $\ell$ is $1$. Then $\Sh_{\cK} (G,X)$ 
admits a natural integral model $\cS_{\cK,\lambda}(G,X)$ over $\cO_{(\lambda)}$. More precisely, 
if one fixes $\cK_\ell$ and allows $\cK^\ell$ to vary, then 
Kisin \cite{kisin} shows that 
the projective system $\varprojlim \Sh_{\cK_\ell \cK^\ell} (G,X)$ admits a
{\it canonical model} $\cS_{\cK_\ell,\lambda}(G,X)$ over $\cO_{(\lambda)}$, which is characterized by 
a certain extension property. 
We will also need integral models of automorphic vector bundles on 
$\Sh_{\cK} (G,X)$. In the abelian case, these will be constructed in the thesis of 
Lovering \cite{lovering}, and we now summarize the relevant results. 

 Recall that the compact dual $\check{X}$ is naturally defined 
over $E(G,X)$. In addition, $\check{X}$ has a natural model $\check{\fX}$ over 
$\cO_{(\lambda)}$ whose $A$-valued points for any $\cO_{(\lambda)}$-algebra $A$
are in bijection with equivalence classes of pairs $(P,\mu)$ consisting of a parabolic subgroup 
$P$ of $\cG_{0,A}$ and a cocharacter $\mu:\GG_{m,A} \rightarrow P$, where 
$(P,\mu)\sim (P',\mu')$ if $P=P'$ and $\bar{\mu}=\bar{\mu'}$. The data needed to 
define integral models of automorphic vector bundles consists of the following:
\begin{itemize}
\item A finite extension $L$ of $E(G,X)$ and a $G_L$-equivariant vector bundle $\check{\cV}$ on $\check{X}_L$.
The corresponding automorphic vector bundle $\cV_{\cK}$ on $\Sh_{\cK} (G,X)$ has 
a canonical model over $L$. 

\item A prime $\lambda$  of $\cO_L$;  we 
write $\lambda$ for the induced prime of $\cO$ as well. 

\item A $\cG_{0}$-equivariant vector bundle 
$\check{\cV}_\lambda$ on $\check{\fX}_{\cO_{L,(\lambda)}}$ 
which extends the $G_L$-equivariant vector bundle $\check{\cV}$ on $\check{X}_L$.

\end{itemize}
 
To this data, one can associate (by the results of \cite{lovering}) in a functorial way a vector bundle $\cV_{\cK,\lambda}$ 
over $\cS_{\cK,\lambda} (G,X) \otimes_{\cO_{(\lambda)}} \cO_{L,(\lambda)}$ which extends $\cV_{\cK}$. 
Likewise, if one fixes $\cK_\ell$ and varies $\cK^{\ell}$, one gets 
a vector bundle $\cV_{\cK_\ell,\lambda}$ 
over $\cS_{\cK_\ell,\lambda} (G,X) \otimes_{\cO_{(\lambda)}} \cO_{L,(\lambda)}$.
If $f:\check{\cV}^1_{\lambda} \rightarrow \check{\cV}^2_{\lambda}$ is
a map of $\cG_0$-equivariant vector bundles over $\check{\fX}_{\cO_{L,(\lambda)}}$, there are 
natural associated maps 
$f_\cK :  \cV_{\cK,\lambda}^1 \rightarrow \cV_{\cK,\lambda}^2$ and 
 $f_{\cK_\ell}:  \cV_{\cK_\ell,\lambda}^1 \rightarrow \cV_{\cK_\ell,\lambda}^2$.

\paragraph{{\it Models over $\cO_L[\frac{1}{N}]$}}
\label{par:patching}
Suppose now that 
we are given a reductive group $\cG_0$ over 
$\Z[\frac{1}{N}]$ such that $\cG_{0,\Q}=G$ and that
$\cK $ is of the form $\prod_\ell \cK_\ell$, where 
$\cK_\ell = \cG_0 (\Z_\ell)$ for all $\ell$ not dividing $N$, so that
$\cK_\ell$ is hyperspecial for such $\ell$. 
Then the integral models of 
$\Sh_\cK (G,X)$ for varying $\ell$ (not dividing $N$) patch together 
to give a canonical model 
$\cS_{\cK,\cO[\frac{1}{N}]} (G,X)$ over 
$ \cO [\frac{1}{N}]$. 

The compact dual 
$\check{X}$ has a natural model 
$\check{\fX}$ over $\cO[\frac{1}{N}]$ as well. If 
we are given moreover:
\begin{itemize}
\item A finite extension $L$ of $E(G,X)$ and a $G_L$-equivariant vector bundle $\check{\cV}$ on $\check{X}_L$.
\item  A $\cG_{0}$-equivariant vector bundle 
$\check{\cV}$ on $\check{\fX}_{\cO_{L}[\frac{1}{N}]}$ 
which extends the $G_L$-equivariant vector bundle $\check{\cV}$ on $\check{X}_L$.
\end{itemize}
Then the integral models $\cV_{\cK,\lambda}$ 
(as $\lambda$ varies over the primes of $\cO_L$ 
not dividing $N$) patch together to give 
an integral model $\cV_{\cK,\cO_L[\frac{1}{N}]}$ over 
$\cO_L[\frac{1}{N}]$.

\subsection{Automorphic vector bundles on quaternionic Shimura varieties}
\label{ssec:avs-on-qvs}

In this section, we 
review the connection between automorphic forms on the multiplicative group of a 
 quaternion algebra over a totally real field and sections of automorphic vector bundles on 
the corresponding Shimura variety.
We will also define canonical metrics on such bundles. 

\begin{rem}
Everything in this section goes through verbatim even in the 
case that the quaternion algebra $B$ is {\it totally definite}, even though 
this does not strictly speaking give a Shimura variety in the sense of 
\S \ref{ssec:shim-var}.
\end{rem}

Let $F$ be  
a totally real field and $B$ a quaternion algebra over $F$. 
 Let 
$G_B$ denote the $\Q$-algebraic group $\Res_{F/\Q} (B^\times)$.
Thus for any $\Q$-algebra $R$, the $R$-valued points of $G_B$
are given by 
\[
G_B(R) =(B\otimes_{\Q} R)^\times.
\]
Let $\Sigma_B$ denote the set of places of $F$ at which $B$ is ramified.

We fix for the moment some choice of isomorphisms
\begin{align}
\label{eqn:isom1}
B\otimes_{F,\sigma} \R & \simeq \M_2( \R),  & & \hspace{-20mm} \text{for }\sigma \in \Sigma_\infty \smallsetminus \Sigma_B; \\
\label{eqn:isom2}
B\otimes_{F,\sigma} \R & \simeq \H,  & & \hspace{-20mm} \text{for }\sigma \in \Sigma_{B,\infty},
\end{align}
where $\H$ is the subalgebra 
\[
\left\{ \begin{pmatrix} \alpha & \beta \\ -\bar{\beta} & \bar{\alpha} \end{pmatrix} : \ \ \alpha,\beta \in \C \right\}
\]
of $\M_2(\C)$. (Later we will fix these isomorphisms more carefully.)
The choice of isomorphisms above gives us identifications
\[
G_B (\R) \simeq \prod_{\sigma \in \Sigma_\infty \smallsetminus \Sigma_B} \GL_2 (\R) \times \prod_{\sigma \in \Sigma_{B,\infty}} \H^\times
\]
and 
\[
G_B (\C) \simeq \prod_{\sigma \in \Sigma_\infty} \GL_2 (\C) .
\]

Let $X_B$ be the $G_B(\R)$-conjugacy class of homomorphisms $\Ss \rightarrow G_{B,\R}$ containing
\[
h_0 : \Ss \rightarrow G_{B,\R}, \quad h_0:=\prod_\sigma h_{0,\sigma}, \quad h_{0,\sigma} (z) = 
\begin{cases}
 z, & \text{ if } \sigma \in \Sigma_\infty \smallsetminus \Sigma_B; \\
 1, & \text{ if }  \sigma \in \Sigma_{B,\infty},
\end{cases}
\]
where we identify $\C$ with a subring of $\M_2(\R)$ (see remark below.)
Denote by $\check{X}_B$ the corresponding compact dual hermitian symmetric space.
The choice of isomorphisms \eqref{eqn:isom1} and \eqref{eqn:isom2} above gives rise to an 
identification $\check{X}_B =(\P^1_\C)^{d_B}$ and $X_B= (\fh^\pm)^{d_B}$, with $d_B$ being the number of infinite places of $F$
where $B$ is split.  

\begin{rem}
(Choices) We embed $\C $ in $ \M_2(\R)$ by identifying
$a+bi$ with the matrix 
\[
\begin{pmatrix} a & b \\ -b & a \end{pmatrix}
\]
In addition, we identify the homomorphism
\[
 \Ss \rightarrow \GL_{2,\R}, \quad a+bi \mapsto a+bi = \begin{pmatrix} a & b \\ -b & a \end{pmatrix}.
\]
with the element $i\in \fh$. Note that this is {\it opposite} to the usual choice made by Shimura. 
Shimura would identify $i \in \fh$ with the map 
\[
 a+bi \mapsto  \begin{pmatrix} a & - b \\ b & a \end{pmatrix}.
\]

\end{rem}

\subsubsection{Hermitian forms}

For $\sigma \in \Sigma_\infty \smallsetminus \Sigma_B$, let $V_{\sigma,\R}$ denote the 
vector space $\R^2$ of column vectors viewed as a left $\M_2(\R)$-module. 
Let $h:\C^\times = \Ss (\R) \rightarrow (B\otimes_{F,\sigma} \R)^\times =\GL_2(\R) $ 
be any homomorphism that is $\GL_2(\R)$-conjugate to $h_{0,\sigma}$. Then 
we can write 
\begin{equation}
\label{eqn:hodgedec-0}
V_{\sigma,\C}= V_{\sigma,\R} \otimes_{\R} \C = V_{\sigma,h}^{-1,0} \oplus V_{\sigma,h}^{0,-1},
\end{equation}
where the decomposition on the right corresponds to the $\C$-subspaces on which 
$h(z) \otimes 1$ acts as $1\otimes z$ and $1\otimes \bar{z}$ respectively. 
The bilinear form 
\begin{equation}
\label{eqn:riemannform}
(x,y) \mapsto {}^t x \begin{pmatrix} 0 & -1 \\  1 & 0 \end{pmatrix} y
\end{equation}
on $V_{\sigma,\R}$ is almost $\GL_2(\R)$-invariant:
\[
(gx,gy) = \det(g) \cdot (x,y).
\]
Further, it satisfies the following conditions:
\begin{enumerate}
\item $(x,y)=-(y,x)$.
\item $(h(i) x, h(i) y) = (x,y)$. 
\item The form $ (x,h(i)y)$ is symmetric. (This follows formally from (i) and (ii).) Further, it is positive definite if $h$ is $\GL_2(\R)^+$-conjugate to $h_0$. (Otherwise it is negative definite.)
\end{enumerate} 

\begin{rem}
Let $\tau$ be the unique point on the complex upper half plane fixed by $K_h$. 
The bilinear form above equals $\frac{1}{2\pi i} \lambda_\tau$ where $\lambda_\tau$ is the Weil pairing on $H_1(E_\tau)$
given in the ordered basis $\{ \tau, 1 \}$. 
\end{rem}

 The composite map 
\[
V_{\sigma,\R} \rightarrow V_\sigma \otimes_{\R} \C \rightarrow V_{\sigma,h}^{-1,0} 
\]
is an $\R$-linear isomorphism; via this isomorphism one gets a skew-smmetric 
bilinear form on $V_{\sigma,h}^{-1,0}$, which is the {\it negative } of the  imaginary part of a (necessarily unique) hermitian form
$H_h$ on $V_{\sigma,h}^{-1,0}$ 
defined by identifying $V_{\sigma,h}^{-1,0}$ with $V_{\sigma,\R}$ and setting 
\[
H_h(x,y) = (x,h(i)y) - i(x,y) = (x,iy) - i (x,y).
\]

\smallskip

\begin{rem}
The form $H_h$ is linear in the first variable and conjugate linear in the second variable. 
If we denote the form \eqref{eqn:riemannform} above by $\tilde{B}$, then $H_h$ agrees with the positive definite form $2\cdot \tilde{B}_{h(i)}$ 
of Appendix \ref{sec:hodge}, where:
\begin{equation}
\label{eqn:Bh}
\tilde{B}_{h(i)} ( v,w) = \tilde{B}_\C (v, h(i) \bar{w}).
\end{equation}
\end{rem}

\smallskip

If $h$ is $\GL_2(\R)^+$-conjugate to $h_0$, the form $H_h$ is positive definite on account of condition (iii) above. 
Note that 
\begin{equation}
\label{eqn:herm-metric1}
H_h(x,x) = (x,h(i)x). 
\end{equation}
The subgroup $K_h$ preserves the decomposition 
\eqref{eqn:hodgedec-0} and the form $H_h$ is $K_h$-invariant up to a scalar. In fact,
for $\kappa \in K_h$, we have 
\[
H_h (\kappa x, \kappa y) = \det(\kappa) H_h (x,y).
\]
Moreover, the natural action of $\GL_2(\R)$ on $V_{\sigma,\C}$ takes 
$V_{\sigma,h}^{-1,0}$ isomorphically onto $V_{\sigma,g\cdot h}^{-1,0}$ (recall $g\cdot h = ghg^{-1}$) and we have 
\[
H_{g\cdot h} (gx,gy) = (gx, (gh(i)g^{-1})gy) = \det(g) (x,h(i)y)= \det(g) H_h (x,y). 
\] 
We note also that $\det(V_{\sigma,\C})$ carries a natural 
bilinear form induced from the $\C$-linear extension of $(\cdot,\cdot)$. 
We equip $\det(V_{\sigma,\C})$ with the positive definite Hermitian form 
\begin{equation}
\label{eqn:herm-metric1'}
H_{\det}(x,y) = (x,\bar{y}),
\end{equation}
where the complex conjugation is with respect to the natural real structure coming from $\det(V_{\sigma,\R})$.
This hermitian form satisfies
\[
H_{\det} (gx,gy) = \det(g)^2 \cdot H_{\det}(x,y)
\]
for all $g\in \GL_2(\R)$.

For $\sigma \in \Sigma_{B,\infty}$, let $V_{\sigma,\C}$ 
denote the $\C$-vector space $\C^2$ of column vectors viewed as 
a left $\M_2(\C)$-module. The form 
\[
(x,y) \mapsto {}^t x  \begin{pmatrix} 0 & -1 \\  1 & 0 \end{pmatrix} y 
\]
is almost $\GL_2(\C)$-invariant:
\[
(g x, g y) =\det(g) \cdot (x,y).
\]
Let $L$ be the $\R$-linear operator
\[
L(x):=\begin{pmatrix} 0 & 1 \\  -1 & 0 \end{pmatrix} \bar{x}.
\]
The operator $L$ is the analog in this case of the operator $w\mapsto h(i)\bar{w}$ 
in \eqref{eqn:Bh} and the operator $x \mapsto h(i)x$ in \eqref{eqn:herm-metric1} above.
Define a hermitian form $H$ on $V_{\sigma,\C}$ by
\begin{equation}
\label{eqn:herm-metric2}
H(x,y) = (x, Ly) = {}^t x \bar{y}.
\end{equation}
Note that $L$ commutes with the left action of $\H$, hence the form $H$ is $\H^\times$-invariant up to a scalar, 
which is also obvious from the formula above. More precisely, for $g\in \H^\times$, we have 
\[
H(gx,gy) = \nu(g) H(x,y)
\]
where $\nu$ is the reduced norm. 

Let $\rho_{\sigma,k,r}$ denote the representation\[
V_{\sigma,k,r} = \Sym^{k} (V_{\sigma,\C}) \otimes {\det(V_{\sigma,\C})}^{\frac{r-k}{2}}
\]
of $\GL_2(\C)$. Note that the central character of
$\rho_{\sigma,k,r}$ is $z\mapsto z^r$. 

\subsubsection{Hermitian metrics on automorphic vector bundles}
Let $(\uk,r)$ be a multi-index of integers with $\uk=(k_\sigma)_{\sigma \in \Sigma_\infty}$ such that 
\[
k_{\sigma} \equiv r \pmod 2 \quad \text{ for all }\sigma \in \Sigma_\infty.
\]
We assume that $k_{\sigma} \ge 1$ if $B$ is split at $\sigma$ and 
that $k_{\sigma} \ge 0$ if $B$ is ramified at $\sigma$. 

Let $\rho_{\uk,r}= \otimes_\sigma \rho_{\sigma,k_\sigma,r} $ be the representation of $G_B(\C)$ on 
\[
V_{\uk,r} = \bigotimes_\sigma V_{\sigma,k_\sigma,r}.
\]
This gives rise to 
a $G_B(\C)$-homogeneous vector bundle $\check{\cV}_{\rho_{\uk,r}}$ on 
$\check{X}_B$:
\[
\check{\cV}_{\rho_{\uk,r}}= \check{X}_B \times V_{\uk,r},
\]
where the $G_B(\C)$ action is:
\[
g \cdot (x,v) = (gx, gv).
\]
 By restriction one gets a 
$G_B(\R)$-homogeneous vector bundle $\cV_{\rho_{\uk,r}}$ on $X_B$.  
Further, the latter admits a 
unique $G_B(\R)$-equivariant sub-bundle $\cV_{\uk,r}$
corresponding to the $K_h$-subrepresentation $\rho_{\uk,r,h}$ on 
\[
\cV_{\uk,r,h} = \bigotimes_{\sigma \in \Sigma_\infty \smallsetminus \Sigma_B} \left( (V_{\sigma,h}^{-1,0})^{\otimes k_\sigma} \otimes \det(V_{\sigma,\C})^{\otimes \frac{r-k_\sigma}{2}} \right) \bigotimes_{\sigma \in \Sigma_{B,\infty}} V_{\sigma,k_\sigma,r}
\]
Let $X^+_B$ denote the connected component of $X_B$ containing $h_0$. 
Note that for $h\in X^+_B$, the $K_h$-representation above carries a natural {\it positive definite} hermitian metric 
$\langle \cdot , \cdot \rangle_h$
obtained from the hermitian metrics in \eqref{eqn:herm-metric1}, \eqref{eqn:herm-metric1'} and \eqref{eqn:herm-metric2} above. 
This gives a metric on $\cV_{\uk,r}$ that is almost $G_B(\R)$-equivariant; in fact, one has 
\[
\langle gx,gy \rangle_{g\cdot h} = \nu(g)^r \langle x,y \rangle_h
\]
for $g\in G_B(\R)$ and $x,y \in \cV_{\uk,r,h}$. 
Now consider the vector bundle ${\cV_{\uk,r}}_{|X^+_B} \times G_B(\A_f) $ on $X^+_B \times G_B(\A_f)$. 
We equip this with the hermitian metric that assigns to the fiber 
$\cV_{\uk,r,h} \times \{g_f\} $ over $(h,g_f) \in X^+_B \times G_B(\A_f)$ 
the metric defined above on $\cV_{\uk,r,h}$ multiplied by the factor $\|\nu(g_f)\|^{r}$. (Here 
$\nu(g_f)\in \A^\times_F$ and $\|\cdot \|$ denotes the idelic norm.)
Recall that 
\[
\Sh_\cK (G_B,X_B) = G_B(\Q) \backslash X_B \times G_B(\A_f) / \cK = G_B(\Q)^+ \backslash X^+_B \times G_B(\A_f) / \cK 
\]
and 
\[
\cV_{\uk,r,\cK} = G_B(\Q) \backslash \cV_{\uk,r} \times G_B(\A_f) / \cK = G_B(\Q)^+ \backslash {\cV_{\uk,r}}_{|X^+_B} \times G_B(\A_f) / \cK,
\]
where $G_B(\Q)^+ = G_B(\R)^+ \cap G_B(\Q)$.

\begin{prop}
\label{prop:hermitian-metric}
The metric on ${\cV_{\uk,r}}_{|X^+_B} \times G_B(\A_f)$ above descends to a (positive definite hermitian) metric 
on the vector bundle $\cV_{\uk,r,\cK}$ over $\Sh_{\cK}(G_B,X_B)$.
\end{prop}

\begin{proof}
Let $(h,g_f)$ and $(h',g'_f)$ be two elements of $X^+_B \times G_B(\A_f)$ 
whose classes in $\Sh_{\cK}(G_B,X_B)$ are equal.
Then there exist elements $\gamma \in G_B(\Q)^+$ and $\kappa \in \cK$ such that 
\[
(h',g'_f) = \gamma (h,g_f) \kappa = (\gamma\cdot h, \gamma_f g_f \kappa).
\]
Here $\gamma_f$ is $\gamma$ viewed as an element of $G_B(\A_f)$.
We need to check that the bijection
\[
\cV_{\uk,r,h} \times \{ g_f \} \rightarrow \cV_{\uk,r,h'} \times \{ g'_f \} =  \cV_{\uk,r,\gamma\cdot h} \times \{ \gamma_f g_f \kappa\} 
\]
given by $(v,g_f) \mapsto (\gamma  v, \gamma g_f \kappa)$ is metric preserving. 
But
\begin{align*}
\langle \gamma v_1, \gamma v_2 \rangle_{\gamma \cdot h} \cdot \| \nu(\gamma_f g_f \kappa) \|^{r} &= 
\prod_{\sigma \in \Sigma_\infty} \sigma(\nu(\gamma))^r \cdot \langle v_1,v_2 \rangle_h \cdot \| \nu(\gamma)_f \|^r \| \nu (g_f) \|^r \\
&=\langle v_1,v_2 \rangle_h \cdot \| \nu(g_f) \|^r,
\end{align*}
using the product formula and the fact that $\| \nu (\kappa) \| =1$. 
\end{proof}

We will need to work with the 
dual vector bundle $\cV_{\uk,r}^\vee$. 
This is motivated by observing that 
in the case of $\GL_2(\Q)$, the bundle $\cV_{\rho_{\uk,r}}$ corresponds to the relative homology 
of the universal elliptic curve and the sub-bundle $\cV_{\uk,r}$ corresponds to its relative 
Lie algebra. The line bundle of usual modular forms corresponds to the bundle of relative differentials, which is why
we need to replace $\cV_{\uk,r}$ by its dual.
We begin by making the following completely elementary remark, 
which we nevertheless state carefully to avoid any confusion.

\smallskip

\begin{rem}
If $\rho$ is a representation of a group $G$ on a finite-dimensional complex vector space $V$, then $\rho^\vee$ is defined by 
\[
\rho^\vee(g) (L) = L\circ \rho(g^{-1})
\]
for $L\in V^\vee =\Hom (V,\C)$ and $g\in G$. Thus for the tautological pairing
\[
(\cdot,\cdot): V^\vee \times V \rightarrow \C, \quad (L,v)= L(v),
\]
we have 
\[
( \rho^\vee (g^{-1}) L , v) = (L, \rho(g) v).
\]
Suppose $V$ is equipped with a non-degenerate hermitian pairing $\langle \cdot, \cdot \rangle$ that is linear 
in the first variable and conjugate linear in the second variable, and such that
\[
\langle g v, g w \rangle = \chi(g) \langle v,w \rangle
\]
for some character $\chi: G \rightarrow \C^\times$. 
Since $\langle \cdot, \cdot \rangle$ is non-degenerate, it induces a conjugate linear isomorphism
\[
 V \simeq V^\vee, \quad w \mapsto L_w, \quad L_w (v) = \langle v,w \rangle.
 \]
 Composing the inverse of this isomorphism with the canonical isomorphism $V \simeq V^{\vee \vee}$
 gives a conjugate linear isomorphism $V^\vee \simeq (V^\vee)^\vee$, which one may view 
 as a hermitian form on $V^\vee$. Explicitly this isomorphism sends $L_w$ to the linear functional
 $\eval_w \in (V^{\vee})^\vee$, so that 
 for any $L\in V^\vee$, we have
 \[
 \langle L,L_w \rangle = L(w).
 \]
  Note that
 \[
 g L_w ( v) = L_w (g^{-1} v) = \langle g^{-1} v ,w \rangle = \chi(g)^{-1} \langle v, gw \rangle = \chi(g)^{-1} L_{gw} (v),
 \]
 so that $g L_w = \chi(g)^{-1} L_{gw}$. 
 For any $L\in V^\vee$, we have 
 \begin{align*}
 \langle g L, g L_w \rangle &= \langle g L, \chi(g)^{-1} L_w \rangle = \overline{\chi(g)}^{-1} \langle gL, L_{gw} \rangle  \\ &=\overline{\chi(g)}^{-1}(gL )(gw) = \overline{\chi(g)}^{-1} L(w) \\ &=\overline{\chi(g)}^{-1} \langle L, L_w \rangle,
 \end{align*}
so for any $L_1, L_2 \in V^\vee$, we have $\langle g L_1 , g L_2 \rangle =\overline{\chi(g)}^{-1} \langle L_1, L_2 \rangle$. 
\end{rem}

\smallskip

From the remark above, it is clear that for $x,y \in \cV_{\uk,r,h}^\vee$ and $g\in G_B(\R)$, we have 
\[
\langle gx,gy \rangle = \nu(g)^{-r} \langle x,y \rangle.
\] 
Thus we take on ${\cV_{\uk,r}^\vee}_{|X^+_B} \times G_B(\A_f)$ the metric which on $\cV_{\uk,r,h}^\vee \times \{ g_f \}$ is 
$\| \nu(g_f) \|^{-r}$ times the induced metric on  $\cV_{\uk,r,h}^\vee$. 
This descends to a positive definite hermitian metric $\llangle \cdot, \cdot \rrangle$ on $\cV_{\uk,r,\cK}^\vee$. (See Prop. \ref{prop:hermitian-metric} above.)

\begin{defn}
\label{defn:pet-norm}
A holomorphic automorphic form of weight $(\uk,r)$ and level $\cK$ on $G_B$ is 
a {\it holomorphic} section $s$ of the bundle $\cV_{\uk,r,\cK}^\vee$ on $\Sh_{\cK}(G_B,X_B)$. 
Let $\cKt \supseteq \cK$ be any open compact subgroup of 
$G_B (\A_f)$ such that $\llangle s(x),s (x) \rrangle$ descends to 
a function on $\Sh_{\cKt} (G_B,X_B)$. Then 
the Petersson norm of the section $s$ (normalized with respect to $\cKt$) is defined to be the integral
\[
\llangle s,s \rrangle_{\cKt} := \int_{\Sh_{\cKt}(G_B,X_B)} \llangle s(x), s(x) \rrangle \,d\mu_x
\]
where $d\mu_x$ is the measure on $\Sh_{\cKt}(G_B, X_B)$ defined in Sec. \ref{sssec:measures-B}.  
\end{defn}

\begin{rem}
Defn. \ref{defn:pet-norm}
above has the advantage that it does not depend on any choice of base point. 
In practice though, one usually needs to pick a base point to 
make any computation at all, and so we shall now discuss the translation
between these two points of view.
\end{rem}

\smallskip

Pick a base point $h \in X_B^+$. 
Via the isomorphism $\Lift_h$, the space of holomorphic automorphic forms $s$ as above is 
identified 
with the space of functions $\cA (G_B, \cK, \cV_{\uk,r}^\vee, h)$. An element 
\[
F: G_B(\Q) \backslash G_B(\A) / \cK \rightarrow \cV_{\uk,r,h}^\vee
\]
in $\cA (G_B, \cK, \cV_{\uk,r}^\vee, h)$ satisfies in particular the condition
\begin{equation}
\label{eqn:Fgkappa}
F(g \kappa_h) = \rho_{\uk,r,h}^\vee (\kappa_h)^{-1} F(g), \quad \text{ for all } \kappa_h\in K_h.
\end{equation}

Henceforth we will fix a character $\xi$ of $F^\times \backslash \A^\times_F$ 
which satisfies
\[
\xi (z \cdot z_\infty) = \N (z_\infty)^r \cdot \xi (z)
\]
for $z\in \A^\times_F$, $z_\infty \in \A^\times_{F,\infty}$,
and 
assume
that the section $s$ satisfies the following invariance under the 
center $Z_{G_B} (\A_f) = \A_{F,f}^\times$:
\begin{equation}
\label{eqn:s-invariant-by-center}
s (x \cdot \alpha ) = \xi (\alpha) \cdot s(x). 
\end{equation}
This enables us to take $\cKt$ containing the maximal open compact subgroup of $Z_{G_B}(\A_f)$, and 
implies that the corresponding function $F$ above satisfies the following 
invariance property: for $\alpha \in \A_F^\times  = Z_{G_B} (\A)$, we have
\[
F (g \cdot \alpha) = \xi (\alpha) \cdot F(g)
\]
and 
\[
\langle F(g\cdot \alpha), F(g\cdot \alpha) \rangle = \| \alpha \|^{2r} \cdot \langle F(g), F(g) \rangle.
\]

\begin{prop}
\label{prop:s-to-F}
Suppose $\Lift_h(s)=F$. 
Let $\cK_0$ denote any maximal compact subgroup of 
$G_B(\A_f)$ containing $\cKt$.
Then 
\[
\llangle s,s \rrangle_{\cKt} = 2^{|\Sigma_\infty \smallsetminus \Sigma_B|} \cdot h_F \cdot [\cK_0:\cKt] \cdot \langle F, F \rangle_h,
\]
where 
\[
\langle F,F \rangle_h = \int_{[G_B]} \langle F(g), F(g)\rangle_h \cdot \| \nu(g) \|^{-r} \,dg.
\]
\end{prop}
Here and henceforth we write $[G_B]$ for $G_B(\Q) Z_{G_B}(\A) \backslash G_B(\A)$. 
Also, $dg$ denotes the {\it standard measure} on $[G_B]$ which is defined in 
\S  \ref{sssec:measures-B}.

\begin{proof} Recall that if $g=(g_\infty,g_f)$, we have 
\[
F(g) = g_\infty^{-1} s[(g_\infty \cdot h, g_f)],
\]
where we view $s[(g_\infty \cdot h, g_f)]$ as an element in $\cV_{\uk,r,g_\infty\cdot h}^\vee$. 
Now
\begin{align*}
\langle F(g), F(g) \rangle_h &= \nu(g_\infty) ^r \langle s[(g_\infty \cdot h, g_f)], s[(g_\infty \cdot h, g_f)] \rangle_{g_\infty \cdot h}\\
&=  \nu(g_\infty) ^r \| \nu(g_f) \|^r \cdot  \llangle s[(g_\infty \cdot h, g_f)], s[(g_\infty \cdot h, g_f)] \rrangle\\
&= \| \nu(g) \|^r  \llangle s[(g_\infty \cdot h, g_f)], s[(g_\infty \cdot h, g_f)] \rrangle.
\end{align*}
The proposition follows from this and the comparison of measures in Lemma \ref{lem:meas-shim-st}. 
\end{proof}

Next, we simplify further to scalar valued forms. 
For $\kappa = (z_\sigma e^{i\theta_\sigma})_{\sigma \in \Sigma_\infty} \in (\C^\times)^d$, let $\kappa_h$ be the 
element of $K_h\subset G_B(\R)$ defined by:
\[
\kappa_{h,\sigma} = \begin{cases}  h_\sigma(z_\sigma e^{i \theta_\sigma}), & \text{if } \sigma \in \Sigma_\infty 
\smallsetminus \Sigma_B; \\
z_\sigma e^{i \theta_\sigma}, & \text{if } \sigma \in \Sigma_{B,\infty},
\end{cases}
\]
where for $\sigma\in \Sigma_{B,\infty}$, we view $z_\sigma e^{i \theta_\sigma}$ as an element in $\C^\times \subset \H^\times \simeq (B\otimes_{F,\sigma} \R)^\times$ via \eqref{eqn:isom2}.
The equation \eqref{eqn:Fgkappa} can be rewritten as 
\[
F(g \kappa_h) = \prod_{\sigma\in \Sigma_\infty} z_\sigma^r \cdot  \prod_{\sigma \in \Sigma_\infty \smallsetminus \Sigma_B} e^{i k_\sigma \theta_\sigma} \cdot 
\bigotimes_{\sigma \in \Sigma_{B,\infty}} \rho_{\sigma,k_\sigma,r}^\vee (e^{-i\theta_\sigma}) F(g).
\]
For $\sigma\in \Sigma_\infty \smallsetminus \Sigma_B$, let $v_{\sigma,k_\sigma}$ be any nonzero vector in 
the one-dimensional $\C$-vector space
\[
(V_{\sigma,h}^{-1,0})^{\otimes{k_\sigma}} \otimes \det(V_{\sigma,\C})  ^{\otimes \frac{r-k_\sigma}{2}},
\]
so that 
\begin{equation}
\label{eqn:choice-of-vksigma}
\rho_{\sigma,k_\sigma,r} (\kappa_{h,\sigma}) \cdot v_{\sigma,k_\sigma} = z_\sigma^r e^{ik_\sigma \theta_\sigma}  \cdot v_{\sigma,k_\sigma}.
\end{equation}
For $\sigma \in \Sigma_{B,\infty}$, let 
$v_{\sigma,k_\sigma} \in V_{\sigma, k_\sigma,r}$ be any nonzero vector such that
the condition \eqref{eqn:choice-of-vksigma} is satisfied for all $\kappa\in (\C^\times)^d$. Such 
a vector is well-defined up to scaling.

Set $v_{\uk} = \otimes_{\sigma \in \Sigma_{\infty}}  v_{\sigma, k_\sigma} \in \cV_{\uk,r,h}$.
Define
\[
\phi_{F} (g) = (F(g), v_{\uk}).
\]
Then $\phi_F (g)$ satisfies 
\begin{equation}
\label{eqn:kappatr}
\phi_F (g \kappa_h) = \prod_{\sigma \in \Sigma_\infty} z_\sigma^r  e^{i k_\sigma \theta_\sigma} \cdot \phi_F (g)
\end{equation}
and 
\begin{equation}
\label{eqn:alphatr}
\phi_F (\alpha g) = \xi(\alpha) \phi_F(g), \quad \text{for } \alpha \in Z_{G_B}(\A) = \A_F^\times.
\end{equation}

\begin{prop}
The map $F \mapsto \phi_F$ is injective.
\end{prop}

\begin{proof} This follows immediately from the fact that $\cV_{\uk,r,h}$ is 
irreducible as a module over $G=\prod_{\sigma \in \Sigma_{B,\infty}} (B\otimes_{F,\sigma} \R)^\times$. 
Indeed, given any $w \in \cV_{\uk,r,h}$, there exist elements $\kappa_i \in G$ and $\alpha_i \in \C$ such that 
\[
w= \sum_{i} \alpha_i \rho(\kappa_i) v_{\uk},
\]
where $\rho$ denotes the natural action of $G$ on $\cV_{\uk,r,h}$. 
Then 
\begin{align*}
(F(g), w) &= \sum_i \alpha_i (F(g), \rho(\kappa_i) v_{\uk}) =\sum_i \alpha_i (\rho^\vee (\kappa_i)^{-1} F(g), v_{\uk}) \\ &= \sum_i \alpha_i (F(g\kappa_i), v_{\uk})  = \sum_i \alpha_i \phi_F (g\kappa_i). 
\end{align*}
Thus if $\phi_F$ is identically zero, then so is $F$. \end{proof}

We will now compare $\langle F,F \rangle $ to $\langle \phi_F, \phi_F \rangle$,
where 
\[
\langle \phi_F, \phi_F \rangle = \int_{[G_B]} \phi_F (g) \overline{\phi_F(g)} \cdot \| \nu(g) \|^{-r} \,dg.
\]

We use the following well known lemma.
\begin{lem} 
\label{lem:schur-orth}
Let $K$ be a compact Lie group and $V$ a (finite dimensional) irreducible complex representation of $K$. 
Let $\langle \cdot,\cdot \rangle$ be a nonzero $K$-invariant hermitian form on $V$ 
(such a form is unique up to scalar multiples)
and denote also by
$\langle \cdot,\cdot \rangle$ the induced hermitian form on $V^\vee$. Then for all $v_1,v_2 \in V$ and $L_1, L_2 \in V^\vee$, we have 
\[
\int_{K} ( \rho^\vee(k) L_1, v_1 ) \overline{( \rho^\vee (k) L_2, v_2 )} \,dk
 = \frac{1}{\dim(V)} \cdot \langle v_1,v_2 \rangle \langle L_1,L_2 \rangle,
\]
where $dk$ is Haar measure normalized to have total volume $1$. 
\end{lem}

\begin{rem}
It is immediate to check that if the form $\langle \cdot, \cdot \rangle$ on $V$ is scaled by $\alpha\in \C^\times$, then the form $\langle \cdot,\cdot\rangle$ on $V^\vee$ is scaled by $\bar{\alpha}^{-1}$, so the right hand side is independent of the 
choice of $\langle \cdot ,\cdot \rangle$. 
\end{rem}

\begin{prop}
\label{prop:F-to-phi}
\[
\langle F,F \rangle_h = \frac{\rank  \cV_{\uk,r}}{\langle v_{\uk},v_{\uk} \rangle_h } \cdot \langle \phi_F,\phi_F \rangle.
\]
\end{prop}

\begin{proof}
Let $K_h^0$ denote the maximal compact subgroup of $K_h$. 
Since $\cV_{\uk,r,h}$ is an irreducible representation of $K_h^0$, using Lemma \ref{lem:schur-orth} we get
\begin{align*}
\langle \phi_F, \phi_F \rangle &= \int_{[G_B]} \phi_F (g) \overline{\phi_F(g)} \cdot \| \nu(g) \|^{-r} \,dg \\
&= \int_{K_h^0} \int_{[G_B]} \phi_F (g) \overline{\phi_F(g)} \cdot \| \nu(g) \|^{-r} \,dg \,d\kappa \\
&= \int_{K_h^0} \int_{[G_B]} \phi_F (g \kappa^{-1}) \overline{\phi_F(g\kappa^{-1})} \cdot \| \nu(g\kappa^{-1}) \|^{-r} \,dg \,d\kappa \\
&= \int_{K_h^0} \int_{[G_B]} (F (g \kappa^{-1}),v_{\uk}) \overline{(F(g\kappa^{-1}),v_{\uk})} \cdot \| \nu(g) \|^{-r} \,dg \,d\kappa \\
&= \int_{[G_B]}   \int_{K_h^0} (\rho^\vee(\kappa)F(g),v_{\uk}) \overline{(\rho^\vee(\kappa)F(g),v_{\uk})} \cdot \| \nu(g) \|^{-r}  \,d\kappa \,dg \\
&= \frac{1 }{\rank  \cV_{\uk,r}} \langle v_{\uk},v_{\uk} \rangle_h \int_{[G_B]} \langle F(g),F(g) \rangle_h \| \nu(g) \|^{-r} \,dg.
\end{align*} 
\end{proof}

\subsection{Rational and integral structures}
\label{ssec:rational-integral-structures}

Let $\Pi=\otimes_v \Pi_v$ be 
an irreducible cuspidal automorphic representation of 
$\GL_2 (\A_F )$ 
corresponding to a Hilbert modular form 
of weight $(\uk,r)$, character $\xi_\Pi$ and conductor $\fN=\fN_\s \cdot \fN_\ps$, 
as in the introduction. 
Thus the character $\xi_\Pi$ satisfies
\[
\xi_\Pi (z \cdot z_\infty) = \N (z_\infty)^r \cdot \xi_\Pi (z)
\]
for $z\in \A^\times_F$ and $z_\infty \in \A^\times_{F,\infty}$. 
We also let $\pi=\otimes_v \pi_v $ denote the corresponding 
{\it unitary } representation:
\[
\pi :=  \Pi \otimes \| \det(\cdot) \|^{-r/2}.
\]

Recall that $\Sigma_\Pi$ denotes the set of all places $v$ of $F$
at which $\Pi_v$ is discrete series. Thus $\Sigma_\Pi$ contains 
$\Sigma_\infty$ but will typically be larger. Let 
$B$ be any quaternion algebra over $F$ such that $\Sigma_B \subseteq \Sigma_\Pi$, where
$\Sigma_B$ denotes the set of places 
$v$ of $F$ where $B$ is ramified. By the Jacquet--Langlands correspondence, there exists (up to isomorphism) a unique 
irreducible (cuspidal) automorphic representation $\Pi_B\simeq \otimes_{v} \Pi_{B,v}$ of $G_B (\A)$ 
such that $\Pi_{B,v} \simeq \Pi_v$ for all $v \not \in \Sigma_B$. 
Let $\uk_B=(k_{B,\sigma})_{\sigma \in \Sigma_\infty}$ be defined by:
\begin{equation}
\label{eqn:kB-defn}
k_{B,\sigma} = \begin{cases} k_\sigma, \quad & \text{if $B$ is split at $\sigma$,} \\
k_\sigma-2, \quad & \text{if $B$ is ramified at $\sigma$.}
\end{cases}
\end{equation}
Then $\Pi_B$ has {\it weight} $(\uk_B,r)$ at infinity.

Choose a maximal order $\cO_B$ in $B$. 
Recall that 
we have assumed that the conductor $\fN$ of $\Pi$ satisfies
\[
\fN= \fN_{\s} \cdot \fN_{\ps}
\]
where $\fN_\s$ is divisible exactly by those primes at which $\Pi_v$ is discrete series 
and $\fN_{\ps}$ is divisible exactly by those primes 
at which $\Pi_v$ is ramified principal series.  
Let $\fd_B$ be the (finite part of the) discriminant of $B$, so that $\fd_B$ divides $\fN_\s$. Then there is a
unique integral ideal $\fN_B$ in $\cO_F$ such that 
\[
\fN = \fN_B \cdot \fd_B,
\]
 and we may choose and fix an Eichler order $\cO_B(\fN_B)$ in 
$\cO_B$ of level $\fN_B$. 
We will also fix an {\it orientation} of this order at the places dividing $\fN_{\ps}$. By this,
we mean 
a ring homomorphism 
\[
o: \cO_B(\fN_B) \rightarrow \cO_F/ \fN_{\ps}.
\] 
This choice determines an open compact subgroup $\cK =\prod \cK_\ell$ of
$G_B (\A_f)$, namely $\cK_\ell = \prod_{v\mid \ell} \cK_v$ where for any finite place 
$v$ of $F$, we have
\[
\cK_v = \ker \left[ o_v : (\cO_B (\fN_B) \otimes_{\cO_F} \cO_{F,v})^\times \rightarrow (\cO_{F,v}/\fN_{\ps} \cO_{F,v})^\times \right].
\]
Here $o_v$ is the natural map induced by the orientation $o$. 
For all rational primes $\ell$ such that $(\ell,N(\Pi))=1$, the subgroup $\cK_\ell$ is a hyperspecial maximal compact subgroup of $G_B (\Q_\ell)$.

Now, we will assume that $B$ is not {\it totally definite}, relegating the 
case of totally definite $B$ to Remark \ref{rem:tot-def-rat-int} at the end of this section. 
Let $\ell$ be such that $(\ell,N(\Pi))=1$. 
Then for each prime $\lambda$ of $E(G_B,X_B)$ dividing such an $\ell$, one has (see \S \ref{sec:int-models}) an associated canonical integral model $\cS_{\cK,\lambda}= \cS_{\cK,\lambda}(G_B,X_B)$ of $\Sh_{\cK} (G_B, X_B)$ defined over $\cO_{E(G_B,X_B),(\lambda)}$.

We will now fix more carefully the isomorphism
\begin{equation}
\label{eqn:isomBinfty}
\phi_B: B \otimes \R \simeq \prod_{\sigma \in \Sigma_\infty \smallsetminus \Sigma_B} \M_2 (\R) \times \prod_{\sigma \in \Sigma_{B,\infty}} \H.
\end{equation}
Note that the vector bundles previously denoted by $\cV_{\rho_{\uk_B},r,\cK}$ and 
$\cV_{\uk_B,r,\cK}$ actually depend on the choice of $\phi_B$. In this section alone,  
we will be pedantic and write $\cV^{\phi_B}_{\rho_{\uk_B},r,\cK}$ and 
$\cV_{\uk_B,r,\cK}^{\phi_B}$ to indicate the dependence on $\phi_B$. 
Let $L\supset F$ be a number field such that $L$ splits $B$. We may assume by 
enlarging $L$ if necessary 
that it is Galois over $\Q$. Then $L$ contains $E(G_B,X_B)$. 
We pick the isomorphism $\phi_B$ above such that $B$ maps into $\prod_{\sigma \in \Sigma_\infty} \M_2 (L)$. 
This data defines an $L$-rational structure (\cite{harris-avb1}, \cite{milne}) on the automorphic 
vector bundle $\cV^{\phi_B}_{\rho_{\uk_B},r,\cK}$ on $\Sh_{\cK}(G_B, X_B)$ associated to the $G_B(\R)$-homogeneous vector bundles $\cV_{\rho_{\uk_B},r}^{\phi_B}$ 
as well as the sub-bundles $\cV_{\uk_B,r,\cK}^{\phi_B}$. To define integral models of these vector bundles, we 
first pick a rational prime $\ell$ prime to $N (\Pi)$  and insist that 
the isomorphism $\phi_B$ satisfy
\begin{equation}
\label{eqn:isomBinfty-integral1}
\phi_B (\cO_B) \subset \prod_\sigma \M_2 (\cO_{L,(\ell)}),
\end{equation}
so that $\phi_B$ gives an isomorphism
\begin{equation}
\label{eqn:isomBinfty-integral2}
\cO_B \otimes \cO_{L,(\ell)} \simeq \prod_\sigma \M_2 (\cO_{L,(\ell)}).
\end{equation}
By the discussion in Sec. \ref{sec:int-models}, this data defines for all primes $\lambda'$ of $L$ with $\lambda' \mid \lambda \mid \ell$,  natural 
integral models for 
the bundles
$\cV^{\phi_B}_{\rho_{\uk_B},r,\cK}$ and  $\cV^{\phi_B}_{\uk_B,r,\cK}$ 
over $\cS_{\cK,\lambda} \otimes_{\cO_{E(G_B,X_B),(\lambda)}} \cO_{L,(\lambda')}$.
Indeed, the choice of the Eichler order $\cO_B (\fN_B)$ determines a 
reductive group $\cG_0$ over 
$\Z_{(\ell)}$ such that $\cG_{0,\Q}=G_B$; namely for any
$\Z_{(\ell)}$-algebra $A$, we have 
\[
\cG_0 (A) = (\cO_B (\fN_B) \otimes A)^\times.
\]
Further, the map $\phi_B$ induces an isomorphism
\begin{equation}
\label{eqn:intmod-group}
\cG_0 \otimes \cO_{L,(\ell)} \simeq \prod_{\sigma \in \Sigma_\infty} \GL_{2/ \cO_{L,(\ell)}}.
\end{equation}
This gives an integral model over $\cO_{L,(\ell)}$ for the compact dual symmetric space and the 
vector bundle $\check{\cV}_{\uk_B,r}$. Via the identification \eqref{eqn:intmod-group} above, the integral model
 $\check{\fX}_{\cO_{L,(\ell)}}$ for the compact dual is simply the  
conjugacy class of the parabolic subgroup
\[
\cP:= \prod_{\sigma \in \Sigma_\infty \smallsetminus \Sigma_B} \mathcal{B} \times \prod_{\sigma \in \Sigma_{B,\infty}} \GL_{2/\cO_{L,(\ell)}}
\]
of $\cG_0 \otimes \cO_{L,(\ell)}$,
where 
\[
\mathcal{B} = \left\{  \begin{pmatrix} * & * \\ 0 & * \end{pmatrix} \right\} \subset \GL_{2/\cO_{L,(\ell)}}.
\]
Thus $\check{\fX}_{\cO_{L,(\ell)}}$ is isomorphic to $\prod_{\sigma \in \Sigma_\infty \smallsetminus \Sigma_B} \P^1_{\cO_{L,(\ell)}}$, the isomorphism 
depending on the choice of $\phi_B$. 
Let $\cL= \cO_{L,(\ell)}^2$ with the obvious left action of $\GL_2 (\cO_{L,(\ell)})$. Then the integral model  of the vector bundle 
$\check{\cV}_{\uk_B,r}$ over $\check{\fX}_{\cO_{L,(\ell)}}$ is the vector bundle  $\check{\cV}_{\cO_{L,(\ell)}}^{\phi_B}$ 
corresponding to the representation 
\[
 \prod_{\sigma \in \Sigma_\infty \smallsetminus \Sigma_B} \chi_{k_{B,\sigma},r}  \cdot \bigotimes_{\sigma \in \Sigma_{B,\infty}} \Sym^{k_{B,\sigma}}  (\cL) \otimes \det (\cL)^{\frac{r-k_{B,\sigma}}{2}}
\]
of $\cP$. (Recall that $\chi_{k,r}$ has been defined in Eg. \ref{eg:gl2}.)

For all finite places $v$ of $F$ at which $B$ is split, we will fix an isomorphism 
\begin{equation}
\label{eqn:Bfinite-isom}
i_v: B \otimes F_v \simeq \M_2 (F_v)
\end{equation}
such that for all but finitely many $v$, we have 
\begin{equation}
\label{eqn:iv-order}
i_v: \cO_B (\fN_B)  \otimes \cO_{F,v} \simeq \M_2 (\cO_{F,v}).
\end{equation}
Let $\Delta$ be a large enough finite set of places of $F$ such that 
\begin{itemize}
\item $\Delta$ contains all the infinite places, and all the finite places $v$ at which $\Pi_v$ is ramified.
\item For all $v\not \in \Delta$, the condition \eqref{eqn:iv-order} holds.
\end{itemize}

For all finite places $v$ not in $\Delta$, we get (using $i_v$) an identification 
\begin{equation}
\label{eqn:hecke-iden}
\cK_v \simeq \GL_2 (\cO_{F,v}), \quad \cH'_v \simeq \cH_v
\end{equation}
where $\cH'_v$ and $\cH_v$ denote the spherical Hecke algebras 
on $B_v^\times$ and $\GL_2 (F_v)$ constructed using the maximal compact subgroups $\cK_v$ and $\GL_2 (\cO_{F,v})$ respectively. 
Let 
\[
\cH'_\Delta = \bigotimes_{v\not \in \Delta} \cH'_v, \quad \cH_\Delta = \bigotimes_{v\not \in \Delta} \cH_v. 
\]
Note that $\cH'_\Delta$ acts naturally on the space of sections of $\cV_{\uk_B,r,\cK}$ and we have 
an identification $\cH'_\Delta \simeq \cH_\Delta$. Also 
$\cH_\Delta$ acts on $\otimes_{v \not \in \Delta} \Pi_v$. Let $\varphi= \otimes_{v\not \in \Delta} \varphi_v$ 
be a new-vector in the space $\otimes_{v \not \in \Delta} \Pi_v$, so that 
$\varphi$ is an eigenvector for the action of $\cH_\Delta$.  
Let $\Lambda_\Pi$ denote the corresponding character of 
$\cH_\Delta$.

\begin{prop}
\label{prop:uniquesection}
There exists up to scaling a unique non-zero section $s_B$ of the bundle $\cV_{\uk_B,r,\cK}^{\phi_B}$ 
 which satisfies the following conditions:
\begin{itemize}
\item $s$ is an eigenvector for the action of $\cH'_\Delta$ and 
 $\cH'_\Delta$ acts on it by $\Lambda_\Pi$,
 via the identification $\cH'_\Delta \simeq \cH_\Delta$ above.
\item $s$ satisfies \eqref{eqn:s-invariant-by-center} for $\xi = \xi_\Pi$. 
\end{itemize} 
 \end{prop}

\begin{proof}
 Let 
 $s$ be any section of $\cV_{\uk_B,r,\cK}^{\phi_B}$. Pick some point $h\in X_B$. Let $F_{s,h}=\Lift_h (s)$ 
 and set $\phi_{s,h}= \phi_{F_{s,h}}$, notations as in the previous section. 
 By strong multiplicity one,
 the assignment 
 $s\mapsto \phi_{s,h}$ gives a bijection of the space of sections 
 of $\cV_{\uk_B,r,\cK}^{\phi_B}$ 
 on which $\cH'_\Delta$ acts by $\Lambda_\Pi$
 with the space of functions 
 \[
 \phi: G_B(\Q)\backslash G_B(\A) / \cK \rightarrow \C
 \]
 that satisfy \eqref{eqn:kappatr} and \eqref{eqn:alphatr} and 
  on which $\cH'_\Delta$ acts by $\Lambda_\Pi$.
 By 
 the Jacquet--Langlands correspondence and
 the uniqueness of newforms \cite{casselman-newforms},
 this latter space is one-dimensional, generated by 
 a nonzero element $\phi$.
 If $s_B$ is such that $\phi_{s_B,h}=\phi$, then $s_B$ is our required section. 
\end{proof} 
 
 Let us 
enlarge $L$ if necessary so that $E_\Pi \subset L$ where $E_\Pi$ is the field generated by the Hecke eigenvalues of $\Pi$. 
By \cite{harris-jams} Prop. 2.2.4, the section $s_B$ of Prop. \ref{prop:uniquesection} can be chosen 
to be $L$-rational. 
Further, for $\lambda' \mid \lambda \mid \ell$ as above,
the integral model of $\cV_{\uk_B,r,\cK}^{\phi_B}$ 
over $\cS_{\cK,\lambda} \otimes_{\cO_{E(G_B,X_B),(\lambda)}} \cO_{L,(\lambda')}$ defines an
$\cO_{L,(\lambda')}$-lattice $\mathscr{M}_{\lambda'}$ in 
\[
H^0 ( \Sh_{\cK} (G_B,X_B)_{/L} ,\cV_{\uk_B,r,\cK}^{\phi_B}).
\]

Fixing $\ell$, choose $s_B$ (by suitably scaling) such that for all $\lambda' \mid \lambda \mid \ell$, 
it is a generator for the rank one $\cO_{L,(\lambda')}$-lattice $\mathscr{M}_{\lambda'} \cap Ls_B$.
We will say that the section $s_B$ is $\ell$-normalized.

\begin{rem}
\label{rem:tot-def-rat-int}
In this remark we deal with the case of totally definite $B$. 
Pick $\phi_B$ satisfying \eqref{eqn:isomBinfty-integral1}, \eqref{eqn:isomBinfty-integral2} above 
with an appropriate choice of $L$. 
Then $X_B=\{ h_0 \}$, and sections $s$ of 
$\cV_{\uk_B,r,\cK}^{\phi_B}$ are identified with  functions
\[
F: G_B(\Q)\backslash G_B(\A) / \cK \rightarrow \cV_{\uk_B,r} =\bigotimes_{\sigma \in \Sigma_\infty} V_{\sigma,k_{B,\sigma},r}
\]
satisfying the appropriate invariance property under the right action of $G_B (\R)$. 
Then $\cV_{\uk_B,r}$ admits a natural $L$-rational structure as well as a natural
$\cO_{L,(\ell)}$-submodule:
\begin{align*}
 \cV_{\uk_B,r} &\supset \cV_{\uk_B,r}(L) =\bigotimes_{\sigma \in \Sigma_\infty} \Sym^{k_{B,\sigma}} L^2 
 \otimes \det(L^2)^{\frac{r-k_{B,\sigma}}{2}} \\
 &\supset \cV_{\uk_B,r} (\cO_{L,(\ell)})=\bigotimes_{\sigma \in \Sigma_\infty} \Sym^{k_{B,\sigma}} \cO_{L,(\ell)} ^2
 \otimes \det(\cO_{L,(\ell)}^2)^{\frac{r-k_{B,\sigma}}{2}}.
 \end{align*}
 This gives an $L$-rational structure and an $\cO_{L,(\ell)}$-integral structure on the space of sections of 
 $\cV_{\uk_B,r,\cK}^{\phi_B}$, namely we take sections $s$ which on $G_B(\A_f)$ take values in $\cV_{\uk_B,r}(L)$ and $\cV_{\uk_B,r} (\cO_{L,(\ell)})$ respectively. 
 We pick isomorphisms $i_v$ as above and Prop. \ref{prop:uniquesection} continues to hold. Finally, 
 we pick $s_B$ to be $\ell$-normalized with respect to the integral structure provided by $\cV_{\uk_B,r} (\cO_{L,(\ell)})$. 
\end{rem}

\subsection{Canonical quadratic period invariants} 
\label{ssec:defn-period-invariant}

We can now define the canonical quadratic period invariants attached
to $\Pi$ and 
state the main conjecture relating these invariants. 
Let $B$ be a quaternion algebra such that $\Sigma_B \subseteq \Sigma_\Pi$. 
As in the introduction, let $R$ be the ring $\cO_{\Qbar} [1/N (\Pi)]$. For any rational prime $\ell$ prime to $N(\Pi)$, we define 
 an invariant $q_B (\Pi,\ell) \in \C^\times / R_{(\ell)}^\times$ as follows. 
 Let $\cKt \supseteq \cK$ be the open compact subgroup of $G_B(\A_f)$ defined by 
$\cKt=\prod_v \cKt_v$ with
\[
\cKt_v = (\cO_B (\fN_B) \otimes _{\cO_F} \cO_{F,v})^\times.
\]
 Choosing a section 
 $s_B$ as above that is $\ell$-normalized, define 
\[
q_B (\Pi, \ell):= \llangle s_B, s_B \rrangle_{\cKt}  \in \C^\times / R_{(\ell)}^\times,
\]
to be the Petersson norm of the section $s_B$ as in Defn. \ref{defn:pet-norm}.

\begin{prop}
The invariant $q_B (\Pi, \ell)$ is well defined, in that as an element of $\C^\times/ R_{(\ell)}^\times$, it 
does not depend on the choices of the number field $L$, the pair $(\cO_B (\fN_B),o)$ consisting of the Eichler order $\cO_B(\fN_B)$ and the orientation $o:\cO_B(\fN_B) \rightarrow \cO_F/\fN_{\ps}$, the isomorphism 
\eqref{eqn:isomBinfty} satisfying \eqref{eqn:isomBinfty-integral1},  \eqref{eqn:isomBinfty-integral2} above and the collection of isomorphisms
\eqref{eqn:Bfinite-isom}. 
\end{prop}

\begin{proof}
We will give the argument in the case when $B$ is not totally definite. 
In the case of a totally definite $B$, a similar (but simpler) argument can be given which 
we leave to the reader. 

Independence of the choice of $L$ is clear since we can always replace $L$ by a larger field
without changing the choice of $s_B$. 
Implicitly in the arguments below we may need to make such a field extension and we do this without comment. 
Let us first show that fixed choices of other data, there is no dependence on the choice of 
isomorphisms \eqref{eqn:Bfinite-isom}. Indeed, for all but finitely many $v$, the isomorphisms
$i_v$ must satisfy \eqref{eqn:iv-order}. Let $\{ i'_v\}$ be a different set of choices. Then for all but finitely 
many $v$, the isomorphisms $i_v$ and $i'_v$ must differ by conjugation by an element of $\cK_v$. 
For such $v$, the identifications $\cH'_v \simeq \cH_v$ given by $i_v$ and $i'_v$
are the same.
This implies that
the same choice of $s_B$ can be used 
if $\{ i_v\}$ is replaced by $\{ i'_v\}$ and the norm $\llangle s_B, s_B \rrangle_{\cKt} $ is unchanged. 

Next let us look at the dependence on the choice of isomorphism $ \phi_B$ in \eqref{eqn:isomBinfty},
for fixed choices of other data. 
Let $\phi'_B$ be a different choice of isomorphism satisfying \eqref{eqn:isomBinfty-integral1}. Then 
$\phi_B$ and $\phi'_B$ differ by conjugation by an element 
\[
t \in \prod_\sigma \GL_2(L) \cap 
\left( \prod_{\sigma \in \Sigma_\infty \smallsetminus \Sigma_B} \GL_2 (\R) \times \prod_{\sigma \in \Sigma_{B,\infty}} \H^\times \right)
\]
 that 
normalizes $\prod_\sigma \M_2 (\cO_{L,(\ell)})$.
The normalizer of $\M_2 (\cO_{L,(\ell)})$ in $\GL_2 (L)$ is 
$L^\times \cdot \GL_2(\cO_{L,(\ell)})$, so we may assume that 
$t$ lies in $\prod_\sigma \GL_2 (\cO_{L,(\ell)})$. 
Then there is a natural morphism of integral models
\[
\check{\cV}_{\cO_{L,(\ell)}}^{\phi_B} \simeq \check{\cV}_{\cO_{L,(\ell)}}^{\phi'_B}
\]
which is just given by the (left) action of $t$ on the fibers. This induces an 
 isomorphism 
between the integral models  of the corresponding automorphic vector bundles
that is also given by the action of $t$ on the fibers. 
(Keep in mind that the $G(\Q)$-action on the fibers are different, and differ by 
conjugation by $t$, so the left action of $t$ on the fibers is indeed a map of bundles.)
Thus if $s_B$ is an $\ell$-normalized
section of $\cV_{\uk_B,r,\cK}^{\phi_B}$, then $t\cdot s_B$ is an $\ell$-normalized section 
of $\cV_{\uk_B,r,\cK}^{\phi'_B}$. Then the 
inner products $\llangle s_B, s_B \rrangle_{\cKt}$ and $\llangle s'_B, s'_B \rrangle_{\cKt}$ differ by
a power of $\| \nu (t) \|$, which is a unit at $\ell$. 

Finally,  we consider dependence on the choice of the pair $(\cO_B (\fN_B),o)$. 
Let 
$(\cO_B (\fN_B)',o')$
be another such pair and let $\phi'_B$ (respectively $i'_v$) denote our choices of isomorphism
\eqref{eqn:isomBinfty} satisfying \eqref{eqn:isomBinfty-integral1} (respectively the isomorphisms 
\eqref{eqn:Bfinite-isom} satisfying \eqref{eqn:iv-order} for all but finitely many $v$). Let us suppose first that the pair $(\cO_B (\fN_B)' ,o')$
 is conjugate to $(\cO_B (\fN_B),o)$ by an element in $B^\times$, say $\cO_B (\fN_B)' = b^{-1} \cO_B(\fN_B)b$ and 
 $o'(x) = o(bxb^{-1})$.  
By what we have shown so far we may assume that  
\[
\phi'_B (x) = \phi_B (b x b^{-1}), \quad i'_v (x) = i_v ( b x b^{-1}).
\]
 The open compact subgroup $\cK'$ of $G_B(\A_f)$ determined by the pair $(\cO_B(\fN_B)',o')$ satisfies $\cK' = b^{-1} \cK b$. 
 Let us write $b = b_\infty \cdot b_f$ where $b_\infty$ and 
 $b_f$ denote the infinite and finite components of $b$ respectively, viewed as elements in $G_B (\A)$. 
 There is a natural isomorphism 
 of Shimura varieties 
 \[
 \Sh_{\cK} (G_B,X_B) =G_B(\Q) \backslash X_B \times G_B (\A_f) / \cK \ \ \stackrel{\xi_b}{\simeq} \ \ G_B(\Q) \backslash X_B \times G_B (\A) / \cK' = \Sh_{\cK'} (G_B,X_B), 
 \]
 given by 
 \[
 (h,g_f) \mapsto (h, g_f b_f).
 \]
 Further, there is a natural isomorphism 
 \[
 \cV_{\uk_B,r,\cK}^{\phi_B} = G_B (\Q) \backslash \cV_{\uk_B,r}^{\phi_B} \times G_B(\A_f)/ \cK \ \ \stackrel{\tilde{\xi}_b}{\simeq}  \ \
 G_B (\Q) \backslash \cV_{\uk_B,r}^{\phi'_B} \times G_B(\A_f)/ \cK' = \cV_{\uk_B,r,\cK'}^{\phi'_B}
\] 
covering $\xi_b$, given by 
\begin{equation}
\label{eqn:maponvb}
(v,g_f) \mapsto  (\phi_B (b) \cdot v, g_f b_f).
\end{equation}
Note that if $\gamma $ is an element in $G_B (\Q)$ then 
\begin{align*}
\gamma \cdot (v,g_f) = (\phi_B (\gamma) \cdot v, \gamma g_f) & \mapsto (\phi_B(b) \phi_B (\gamma) \cdot v, \gamma g_f b_f ) \\ &= (\phi'_B (\gamma) 
\phi_B (b)\cdot v, \gamma g_f b_f ) = \gamma \cdot (\phi_B (b)\cdot v, g_f b_f),
\end{align*}
so that the assignment in \eqref{eqn:maponvb} does descend to equivalence classes for the $G_B(\Q)$-action. 
Note also that $\tilde{\xi}_B$ is the map on automorphic vector bundles corresponding to a
morphism of vector bundles that extend to the integral models, since these integral models are defined using 
the triples $(\cO_B(\fN_B), o,\phi_B)$ and $(\cO_B(\fN_B)',o', \phi'_B)$ respectively. Thus 
$\tilde{\xi}_b$ is an isomorphism at the level of integral models, and so 
we may assume that $s'_B = \tilde{\xi}_b (s_B)$. But then we see from the definition 
of the metrics on the vector bundles $\cV_{\uk_B,r,\cK}^{\phi_B}$ and $\cV_{\uk_B,r,\cK'}^{\phi'_B}$ 
and the product formula that 
\[
\llangle s'_B, s'_B \rrangle_{\cKt'} = \| \nu(b_\infty) \|^{-r} \cdot \| \nu(b_f) \|^{-r} \cdot \llangle s_B, s_B \rrangle_{\cKt} = \llangle s_B, s_B \rrangle_{\cKt}.
\]
In general, it may not be true that the pairs $(\cO_B(\fN_B),o)$ and $(\cO_B(\fN_B)',o')$ are conjugate by an element of $B^\times$. 
Nevertheless, we can always find an element $\beta_f \in B^\times (\A_f)$ such that 
\[
\cO_B(\fN_B)'=\beta_f^{-1} \cO_B(\fN_B) \beta_f, \quad o'(x)=o(\beta_f x \beta_f^{-1}).
\]
Let $b$ be an element of $B^\times$ approximating $\beta_f=(\beta_v)$ at $\ell$ so that 
\[
 \cO_B(\fN_B)' \otimes \Z_{(\ell)}=(b^{-1} \cO_B(\fN_B) b) \otimes \Z_{(\ell)}.
\]
Then we may assume that 
\[
\phi'_B (x) = \phi_B (b x b^{-1}), \quad i'_v (x) = i_v ( \beta_v x \beta_v^{-1}).
\]
 The open compact subgroup $\cK'$ satisfies $\cK' = \beta_f^{-1} \cK \beta_f$. We now run through the same argument as above, defining 
 \[
  \xi_b [(h,g_f)]= [(h, g_f \beta_f)], \quad \tilde{\xi}_b [(v,g_f)]= [(\phi_B (b) \cdot v, g_f \beta_f)].
  \]
The result follows from observing that $\| \nu(b_\infty) \| \cdot \|\nu (\beta_f) \|$, while not necesarily $1$, is still an 
element in $R_{(\ell)}^\times$.   
\end{proof}

We can also define an invariant $q_B (\Pi) \in \C^\times /R^\times$ such that the class of 
$q_B (\Pi)$ in $\C^\times/ R_{(\ell)}^\times$ equals $q_B (\Pi,\ell)$. Indeed, 
pick the isomorphism $\phi_B$, the number field $L$ and the maximal order $\cO_B$ in $B$ 
such that 
\[
\phi_B (\cO_B) \subset \prod_\sigma \M_2 (\cO_{L}).
\]
Choose a pair $(\cO_B (\fN_B),o)$ consisting of an Eichler order and an orientation.
The constructions in \S \ref{ssec:rational-integral-structures} can be 
copied verbatim to give integral models for everything in sight over 
$\cO_L [\frac{1}{N(\Pi)}]$. (See Sec. \ref{par:patching}.)
By enlarging $L$ if need be, we can pick a section $s_B$ that is $\ell$-normalized 
at {\it all} rational primes that are prime to $N(\Pi)$ and then
define $q_B (\Pi)$ to equal $\llangle s_B, s_B \rrangle_{\cKt}$ for such a choice of $s_B$. 
This is an element of $\C^\times$ that maps to $q_B(\Pi, \ell)$ 
under the natural map $\C^\times \rightarrow \C^\times / R_{(\ell)}^\times$ for 
all $\ell$ such that $(\ell, N (\Pi))=1$. Since the map 
\[
\C^\times / R^\times \rightarrow \prod_{(\ell, N (\Pi)) =1} \C^\times / R_{(\ell)}^\times
\]
is injective, the class of $q_B (\Pi)$ in $\C^\times/ R^\times$ is well defined. 
This defines the invariants needed in the formulation of Conjecture \ref{conj:oldintro} of the introduction.

%% file: un-qu-groups.tex
\section{Unitary and quaternionic unitary groups}
\label{sec:uqu}

In \S \ref{subsec:hsh}, we review the general theory of hermitian and skew-hermitian forms
over local fields and number fields. In \S \ref{subsec:key}, we describe the construction of a certain 
skew-hermitian space over a quaternion algebra (over a number field), which plays an important role throughout the paper, while 
in \S \ref{subsec:hasse} we review the connection between this construction and the failure of the Hasse principle for quaternionic skew-hermitian forms.

\subsection{Hermitian  and skew-hermitian spaces}
\label{subsec:hsh}

\subsubsection{Hermitian spaces}

Let $F$ be a field of characteristic zero and $E$ a quadratic extension of $F$, possibly split.  
Let $\VV$ be a right $E$-vector space of dimension $n$ (i.e., a free $E$-module of rank $n$), equipped with a 
Hermitian form 
\[
(\cdot, \cdot):\VV \times \VV \rar E.
\]
Such a form is linear in one variable and antilinear in the other, and 
we fix any one convention at this point. For example, if $(\cdot, \cdot)$ is antilinear in the first variable 
and linear in the second, then:
\[
(v \alpha, v' \beta) = \alpha^\rho (v,v') \beta, \quad (v,v')=(v',v)^\rho,
\]
where $\rho$ denotes the nontrivial involution of $E/F$.  

To such a $\VV$, one associates the following invariants: $\dim(\VV)=n$ and $\disc (\VV)\in F^\times/\N_{E/F} E^\times$, where
\[
\disc (\VV) = (-1)^{n(n-1)/2} \det\left( (v_i,v_j) \right),
\]
with $\{ v_i \}$ an $E$-basis for $\VV$. Since $(\cdot, \cdot)$ is Hermitian, $\disc(\VV)$ lies in $F^\times$
 and its class in $F^\times/\N_{E/F} E^\times$ is independent of the choice of basis. 

Let $\GU(\VV)$ denote the unitary similitude group of $\VV$. (Occasionally, we will write $\GU_E(\VV)$ for clarity.) This is an algebraic group over $F$ such that
for any $F$-algebra $R$, we have
\[
\GU(\VV) (R):= \{ g \in \GL(\VV \otimes R): \quad (gv,gv')= \nu(g)(v,v') \text{ for all }v,v' \text{ and } \nu(g)\in R^\times\}.
 \]

If $E=F\times F$, then $\GU(\VV) \simeq \GL_n \times \GL_1$. If $E$ is a field, the various possibilities for $\GU(\VV)$ are 
discussed below.

\paragraph{$p$-adic local fields}
Let $F$ be $p$-adic. 
As a Hermitian space, 
$\VV$ is determined up to isomorphism by its dimension and discriminant. Further, given any choice of dimension and 
discriminant, there is a space $\VV$ with these as its invariants. 
If $\dim(\VV)$ is odd, the group $\GU(\VV)$ is (up to isomorphism) independent of $\disc(\VV)$ and is quasi-split. 
If $\dim(\VV)$ is even, there are two posibilities for $\GU(\VV)$ up to isomorphism and 
$\GU(\VV)$ is quasi-split if and only if $\disc(\VV) =1$. 

\paragraph{Archimedean fields} Let $F=\R$ and $E=\C$. Then 
the form $(\cdot,\cdot)$ can be put into the diagonal form $(1,\ldots, 1, -1, \ldots, -1)$ which is called the signature of $\VV$
; we say $\VV$ is of type $(p,q)$ if the number of $1$s is $p$ and the number of $-1$s is $q$. 
Hermitian spaces are classified up to isomorphism by their signature (which determines both the dimension and discrminant) and we write $\GU(p,q)$ for the associated group.
The only isomorphisms between these groups are $\GU(p,q) \simeq \GU(q,p)$.

\paragraph{Number fields} Let $E/F$ be a quadratic extension of number fields. 
If $\VV$ is a Hermitian $E$-space, then for each place $v$ of $F$, one gets 
a local space $\VV_v$ which is a Hermitian space for $E_v/F_v$
and such that for almost all $v$, the discriminant of $\VV_v$ is $1$. The Hasse principle
says that $\VV$ is determined up to isomorphism by this collection of local spaces. 
Conversely, suppose we are given for each place $v$ a local space $\VV_v$ (of some fixed dimension $n$) such that 
almost all of the local discriminants are equal to $1$. The collection of 
local discriminants gives an element of 
$\A^\times_F/\N_{E/F} \A^\times_E$. Such a collection of local spaces comes from a global space
if and only if this element lies in the image of $F^\times$, i.e., is trivial in the quotient
$\A^\times_F/F^\times \N_{E/F} \A^\times_E$, which has order $2$.

\subsubsection{Skew-Hermitian spaces}
Let $E/F$ be a quadratic extension as in the beginning of the previous section. 
Skew-hermitian $E$-spaces are defined similarly to hermitian spaces but with the condition
\[
(v,v') = -(v',v)^\rho.
\]
We can go back and forth between hermitian and skew-hermitian spaces simply by multiplying 
the form by an element in $E^\times$ of trace zero. Indeed, pick a trace zero element $\i \in E^\times$. If $(\cdot,\cdot)$ is skew-hermitian form on $\VV$, the product $(\cdot,\cdot)':=\i \cdot (\cdot,\cdot)$ is hermitian. The group $\GU(\VV)$ is the same for both $(\cdot,\cdot)$ and $(\cdot,\cdot)'$. Thus the classification of 
skew-hermitian forms (and the corresponding groups) can be deduced from the hermitian case. 

 \subsubsection{Quaternionic hermitian spaces}
Let $F$ be a field and $B$ a quaternion algebra over $F$. Let $a \mapsto a^*$ denote the main involution on $B$. A $B$-Hermitian space is a right $B$-space $V$
equipped with a $B$-valued form
\[
\langle \cdot,\cdot \rangle : V \times V \rar B
\]
satisfying 
\[
\langle v\alpha, v' \beta \rangle = \alpha^* \langle v,v' \rangle \beta, \quad \langle v, v' \rangle = \langle v', v \rangle ^*,
\]
for $v,v' \in V$ and $\alpha,\beta \in B$. 

Let $\GU(V)$ denote the unitary similitude group of $V$. (Sometimes, we write $\GU_B(V)$ for clarity.) This is an algebraic group over $F$ such that
for any $F$-algebra $R$, we have
\[
\GU(V) (R):= \{ g \in \GL(V \otimes R): \quad \langle gv,gv' \rangle= \nu(g)\langle v,v'\rangle \text{ for all }v,v' \text{ and } \nu(g)\in R^\times\}.
 \]

If $B$ is split, there is a unique such space $V$ of any given dimension $n$ over $B$. The corresponding group 
$\GU(V)$ is identified with $\GSp(2n)$. 
If $B$ is nonsplit, the classification of such spaces over $p$-adic fields and number fields is recalled below. 

\paragraph{$p$-adic fields}
If $F$ is a $p$-adic field, there is a unique such space of any given dimension, up to isometry.
The corresponding group is the unique nontrivial inner form of $\GSp (2n)$.

\paragraph{Archimedean fields}
If $F=\R$, such spaces are classified by dimension and signature. If the signature is of type $(p,q)$, the associated 
group is denoted $\GSp(p,q)$. The only isomorphisms between these are $\GSp(p,q)\simeq \GSp(q,p)$. 

\paragraph{Global fields} 
The Hasse principle holds in this case, so a global $B$-hermitian space is determined up to isometry by 
the collection of corresponding local $B_v$-Hermitian spaces. Conversely, given any collection of 
$B_v$-hermitian spaces, there is a (unique) $B$-Hermitian space that gives rise to this local collection up to isometry.

\subsubsection{Quaternionic skew-hermitian spaces}
These are defined similarly to $B$-hermitian spaces but with the condition
\[
\langle v,v' \rangle = - \langle v',v \rangle^*.
\]
To such a space $V$ is associated the invariant $\det (V) \in F^\times/(F^\times)^2$ as follows. Pick a $B$-basis
$\{ v_i \}$ for $V$ and set 
\[
\det(V) = \nu_B (\langle v_i,v_j \rangle ).
\]
Here $\nu_B$ denotes the reduced norm. (Often, we will omit the subscript $B$ when the choice of 
quaternion algebra is clear.)
The group $\GU(V)$ is defined similarly as above. It is however not connected as an algebraic group.
We now recall the classification of such spaces and the associated groups. Note that if $B$ is split,
we can associate to $V$ a quadratic space $V^\dagger$ over $F$ of dimension $2n$ (where 
$n=\dim_B (V)$) and $\GU(V)\simeq \GO(V^\dagger)$. 

\paragraph{$p$-adic fields}
Let $F$ be $p$-adic. If $B$ is split, $V$ is determined by $\dim(V)$, $\det(V)$ and the Hasse invariant of $V^\dagger$. 
If $B$ is nonsplit, $V$ is determined by $\dim(V)$ and $\det(V)$. 

\paragraph{Archimedean fields}
If $F=\R$ and $B$ is split, $V$ is determined by the signature of $V^\dagger$. The group $\GU(V)$ is isomorphic to $\GO(p,q)$
where $(p,q)$ is the signature.
If $B$ is nonsplit, $V$ is determined just by $n=\dim_B (V)$. The group $\GU(V)$ is isomorphic to $\GO^*(2n)$.
If $F=\C$, then $B$ must be split and there is a unique skew-hermitian space of any given dimension. 
Then $\GU(V)\simeq \GO(2n,\C)$. 

\paragraph{Global fields}
Let $F$ be a number field. If $B$ is split, then the classification reduces to that for quadratic spaces via the assignment $V\mapsto V^\dagger$. 
In this case, the Hasse principle holds. 
If $B$ is nonsplit, then the Hasse principle does not hold. Let $\Sigma_B$ be 
the set of places $v$ where $B$ is ramified and let $s=|\Sigma_B|$. 
The space $V$ gives rise to a collection of local spaces and up to isometry
there are exactly $2^{s-2}$ global $B$-skew-hermitian spaces that give rise to the same 
set of local spaces. 
Conversely a collection of local $B_v$-skew-hermitian spaces $V_v$ arises from 
a global $B$-skew-hermitian space $V$ if and only if there exists a global 
element $d\in F^\times$ such that $\det(V_v)=d $ in $F_v^\times/(F_v^\times)^2$ 
for all $v$ and for almost all $v$, the Hasse invariant of $V_v^\dagger$ is trivial.

\subsection{The key constructions}
\label{subsec:key}

In this section, we assume 
that $B_1$ and $B_2$ are two quaternion algebras 
 over a number field $F$ and $E/F$ is a quadratic extension that 
embeds in both $B_1$ and $B_2$.
We will fix embeddings $E \hookrightarrow B_1$ and $E \hookrightarrow B_2$. Via these embeddings, $B_1$ and $B_2$ are hermitian spaces over $E$. Let $\tau_i$ and $\nu_i$ be respectively the reduced trace and norm on $B_i$. We think of $B_1$ and $B_2$ as right $E$-vector spaces, the Hermitian form being described below.
Write
$$ B_1 = E + E \j_1= E + \j_1 E ,\qquad B_2 = E + E \j_2= E + \j_2 E ,$$
where $\tau_1(\j_1) = \tau_2 (\j_2) =0$. We write $\pr_i$ for the projection $B_i \rightarrow E$ onto the ``first coordinate" and $*_i$ for the main involution on $B_i$. 
Then $B_i$ is a right Hermitian $E$-space, the form being given by:
$$ (x,y)_i = \pr_i (x^{*_i} y).$$
If $x=a+\j_i b$, $y= c+\j_i d$, then 
$$ (x,y)_i = (a + \j_i b, c+\j_i d)_i =  a^\rho c - J_i  b^\rho d ,$$
where $J_i := - \nu_i (\j_i) = \j_i^2$. 
This form satisfies the relations
$$ (x\alpha, y\beta)_i = \alpha^\rho (x,y)_i \beta, \qquad \text{for} \qquad \alpha,\beta \in E,$$
and
$$ (x,y)_i = (y,x)_i^\rho.$$

We note that $B_i^\times$ acts naturally on $B_i$ by left multiplication, and this action is $E$-linear. Further,
\begin{equation}
\label{eqn:B_i-unitary}
(\ba x, \ba y)_i = \nu_i (\ba) (x,y)_i
\end{equation}
for all $\ba \in B_i$. Thus $B_i^\times$ embeds naturally in $\GU_E (B_i)$. In fact, 
$$F^\times \backslash (B_i^\times \times E^\times)  \simeq \GU_E(B_i), $$
where $E^\times$ acts on $B_i$ by right multiplication, and we think of $F^\times$ as a subgroup of 
$B_i^\times \times E^\times $ via $\lambda \mapsto (\lambda^{-1}, \lambda)$. 

Consider the (right) $E$-vector space 
$$ V:= B_1 \otimes_E B_2.$$

\begin{rem}
 If $x\in B_1$, $y \in B_2$, $\alpha \in E$, then by definition,
$$ (x\otimes y )\alpha = x\alpha \otimes y = x \otimes y \alpha. $$
\end{rem}

 The $E$-vector space $V$ is equipped with a natural Hermitian form given by the tensor product $(\cdot,\cdot)_1 \otimes (\cdot,\cdot)_2$. We fix a nonzero element $\i \in E$ of trace $0$, and define $(\cdot , \cdot )$ on $V$ by
 \begin{equation}
 \label{eqn:()defn}
  (\cdot, \cdot) : = \i \cdot (\cdot,\cdot)_1 \otimes (\cdot,\cdot)_2.
  \end{equation}
Clearly, $(\cdot , \cdot )$ satisfies
$$ (x\alpha, y \beta) = \alpha^\rho (x,y) \beta , \qquad \text{for} \qquad \alpha,\beta \in E,$$
$$ ( x,y ) = - ( y,x )^\rho .$$
Thus $(\cdot, \cdot)$ is an $E$-skew-Hermitian form on $V$. 

It will be useful to write down the form $(\cdot, \cdot)$ explicitly in terms of coordinates with respect to a suitable $E$-basis. We pick the following (orthogonal) basis:
\begin{equation}
\label{eq:basisX}
 \e_1:= 1\otimes 1\qquad \e_2 := \j_1 \otimes 1 \qquad \e_3:= 1\otimes \j_2 \qquad \e_4:= \j_1 \otimes \j_2.
\end{equation}
In this basis,
$$( \e_1 a + \e_2 b + \e_3 c + \e_4 d, \e_1 a' + \e_2 b' + \e_3 c' + \e_4 d') = \i \cdot \left[ a^\rho a' - J_1 b^\rho b' - J_2 c^\rho c' + J_1 J_2 d^\rho d'\right] .$$

There is a natural map
$$ \GU_E (B_1) \times \GU_E (B_2) \rightarrow \GU_E(V),$$
given by the actions of $\GU_E(B_1)$ and $\GU_E(B_2)$ on the first and second component respectively of $V=B_1\otimes_E B_2$. The kernel of this map is 
$$ \mathrm{Z} := \left\{  ([\lambda_1, \alpha_1 ], [\lambda_2, \alpha_2] ) : \lambda_i \in F^\times, \alpha_i \in E^\times, \lambda_1\lambda_2 \alpha_1 \alpha_2 =1 \right\}. $$

Let $B:= B_1 \cdot B_2$ be the product in the Brauer group over $F$. Then $E$ embeds in $B$ as 
well, and we will fix an embedding $E \rightarrow B$. We may write 
$$B = E + E \j$$
where $\tau (\j)=0$ and $J: = - \nu (\j) = \j^2$ satisfies 
$$ J= J_1 J_2.$$
Here $\tau$ and $\nu$ are respectively the reduced trace and reduced norm on $B$.
We define a right action of $B$ on $V$ (extending the right $E$-action on $V$) by setting
\begin{eqnarray}
\label{eqn:jaction11} (1\otimes 1)  \cdot \j&=& \j_1 \otimes \j_2 \\
\label{eqn:jactionj1}(\j_1 \otimes 1)  \cdot \j &=& J_1 (1 \otimes \j_2) \\
\label{eqn:jaction1j} (1 \otimes \j_2)  \cdot \j&=& J_2 (\j_1 \otimes 1) \\
\label{eqn:jactionjj}(\j_1 \otimes \j_2)  \cdot \j&=& J_1 J_2 (1\otimes 1)
\end{eqnarray}
and requiring the right action by $\j$ on $V$ to be conjugate $E$-linear. (It is straightforward to check that this gives an action.)
Then $V$ is a free rank-$2$ right $B$-module. For example, a basis for $V$ as a right $B$-module is given either by 
$$ \{ 1\otimes 1, \j_1 \otimes 1 \}$$
or 
$$ \{ 1\otimes 1, 1 \otimes \j_2 \}.$$
Further, one checks that (equivalently)
\begin{eqnarray}
\label{eqn:jaction1} (x\j, y ) &=& (y\j, x)\\
\label{eqn:jaction2} (x\j, y\j)^\rho &=& - J (x,y) 
\end{eqnarray}
for all $x,y \in V$.

We will now show that there is a $B$-valued skew-Hermitian form $\langle \cdot, \cdot \rangle$ on $V$ such that 
$$ \pr \circ \langle \cdot, \cdot \rangle = (\cdot, \cdot).$$
Indeed, define
\begin{equation}
\label{eqn:<>defn}
 \langle x,y \rangle = (x,y) - \frac{1}{J} \cdot \j \cdot (x\j, y). 
\end{equation}
It may be checked (using (\ref{eqn:jaction1}) and (\ref{eqn:jaction2})) that 
\begin{eqnarray}
\label{eqn:skewherm1} \langle x\ba, y \bb\rangle &=& \ba^* \langle x,y\rangle \bb, \\
\label{eqn:skewherm2}  \langle x,y\rangle &=& -  \langle y,x \rangle^*,
\end{eqnarray}
for all $\ba,\bb \in B$. 
For future reference, we write down the matrix of inner products $\langle \e_i, \e_j \rangle$. 

$$
\begin{array}{|c|c|c|c|c|}
\hline
\langle \cdot, \cdot \rangle & \e_1& \e_2 & \e_3 & \e_4 \\
\hline
\e_1 & \i & 0&0 & \i \j \\
\hline
\e_2 & 0 &  -J_1 \i& - \i \j & 0 \\
\hline
\e_3 & 0 & -\i \j& -J_2 \i & 0 \\
\hline
\e_4 & \i \j & 0&  0& J \i\\
\hline
\end{array}
$$

\noindent From the table, we see that $\det(V) = \nu (-J_1u) = 1$ in $F^\times/(F^\times)^2$.

Notice that $B_1^\times$ and $B_2^\times$ act on $V$ by left multiplication on the first and second factor of $V=B_1 \otimes _E B_2$ respectively. These actions are (right) $E$-linear and in fact (right) $B$-linear as is easily checked using (\ref{eqn:jaction11}) through (\ref{eqn:jactionjj}). Further, it follows from (\ref{eqn:B_i-unitary}), (\ref{eqn:()defn}), and (\ref{eqn:<>defn}) that 
\begin{equation}
\langle \ba_i \cdot x, \ba_i \cdot y\rangle = \nu_i (\ba_i) \langle x,y \rangle 
\end{equation}
for $\ba_i \in B_i^\times$. Clearly, the actions of $B_1^\times $ and $B_2^\times$ commute, hence one gets an embedding
\begin{equation}
\label{eqn:qu-gubv}
 F^\times \backslash (B_1 ^\times \times B_2^\times ) \hookrightarrow \GU_B (V), 
 \end{equation}
the quaternionic unitary group of the $B$-skew-Hermitian form $\langle \cdot , \cdot \rangle$. 
(Here we think of $F^\times$ as embedded antidiagonally in $B_1^\times \times B_2^\times$ via $\lambda \mapsto (\lambda^{-1}, \lambda)$.)
Then (\ref{eqn:qu-gubv}) gives 
an isomorphism
$$ F^\times \backslash (B_1 ^\times \times B_2^\times ) \simeq \GU_B (V)^0,$$
where $\GU_B(V)^0$ denotes the identity component of $\GU_B (V)$. 
Further, one has a commutative diagram
$$ \xymatrix{
F^\times \backslash (B_1^\times \times B_2^\times) \ar[rr]^{\simeq} \ar@{^{ (}->}[d] & & \GU_B(V)^0 \ar@{^{ (}->}[d] \\
\mathrm{Z} \backslash (\GU_E(B_1) \times \GU_E (B_2)) \ar[rr]  & & \GU_E(V)  \\
}
$$
where the vertical map
$$ F^\times \backslash (B_1^\times \times B_2^\times) \hookrightarrow \mathrm{Z} \backslash (\GU_E(B_1) \times \GU_E(B_2)) $$ 
is 
$$ [b_1, b_2] \mapsto \left[ [(b_1,1)], [ (b_2 , 1)]\right],$$
and the vertical map $\GU_B(V)^0 \hookrightarrow \GU_E(V)$ is just the natural inclusion.

Let $\V=\mathrm{Res}_{E/F} (V)$, that is $\V$ is just $V$ thought of as an $F$-space, with non-degenerate symplectic form
$$ \llangle v_1,v_2 \rrangle:= \frac{1}{2} \tr_{E/F} (v_1,v_2).$$
Let
$$ \X= F\e_1 \oplus F\e_2 \oplus F\e_3 \oplus F\e_4 \subset \V.$$
Since $\X$ is maximal isotropic for $\llangle \cdot,\cdot \rrangle$, there exists a unique maximal isotropic subspace $\Y$ in $\V$, such that 
$\V=\X\oplus \Y$. Let  $(\e_1^*, \ldots, \e_4^*)$ be an $F$-basis for $\Y$ that is dual to $(\e_1, \ldots,\e_4)$. We can identify this basis precisely: letting $\i^2=u \in F^\times$, we have
\begin{equation}
\label{eq:basisY}
  \e_1^*= \frac{1}{u} \e_1 \i, \quad \e_2^*= -\frac{1}{J_1 u} \e_2\i, \quad \e_3^*= -\frac{1}{J_2 u}\e_3 \i, \quad \e_4^* = \frac{1}{Ju} \e_4 \i. 
\end{equation}

\subsubsection{The unitary group $\U_E (V)$ at infinite places}

This section will not be relevant in this paper. We simply record for future use the 
isomorphism class of the unitary group $\U_E(V)$ at the infinite places $v$ 
assuming that $F_v =\R$ and $E_v=\C$. 
The skew-hermitian form is given by  
\[
\i \cdot \left[ a^\rho a' - J_1 b^\rho b' - J_2 c^\rho c' + J d^\rho d'\right] 
\]
Thus we have the following table which summarizes the relation between 
the ramification of $B_1$ and $B_2$ at $v$ and the isomorphism class of $\U_E(V)$. 

\begin{center}
{\renewcommand{\arraystretch}{1.2}
\begin{tabular}{|c | c | c |}
\hline
$B_1, B_2$ & $J_1, J_2 $& $\U_E(V)$\\
\hline
split, split & $J_1>0,J_2 >0 $ &   $\U (2,2)$ \\
\hline
ramified, split & $J_1 <0 ,J_2 >0 $ &    $\U (2,2)$ \\
\hline
split, ramified & $J_1 >0 , J_2 <0 $ &    $\U (2,2)$ \\
\hline
ramified, ramified & $J_1 <0,J_2<0$ &    $\U (4,0)$ \\
\hline
\end{tabular}
}
\end{center}

\subsection{The failure of the Hasse principle}
\label{subsec:hasse}
The constructions above illustrate the failure of the Hasse principle for skew-hermitian $B$-spaces.
Indeed, let us fix a $B$ and consider pairs $(B_1,B_2)$ such that $\Sigma_{B_1} \cap \Sigma_{B_2} =\Sigma_0$,
where $\Sigma_0$ is some fixed set of places not intersecting $\Sigma_B$.  
Let $E/F$ be a quadratic extension that is nonsplit at all the places in $\Sigma_B \cup \Sigma_0$. 
Then $E$ embeds in $B$,$B_1$,$B_2$, so the constructions from the previous section apply. 
The various spaces $V$ obtained by taking different choices of $B_1$ and $B_2$ are all locally isometric,
since all of them have $\det(V)=1$ and the Hasse invariant of $V_v^\dagger$ is independent of $V$ for $v \not \in \Sigma_B$. 
Since interchanging $B_1$ and $B_2$ gives an isometric global space, the number of different global spaces obtained in this way (up to isometry) is exactly $2^{s-2}$, where $s=|\Sigma_B|$. 

Conversely, suppose that we are given a quaternion {\it division} algebra $B$ and a collection of local $B_v$-skew-hermitian spaces $V_v$ such that $\det(V_v)=1$ for all $v$ and the Hasse invariant of $V_v^\dagger$ (for $v\not \in \Sigma_B$) is 
$1$ for all but finitely many $v$. Then there are up to isometry $2^{s-2}$ different global skew-hermitian spaces that give rise to this collection of local spaces, and all of them may be obtained by the construction above, by suitably choosing $B_1$, $B_2$ and $\i$.

%% file: theta-correspondence.tex
\section{Theta correspondences}
\label{sec:theta}

\subsection{Preliminaries}
\label{subsec:theta-prelim}

\subsubsection{Weil indices}

Let $F$ be a local field of characteristic not $2$ and fix a non-trivial additive character $\psi$ of $F$.
For a non-degenerate symmetric $F$-bilinear form $q$,
we let $\gamma_F(\psi \circ q) \in \mu_8$ denote the Weil index
associated to the character of second degree $x \mapsto \psi(q(x,x))$
(see \cite{weil}, \cite[Appendix]{rangarao}).
When $q(x,y) = xy$ for $x, y \in F$, we write $\gamma_F(\psi) := \gamma_F(\psi \circ q)$.
Put 
\[
 \gamma_F(a, \psi) := \frac{\gamma_F(a \psi)}{\gamma_F(\psi)}
\]
for $a \in F^{\times}$,
where $(a \psi)(x) := \psi(ax)$ for $x \in F$.
Then we have
\begin{align*}
 \gamma_F(ab^2, \psi) & = \gamma_F(a, \psi), \\
 \gamma_F(a b, \psi) & = \gamma_F(a, \psi) \cdot \gamma_F(b, \psi)
 \cdot (a,b)_F, \\
 \gamma_F(a, b \psi) & = \gamma_F(a,\psi) \cdot (a,b)_F, \\
 \gamma_F(a, \psi)^2 & = (-1, a)_F, \\
 \gamma_F(a, \psi)^4 & = 1, \\
 \gamma_F(\psi)^2 & = \gamma_F(-1, \psi)^{-1}, \\
 \gamma_F(\psi)^8 & = 1
\end{align*}
for $a,b \in F^{\times}$ (see \cite[p.~367]{rangarao}).
Here $(\cdot,\cdot)_F$ is the quadratic Hilbert symbol of $F$.

Let $q$ be a non-degenerate symmetric $F$-bilinear form.
Let $\det q \in F^{\times} / (F^{\times})^2$ and $h_F(q) \in \{ \pm 1\}$
denote the determinant and the Hasse invariant of $q$.
For example, when  
\[
 q(x,y) = a_1 x_1 y_1 + \cdots + a_m x_m y_m
\]
for $x=(x_1, \ldots, x_m)$, $y = (y_1, \ldots, y_m) \in F^m$, then
\[
 \det q = \prod_{1 \le i \le m} a_i, \qquad
 h_F(q) = \prod_{1 \le i<j \le m} (a_i, a_j)_F.
\]
Moreover, we have
\begin{equation}
\label{eqn:weil_index}
 \gamma_F(\psi \circ q) = 
 \gamma_F(\psi)^m \cdot \gamma_F(\det q, \psi) \cdot h_F(q)
\end{equation}
(see \cite[pp.~367--368]{rangarao}).

\subsubsection{Leray invariants}

Let $\V$ be a $2n$-dimensional $F$-vector space with a non-degenerate symplectic form $\llangle \cdot, \cdot \rrangle : \V \times \V \rightarrow F$.
For maximal isotropic subspaces $\Y$, $\Y'$, $\Y''$ of $\V$,
the Leray invariant $q(\Y,\Y',\Y'')$ is a non-degenerate symmetric $F$-bilinear form defined as follows.
(See also Definitions 2.4 and 2.10 of \cite{rangarao}.)

Suppose first that $\Y$, $\Y'$, $\Y''$ are pairwise transverse.
Let $P_{\Y}$ be the maximal parabolic subgroup of $\Sp(\V)$ stabilizing $\Y$
and let $N_{\Y}$ be its unipotent radical.
By Lemma 2.3 of \cite{rangarao}, there exists a unique $g \in N_{\Y}$ such that $\Y' g = \Y''$.
We write
\[
 g = \begin{pmatrix} 1 & b \\ & 1 \end{pmatrix},
 \qquad b \in \Hom(\Y', \Y)
\]
with respect to the complete polarization $\V = \Y' + \Y$.
Then $q = q(\Y,\Y',\Y'')$ is a symmetric bilinear form on $\Y'$ defined by
\[
 q(x', y') : = \llangle x', y' b \rrangle.
\]

In general, we consider a symplectic space $\V_{\R} := \R^{\perp} / \R$,
where 
\[
 \R := (\Y \cap \Y') + (\Y' \cap \Y'') + (\Y'' \cap \Y),
\]
and maximal isotropic subspaces
\[
 \Y_{\R}:= (\Y \cap \R^\perp) / \R, \qquad
 \Y'_{\R}:= (\Y' \cap \R^\perp) / \R, \qquad
 \Y''_{\R}:= (\Y'' \cap \R^\perp) / \R
\]
of $\V_{\R}$.
By Lemma 2.9 of \cite{rangarao}, $\Y_{\R}$, $\Y'_{\R}$, $\Y''_{\R}$ are pairwise transverse.
We put
\[
 q(\Y,\Y',\Y'') := q(\Y_{\R},\Y'_{\R},\Y''_{\R}).
\]

By Theorem 2.11 of \cite{rangarao}, we have
\[
 q(\Y g,\Y' g,\Y'' g) = q(\Y,\Y',\Y'')  
\]
for $g \in \Sp(\V)$.

\subsection{\texorpdfstring{Weil representation for $\Mp$}{Weil representation for Mp} } 
\label{subsec:theta-weil}

\subsubsection{Heisenberg group, Heisenberg representations}

Let $F$ be a local field of characteristic not $2$.
For simplicity, we assume that $F$ is non-archimedean.

Let $\V$ be a finite dimensional vector space over $F$ equipped with 
a non-degenerate symplectic form
\[
 \llangle \cdot, \cdot \rrangle : \V \times \V \longrightarrow F.
\]
The Heisenberg group $H(\V)$ is defined by
$$ H(\V):= \V \oplus F$$
as a set, with group law
$$ (v_1, z_1) \cdot (v_2, z_2) = \left(v_1+v_2, z_1 + z_2 + \frac{1}{2} \llangle v_1, v_2 \rrangle \right).$$
The center of $H(\V)$ is $F$. 

Let $\psi$ be a nontrivial additive character of $F$. Then by the Stone--von Neumann theorem, $H(\V)$ admits a unique (up to isomorphism) irreducible representation $\rho_\psi$ on which $F$ acts via $\psi$. This representation can be realized as follows.
Fix a complete polarization
$$\V=\X \oplus \Y,$$
i.e., $\X$ and $\Y$ are maximal isotropic subspaces of $\V$.
We construct a character $\psi_{\Y}$ of $H(\Y)=\Y\oplus F$ by setting 
$$\psi_\Y (y,z) = \psi (z).$$
Define 
$$S_\Y := \Ind_{H(\Y)}^{H(\V)} \psi_\Y.$$
i.e. 
$ S_\Y$ is the space of functions $f: H(\V) \rightarrow \C$ satisfying
\begin{enumerate}
 \item $f(\tilde{y}\tilde{v}) = \psi_\Y (\tilde{y}) f(\tilde{v})$ for $\tilde{y} \in H(\Y)$, $\tilde{v} \in H(\V)$.
 \item  $f$ is smooth i.e. there exists an open compact subgroup (a lattice !) $L$ in $\V$ such that 
$$ f( \tilde{v} \ell) = f(\tilde{v}) \quad \text{for all} \ \ell \in L\subset \V \subset H(\V).$$
\end{enumerate}
Then $H(\V)$ acts on $S_\Y$ on the right naturally. 
We can identify $S_\Y$ with $\SS (\X)$, the Schwartz space of locally constant functions with compact support on $\X$, via the restriction of functions to $\X$. 
\subsubsection{Metaplectic group}
\label{sssec:theta-metaplectic}

Let $\Sp(\V)$ be the symplectic group of $\V$.
Following Weil, we let $\Sp(\V)$ act on $\V$ on the right.
Recall that $\Sp(\V)$ acts on $H(\V)$ by
\[
 (v, z)^g :=  (vg, z).
\]

Let $\widetilde{\Sp}(\V)$ be the unique non-trivial $2$-fold central extension
of $\Sp(\V)$.
The metaplectic group $\Mp(\V)$ is a central extension 
\[
 1 \longrightarrow \C^1 \longrightarrow \Mp(\V)
 \longrightarrow \Sp(\V) \longrightarrow 1
\]
defined by
\[
 \Mp(\V) := \widetilde{\Sp}(\V) \times_{\{ \pm 1 \}} \C^1.
\]

\begin{lem}
\label{lem:aut_Mp_triv}
Any automorphism of $\Mp(\V)$ which lifts the identity map of $\Sp(\V)$
and which restricts to the identity map of $\C^1$
must be the identity map of $\Mp(\V)$.
\end{lem}

\begin{proof}
Let $p: \Mp(\V) \rightarrow \Sp(\V)$ be the projection.
Let $f: \Mp(\V) \rightarrow \Mp(\V)$ be such an automorphism.
Since $p(f(g)) \cdot p(g)^{-1} = 1$ for $g \in \Mp(\V)$, 
there exists a character $\kappa : \Mp(\V) \rightarrow \C^1$ such that
$f(g) \cdot g^{-1} = \kappa(g)$.
Since $f(z) = z$ for $z \in \C^1$, $\kappa$ is trivial on $\C^1$,
and hence induces a character of $\Sp(\V)$.
Since $[\Sp(\V), \Sp(\V)] = \Sp(\V)$, this character must be trivial.
\end{proof}

One can realize $\Mp(\V)$ explicitly as follows.
Put 
\[
 z_{\Y}(g_1, g_2) = \gamma_F(\frac{1}{2} \psi \circ q(\Y, \Y g_2^{-1}, \Y g_1)) 
\]
for $g_1, g_2 \in \Sp(\V)$.
By Theorem 4.1 of \cite{rangarao}, $z_{\Y}$ is a $2$-cocycle,
and the group
\[
 \Mp(\V)_{\Y} := \Sp(\V) \times \C^1
\]
with group law
\[
 (g_1, z_1) \cdot (g_2, z_2) = (g_1 g_2, z_1 z_2 \cdot z_{\Y}(g_1, g_2))
\]
is isomorphic to $\Mp(\V)$.
Moreover, by Lemma \ref{lem:aut_Mp_triv}, this isomorphism is canonical.
If there is no confusion, we identify $\Mp(\V)_{\Y}$ with $\Mp(\V)$.

Let $\V = \X' \oplus \Y'$ be another complete polarization.

\begin{lem}
\label{lem:compare-2cocycles}
We have
\[
 z_{\Y'}(g_1, g_2) = 
 \lambda(g_1 g_2) \lambda(g_1)^{-1} \lambda(g_2)^{-1}
 \cdot z_{\Y}(g_1, g_2),
\]
where $\lambda: \Sp(\V) \rightarrow \C^1$ is given by
\[
 \lambda(g) :=
 \gamma_F(\frac{1}{2} \psi \circ q(\Y, \Y' g^{-1}, \Y'))
 \cdot \gamma_F(\frac{1}{2} \psi \circ q(\Y, \Y', \Y g)).
\]
In particular, the bijection
\begin{align*}
 \Mp(\V)_{\Y} & \longrightarrow \Mp(\V)_{\Y'} \\
 (g, z) & \longmapsto (g, z \cdot \lambda(g))
\end{align*}
is an isomorphism.
\end{lem}

\begin{proof}
See Lemma 4.2 of \cite{kudla-splitting}.
\end{proof}

Suppose that $\V = \V_1 \oplus \V_2$, where each $\V_i$ is a non-degenerate symplectic subspace.
One can lift the natural embedding $\Sp(\V_1) \times \Sp(\V_2) \hookrightarrow \Sp(\V)$ to a homomorphism
\[
 \Mp(\V_1) \times \Mp(\V_2) \longrightarrow \Mp(\V).
\]
If $\V_i = \X_i \oplus \Y_i$ is a complete polarization and 
\[
 \X = \X_1 \oplus \X_2, \qquad \Y = \Y_1 \oplus \Y_2,
\]
then this homomorphism is given by 
\begin{align*}
 \Mp(\V_1)_{\Y_1} \times \Mp(\V_2)_{\Y_2} & \longrightarrow \Mp(\V)_{\Y}, \\
 \left((g_1, z_1), (g_2, z_2) \right)
 & \longmapsto (g_1 g_2, z_1 z_2)
\end{align*}
i.e., we have
\[
 z_{\Y_1}(g_1, g_1') \cdot  z_{\Y_2}(g_2, g_2') = z_{\Y}(g_1 g_2, g_1' g_2')
\]
for $g_i, g_i' \in \Sp(\V_i)$ (see Theorem 4.1 of \cite{rangarao}).

Let $L$ be a self-dual lattice in $\V$ and 
let $K$ be the stabiliser of $L$ in $\Sp(\V)$.
If the \emph{residual} characteristic of $F$ is not $2$, then there exists a splitting 
\[
 \xymatrix{
  & \Mp(\V) \ar@{->}[d] \\
  K \ar@{->}[r] \ar@{->}[ur] & \Sp(\V)}.
\]
Moreover, if the residue field of $F$ has at least four elements, 
then $[K, K] = K$
(see Lemma 11.1 of \cite{moore}),
and hence such a splitting is unique.
In the next section, we shall describe this splitting by using the Schr\"odinger model.

\subsubsection{Weil representation, Schr\"odinger model}
\label{sssec:theta-weil}

Recall that $\rho_{\psi}$ is the unique (up to isomorphism) 
irreducible smooth representation of $H(\V)$ with central character $\psi$.
Let $S$ be the underlying space of $\rho_{\psi}$.
The Weil representation $\omega_{\psi}$ of $\Mp(\V)$ on $S$
is a smooth representation characterized by the following properties:
\begin{itemize}
 \item $\rho_{\psi}(h^g) = \omega_{\psi}(g)^{-1} \rho_{\psi}(h) \omega_{\psi}(g)$ for all $h \in H(\V)$ and $g \in \Mp(\V)$.
 \item $\omega_{\psi}(z) = z \cdot \id_S$ for all $z \in \C^1$.
\end{itemize}

One can realize $\omega_{\psi}$ explicitly as follows.
We regard $\V = F^{2n}$ as the space of row vectors.
Fix a complete polarization $\V = \X \oplus \Y$.
Choose a basis $\e_1, \ldots, \e_n, \e_1^*, \ldots, \e_n^*$ of $\V$ such that
\[
 \X = F \e_1 + \cdots + F \e_n, \qquad 
 \Y = F \e_1^* + \cdots + F \e_n^*, \qquad 
 \llangle \e_i, \e_j^* \rrangle = \delta_{ij}.
\]
Then we have
\[
 \Sp(\V) = \left\{ g \in \GL_{2n}(F) \, | \,
 g \begin{pmatrix} & \1_n \\ -\1_n & \end{pmatrix}
 {}^t g = \begin{pmatrix} & \1_n \\ -\1_n & \end{pmatrix} \right\}.
\]
The Weil representation $\omega_{\psi}$ of $\Mp(\V)_{\Y}$ on the Schwartz space $\SS(\X)$ is given as follows:
\begin{align*}
 \omega_{\psi} \left( \begin{pmatrix} a & \\ & {}^t a^{-1} \end{pmatrix}, z \right) \varphi(x) & = 
 z \cdot |\det a|^{1/2} \cdot \varphi(xa) \\
 \omega_{\psi} \left( \begin{pmatrix} \1_n & b \\ & \1_n \end{pmatrix}, z \right) \varphi(x)
 & = z \cdot \psi \left(\frac{1}{2} x b {}^t x \right) \cdot \varphi(x) \\
 \omega_{\psi} \left( \begin{pmatrix} & \1_n \\ -\1_n \end{pmatrix}, z \right) \varphi(x)
 & = z \cdot \int_{F^n} \varphi(y) \psi(x{}^t y) \, dy.
\end{align*}
for $\varphi \in \SS(\X)$, $x \in \X \cong F^n$,
$a \in \GL(\X) \cong \GL_n(F)$,
$b \in \Hom(\X, \Y) \cong \M_n(F)$ with ${}^t b = b$,
and $z \in \C^1$.
More generally, for $(g, z) \in \Mp(\V)_{\Y}$ with $g = \smat{a}{b}{c}{d} \in \Sp(\V)$, we have
\[
 \omega_{\psi}(g, z) \varphi(x) = z \cdot 
 \int_{F^n / \ker(c)}
 \psi \left( \frac{1}{2}(xa, xb) + (xb, yc) + \frac{1}{2}(yc, yd) \right)
 \varphi(xa + yc) \, d \mu_g(y),
\]
where $(x, y) = x {}^t y$ for row vectors $x, y \in F^n$, and the measure $d \mu_g(y)$ on $F^n / \ker(c)$ is normalized so that this operator is unitary (see \cite[Proposition 2.3]{kudla-note}).
In particular, if $\det c \ne 0$, then $\omega_{\psi}(g, 1) \varphi(x)$ is equal to 
\begin{align*}
 & \int_{F^n}
 \psi \left( \frac{1}{2}(xa, xb) + (xb, yc) + \frac{1}{2}(yc, yd) \right)
 \varphi(xa + yc) \, d \mu_g(y) \\
 & = |\det c|^{-1} \cdot \psi \left( \frac{1}{2}(xa, xb) \right)
 \cdot \int_{F^n}
 \psi \left( (xb, y - xa) + \frac{1}{2}(y - xa, yc^{-1}d - xac^{-1} d) \right)
 \varphi(y)  \, d \mu_g(y) \\
 & = |\det c|^{-1} \cdot \psi \left( \frac{1}{2}(xa, x(ac^{-1}d - b)) \right) 
 \cdot \int_{F^n}
 \psi \left( (y, x(b - a c^{-1} d)) + \frac{1}{2}(y, yc^{-1}d) \right)
 \varphi(y)  \, d \mu_g(y) \\
 & = |\det c|^{-1} \cdot \psi \left( \frac{1}{2}(xa, x {}^t c^{-1}) \right)
 \cdot \int_{F^n}
 \psi \left( - (y, x {}^t c^{-1}) + \frac{1}{2}(y, yc^{-1}d) \right)
 \varphi(y)  \, d \mu_g(y) \\
 & = |\det c|^{-1/2} \cdot \psi \left( \frac{1}{2}(xa c^{-1}, x) \right)
 \cdot \int_{F^n}
 \psi \left( - (x {}^t c^{-1}, y) + \frac{1}{2}(yc^{-1}d, y) \right)
 \varphi(y)  \, d y,
\end{align*}
where $dy$ is the self-dual Haar measure on $F^n$ with respect to the pairing $\psi \circ (\cdot,\cdot)$.

If the residual characteristic of $F$ is not $2$,
let $K$ be the stabilizer of the self-dual lattice
\[
 \o \e_1 + \cdots + \o \e_n + \o \e_1^* + \cdots + \o \e_n^*.
\]
Then the splitting $K \rightarrow \Mp(\V)_{\Y}$ is given as follows.
Assume that $\psi$ is of order zero.
Let $\varphi^0 \in \SS(\X)$ be the characteristic function of
$\o \e_1 + \cdots + \o \e_n \cong \o^n$.
Since the residual characteristic of $F$ is not $2$, we see that 
\[
 \omega_{\psi}(k,1) \varphi^0 = \varphi^0
\]
for 
\[
 k =  \begin{pmatrix} a & \\ & {}^t a^{-1} \end{pmatrix}, \
\begin{pmatrix} \1_n & b \\ & \1_n \end{pmatrix} \ \text{and} \
\begin{pmatrix} & \1_n \\ -\1_n \end{pmatrix},
\]
where $a \in \GL_n(\o)$ and $b \in \M_n(\o)$.
Since these elements generate $K$, there exists a function $s_\Y:K \rightarrow \C^1$ such that
\[
 \omega_{\psi}(k,1) \varphi^0 = s_\Y(k)^{-1} \cdot 
 \varphi^0 \quad \text{for all} \ k \in K.
\]
Thus we obtain
\[
 z_{\Y}(k_1, k_2) = s_\Y(k_1 k_2) s_\Y(k_1)^{-1} s_\Y(k_2)^{-1}
\]
for $k_1, k_2 \in K$,
so that the map $k \mapsto (k, s_\Y (k))$ is the desired splitting.

\subsubsection{Change of models}

Suppose $\V= \X'\oplus \Y'$ is another polarization of $\V$. Then likewise the representation $\rho_\psi$ can be realized on $S_{\Y'} \simeq \SS (\X')$.
We will need an explicit isomorphism between these representations of $H(\V)$.
At the level of the induced representations, this is given by the map
$$ S_{\Y'} \rightarrow S_{\Y}, \qquad f' \mapsto f $$
$$ f(\tilde{v}) = \int_{\Y\cap \Y' \backslash \Y} f'((y,0) \cdot \tilde{v}) \cdot \psi_{\Y} (y,0)^{-1} \, dy =  \int_{\Y \cap \Y' \backslash \Y} f'((y,0) \cdot \tilde{v}) \, dy.$$

For now, we will take any Haar measure on $\Y$ to define this. We will fix this more carefully later. Let us now write down this isomorphism in terms of Schwartz spaces. 

\begin{lem}
Suppose that $\varphi \in \SS(\X)$ and $\varphi' \in \SS(\X')$
correspond to $f \in S_{\Y}$ and $f' \in S_{\Y'}$ respectively.
Then we have
\begin{equation}
\label{eqn:partial-fourier}
 \varphi(x) = \int_{\Y\cap \Y' \backslash \Y}
 \psi \left( \frac{1}{2} \llangle x', y' \rrangle - \frac{1}{2} \llangle x, y \rrangle \right) \varphi'(x') \, dy, 
\end{equation}
where $x' = x'(x,y) \in \X'$ and $y' = y'(x,y) \in \Y'$
are given by $x'+y' = x+y \in \V$.
\end{lem}

\begin{proof}
Let $\varphi' \in \SS (\X')$. Let $(x'+y',z) \in H(\V)$. We write this as:
$$ (x'+y',z) = \left(y', z- \frac{1}{2} \llangle y',x' \rrangle \right) \cdot (x',0).$$ 
Thus if $f' \in S_{\Y'}$ corresponds to $\varphi'$, then 
$$ f' (x'+y',z) = \psi \left( z- \frac{1}{2} \llangle y',x' \rrangle \right) \cdot \varphi'(x').$$
We can rewrite this as: (with $v= x'+y'$)
$$ f' (v,z) = \psi \left( z- \frac{1}{2} \llangle v,x' \rrangle \right) \cdot \varphi'(x').$$
Thus $f'$ corresponds to $f \in S_\Y$ given by
$$ f (x+y,z) =  \int_{\Y\cap \Y' \backslash \Y} f'((y_0, 0) \cdot (x+y,z) ) \, dy_0= \int_{\Y\cap \Y' \backslash \Y} f'\left(x+y+y_0, z+ \frac{1}{2} \llangle y_0,x \rrangle \right) dy_0 .$$
Thus 
\begin{align*}
 \varphi(x) &= \int_{\Y\cap \Y' \backslash \Y} f'\left(x+y_0, \frac{1}{2} \llangle y_0,x \rrangle \right)dy_0 \\ &=  \int_{\Y\cap \Y' \backslash \Y}
 \psi\left(  \frac{1}{2} \llangle y_0,x \rrangle \right)
 f'\left(x+y_0, 0\right) \, dy_0 \\
&=  \int_{\Y\cap \Y' \backslash \Y}  \psi \left( -\frac{1}{2} \llangle x+y_0,x' \rrangle + \frac{1}{2} \llangle y_0,x \rrangle \right) \varphi'(x') \, dy_0. 
\end{align*}
\end{proof}

Thus we obtain an $H(\V)$-equivariant isomorphism
$\SS (\X') \cong \SS(\X)$ defined by the partial Fourier transform
\eqref{eqn:partial-fourier}.
Using the characterization of the Weil representation of $\Mp(\V)$,
one sees that this isomorphism is also $\Mp(\V)$-equivariant.

The isomorphism $\SS (\X') \cong \SS(\X)$ is in fact a partial Fourier transform.
To see this, using \cite[Lemma 2.2]{rangarao}, we choose a basis $\e_1, \ldots, \e_n, \e_1^*, \ldots, \e_n^*$ of $\V$ such that 
\begin{align*}
 \X & = F \e_1 + \cdots + F \e_n, & 
 \X'& = F \e_1 + \cdots + F \e_k + F \e^*_{k+1} + \cdots + F \e^*_n, \\
 \Y & = F \e^*_1 + \cdots + F \e^*_n, & 
 \Y'& = F \e^*_1 + \cdots + F \e^*_k + F \e_{k+1} + \cdots + F \e_n,
\end{align*}
and $\llangle \e_i, \e_j^* \rrangle = \delta_{ij}$, where $k = \dim (\Y \cap \Y')$.
In particular, we have $\Y \cap \Y' = F \e^*_1 + \cdots + F \e^*_k$.
Let $\varphi \in \SS(\X)$ and $\varphi' \in \SS(\X')$ be as in \eqref{eqn:partial-fourier}.
We also regard $\varphi'$ as a function on $F^n$ via the basis $\e_1, \ldots, \e_k, \e^*_{k+1}, \ldots, \e^*_n$.
Write $x + y = x' + y'$ with $x \in \X$, $y \in \Y$, $x' \in \X'$, $y' \in \Y'$.
If we write 
\[
 x = x_1 \e_1 + \cdots + x_n \e_n, \qquad 
 y = y_1 \e^*_1 + \cdots + y_n \e^*_n
\]
with $x_i, y_j \in F$, then 
\begin{align*}
 x' & = x_1 \e_1 + \cdots + x_k \e_k + y_{k+1} \e^*_{k+1} + \cdots + y_n \e^*_n,\\
 y' & = y_1 \e^*_1 + \cdots + y_k \e^*_k + x_{k+1} \e_{k+1} + \cdots + x_n \e_n,
\end{align*}
and 
\[
 \llangle x, y \rrangle = x_1 y_1 + \cdots + x_n y_n, \qquad
 \llangle x', y' \rrangle = x_1 y_1 + \cdots + x_k y_k - x_{k+1} y_{k+1} - \cdots - x_n y_n.
\]
Hence we have
\begin{align*}
 \varphi(x) & = \int_{\Y \cap \Y' \backslash \Y}
 \psi \left( \frac{1}{2} \left( \llangle x', y' \rrangle - \llangle x, y \rrangle \right) \right)
 \varphi'(x') \, dy \\
 & = \int_{F^{n-k}}
 \psi \left( - x_{k+1} y_{k+1} - \cdots - x_n y_n \right)
 \varphi'(x_1, \ldots, x_k, y_{k+1}, \ldots, y_n) \, dy_{k+1} \cdots \, dy_n.
\end{align*}

\subsubsection{Over global fields}

In this section, let $F$ be a number field with ring of adeles $\A$.
Let $\V$ be a symplectic space over $F$.
The global metaplectic group $\Mp(\V)_{\A}$ is defined as follows.

Fix a lattice $L$ in $\V$.
For each finite place $v$, let $K_v$ be the stabilizer of $L_v$ in $\Sp(\V_v)$.
For almost all $v$, $L_v$ is self-dual and there exists a unique splitting
$K_v \hookrightarrow \Mp(\V_v)$, in which case we identify $K_v$ with its image in $\Mp(\V_v)$.

For a finite set $S$ of places of $F$ including all archimedean places,
we define a central extension
\[
 1 \longrightarrow \C^1 \longrightarrow \Mp(\V)_S
 \longrightarrow \prod_{v \in S} \Sp(\V_v) \longrightarrow 1
\]
by 
\[
 \Mp(\V)_S := \left( \prod_{v \in S} \Mp(\V_v)  \right) /
 \left\{ (z_v) \in \prod_{v \in S} \C^1 \, | \, \prod_{v \in S} z_v = 1 \right\}.
\]
Put $K^S := \prod_{v \notin S} K_v$.
If $S \subset S'$ are sufficiently large,
the splitting $K_v \hookrightarrow \Mp(\V_v)$ induces an embedding
\[
 \Mp(\V)_S \times K^S \hookrightarrow \Mp(\V)_{S'} \times K^{S'}.
\]
Then $\Mp(\V)_{\A}$ is defined by
\[
 \Mp(\V)_{\A} := \varinjlim_{S} \left(  \Mp(\V)_S \times K^S \right).
\]

There exists a unique splitting
\[
 \xymatrix{
  & \Mp(\V)_{\A} \ar@{->}[d] \\
 \Sp(\V)(F) \ar@{->}[r] \ar@{->}[ur] & \Sp(\V)(\A)}
\]
given as follows.
Fix a complete polarization $\V = \X \oplus \Y$ over $F$.
Recall that the metaplectic group $\Mp(\V_v) = \Sp(\V_v) \times \C^1$
is determined by the $2$-cocycle $z_{\Y_v}$.
Moreover, for almost all $v$, there exists a function $s_{\Y_v} : K_v \rightarrow \C^1$ such that
\[
 z_{\Y_v}(k_1, k_2) = s_{\Y_v}(k_1 k_2) s_{\Y_v}(k_1)^{-1} s_{\Y_v}(k_2)^{-1}
\]
for $k_1, k_2 \in K_v$.

\begin{lem}
\label{lem:s_trivial}
Let $\gamma \in \Sp(\V)(F)$.
Then we have
\[
 s_{\Y_v}(\gamma) = 1
\]
for almost all $v$.
\end{lem}

\begin{proof}
By the Bruhat decomposition, we may write $\gamma = p_1 w p_2$ with some
\[
 p_i = \begin{pmatrix} a_i & b_i \\ & {}^t a_i^{-1} \end{pmatrix}, \qquad
 w = \begin{pmatrix} & & \1_k & \\ & \1_{n-k} & & \\ -\1_k & & & \\ & & & \1_{n-k} \end{pmatrix},
\]
where $a_i \in \GL_n(F)$ and $b_i \in \M_n(F)$.
By Theorem 4.1 of \cite{rangarao}, we have
\[
 z_{\Y_v}(p_1, g) = z_{\Y_v}(g, p_2) = 1
\]
for all $v$ and $g \in \Sp(\V_v)$, so that 
\[
 (p_1 w p_2, 1) = (p_1, 1) \cdot (w, 1) \cdot (p_2, 1) 
 \quad \text{in} \ \Mp(\V_v).
\]
On the other hand, for almost all $v$, we have $p_i \in K_v$ and
\[
 \omega_{\psi_v}(p_i, 1) \varphi^0_v = \varphi^0_v, \qquad 
 \omega_{\psi_v}(w, 1) \varphi^0_v = \varphi^0_v,
\]
where $\psi_v$ is a non-trivial character of $F_v$ of order zero
and $\varphi^0_v$ is the characteristic function of $\o_v^n$.
Thus we obtain
\[
 \omega_{\psi_v}(\gamma, 1) \varphi^0_v = 
 \omega_{\psi_v}(p_1, 1) \omega_{\psi_v}(w, 1) \omega_{\psi_v}(p_2, 1) \varphi^0_v = \varphi^0_v
\]
for almost all $v$.
\end{proof}

For $\gamma \in \Sp(\V)(F)$, let $(\gamma, 1)$ be the element in $\prod_v \Mp(\V_v)$ such that $(\gamma, 1)_v = (\gamma, 1)$ for all $v$.
By Lemma \ref{lem:s_trivial}, we have $(\gamma, 1)_v \in K_v$ for almost all $v$.
Hence, if $S$ is a sufficiently large finite set of places of $F$,
then $(\gamma, 1)$ maps to an element $i(\gamma)$ in $\Mp(\V)_S \times K^S$.

\begin{lem}
The map
\[
 i: \Sp(\V)(F) \longrightarrow \Mp(\V)_{\A}
\] 
is a homomorphism.
\end{lem}

\begin{proof}
Let $\gamma_1, \gamma_2 \in \Sp(\V)(F)$.
For each $v$, we have
\[
 (\gamma_1, 1)_v \cdot  (\gamma_2, 1)_v
 = (\gamma_1 \gamma_2, z_{\Y_v}(\gamma_1, \gamma_2))
 \quad \text{in} \ \Mp(\V_v).
\]
Choose a finite set $S$ of places of $F$ such that
\[
 \gamma_1, \gamma_2 \in K_v, \qquad s_{\Y_v}(\gamma_1) = s_{\Y_v}(\gamma_2) = s_{\Y_v}(\gamma_1 \gamma_2) = 1
\]
for $v \notin S$.
Then we have
\[
 z_{\Y_v}(\gamma_1, \gamma_2) = 1
\]
for $v \notin S$.
Moreover, by the product formula for the Weil index, we have
\[
 \prod_{v \in S} z_{\Y_v}(\gamma_1, \gamma_2) = 1.
\]
Hence the image of $(\gamma_1, 1) \cdot (\gamma_2, 1)$
in $\Mp(\V)_S \times K^S$ is equal to $i(\gamma_1 \gamma_2)$.
\end{proof}

Fix a non-trivial additive character $\psi$ of $\A/F$.
We have the global Weil representation $\omega_{\psi}$ of $\Mp(\V)_{\A}$
on the Schwartz space $\SS(\X(\A))$.
For each $\varphi \in \SS(\X(\A))$, the associated theta function 
on $\Mp(\V)_{\A}$ is defined by
\[
 \Theta_{\varphi}(g) := \sum_{x \in \X} \omega_{\psi}(g) \varphi(x).
\]
Then $\Theta_{\varphi}$ is a left $\Sp(\V)(F)$-invariant
slowly increasing smooth function on $\Mp(\V)_{\A}$.

\subsection{Reductive dual  pairs}
\label{subsec:theta-dualpairs}

In this section, we consider the reductive dual pair $(\GU(V), \GU(W))$, where $V$ is a skew-Hermitian right $B$-space of dimension two and $W$ is a Hermitian left $B$-space of dimension one. 

\subsubsection{Reductive dual pairs; examples}
\label{sssec:theta-dualpairs-examples}

Recall that in \S \ref{subsec:key}, we have constructed the $2$-dimensional skew-Hermitian right $B$-space $V = B_1 \otimes_E B_2$ with the skew-Hermitian form $\langle \cdot, \cdot \rangle : V \times V \rightarrow B$.
Let $W = B$ be the $1$-dimensional Hermitian left $B$-space with the Hermitian form $\langle \cdot, \cdot \rangle : W \times W \rightarrow B$ defined by 
\[
 \langle x, y \rangle = x y^*.
\]
These forms satisfy that 
\begin{align*}
 \langle v \ba, v' \bb \rangle & = \ba^* \langle v, v' \rangle \bb &
 \langle v', v \rangle  & = - \langle v, v' \rangle^* \\
 \langle \ba w, \bb w' \rangle & = \ba \langle w, w' \rangle \bb^* &
 \langle w', w \rangle  & = \langle w, w' \rangle^* 
\end{align*}
for $v, v' \in V$, $w,w' \in W$ and $\ba, \bb \in B$.
We let $\GL(V)$ act on $V$ on the left and let $\GL(W)$ act on $W$ on the right.
Let $\GU(V)$ and $\GU(W)$ be the similitude groups of $V$ and $W$
with the similitude characters
$\nu:\GU(V) \rightarrow F^{\times}$ and $\nu:\GU(W) \rightarrow F^{\times}$
respectively:
\begin{align*}
 \GU(V) & := \{ g \in \GL(V) \, | \, \langle g v, gv' \rangle =
 \nu(g) \cdot  \langle v, v' \rangle \text{ for all } v, v' \in V \}, \\
 \GU(W) & := \{ g \in \GL(W) \, | \, \langle w g, w' g \rangle =
 \nu(g) \cdot  \langle w, w' \rangle \text{ for all } w, w' \in W \}.
\end{align*}
Let $\U(V) := \ker \nu$ and $\U(W) := \ker \nu$
be the unitary groups of $V$ and $W$ respectively.

Put
\[
 \V := V \otimes_B W.
\]
Then $\V$ is an $F$-space equipped with a symplectic form
\[
 \llangle \cdot, \cdot \rrangle :=
 \frac{1}{2} \tr_{B/F}
 \left( \langle \cdot, \cdot \rangle \otimes \langle \cdot, \cdot \rangle^* \right).
\]
If we identify $\Res_{B/F}(V)$ with $\V$ via the map $v \mapsto v \otimes 1$, 
then the associated symplectic form on $\V$ is given by 
\[
 \llangle \cdot, \cdot \rrangle =
 \frac{1}{2} \tr_{B/F} \langle \cdot, \cdot \rangle
 = \frac{1}{2} \tr_{E/F} ( \cdot, \cdot ),
\]
where $( \cdot, \cdot ) = \pr \circ \langle \cdot, \cdot \rangle$ is the associated $E$-skew-Hermitian form.
We let $\GL(V) \times \GL(W)$ act on $\V$ on the right:
\[
(v \otimes w) \cdot (g, h) := g^{-1} v \otimes w h.
\]
This gives a natural homomorphism
\[
 \G(\U(V) \times \U(W)) \longrightarrow \Sp(\V),
\]
where 
\[
 \G(\U(V) \times \U(W)) 
 := \{ (g,h) \in \GU(V) \times \GU(W) \, | \, \nu(g) = \nu(h) \}.
\] 

\subsubsection{Splittings}

Recall that
\[
 \V = \e_1 E + \e_2 E + \e_3 E + \e_4 E.
\]
Let $\V = \X + \Y$ be the complete polarization given by
\[
 \X = F \e_1 + F \e_2 + F \e_3 + F \e_4, \qquad
 \Y = F \e_1^* + F \e_2^* + F \e_3^* + F \e_4^*.
\]

We first suppose that $F$ is a local field.
In Appendix \ref{sec:spl-B}, we define a function 
\[
 s : \G(\U(V) \times \U(W))^0 \longrightarrow \C^1
\]
such that
\[
 z_{\Y}(g_1,g_2) = s(g_1 g_2) s(g_1)^{-1} s(g_2)^{-1},
\]
so that the map 
\begin{align*}
 \iota: \G(\U(V) \times \U(W))^0 & \longrightarrow \Mp(\V)_{\Y} \\
 g & \longmapsto (g, s(g))
\end{align*}
is a homomorphism.
Thus we have a commutative diagram
\[
 \xymatrix{
  & \Mp(\V) \ar@{->}[d] & \\
  \G(\U(V) \times \U(W))^0 \ar@{->}[ru]^{\iota}
  \ar@{->}[r] &
  \Sp(\V)}.
\]
If $\V = \X' + \Y'$ is another complete polarization, we choose $g_0 \in \Sp(\V)$ such that $\Y' = \Y g_0$ and define a function
\[
 s' : \G(\U(V) \times \U(W))^0 \longrightarrow \C^1
\]
by
\begin{align*}
 s'(g) & = s(g) \cdot z_{\Y}(g_0, g g_0^{-1}) \cdot z_{\Y}(g, g_0^{-1}) \\
 & = s(g) \cdot z_{\Y}(g_0 g g_0^{-1}, g_0)^{-1} \cdot z_{\Y}(g_0, g).
\end{align*}
By Lemma \ref{lem:compare-2cocycles}, we have
\[
 z_{\Y'}(g_1, g_2) = s'(g_1 g_2) s'(g_1)^{-1} s'(g_2)^{-1},
\]
so that the map
\begin{align*}
 \G(\U(V) \times \U(W))^0 & \longrightarrow \Mp(\V)_{\Y'} \\
 g & \longmapsto (g, s'(g))
\end{align*}
is a homomorphism.

We next suppose that $F$ is a number field.
For each place $v$ of $F$, we have defined a function 
\[
 s_v : \G(\U(V_v) \times \U(W_v))^0 \longrightarrow \C^1
\]
with associated homomorphism 
\[
 \iota_v :\G(\U(V_v) \times \U(W_v))^0 \longrightarrow \Mp(\V_v).
\]

\begin{lem}
The homomorphisms $\iota_v$ induce a homomorphism
\[
 \iota:\G(\U(V) \times \U(W))^0(\A) \longrightarrow \Mp(\V)_{\A}.
\]
Moreover, the diagram
\[
 \xymatrix{
 \G(\U(V) \times \U(W))^0(F) \ar@{->}[d] \ar@{^(->}[r] &
 \G(\U(V) \times \U(W))^0(\A) \ar@{->}[d]^{\iota}  \\
 \Sp(\V)(F) \ar@{->}[r]^i & \Mp(\V)_{\A}
 }
\]
is commutative.
\end{lem}

\begin{proof}
Recall that, for almost all $v$, $K_v$ is the maximal compact subgroup of $\Sp(\V_v)$ and $s_{\Y_v}:K_v \rightarrow \C^1$ is the function which defines the splitting
$K_v \hookrightarrow \Mp(\V_v)$.
Put
\[
 \KK_v := \G(\U(V_v) \times \U(W_v))^0 \cap K_v.
\]
Then $\KK_v$ is a maximal compact subgroup of $\G(\U(V_v) \times \U(W_v))^0$ for almost all $v$.
By Lemma \ref{lem:spl-B-K}, we have
\[
 s_v|_{\KK_v} = s_{\Y_v}|_{\KK_v}
\]
for almost all $v$.
Hence, for $g =(g_v) \in \G(\U(V) \times \U(W))^0(\A)$,
the element $(\iota_v(g_v)) \in \prod_v \Mp(\V_v)$ maps to
an element $\iota(g)$ in $\Mp(\V)_S \times K^S$ if $S$ is sufficiently large.
This proves the first assertion.

Let $\gamma \in \G(\U(V) \times \U(W))^0(F)$.
By Proposition \ref{prop:spl-B-prod}, we have
\[
 \prod_v s_v(\gamma) = 1. 
\]
Hence, if $S$ is sufficiently large,
the image of $(\iota_v(\gamma))$ in $\Mp(\V)_S \times K^S$ is equal to 
that of $(\gamma, 1)$.
This proves the second assertion.
\end{proof}

\subsubsection{Weil representation for the above dual reductive pair}

If $F$ is local, we get the Weil representation $\omega_\psi \circ \iota$ of $\G(\U(V) \times \U(W))^0$ on $\SS(\X)$, where $\omega_\psi$ is the Weil representation of $\Mp(\V)_{\Y}$ and $\iota: \G(\U(V) \times \U(W))^0 \rightarrow \Mp(\V)_{\Y}$ is the above homomorphism. 
Similarly, if $F$ is global, we get the global Weil representation $\omega_\psi \circ \iota$ of $\G(\U(V) \times \U(W))^0(\A)$ on $\SS(\X(\A))$.
If there is no confusion, we suppress $\iota$ from the notation.

%% file: rallis-B-general.tex
\section{The Rallis inner product formula and the Jacquet--Langlands correspondence}
\label{sec:rallis-general}

\subsection{Setup}
\label{subsec:rallissetup}

Let $F$ be a number field and $B$ a quaternion algebra over $F$.
As in Appendix \ref{sec:spl-double-B}, we consider the following spaces:
\begin{itemize}
\item $V = B_1 \otimes_E B_2$ is the $2$-dimensional right skew-hermitian $B$-space.
\item $W = B$ is the $1$-dimensional left hermitian $B$-space.
\item $W^{\square} = W \oplus W$ is the $2$-dimensional left hermitian $B$-space.
\item $\V = V \otimes_B W$ is the $8$-dimensional symplectic $F$-space.
\item $\V^{\square} = V \otimes_B W^{\square} = \V \oplus \V$ is the $16$-dimensional $F$-space. \\
\item $W^{\square} = W^{\bigtriangledown} \oplus W^{\triangle}$ is the complete polarization over $B$.
\item $\V = \X \oplus \Y$ is the complete polarization over $F$.
\item $\V_v = \X'_v \oplus \Y'_v$ is the complete polarization over $F_v$.
\item $\V^{\square} = \V^{\bigtriangledown} \oplus \V^{\triangle}$ is the complete polarization over $F$.
\item $\V^{\square} = \X^{\square} \oplus \Y^{\square}$ is the complete polarization over $F$.
\item $\V_v^{\square} = \X_v'^{\square} \oplus \Y_v'^{\square}$ is the complete polarization over $F_v$.
\end{itemize}
We have a natural map
\[
 \iota : \G(\U(W) \times \U(W)) \longrightarrow \GU(W^{\square})
\]
and a see-saw diagram
\[
 \xymatrix{
  \GU(W^{\square}) \ar@{-}[dr] \ar@{-}[d] &
  \G(\U(V) \times \U(V)) \ar@{-}[dl] \ar@{-}[d] \\
  \G(\U(W) \times \U(W)) & \GU(V)}.
\]

\subsubsection{Partial Fourier transform}

Fix a non-trivial character $\psi = \otimes_v \psi_v$ of $\A / F$.
Recall that $\e_1, \ldots, \e_4$ is a basis of $\X$ over $F$.
For each place $v$ of $F$, this basis and the self-dual measure on $F_v$ with respect to $\psi_v$ define a Haar measure $dx_v$ on $\X_v$.
Then the product measure $dx = \prod_v dx_v$ is the Tamagawa measure on $\X(\A)$.
We define a hermitian inner product $\langle \cdot, \cdot \rangle$ on $\SS(\X(\A))$ by
\[
 \langle \varphi_1, \varphi_2 \rangle := \int_{\X(\A)} \varphi_1(x) \overline{\varphi_2(x)} \, dx.
\]
Recall that $\V^{\square} = \V^{\bigtriangledown} \oplus \V^{\triangle} = \X^{\square} \oplus \Y^{\square}$.
We define a partial Fourier transform
\begin{align*}
\SS(\X^{\square}(\A)) & \longrightarrow \SS(\V^{\bigtriangledown}(\A)) \\
\varphi & \longmapsto \hat{\varphi}
\end{align*}
by
\[
 \hat{\varphi}(u) = \int_{(\V^{\triangle} \cap \Y^{\square} \backslash \V^{\triangle})(\A)} \varphi(x)
 \psi \left( \frac{1}{2} \left( \llangle x, y \rrangle - \llangle u, v \rrangle \right) \right) dv,
\]
where we write $u+v = x+y$ with $u \in \V^{\bigtriangledown}(\A)$, $v \in \V^{\triangle}(\A)$, $x \in \X^{\square}(\A)$, $y \in \Y^{\square}(\A)$, and $dv$ is the Tamagawa measure.

\begin{lem}[{\cite[p.~182, (13)]{li92}}]
If $\varphi = \varphi_1 \otimes \bar{\varphi}_2 \in \SS(\X^{\square}(\A))$ with $\varphi_i \in \SS(\X(\A))$, then we have
\[
 \hat{\varphi}(0) = \langle \varphi_1, \varphi_2 \rangle.
\]
\end{lem}

\begin{proof}
We include a proof for convenience.
Since $\V^{\triangle} \cap \Y^{\square} = \Y^{\triangle}$, we have
\[
 \hat{\varphi}(u) = \int_{\X^{\triangle}(\A)} \varphi(x)
 \psi \left( \frac{1}{2} \left( \llangle x, y \rrangle - \llangle u, v \rrangle \right) \right) dv.
\]
We write 
\[
 v = (v_0, v_0), \qquad u = (u_0, -u_0), \qquad u_0 = x_0 + y_0
\]
with $v_0, x_0 \in \X(\A)$ and $y_0 \in \Y(\A)$, so that 
\[
 x = (v_0 + x_0, v_0 - x_0), \qquad y = (y_0, - y_0).
\]
We have
\[
 \llangle x, y \rrangle
 = \llangle v_0 + x_0, y_0 \rrangle - \llangle v_0 - x_0, -y_0 \rrangle
 = 2 \llangle v_0, y_0 \rrangle, \qquad
 \llangle u, v \rrangle = 2 \llangle u_0, v_0 \rrangle = 2 \llangle y_0, v_0 \rrangle,
\]
and hence 
\[
 \hat{\varphi}(u) = \int_{\X(\A)} \varphi(v_0+x_0, v_0-x_0)
 \psi (2 \llangle v_0, y_0 \rrangle) \, dv_0,
\]
where $dv_0$ is the Tamagawa measure on $\X(\A)$.
In particular, we have
\[
 \hat{\varphi}(0) = \int_{\X(\A)} \varphi(v_0, v_0) \, dv_0.
\]
\end{proof}

For each place $v$ of $F$, we define a hermitian inner product $\langle \cdot, \cdot \rangle$ on $\SS(\X_v)$ with respect to the Haar measure $dx_v$ on $\X_v$ given above.
Fix a Haar measure on $\X'_v$ and define a hermitian inner product $\langle \cdot, \cdot \rangle$ on $\SS(\X'_v)$ similarly.
For $\varphi' \in \SS(\X'_v)$, we define its partial Fourier transform $\varphi \in \SS(\X_v)$ by
\[
 \varphi(x) = \int_{\Y_v \cap \Y'_v \backslash \Y_v} \varphi'(x')
 \psi_v \left( \frac{1}{2} \left( \llangle x', y' \rrangle - \llangle x, y \rrangle \right) \right) dy,
\]
where we write $x+y = x'+y'$ with $x \in \X_v$, $y \in \Y_v$, $x' \in \X'_v$, $y' \in \Y'_v$, and we normalize a Haar measure $dy$ so that
\[
 \langle \varphi_1, \varphi_2 \rangle  = \langle \varphi_1', \varphi_2' \rangle
\]
holds for $\varphi_1', \varphi_2' \in \SS(\X'_v)$ and their partial Fourier transforms $\varphi_1, \varphi_2 \in \SS(\X_v)$.

\subsubsection{Weil representations}

Fix a place $v$ of $F$ and suppress the subscript $v$ from the notation.
In Appendices \ref{sec:spl-B}, \ref{sec:spl-double-B}, we have defined the maps
\begin{itemize}
 \item $\hat{s} : \G(\U(V) \times \U(W^{\square})) \rightarrow \C^1$ such that $z_{\V^{\triangle}} = \partial \hat{s}$,
 \item $s: \GU(V)^0 \times \GU(W) \rightarrow \C^1$ such that $z_{\Y} = \partial s$,
 \item $s': \GU(V)^0 \times \GU(W) \rightarrow \C^1$ such that $z_{\Y'} = \partial s'$.
\end{itemize}
Let $\omega_{\psi}$ and $\omega_{\psi}^{\square}$ be the Weil representations of $\Mp(\V)$ and $\Mp(\V^{\square})$ with respect to $\psi$, respectively.
Composing these with $\hat{s}$, $s$, $s'$, we obtain:
\begin{itemize}
 \item a representation $\omega_{\psi}^{\square}$ of $\G(\U(V) \times \U(W^{\square}))$ on $\SS(\V^{\bigtriangledown})$, 
 \item a representation $\omega_{\psi}$ of $\G(\U(V)^0 \times \U(W))$ on $\SS(\X)$, 
 \item a representation $\omega_{\psi}$ of $\G(\U(V)^0 \times \U(W))$ on $\SS(\X')$.
\end{itemize}
By \S \ref{ss:spl-double-B-Y}, the partial Fourier transform
\[
 \SS(\V^{\bigtriangledown}) \cong \SS(\X^{\square}) = \SS(\X) \otimes \SS(\X)
\]
induces an isomorphism
\[
 \omega_{\psi}^{\square} \circ (\id_V \otimes \iota) \cong \omega_{\psi} \otimes \bar{\omega}_{\psi}
\]
as representations of $\G(\U(V)^0 \times \U(W) \times \U(W))$.
By definition, the partial Fourier transform $\SS(\X') \cong \SS(\X)$ is $\G(\U(V)^0 \times \U(W))$-equivariant.

\subsection{The Jacquet--Langlands--Shimizu correspondence}
\label{subsec:jls}

Let $F$ be a number field and $B$ a quaternion algebra over $F$.
We assume that $B$ is division.
Set 
\begin{align*}
 G & = \GU(W), & H & = \GU(V), & H^0 & = \GU(V)^0, \\
 G_1 & = \U(W), & H_1 & = \U(V), & H_1^0 & = \U(V)^0.
\end{align*}
Recall that $G \cong B^{\times}$ and 
\[
 1 \longrightarrow F^{\times} \overset{i}\longrightarrow
 B_1^{\times} \times B_2^{\times} \longrightarrow H^0 \longrightarrow 1,
\]
where $B_1$ and $B_2$ are quaternion algebras over $F$ such that $B_1 \cdot B_2 = B$ in the Brauer group and $i(z) = (z, z^{-1})$.

Put $(\A^{\times})^+ = \nu(G(\A)) \cap \nu(H^0(\A))$,
\[
 G(\A)^+ = \{ g \in G(\A) \, | \, \nu(g) \in (\A^{\times})^+ \}, \qquad
 H^0(\A)^+ = \{ h \in H^0(\A) \, | \, \nu(h) \in (\A^{\times})^+ \}.
\]
For each place $v$ of $F$, we define $(F_v^{\times})^+$, $G_v^+$, $(H^0_v)^+$ similarly.
We have $(F_v^{\times})^+ = F_v^{\times}$ if $v$ is either finite or complex.
If $v$ is real, then we have
\[
 (F_v^{\times})^+ =
 \begin{cases}
  \R^{\times} & \text{if $B_v$, $B_{1,v}$, $B_{2,v}$ are split,}\\
  \R^{\times}_+ &  \text{otherwise.}
 \end{cases}
\]
We have $(\A^{\times})^+ = \prod'_v (F_v^{\times})^+$, $G(\A)^+ = \prod'_v G_v^+$, and $H^0(\A)^+ = \prod'_v (H^0_v)^+$.

Let $\pi$ be an irreducible unitary cuspidal automorphic representation of $\GL_2(\A)$.
We assume that its Jacquet--Langlands transfers $\pi_B$, $\pi_{B_1}$, $\pi_{B_2}$ to $B^{\times}(\A)$, $B_1^{\times}(\A)$, $B_2^{\times}(\A)$ exist.
We regard $\pi_B$ and $\pi_{B_1} \boxtimes \pi_{B_2}$ as irreducible unitary automorphic representations of $G(\A)$ and $H^0(\A)$ respectively.

We define a theta distribution $\Theta : \SS(\X(\A)) \rightarrow \C$ by
\[
 \Theta(\varphi) = \sum_{x \in \X(F)} \varphi(x)
\]
for $\varphi \in \SS(\X(\A))$.
Let $\varphi \in \SS(\X(\A))$ and $f \in \pi_B$.
For $h \in H^0(\A)^+$, choose $g \in G(\A)^+$ such that $\nu(g) = \nu(h)$ and put
\begin{equation}
\label{eq:jls_theta_lift_def}
 \theta_\varphi(f)(h) :=
 \int_{G_1(F) \backslash G_1(\A)} \Theta(\omega_{\psi}(g_1 g h) \varphi)
 f(g_1 g) \, dg_1.
\end{equation}
Here $dg_1 = \prod_v dg_{1,v}$ is the Tamagawa measure on $G_1(\A)$ and we may assume that the volume of a hyperspecial maximal compact subgroup of $G_{1,v}$ with respect to $dg_{1,v}$ is $1$ for almost all $v$.
Note that $\vol(G_1(F) \backslash G_1(\A)) = 1$. 
Using Eichler's norm theorem, one can see that $\theta_\varphi(f)(\gamma h) = \theta_\varphi(f)(h)$ for $\gamma \in H^0(F) \cap H^0(\A)^+$ and $h \in H^0(\A)^+$.
Since $H^0(\A) = H^0(F) H^0(\A)^+$, $\theta_\varphi(f)$ defines an automorphic form on $H^0(\A)$.
Let $\Theta(\pi_B)$ be the automorphic representation of $H^0(\A)$ generated by $\theta_\varphi(f)$ for all $\varphi \in \SS(\X(\A))$ and $f \in \pi_B$.

\begin{lem}
\label{lem:cusp-B}
The automorphic representation $\Theta(\pi_B)$ is cuspidal.
\end{lem}

\begin{proof}
If both $B_1$ and $B_2$ are division, then $H^0$ is anisotropic and the assertion is obvious.
Hence we may assume that either $B_1$ or $B_2$ is split.
Then there exists a complete polarization $V = \tilde{X} \oplus \tilde{Y}$ over $B$.
As in \S \ref{ss:B1-spl}, we regard $V$ as a left $B$-space.
Choosing a basis $\tilde{\v}, \tilde{\v}^*$ of $V$ such that $\tilde{X} = B \tilde{\v}$, $\tilde{Y} = B \tilde{\v}^*$, $\langle \tilde{\v}, \tilde{\v}^* \rangle = 1$, we may write
\[
 H = \left\{ h \in \GL_2(B) \, \left| \,
 h \begin{pmatrix} & 1 \\ -1 & \end{pmatrix} {}^t h^*
 = \nu(h) \cdot \begin{pmatrix} & 1 \\ - 1 & \end{pmatrix} \right. \right\}.
\]
Put 
\[
 \n(b) := \begin{pmatrix} 1 & b \\ & 1 \end{pmatrix} \in H
\]
for $b \in F$.
It remains to show that
\[
 \int_{F \backslash \A} \theta_{\varphi}(f)(\n(b)) \, db = 0
\]
for all $\varphi \in \SS(\X(\A))$ and $f \in \pi_B$.

Let $\V = \tilde{\X} \oplus \tilde{\Y}$ be another complete polarization over $F$ given by $\tilde{\X} = W \otimes_B \tilde{X}$ and $\tilde{\Y} = W \otimes_B \tilde{Y}$, where we regard $W$ as a right $B$-space.
As in \cite[\S 5]{kudla-splitting}, we define a Weil representation $\tilde{\omega}_{\psi}$ of $G_1(\A) \times H_1(\A)$ with respect to $\psi$ on $\SS(\tilde{\X}(\A)) \cong \SS(W(\A))$.
Note that
\[
 \tilde{\omega}_{\psi}(g_1) \tilde{\varphi}(x) = \tilde{\varphi}(g_1^{-1} x), 
 \qquad
 \tilde{\omega}_{\psi}(\n(b)) \tilde{\varphi}(x) = \psi(\tfrac{1}{2} \langle x, x \rangle b) \tilde{\varphi}(x)
\]
for $\tilde{\varphi} \in \SS(W(\A))$, $x \in W(\A)$, $g_1 \in G_1(\A)$, and $b \in \A$.
Let $\tilde{\varphi} \in \SS(W(\A))$ be the partial Fourier transform of $\varphi \in \SS(\X(\A))$.
Then we have
\[
 \Theta(\omega_{\psi}(\g) \varphi)
 = \chi(\g) \sum_{x \in W(F)} \tilde{\omega}_{\psi}(\g) \tilde{\varphi}(x)
\]
for $\g \in G_1(\A) \times H_1^0(\A)$ with some character $\chi$ of $G_1(\A) \times H_1^0(\A)$ trivial on $G_1(F) \times H_1^0(F)$.
One can see that $\chi(g_1) = \chi(\n(b)) = 1$ for $g_1 \in G_1(\A)$ and $b \in \A$.
Since $W$ is anisotropic, we have
\begin{align*}
 \int_{F \backslash \A} \theta_{\varphi}(f)(\n(b)) \, db 
 & = \int_{F \backslash \A} \int_{G_1(F) \backslash G_1(\A)} 
 \sum_{x \in W(F)} \psi(\tfrac{1}{2} \langle x, x \rangle b) 
 \tilde{\omega}_{\psi}(g_1) \tilde{\varphi}(x) f(g_1) \, dg_1 \, db \\
 & = \int_{G_1(F) \backslash G_1(\A)} 
 \tilde{\omega}_{\psi}(g_1) \tilde{\varphi}(0) f(g_1) \, dg_1 \\
 & = \tilde{\varphi}(0) \int_{G_1(F) \backslash G_1(\A)} f(g_1) \, dg_1.
\end{align*}
Since $\pi$ is cuspidal, the restriction of $\pi_B$ to $G_1(\A)$ is orthogonal to the trivial representation of $G_1(\A)$, so that this integral vanishes.
This completes the proof.
\end{proof}

\begin{lem}
\label{lem:nonvanish-B}
The automorphic representation $\Theta(\pi_B)$ is non-zero.
\end{lem}

The proof of this lemma will be given in \S \ref{subsec:rallis-B} below.

\begin{prop}
We have
\[
 \Theta(\pi_B) = \pi_{B_1} \boxtimes \pi_{B_2}.
\]
\end{prop}

\begin{proof}
Since $\Theta(\pi_B)$ is cuspidal and non-zero, the assertion follows from the local theta correspondence for unramified representations and the strong multiplicity one theorem.
\end{proof}

\subsection{The doubling method}
\label{subsec:rallis-general-doubling}

\subsubsection{Degenerate principal series representations}

Set
\[
 G^{\square} = \GU(W^{\square}), \qquad G_1^{\square} = \U(W^{\square}).
\]
Choosing a basis $\w, \w^*$ of $W^{\square}$ such that $W^{\bigtriangledown} = B \w$, $W^{\triangle} = B \w^*$, $\langle \w, \w^* \rangle = 1$, we may write
\[
 G^{\square} = \left\{ g \in \GL_2(B) \, \left| \,
 g \begin{pmatrix} & 1 \\ 1 & \end{pmatrix} {}^t g^*
 = \nu(g) \cdot \begin{pmatrix} & 1 \\ 1 & \end{pmatrix} \right. \right\}.
\]
Let $P$ and $P_1$ be the Siegel parabolic subgroups of $G^{\square}$ and $G^{\square}_1$ given by 
\[
 P = \left\{ \left.
 \begin{pmatrix} a & * \\ & \nu \cdot (a^*)^{-1} \end{pmatrix} \in G^{\square}
 \, \right| \, a \in B^{\times}, \, \nu \in F^{\times} \right\}
\]
and $P_1 = P \cap G^{\square}_1$ respectively.
Let $\delta_P$ and $\delta_{P_1}$ denote the modulus characters of $P(\A)$ and $P_1(\A)$ respectively.
We have
\[
 \delta_P \left( \begin{pmatrix} a & * \\ & \nu \cdot (a^*)^{-1} \end{pmatrix} \right) = |\nu(a)|^3 \cdot |\nu|^{-3}
\]
and $\delta_{P_1} = \delta_P|_{P_1(\A)}$.
Put
\[
 d(\nu) := \begin{pmatrix} 1 & \\ & \nu \end{pmatrix} \in P
\]
for $\nu \in F^{\times}$.
We fix a maximal compact subgroup $K$ of $G^{\square}(\A)$ such that $G^{\square}(\A) = P(\A) K$ and $G^{\square}_1(\A) = P_1(\A) K_1$, where $K_1 = K \cap G^{\square}_1(\A)$.

For $s \in \C$, we consider a degenerate principal series representation $\II(s) := \Ind^{G^{\square}}_P(\delta_P^{s/3})$ of $G^{\square}(\A)$ consisting of smooth functions $\mathcal{F}$ on $G^{\square}(\A)$ which satisfy
\[
 \mathcal{F} \left( \begin{pmatrix} a & * \\ & \nu \cdot (a^*)^{-1} \end{pmatrix} g \right) = |\nu(a)|^{s + \frac{3}{2}} \cdot |\nu|^{-s - \frac{3}{2}} \cdot \mathcal{F}(g).
\]
We define a degenerate principal series representation $\II_1(s) := \Ind^{G_1^{\square}}_{P_1}(\delta_{P_1}^{s/3})$ of $G_1^{\square}(\A)$ similarly.
Then the restriction $\II(s) \rightarrow \II_1(s)$ to $G_1^{\square}(\A)$ as functions is a $G_1^{\square}(\A)$-equivariant isomorphism.
For each place $v$ of $F$, we define degenerate principal series representations $\II_v(s)$ and $\II_{1,v}(s)$ of $G^{\square}_v$ and $G^{\square}_{1,v}$ similarly.

For $\varphi \in \SS(\V^{\bigtriangledown}(\A))$, we define $\mathcal{F}_{\varphi} \in \II(\frac{1}{2})$ by
\[
 \mathcal{F}_{\varphi}(g) = |\nu(g)|^{-2} \cdot (\omega^{\square}_{\psi}(d(\nu(g)^{-1}) g) \varphi)(0).
\]
One can see that the map $\varphi \mapsto \mathcal{F}_{\varphi}$ is $\G(\U(V) \times \U(W^{\square}))(\A)$-equivariant, where $\GU(V)(\A)$ acts trivially on $\II(\frac{1}{2})$.

\subsubsection{Eisenstein series}

For a holomorphic section $\mathcal{F}_s$ of $\II(s)$, we define an Eisenstein series $E(\mathcal{F}_s)$ on $G^{\square}(\A)$ by (the meromorphic continuation of)
\[
 E(g, \mathcal{F}_s)
 = \sum_{\gamma \in P(F) \backslash G^{\square}(F)} \mathcal{F}_s(\gamma g).
\]
For a holomorphic section $\mathcal{F}_{1,s}$ of $\II_1(s)$, we define an Eisenstein series $E(\mathcal{F}_{1,s})$ on $G^{\square}_1(\A)$ similarly.
If $\mathcal{F}_s$ is a holomorphic section of $\II(s)$,
then $\mathcal{F}_s|_{G^{\square}_1(\A)}$ is a holomorphic section of $\II_1(s)$
and $E(\mathcal{F}_s)|_{G^{\square}_1(\A)} = E(\mathcal{F}_s|_{G^{\square}_1(\A)})$.
By \cite[Theorem 3.1]{yamana:sw}, $E(\mathcal{F}_s)$ is holomorphic at $s = \tfrac{1}{2}$.
In particular, the map
\[
 E:\II(\tfrac{1}{2}) \longrightarrow \mathscr{A}(G^{\square}) 
\]
given by $E(\mathcal{F}) := E(\mathcal{F}_s)|_{s=\frac{1}{2}}$ is $G^{\square}(\A)$-equivariant, where $\mathscr{A}(G^{\square})$ is the space of automorphic forms on $G^{\square}(\A)$ and $\mathcal{F}_s$ is the holomorphic section of $\II(s)$ such that $\mathcal{F}_{\frac{1}{2}} = \mathcal{F}$ and $\mathcal{F}_s|_K$ is independent of $s$.

\subsubsection{Doubling zeta integrals}

Let $\langle \cdot, \cdot \rangle$ be the invariant hermitian inner product on $\pi_B$ given by
\[
 \langle f_1, f_2 \rangle := \int_{Z_G(\A) G(F) \backslash G(\A)} f_1(g) \overline{f_2(g)} \, dg
\]
for $f_1, f_2 \in \pi_B$.
Here $Z_G$ is the center of $G$ and $dg$ is the Tamagawa measure on $Z_G(\A) \backslash G(\A)$.
Note that $\vol(Z_G(\A) G(F) \backslash G(\A)) = 2$.
Fix an isomorphism $\pi_B \cong \otimes_v \pi_{B,v}$.
For each place $v$ of $F$, we choose an invariant hermitian inner product $\langle \cdot, \cdot \rangle$ on $\pi_{B,v}$ so that $\langle f_1,f_2 \rangle = \prod_v \langle f_{1,v}, f_{2,v} \rangle$ and $\langle f_{1,v}, f_{2,v} \rangle = 1$ for almost all $v$ for $f_1 = \otimes_v f_{1,v}, f_2 = \otimes_v f_{2,v} \in \pi_B$.
Set 
\[
 \GG = \G(\U(W) \times \U(W)) = \{ (g_1, g_2) \in G \times G \, | \, \nu(g_1) = \nu(g_2) \}.
\]
Then the doubling zeta integral of Piatetski-Shapiro and Rallis \cite{psr} is given by
\[
 Z(\mathcal{F}_s, f_1, f_2) = \int_{Z(\A) \GG(F) \backslash \GG(\A)} E(\iota(g_1, g_2), \mathcal{F}_s) f_1(g_1) \overline{f_2(g_2)} \, d \g
\]
for a holomorphic section $\mathcal{F}_s$ of $\II(s)$ and $f_1, f_2 \in \pi_B$.
Here $Z$ is the center of $G^{\square}$ and $d \g$ is the Tamagawa measure on $Z(\A) \backslash \GG(\A)$.
Note that $\vol(Z(\A) \GG(F) \backslash \GG(\A)) = 2$.
For each place $v$ of $F$, put
\[
 Z(\mathcal{F}_{s,v}, f_{1,v}, f_{2,v})
 = \int_{G_{1,v}} \mathcal{F}_{s,v} (\iota(g_{1,v}, 1)) \langle \pi_{B,v}(g_{1,v}) f_{1,v}, f_{2,v} \rangle \, dg_{1,v}
\]
for a holomorphic section $\mathcal{F}_{s,v}$ of $\II_v(s)$ and $f_{1,v}, f_{2,v} \in \pi_{B,v}$.
Note that, for fixed $f_{1,v}$ and $f_{2,v}$, this integral depends only on the holomorphic section $\mathcal{F}_{s,v}|_{G^{\square}_{1,v}}$ of $\II_{1,v}(s)$.

\begin{lem}
\label{lem:doubling-B}
We have 
\[
 Z(\mathcal{F}_s, f_1, f_2) = \frac{L^S(s+\tfrac{1}{2}, \pi, \ad)}
 {\zeta^S(s+\tfrac{3}{2}) \zeta^S(2s+1)}
 \cdot \prod_{v \in S} Z(\mathcal{F}_{s,v}, f_{1,v}, f_{2,v})
\]
for a holomorphic section $\mathcal{F}_s = \otimes_v \mathcal{F}_{s,v}$ of $\II(s)$ and $f_1 = \otimes_v f_{1,v}, f_2 = \otimes_v f_{2,v} \in \pi_B$.
Here $S$ is a sufficiently large finite set of places of $F$.
\end{lem}

\begin{proof}
The assertion follows from the doubling method \cite{psr}.
Indeed, as in \cite{psr}, \cite[\S 6.2]{harris-jams}, we unfold the Eisenstein series $E(\iota(g_1, g_2), \mathcal{F}_s)$ and see that only the open $\GG$-orbit $P \backslash P \GG$ in $P \backslash G^\square$ contributes to the integral $Z(\mathcal{F}_s, f_1, f_2)$.
Hence we have
\[
 Z(\mathcal{F}_s, f_1, f_2) = \int_{Z(\A) G^{\triangle}(F) \backslash \GG(\A)} \mathcal{F}_s(\iota(g_1, g_2) ) f_1(g_1) \overline{f_2(g_2)} \, d \g, 
\]
where $G^{\triangle} = \{ (g, g) \, | \, g \in G \}$.
We have $\mathcal{F}_s(\iota(g_1, g_2)) = \mathcal{F}_s(\iota(g_2^{-1} g_1, 1))$ for $(g_1, g_2) \in \GG$.
Writing $g = g_2$ and $g' = g_2^{-1} g_1$, we have
\begin{align*}
 Z(\mathcal{F}_s, f_1, f_2)
 & = \int_{G_1(\A)} \int_{Z_G(\A) G(F) \backslash G(\A)}
 \mathcal{F}_s(\iota(g', 1)) f_1(g g') \overline{f_2(g)} \, d g \, d g' \\
 & = \int_{G_1(\A)} \mathcal{F}_s(\iota(g', 1)) \langle \pi_B(g') f_1, f_2 \rangle \, d g' \\
 & = \prod_v Z(\mathcal{F}_{s,v}, f_{1,v}, f_{2,v}).
\end{align*}
By \cite{psr}, we have
\[
 Z(\mathcal{F}_{s,v}, f_{1,v}, f_{2,v})
 = \frac{L(s+\tfrac{1}{2}, \pi_v, \ad)}{\zeta_v(s+\tfrac{3}{2}) \zeta_v(2s+1)}
\]
for almost all $v$.
This completes the proof.
\end{proof}

\subsubsection{Local zeta integrals}

\begin{lem}
\label{lem:conv-B}
The integral $Z(\mathcal{F}_v, f_{1,v}, f_{2,v})$ is absolutely convergent for $\mathcal{F}_v \in \II_v(\frac{1}{2})$ and $f_{1,v}, f_{2,v} \in \pi_{B,v}$.
\end{lem}

\begin{proof}
If $B_v$ is split, then the lemma is proved in \cite[Lemma 6.5]{gi}.
If $B_v$ is division, then $G_{1,v}$ is compact and the assertion is obvious.
\end{proof}

\begin{lem}
\label{lem:nonvanish-local-B}
There exist $\varphi_v \in \SS(\V^{\bigtriangledown}_v)$ and $f_{1,v}, f_{2,v} \in \pi_{B,v}$ such that $Z(\mathcal{F}_{\varphi_v}, f_{1,v}, f_{2,v}) \ne 0$.
\end{lem}

\begin{proof}
If $B_v$ is split, then the lemma is proved in \cite[Lemma 6.6]{gi}.
Assume that $B_v$ is division.
As in \cite[Theorem 3.2.2]{kr90}, \cite[Proposition 7.2.1]{kr94}, one can see that there exist $\mathcal{F}_v \in \II_v(\frac{1}{2})$ and $f_{1,v}, f_{2,v} \in \pi_{B,v}$ such that $Z(\mathcal{F}_v, f_{1,v}, f_{2,v}) \ne 0$.
On the other hand, by \cite[Theorems 1.2, 9.2]{yamana:dps}, the map
\begin{align*}
 \SS(\V^{\bigtriangledown}_v) & \longrightarrow \II_{1,v}(\tfrac{1}{2}) \\
 \varphi_v & \longmapsto \mathcal{F}_{\varphi_v}|_{G^{\square}_{1,v}}
\end{align*}
is surjective .
This yields the lemma.
\end{proof}

If $\varphi_v$ is the partial Fourier transform of $\varphi_{1,v} \otimes \bar{\varphi}_{2,v} \in \SS(\X_v^{\square})$ with $\varphi_{i,v} \in \SS(\X_v)$, then we have
\begin{equation}
\label{eq:local_zeta_integral}
 Z(\mathcal{F}_{\varphi_v}, f_{1,v}, f_{2,v})
 = \int_{G_{1,v}} \langle \omega_{\psi}(g_{1,v}) \varphi_{1,v}, \varphi_{2,v} \rangle
 \langle \pi_{B,v}(g_{1,v}) f_{1,v}, f_{2,v} \rangle \, d g_{1,v}.
\end{equation}
This will be used later to explicate the Rallis inner product formula.

\subsection{The Rallis inner product formula}
\label{subsec:rallis-B}

\subsubsection{Theta integrals}

Recall that $\G(\U(V) \times \U(W^{\square}))(\A)$ acts on $\SS(\V^{\bigtriangledown}(\A))$ via the Weil representation $\omega_{\psi}^{\square}$.
We define a $G_1^{\square}(\A)$-equivariant and $H_1(\A)$-invariant map 
\[
 I:\SS(\V^{\bigtriangledown}(\A)) \longrightarrow \mathscr{A}(G_1^{\square})
\]
as follows.
Here $\mathscr{A}(G_1^{\square})$ is the space of automorphic forms on $G_1^{\square}(\A)$.

Let $\Theta : \SS(\V^{\bigtriangledown}(\A)) \rightarrow \C$ be the theta distribution given by 
\[
 \Theta(\varphi) = \sum_{x \in \V^{\bigtriangledown}(F)} \varphi(x)
\]
for $\varphi \in \SS(\V^{\bigtriangledown}(\A))$.
Let $d h_1$ be the Haar measure on $H_1(\A)$ such that $\vol(H_1(F) \backslash H_1(\A)) = 1$.

First we assume that either $B_1$ or $B_2$ is split.
Then the integral
\begin{equation} \label{eq:theta_int}
 \int_{H_1(F) \backslash H_1(\A)} \Theta(\omega_{\psi}^{\square}(g_1h_1) \varphi) \, dh_1
\end{equation}
may not be convergent.
Following Yamana \cite[\S 2]{yamana:sw}, we choose a place $v \in \Sigma_B$ and an element $z_0$ in the Bernstein center of $H_{1,v}$ or the universal enveloping algebra of the complexified Lie algebra of $H_{1,v}$.
Then the integral
\[
 I(g_1, \varphi) := \int_{H_1(F) \backslash H_1(\A)} \Theta(\omega_{\psi}^{\square}(g_1h_1) (z_0 \cdot \varphi)) \, dh_1
\]
is absolutely convergent for all $g_1 \in G^{\square}_1(\A)$ and $\varphi \in \SS(\V^{\bigtriangledown}(\A))$, and defines an automorphic form on $G^{\square}_1(\A)$.
Note that $I(g_1, \varphi) = \eqref{eq:theta_int}$ if the right-hand side is absolutely convergent for all $g_1$.
In particular, $I(g_1, \varphi)$ does not depend on choice of $v$ and $z_0$.
Next we assume that both $B_1$ and $B_2$ are division.
Then $H_1(F) \backslash H_1(\A)$ is compact.
For $\varphi \in \SS(\V^{\bigtriangledown}(\A))$, we define an automorphic form $I(\varphi)$ on $G^{\square}_1(\A)$ by 
\[
 I(g_1, \varphi)
 := \int_{H_1(F) \backslash H_1(\A)} \Theta(\omega_{\psi}^{\square}(g_1h_1) (z_0 \cdot \varphi)) \, dh_1,
\]
where we write $z_0$ for the identity operator for uniformity.

Similarly, we define a $G_1^{\square}(\A)$-equivariant and $H^0_1(\A)$-invariant map 
\[
 I^0:\SS(\V^{\bigtriangledown}(\A)) \longrightarrow \mathscr{A}(G_1^{\square})
\]
by
\[
 I^0(g_1, \varphi) := \int_{H^0_1(F) \backslash H^0_1(\A)} \Theta(\omega_{\psi}^{\square}(g_1h^0_1) (z_0 \cdot \varphi)) \, dh^0_1,
\]
where $dh^0_1$ is the Tamagawa measure on $H^0_1(\A)$.
Note that $\vol(H^0_1(F) \backslash H^0_1(\A)) = 2$.

\begin{lem}
\label{lem:theta_int^0}
We have
\[
 I^0 = 2 \cdot I.
\]
\end{lem}

\begin{proof}
The lemma follows from \cite[Proposition 4.2]{kudla03} with slight modifications.
We include a proof for convenience.
For each place $v \notin \Sigma_B$, we consider the space $\Hom_{H_{1,v}^0}(\SS(\V_v^{\bigtriangledown}), \C)$ with the natural action of $H_{1,v}^0 \backslash H_{1,v}$.
Let $V_v^{\dagger}$ and $(W_v^{\square})^{\dagger}$ be the $4$-dimensional quadratic $F_v$-space and the $4$-dimensional symplectic $F_v$-space associated to $V_v$ and $W_v^{\square}$ respectively.
Since $\dim V_v^{\dagger} > \frac{1}{2} \dim (W_v^{\square})^{\dagger}$, 
we have $\Hom_{H_{1,v}}(\SS(\V_v^{\bigtriangledown}), \sgn_v) = \{ 0\}$ by \cite[p.~399]{rallis84}, where $\sgn_v$ is the non-trivial character of $H_{1,v}^0 \backslash H_{1,v}$.
Hence $H_{1,v}$ acts trivially on $\Hom_{H_{1,v}^0}(\SS(\V_v^{\bigtriangledown}), \C)$.
On the other hand, we have $H_{1,v}^0 = H_{1,v}$ for all $v \in \Sigma_B$.
Hence $H_1(\A)$ acts trivially on $\Hom_{H_1^0(\A)}(\SS(\V^{\bigtriangledown}(\A)), \C)$, so that
\begin{align*}
 I(g_1, \varphi) 
 & = \int_{H_1^0(\A) H_1(F) \backslash H_1(\A)} I^0(g_1, \omega_{\psi}^{\square}(\dot{h}_1) \varphi) \, d \dot{h}_1 \\
 & = \int_{H_1^0(\A) H_1(F) \backslash H_1(\A)} I^0(g_1, \varphi) \, d \dot{h}_1 \\
 & = \frac{1}{2} \cdot I^0(g_1, \varphi),
\end{align*}
where $d \dot{h}_1$ is the Haar measure on $H_1^0(\A) \backslash H_1(\A)$ such that $\vol(H_1^0(\A) H_1(F) \backslash H_1(\A)) = \frac{1}{2}$.
\end{proof}

\subsubsection{The Siegel--Weil formula}

The Siegel--Weil formula \cite[Theorem 3.4]{yamana:sw} due to Yamana says that $I(\varphi) = E(\mathcal{F}_\varphi)|_{G^{\square}_1(\A)}$ for $\varphi \in \SS(\V^{\bigtriangledown}(\A))$.
Hence, by Lemma \ref{lem:theta_int^0}, we have
\begin{equation}
\label{eq:SW-B}
 I^0(\varphi) = 2 \cdot E(\mathcal{F}_\varphi)|_{G^{\square}_1(\A)}
\end{equation}
for $\varphi \in \SS(\V^{\bigtriangledown}(\A))$.

\subsubsection{The Rallis inner product formula}

Let $Z_{H^0}$ be the center of $H^0$ and $dh^0$ the Tamagawa measure on $Z_H^0(\A) \backslash H^0(\A)$.
Note that $\vol(Z_{H^0}(\A) H^0(F) \backslash H^0(\A)) = 4$.

\begin{prop}
\label{prop:rallis-B}
Let $\varphi = \otimes_v \varphi_v \in \SS(\V^{\bigtriangledown}(\A))$ be the partial Fourier transform of $\varphi_1 \otimes \bar{\varphi}_2 \in \SS(\X^{\square}(\A))$ with $\varphi_i = \otimes_v \varphi_{i,v} \in \SS(\X(\A))$.
Let $f_1 = \otimes_v f_{1,v}, f_2 = \otimes_v f_{2,v} \in \pi_B$.
Then we have
\[
 \int_{Z_{H^0}(\A) H^0(F) \backslash H^0(\A)} \theta_{\varphi_1}(f_1)(h^0)
 \cdot \overline{\theta_{\varphi_2}(f_2)(h^0)} \, dh^0
 = 2 \cdot \frac{L^S(1, \pi, \ad)}{\zeta^S(2)^2} \cdot 
 \prod_{v \in S} Z(\mathcal{F}_{\varphi_v}, f_{1,v}, f_{2,v}).
\]
Here $S$ is a sufficiently large finite set of places of $F$.
\end{prop}

\begin{proof}
Put $(F^{\times})^+ = F^{\times} \cap (\A^{\times})^+$,
\[
 G(F)^+ = G(F) \cap G(\A)^+, \qquad H^0(F)^+ = H^0(F) \cap H^0(\A)^+.
\]
Set $\mathcal{C} = (\A^{\times})^2 (F^{\times})^+ \backslash (\A^{\times})^+$.
Then the similitude characters induce isomorphisms
\[
 Z_G(\A) G_1(\A) G(F)^+ \backslash G(\A)^+ \cong \mathcal{C}, \qquad
 Z_{H^0}(\A) H^0_1(\A) H^0(F)^+ \backslash H^0(\A)^+ \cong \mathcal{C}.
\]
Fix cross sections $c \mapsto g_c$ and $c \mapsto h_c$ of $G(\A)^+ \rightarrow \mathcal{C}$ and $H^0(\A) \rightarrow \mathcal{C}$ respectively.
Since
\[
 \GG(\A) = Z(\A) \cdot \GG(F) \cdot (G_1 \times G_1)(\A) \cdot \{ (g_c, g_c ) \, | \, c \in \mathcal{C} \},
\]
we have
\[
 Z(\mathcal{F}_{\varphi, s}, f_1, f_2)
 = 2 \int_{\mathcal{C}} \int_{G_1(F) \backslash G_1(\A)} \int_{G_1(F) \backslash G_1(\A)} E(\iota(g_1 g_c, g_2 g_c), \mathcal{F}_{\varphi, s})
 f_1(g_1 g_c) \overline{f_2(g_2 g_c)} \, dg_1 \, dg_2 \, dc,
\]
where $dg_1$, $dg_2$ are the Tamagawa measures on $G_1(\A)$ and $dc$ is the Haar measure on $\mathcal{C}$ such that $\vol(\mathcal{C}) = 1$.
For each $c \in \mathcal{C}$, put $\varphi_c = \omega_{\psi}^{\square}(\iota(g_c, g_c), h_c) \varphi$.
Since $ E(g \iota(g_c, g_c), \mathcal{F}_{\varphi}) = E(g, \mathcal{F}_{\varphi_c})$, we have
\begin{align*}
 & Z(\mathcal{F}_{\varphi}, f_1, f_2) \\
 & = 2 \int_{\mathcal{C}} \int_{G_1(F) \backslash G_1(\A)} \int_{G_1(F) \backslash G_1(\A)}
 E(\iota(g_1, g_2), \mathcal{F}_{\varphi_c})
 f_1(g_1 g_c) \overline{f_2(g_2 g_c)} \, dg_1 \, dg_2 \, dc \\
 & = \int_{\mathcal{C}} \int_{G_1(F) \backslash G_1(\A)} \int_{G_1(F) \backslash G_1(\A)}
 I^0(\iota(g_1, g_2), \varphi_c)
 f_1(g_1 g_c) \overline{f_2(g_2 g_c)} \, dg_1 \, dg_2 \, dc \\
 & = \int_{\mathcal{C}} \int_{G_1(F) \backslash G_1(\A)} \int_{G_1(F) \backslash G_1(\A)} \int_{H_1^0(F) \backslash H_1^0(\A)}
 \Theta(\omega_{\psi}^{\square} (\iota(g_1, g_2) h_1^0)(z_0 \cdot \varphi_c)) \\
 & \hspace{10cm} \times f_1(g_1 g_c) \overline{f_2(g_2 g_c)} \, dh_1^0 \, dg_1 \, dg_2 \, dc \\
 & = \int_{\mathcal{C}} \int_{H_1^0(F) \backslash H_1^0(\A)}
 \int_{G_1(F) \backslash G_1(\A)} \int_{G_1(F) \backslash G_1(\A)}
 \Theta(\omega_{\psi}^{\square} (\iota(g_1, g_2) h_1^0) \varphi_c) * z_0 \\
 & \hspace{10cm} \times f_1(g_1 g_c) \overline{f_2(g_2 g_c)} \, dg_1 \, dg_2 \, dh_1^0 \, dc
\end{align*}
by the Siegel--Weil formula \eqref{eq:SW-B}.
On the other hand, we have
\[
 \Theta(\omega_{\psi}^{\square} (\iota(g_1, g_2) h_1^0) \varphi_c)
 = \Theta(\omega_{\psi} (g_1 g_c h_1^0 h_c) \varphi_1) \cdot 
 \overline{\Theta(\omega_{\psi} (g_2 g_c h_1^0 h_c) \varphi_2)}.
\]
Hence we have
\begin{multline*}
 \int_{G_1(F) \backslash G_1(\A)} \int_{G_1(F) \backslash G_1(\A)}
 \Theta(\omega_{\psi}^{\square} (\iota(g_1, g_2) h_1^0) \varphi_c) * z_0
 \cdot f_1(g_1 g_c) \overline{f_2(g_2 g_c)} \, dg_1 \, dg_2 \\
 = \left( \theta_{\varphi_1}(f_1)(h_1^0 h_c) \cdot 
 \overline{\theta_{\varphi_2}(f_2)(h_1^0 h_c)} \right) * z_0.
\end{multline*}
By Lemma \ref{lem:cusp-B}, the function $h_1^0 \mapsto \theta_{\varphi_1}(f_1)(h_1^0 h_c) \cdot \overline{\theta_{\varphi_2}(f_2)(h_1^0 h_c)}$ is integrable over $H_1^0(F) \backslash H_1^0(\A)$, so that 
\begin{multline*}
 \int_{H_1^0(F) \backslash H_1^0(\A)} 
 \left( \theta_{\varphi_1}(f_1)(h_1^0 h_c) \cdot 
 \overline{\theta_{\varphi_2}(f_2)(h_1^0 h_c)} \right) *z_0 \, dh^0_1 \\
 = \int_{H_1^0(F) \backslash H_1^0(\A)} 
 \theta_{\varphi_1}(f_1)(h_1^0 h_c) \cdot 
 \overline{\theta_{\varphi_2}(f_2)(h_1^0 h_c)} \, dh^0_1
\end{multline*}
and hence
\[
 Z(\mathcal{F}_{\varphi}, f_1, f_2)
 = \int_{\mathcal{C}} \int_{H_1^0(F) \backslash H_1^0(\A)}
 \theta_{\varphi_1}(f_1)(h_1^0 h_c) \cdot 
 \overline{\theta_{\varphi_2}(f_2)(h_1^0 h_c)} \, dh_1^0 \, dc.
\]
Since $H^0(\A) = Z_{H^0}(\A) \cdot H^0(F) \cdot H^0_1(\A) \cdot \{ h_c \, | \, c \in \mathcal{C} \}$, this integral is equal to 
\[
 \frac{1}{2} \int_{Z_{H^0}(\A) H^0(F) \backslash H^0(\A)} 
 \theta_{\varphi_1}(f_1)(h^0) \cdot 
 \overline{\theta_{\varphi_2}(f_2)(h^0)} \, dh^0.
\]
Now the assertion follows from this and Lemma \ref{lem:doubling-B}.
\end{proof}

Now Lemma \ref{lem:nonvanish-B} follows from Proposition \ref{prop:rallis-B} and Lemma \ref{lem:nonvanish-local-B}.

%% file: schwartz.tex
\section{Schwartz functions}
\label{sec:schwartz}

Let $F$ be a number field.
Let $\o$ be the integer ring of $F$ and $\mathfrak{d}$ the different of $F$ over $\Q$.
Let $D$ be the discriminant of $F$.
For each finite place $v$ of $F$, let $\o_v$ be the integer ring of $F_v$,
$\p_v = \varpi_v \o_v$ the maximal ideal of $\o_v$,
$\varpi_v$ a uniformizer of $\o_v$, and $q_v$ the cardinality of the residue field $\o_v/\p_v$.
Let $d_v$ be the non-negative integer such that $\mathfrak{d} \otimes_{\o} \o_v = \varpi_v^{d_v} \o_v$.
Then we have $|D| = \prod_{v \in \Sigma_\fin} q_v^{d_v}$.

Let $\psi_0 = \otimes_v \psi_{0,v}$ be the non-trivial character of $\A_{\Q}/\Q$ given by 
\begin{itemize}
 \item $\psi_{0,\infty}(x) = e^{2 \pi \sqrt{-1} x}$ for $x \in \R$,
 \item $\psi_{0,p}(x) = e^{-2 \pi \sqrt{-1} x}$ for $x \in \Q_p$.
\end{itemize}
Let $\psi = \otimes_v \psi_v$ be the non-trivial character of $\A/F$ defined by $\psi = \psi_0 \circ \tr_{F/\Q}$.
We call $\psi$ the standard additive character of $\A/F$.
If $v$ is a real place of $F$, then $\psi_v(x) = e^{2 \pi \sqrt{-1} x}$ for $x \in F_v$.
If $v$ is a complex place of $F$, then $\psi_v(x) = e^{2 \pi \sqrt{-1} (x+\bar{x})}$ for $x \in F_v$, where $\bar{x}$ is the complex conjugate of $x$.
If $v$ is a finite place of $F$, then $\psi_v$ is trivial on $\varpi_v^{-d_v} \o_v$ but non-trivial on $\varpi_v^{-d_v-1} \o_v$.
For each place $v$ of $F$, we define a Fourier transform
\begin{align*}
 \SS(F_v) & \longrightarrow \SS(F_v) \\
 \phi & \longmapsto \hat\phi
\end{align*}
by 
\[
 \hat{\phi}(x) = \int_{F_v} \phi(y) \psi_v(xy) \, dy,
\]
where $dy$ is the self-dual Haar measure on $F_v$ with respect to $\psi_v$.

Let $V = B_1 \otimes_E B_2$ be the $2$-dimensional right skew-hermitian $B$-space given in \S \ref{subsec:key} and $W = B$ the $1$-dimensional left hermitian $B$-space given in \S \ref{sssec:theta-dualpairs-examples}.
Recall that
\begin{align*}
 E & = F + F \i, & B & = E + E \j, & B_1 & = E + E \j_1, & B_2 & = E + E \j_2, \\
 u & = \i^2, & J & = \j^2, & J_1 & = \j_1^2, & J_2 & = \j_2^2,
\end{align*}
where $J = J_1 J_2$.
Let $\V = V \otimes_B W$ be the $8$-dimensional symplectic $F$-space.
We identify $\V$ with $\Res_{B/F}(V)$ via the map $v \mapsto v \otimes 1$.
In \S \ref{subsec:key}, we chose a complete polarization $\V = \X \oplus \Y$ over $F$.
Let $\e_1,\dots,\e_4$ and $\e_1^*,\dots,\e_4^*$ be the bases of $\X$ and $\Y$, respectively, given by \eqref{eq:basisX}, \eqref{eq:basisY}.

\subsection{Complete polarizations}
\label{ssec:complete-pol}

In Appendix \ref{sec:spl-B}, we also choose a complete polarization $\V_v = \X_v' \oplus \Y_v'$ over $F_v$ for each place $v$ of $F$.
Note that in picking the polarization, we use the assumption that for any place $v$ of $F$, at least one of $u$, $J$, $J_1$, $J_2$ is a square in $F_v$.
In this subsection, we recall the choice of this polarization.
Later, we will pick a Schwartz function on $\X_v'$ and then transfer it to a Schwartz function on $\X_v$ by a partial Fourier transform. 
From now on, we fix a place $v$ of $F$ and suppress the subscript $v$ from the notation.

\subsubsection{The case $u \in (F^{\times})^2$}
\label{sssec:cpu}

Choose $t \in F^{\times}$ such that $u = t^2$.
We define an isomorphism $\ii : B \rightarrow \M_2(F)$ of quaternion $F$-algebras by
\begin{equation}
\label{eq:isomB-u}
 \ii(1) = \begin{pmatrix} 1 & \\ & 1 \end{pmatrix}, \qquad
 \ii(\i) = \begin{pmatrix} t & \\ & -t \end{pmatrix}, \qquad
 \ii(\j) = \begin{pmatrix} & 1 \\ J & \end{pmatrix}, \qquad
 \ii(\i \j) = \begin{pmatrix} & t \\ - tJ & \end{pmatrix}.
\end{equation}
Put 
\[
 e = \frac{1}{2} + \frac{1}{2t} \i, \qquad  
 e' = \frac{1}{2} \j + \frac{1}{2t} \i \j, \qquad 
 e'' = \frac{1}{2J} \j - \frac{1}{2tJ} \i \j, \qquad 
 e^* = \frac{1}{2} - \frac{1}{2t} \i,
\]
so that 
\[
 \ii(e) = \begin{pmatrix} 1 & 0 \\ 0 & 0 \end{pmatrix}, \qquad
 \ii(e') = \begin{pmatrix} 0 & 1 \\ 0 & 0 \end{pmatrix}, \qquad
 \ii(e'') = \begin{pmatrix} 0 & 0 \\ 1 & 0 \end{pmatrix}, \qquad
 \ii(e^*) = \begin{pmatrix} 0 & 0 \\ 0 & 1 \end{pmatrix}.
\]
Let $W^\dagger := eW$ be the $2$-dimensional $F$-space associated to $W$ equipped with a non-degenerate symplectic form $\langle \cdot, \cdot \rangle^\dagger$ defined by 
\begin{equation}
\label{eq:W^dagger-form}
 \langle x, y \rangle^* = \langle x, y \rangle^{\dagger} \cdot e' 
\end{equation}
for $x, y \in W^\dagger$.
Then the restriction to $W^\dagger$ induces a natural isomorphism $\GU(W) \cong \GSp(W^\dagger)$.
We have
\[
 \langle e, e \rangle^{\dagger} = \langle e', e' \rangle^{\dagger} = 0, \qquad
 \langle e, e' \rangle^{\dagger} = 1,
\]
and 
\[
 \begin{bmatrix} e \cdot \ba \\ e' \cdot \ba \end{bmatrix}
 = \ii(\ba) \cdot \begin{bmatrix} e \\ e' \end{bmatrix}
\]
for $\ba \in B$.
We take a complete polarization $W^\dagger = X \oplus Y$ given by 
\[
 X = F e, \qquad Y = F e'.
\]
Similarly, let $V^\dagger := V e$ be the $4$-dimensional $F$-space associated to $V$ equipped with a non-degenerate symmetric bilinear form $\langle \cdot, \cdot \rangle^{\dagger}$ defined by 
\begin{equation}
\label{eq:V^dagger-form} 
 \frac{1}{2} \cdot \langle x, y \rangle = 
 \langle x, y \rangle^{\dagger} \cdot e''
\end{equation}
for $x, y \in V^\dagger$.
Then the restriction to $V^\dagger$ induces a natural isomorphism $\GU(V) \cong \GO(V^\dagger)$.
We take a complete polarization $\V = \X' \oplus \Y'$ given by
\[
 \X' = V^\dagger \otimes X, \qquad \Y' = V^\dagger \otimes Y.
\]
We identify $\X'$ with $V^\dagger$ via the map $v \mapsto v \otimes e$.
Put
\begin{equation}
\label{eq:basisX'Y'-u}
\begin{aligned}
 \v_1 & = 2 \e_1 e = \e_1 + t \e_1^*, &
 \v_1^* & = -\frac{1}{t} \e_1 e^* = -\frac{1}{2t} \e_1 + \frac{1}{2} \e_1^*, \\
 \v_2 & = 2 \e_2 e = \e_2 - tJ_1 \e_2^*, &
 \v_2^* & = \frac{1}{tJ_1} \e_2 e^* = \frac{1}{2 tJ_1} \e_2 + \frac{1}{2} \e_2^*, \\
 \v_3 & = -\frac{1}{t J_1} \e_2 e'' = -\frac{1}{2 tJ} \e_3 + \frac{1}{2 J_1} \e_3^*, &
 \v_3^* & = - 2\e_2 e' = - J_1 \e_3 - tJ \e_3^*, \\
 \v_4 & = -\frac{1}{t} \e_1 e'' = -\frac{1}{2tJ} \e_4 - \frac{1}{2} \e_4^*, &
 \v_4^* & = 2 \e_1 e' = \e_4 - tJ \e_4^*. 
\end{aligned}
\end{equation}
Then $\v_1, \dots, \v_4$ and $\v_1^*, \dots, \v_4^*$ are bases of $\X'$ and $\Y'$, respectively, such that $\llangle \v_i, \v_j^* \rrangle = \delta_{ij}$.

We may identify the quadratic space $V^\dagger$ with the space $\M_2(F)$ equipped with a non-degenerate symmetric bilinear form
\begin{equation}
\label{eq:M_2(F)-form}
 \tr(x y^*) = x_1 y_4 - x_2 y_3 - x_3 y_2 + x_4 y_1 
\end{equation}
for $x = \smat{x_1}{x_2}{x_3}{x_4}$, $y = \smat{y_1}{y_2}{y_3}{y_4}$.
Indeed, the basis $\v_1, \dots, \v_4$ of $V^\dagger$ gives rise to an isomorphism $V^\dagger \cong \M_2(F)$ of quadratic spaces by 
\[
 \v_1 \longmapsto \mat{1}{0}{0}{0}, \qquad \v_2 \longmapsto \mat{0}{1}{0}{0}, \qquad \v_3 \longmapsto \mat{0}{0}{1}{0}, \qquad \v_4 \longmapsto \mat{0}{0}{0}{1}.
\]
Under this identification, we have
\[
 \ba_1 \v = \v \cdot \ii_1(\ba_1)^*, \qquad 
 \ba_2 \v = \ii_2(\ba_2) \cdot \v
\]
for $\ba_i \in B_i$ and $\v \in V^\dagger \cong \M_2(F)$, where $\ii_1 : B_1 \rightarrow \M_2(F)$ and $\ii_2 : B_2 \rightarrow \M_2(F)$ are isomorphisms of quaternion $F$-algebras given by 
\begin{equation}
\label{eq:isomB1B2-u}
\begin{aligned}
 \ii_1(a + b \i + c \j_1 + d \i \j_1) & =
 \begin{pmatrix}
  a - bt & -(c-dt) \\
  -J_1 (c+dt) & a + bt
 \end{pmatrix}, \\
 \ii_2(a + b \i + c \j_2 + d \i \j_2) & =
\begin{pmatrix}
 a+bt & - \frac{1}{2tJ_1} (c + dt) \\
 -2 tJ (c - dt) & a - bt
\end{pmatrix}.
\end{aligned}
\end{equation}

\subsubsection{The case $J \in (F^{\times})^2$}
\label{sssec:cpJ}

Choose $t \in F^{\times}$ such that $J = t^2$.
We define an isomorphism $\ii : B \rightarrow \M_2(F)$ of quaternion $F$-algebras by
\begin{equation}
\label{eq:isomB-J}
 \ii(1) = \begin{pmatrix} 1 & \\ & 1 \end{pmatrix}, \qquad
 \ii(\i) = \begin{pmatrix} & 1 \\ u & \end{pmatrix}, \qquad
 \ii(\j) = \begin{pmatrix} t & \\ & -t \end{pmatrix}, \qquad
 \ii(\i \j) = \begin{pmatrix} &  -t \\ tu & \end{pmatrix}.
\end{equation}
Put 
\[
 e = \frac{1}{2} + \frac{1}{2t} \j, \qquad  
 e' = \frac{1}{2} \i - \frac{1}{2t} \i \j, \qquad 
 e'' = \frac{1}{2u} \i + \frac{1}{2tu} \i \j, \qquad 
 e^* = \frac{1}{2} - \frac{1}{2t} \j,
\]
so that 
\[
 \ii(e) = \begin{pmatrix} 1 & 0 \\ 0 & 0 \end{pmatrix}, \qquad
 \ii(e') = \begin{pmatrix} 0 & 1 \\ 0 & 0 \end{pmatrix}, \qquad
 \ii(e'') = \begin{pmatrix} 0 & 0 \\ 1 & 0 \end{pmatrix}, \qquad
 \ii(e^*) = \begin{pmatrix} 0 & 0 \\ 0 & 1 \end{pmatrix}.
\]
Let $W^\dagger := eW$ be the $2$-dimensional $F$-space associated to $W$ equipped with a non-degenerate symplectic form $\langle \cdot, \cdot \rangle^\dagger$ defined by \eqref{eq:W^dagger-form}.
We have
\[
 \langle e, e \rangle^{\dagger} = \langle e', e' \rangle^{\dagger} = 0, \qquad
 \langle e, e' \rangle^{\dagger} = 1,
\]
and 
\[
 \begin{bmatrix} e \cdot \ba \\ e' \cdot \ba \end{bmatrix}
 = \ii(\ba) \cdot \begin{bmatrix} e \\ e' \end{bmatrix}
\]
for $\ba \in B$.
We take a complete polarization $W^\dagger = X \oplus Y$ given by 
\[
 X = F e, \qquad Y = F e'.
\]
Similarly, let $V^\dagger := V e$ be the $4$-dimensional $F$-space associated to $V$ equipped with a non-degenerate symmetric bilinear form $\langle \cdot, \cdot \rangle^{\dagger}$ defined by \eqref{eq:V^dagger-form}.
We take a complete polarization $\V = \X' \oplus \Y'$ given by
\[
 \X' = V^\dagger \otimes X, \qquad \Y' = V^\dagger \otimes Y.
\]
We identify $\X'$ with $V^\dagger$ via the map $v \mapsto v \otimes e$.
Put
\begin{align*}
 \tilde{\v}_1 & = \e_1 e = \frac{1}{2} \e_1 + \frac{1}{2t} \e_4, &
 \tilde{\v}_1^* & = \frac{2}{u} \e_1 e' = \e_1^* + t \e_4^*, \\
 \tilde{\v}_2 & = \e_1e'' = \frac{1}{2} \e_1^* - \frac{t}{2} \e_4^*, &
 \tilde{\v}_2^* & = -2\e_1 e^* = - \e_1 +\frac{1}{t} \e_4, \\
 \tilde{\v}_3 & = \e_2 e=  \frac{1}{2} \e_2 + \frac{J_1}{2t} \e_3, &
 \tilde{\v}_3^* & = -\frac{2}{uJ_1} \e_2 e' = \e_2^* + \frac{t}{J_1} \e_3^*, \\
 \tilde{\v}_4 & = \e_2 e'' = -\frac{J_1}{2} \e_2^* + \frac{t}{2} \e_3^*, & 
 \tilde{\v}_4^* & = \frac{2}{J_1} \e_2 e^* = \frac{1}{J_1} \e_2 - \frac{1}{t} \e_3.
\end{align*}
Then $\tilde{\v}_1, \dots, \tilde{\v}_4$ and $\tilde{\v}_1^*, \dots, \tilde{\v}_4^*$ are bases of $\X'$ and $\Y'$, respectively, such that $\llangle \tilde{\v}_i, \tilde{\v}_j^* \rrangle = \delta_{ij}$.

We need to use two coordinate systems given as follows:

\paragraph{The case (i)}
\label{par:cpJ-i}

We fix $s \in F^\times$ and define bases $\v_1, \dots, \v_4$ and $\v_1^*, \dots, \v_4^*$ of $\X'$ and $\Y'$, respectively, such that $\llangle \v_i, \v_j^* \rrangle = \delta_{ij}$ by 
\begin{equation}
\label{eq:basisX'Y'-J-i}
\begin{aligned}
 \v_1 & = \tilde{\v}_1, & 
 \v_2 & = \tilde{\v}_2, & 
 \v_3 & = \frac{1}{s} \tilde{\v}_3, &
 \v_4 & = \frac{1}{s} \tilde{\v}_4, \\
 \v_1^* & = \tilde{\v}_1^*, &
 \v_2^* & = \tilde{\v}_2^*, &
 \v_3^* & = s \tilde{\v}_3^*, &
 \v_4^* & = s \tilde{\v}_4^*.
\end{aligned}
\end{equation}

We may identify the quadratic space $V^\dagger$ with the space $B_1$ equipped with a non-degenerate symmetric bilinear form
\[
 -\frac{1}{4} \tr_{B_1/F}(xy^*).
\]
Indeed, since 
\[
 \langle \tilde{\v}_1, \tilde{\v}_1 \rangle^{\dagger} = \frac{u}{2}, \qquad
 \langle \tilde{\v}_2, \tilde{\v}_2 \rangle^{\dagger} = - \frac{1}{2}, \qquad 
 \langle \tilde{\v}_3, \tilde{\v}_3 \rangle^{\dagger} = -\frac{uJ_1}{2}, \qquad 
 \langle \tilde{\v}_4, \tilde{\v}_4 \rangle^{\dagger} =  \frac{J_1}{2}, 
\]
and $\langle \tilde{\v}_i, \tilde{\v}_j \rangle^{\dagger} = 0$ if $i \ne j$, 
the basis $\tilde{\v}_1, \dots, \tilde{\v}_4$ of $V^\dagger$ gives rise to an isomorphism $V^\dagger \cong B_1$ of quadratic spaces by
\[
 \tilde{\v}_1 \longmapsto \i, \qquad
 \tilde{\v}_2 \longmapsto 1, \qquad
 \tilde{\v}_3 \longmapsto \j_1\i, \qquad
 \tilde{\v}_4 \longmapsto \j_1.
\]
Under this identification, we have
\[
 \ba_1 \v = \ba_1 \cdot \v, \qquad 
 \ba_2 \v = \v \cdot \ii_2(\ba_2)^*
\]
for $\ba_i \in B_i$ and $\v \in V^\dagger \cong B_1$, where $\ii_2 : B_2 \rightarrow B_1$ is an isomorphism of quaternion $F$-algebras given by
\begin{equation}
\label{eq:isomB2-J-i}
 \ii_2(\alpha + \beta \j_2) = \alpha^\rho + \frac{t \beta^\rho}{J_1} \j_1
\end{equation}
for $\alpha, \beta \in E$.

\paragraph{The case (ii)}
\label{par:cpJ-ii}

Assume that $J_1 \in (F^{\times})^2$.
We choose $t_1 \in F^\times$ such that $J_1 = t_1^2$ and define bases $\v_1, \dots, \v_4$ and $\v_1^*, \dots, \v_4^*$ of $\X'$ and $\Y'$, respectively, such that $\llangle \v_i, \v_j^* \rrangle = \delta_{ij}$ by 
\begin{equation}
\label{eq:basisX'Y'-J-ii}
\begin{aligned}
 \v_1 & = \tilde{\v}_1 + \frac{1}{t_1} \tilde{\v}_3
 = \frac{1}{2} \e_1 + \frac{1}{2t_1} \e_2 + \frac{t_1}{2t} \e_3 + \frac{1}{2t} \e_4, \\
 \v_2 & = \tilde{\v}_2 + \frac{1}{t_1} \tilde{\v}_4
 = \frac{1}{2} \e_1^* - \frac{t_1}{2} \e_2^* + \frac{t}{2t_1} \e_3^* - \frac{t}{2} \e_4^*, \\
 \v_3 & = \tilde{\v}_2 - \frac{1}{t_1} \tilde{\v}_4
 = \frac{1}{2} \e_1^* + \frac{t_1}{2} \e_2^* - \frac{t}{2t_1} \e_3^* - \frac{t}{2} \e_4^*, \\
 \v_4 & = \frac{1}{u} \tilde{\v}_1 - \frac{1}{t_1u} \tilde{\v}_3
 = \frac{1}{2u} \e_1 - \frac{1}{2t_1u} \e_2 - \frac{t_1}{2tu} \e_3 + \frac{1}{2tu} \e_4, \\
 \v_1^* & = \frac{1}{2} \tilde{\v}_1^* + \frac{t_1}{2} \tilde{\v}_3^*
 = \frac{1}{2} \e_1^* + \frac{t_1}{2} \e_2^* + \frac{t}{2t_1} \e_3^* + \frac{t}{2} \e_4^*, \\
 \v_2^* & = \frac{1}{2} \tilde{\v}_2^* + \frac{t_1}{2} \tilde{\v}_4^*
 = -\frac{1}{2} \e_1 + \frac{1}{2t_1} \e_2 - \frac{t_1}{2t} \e_3 + \frac{1}{2t} \e_4, \\
 \v_3^* & = \frac{1}{2} \tilde{\v}_2^* - \frac{t_1}{2} \tilde{\v}_4^*
 = -\frac{1}{2} \e_1 - \frac{1}{2t_1} \e_2 + \frac{t_1}{2t} \e_3 + \frac{1}{2t} \e_4, \\
 \v_4^* & = \frac{u}{2} \tilde{\v}_1^* - \frac{t_1u}{2} \tilde{\v}_3^*
 = \frac{u}{2} \e_1^* - \frac{t_1u}{2} \e_2^* - \frac{tu}{2t_1} \e_3^* + \frac{tu}{2} \e_4^*.
\end{aligned}
\end{equation}

We may identify the quadratic space $V^\dagger$ with the space $\M_2(F)$ equipped with the non-degenerate symmetric bilinear form \eqref{eq:M_2(F)-form}.
Indeed, the basis $\v_1, \dots, \v_4$ of $V^\dagger$ gives rise to an isomorphism $V^\dagger \cong \M_2(F)$ of quadratic spaces by 
\[
 \v_1 \longmapsto \mat{1}{0}{0}{0}, \qquad \v_2 \longmapsto \mat{0}{1}{0}{0}, \qquad \v_3 \longmapsto \mat{0}{0}{1}{0}, \qquad \v_4 \longmapsto \mat{0}{0}{0}{1}.
\]
Under this identification, we have
\[
 \ba_1 \v = \ii_1(\ba_1) \cdot \v, \qquad 
 \ba_2 \v = \v \cdot \ii_2(\ba_2)^*
\]
for $\ba_i \in B_i$ and $\v \in V^\dagger \cong \M_2(F)$, where $\ii_1 : B_1 \rightarrow \M_2(F)$ and $\ii_2 : B_2 \rightarrow \M_2(F)$ are isomorphisms of quaternion $F$-algebras given by 
\begin{equation}
\label{eq:isomB1B2-J-ii}
\begin{aligned}
 \ii_1(a + b \i + c \j_1 + d \i \j_1) & =
 \begin{pmatrix}
  a + ct_1 & b - dt_1 \\
  u(b + dt_1) & a - ct_1
 \end{pmatrix}, \\
 \ii_2(a + b \i + c \j_2 + d \i \j_2) & =
\begin{pmatrix}
 a - c \frac{t}{t_1} & -u (b + d \frac{t}{t_1}) \\
 - (b - d \frac{t}{t_1}) & a + c \frac{t}{t_1}
\end{pmatrix}.
\end{aligned}
\end{equation}

\subsubsection{The case $J_1 \in (F^{\times})^2$ or $J_2 \in (F^{\times})^2$}
\label{sssec:cpJ1}

We only consider the case $J_1 \in (F^{\times})^2$; we switch the roles of $B_1$ and $B_2$ in the other case.
Choose $t \in F^{\times}$ such that $J_1 = t^2$.
We define isomorphisms $\ii_1 : B_1 \rightarrow \M_2(F)$ and $\ii_2 : B_2 \rightarrow B$ of quaternion $F$-algebras by
\begin{equation}
\label{eq:isomB1-J1}
 \ii_1(1) = \begin{pmatrix} 1 & \\ & 1 \end{pmatrix}, \qquad
 \ii_1(\i) = \begin{pmatrix} & 2 \\ \frac{u}{2} & \end{pmatrix}, \qquad
 \ii_1(\j_1) = \begin{pmatrix} t & \\ & -t \end{pmatrix}, \qquad
 \ii_1(\i \j_1) = \begin{pmatrix} & -2t \\ \frac{tu}{2} & \end{pmatrix},
\end{equation}
and 
\begin{equation}
\label{eq:isomB2-J1}
 \ii_2 (\alpha + \beta \j_2) =  \alpha+ \frac{\beta}{t} \j 
\end{equation}
for $\alpha, \beta \in E$.
Put
\[
 \v := \frac{1}{2} \e_1 + \frac{1}{2 t} \e_2, \qquad
 \v^* := \e_1^* + t \e_2^* = \frac{1}{u} \e_1 \i - \frac{1}{tu} \e_2 \i.
\]
Then $\v,\v^*$ is a basis of $V$ over $B$ such that 
\[
 \langle \v, \v \rangle = \langle \v^*, \v^* \rangle = 0, \qquad
 \langle \v, \v^* \rangle = 1.
\]
Moreover, we have 
\[
 \begin{bmatrix} \ba_i \cdot \v & \ba_i \cdot \v^* \end{bmatrix}
 = \begin{bmatrix} \v & \v^* \end{bmatrix} \cdot \ii_i(\ba_i)
\]
for $\ba_i \in B_i$.
Here we identify $\ii_2(\ba_2)$ with the scalar matrix $\ii_2(\ba_2) \cdot \1_2$ in $\M_2(B)$.
Let $V' := V$, regarded as a left $B$-space via $\ba \cdot x' := (x \cdot \ba^*)'$, where for an element $x \in V$, we write $x'$ for the corresponding element in $V'$.
We have a natural skew-hermitian form $\langle \cdot, \cdot \rangle'$ on $V'$ defined by $\langle x', y' \rangle' = \langle x, y \rangle$.
Let $\GL(V')$ act on $V'$ on the right.
We may identify $\GU(V)$ with $\GU(V')$ via the isomorphism
\begin{align*}
 \GL(V) & \longrightarrow \GL(V'). \\
 g & \longmapsto \left[ x' \mapsto (g^{-1} \cdot x)' \right]
\end{align*}
Under this identification, we have
\[
 \begin{bmatrix} \v' \cdot \ba_i \\ (\v^*)' \cdot \ba_i \end{bmatrix}
 = {}^t (\ii_i(\ba_i)^{-1})^* \cdot
 \begin{bmatrix} \v' \\ (\v^*)' \end{bmatrix}
\]
for $\ba_i \in B_i$.
We take a complete polarization $V' = X' \oplus Y'$ given by
\[
 X' = B \cdot \v', \qquad Y' = B \cdot (\v^*)'.
\]
Similarly, let $W' := W$, regarded as a right $B$-space via $x' \cdot \ba := (\ba^* \cdot x)'$.
We have a natural hermitian form $\langle \cdot, \cdot \rangle'$ on $W'$ defined by $\langle x',y' \rangle' = \langle x,y \rangle$.
Let $\GL(W')$ act on $W'$ on the left.
We may identify $\GU(W)$ with $\GU(W')$ via the isomorphism
\begin{align*}
 \GL(W) & \longrightarrow \GL(W'). \\
 g & \longmapsto \left[ x' \mapsto (x \cdot g^{-1})' \right]
\end{align*}
We now consider an $F$-space $\V' := W' \otimes_B V'$ equipped with a non-degenerate symplectic form
\[
 \llangle \cdot, \cdot \rrangle' := 
 \frac{1}{2} \tr_{B/F} (\langle \cdot, \cdot \rangle' \otimes \langle \cdot, \cdot \rangle'^*).
\]
Let $\GL(\V')$ act on $\V'$ on the right.
We identify $\V$ with $\V'$ via the map $\x = x \otimes y \mapsto \x' = y' \otimes x'$.
Then by Lemma \ref{lem:GSp(V)-GSp(V')-identify}, we may identify $\GSp(\V)$ with $\GSp(\V')$ via the isomorphism
\begin{align*}
 \GL(\V) & \longrightarrow \GL(\V'), \\
 \g & \longmapsto \left[ \x' \mapsto (\x \cdot \g)' \right]
\end{align*}
which induces a commutative diagram
\[
 \xymatrix{
  \GU(V) \times \GU(W) \ar@{->}[r] \ar@{->}[d] &
  \GSp(\V) \ar@{->}[d] \\
  \GU(W') \times \GU(V') \ar@{->}[r] & \GSp(\V')
 }.
\]
We take a complete polarization
\[
 \V' = (W' \otimes_B X') \oplus (W' \otimes_B Y').
\]
Under the identification $\V = \V'$, this gives a complete polarization $\V = \X' \oplus \Y'$, where
\[
 \X' = (\v \cdot B) \otimes_B W, \qquad \Y' = (\v^* \cdot B) \otimes_B W.
\]
We identify $\X'$ with $W$ via the map $w \mapsto \v \otimes w$.
We fix $s \in F^{\times}$ and put 
\begin{equation}
\label{eq:basisX'Y'-J1}
\begin{aligned}
 \v_1 & = \v = \frac{1}{2} \e_1 + \frac{1}{2t} \e_2, &
 \v_1^* & = \v^* = \e_1^* + t \e_2^*, \\
 \v_2 & = \frac{1}{u} \v \i = \frac{1}{2} \e_1^* - \frac{t}{2} \e_2^*, &
 \v_2^* & = - \v^* \i = -\e_1 + \frac{1}{t} \e_2, \\
 \v_3 & = \frac{1}{s} \v \j = \frac{1}{2s} \e_4 + \frac{t}{2s} \e_3, &
 \v_3^* & = - \frac{s}{J} \v^* \j = s \e_4^* + \frac{s}{t} \e_3^*, \\
 \v_4 & = \frac{1}{su} \v \i \j = -\frac{J}{2s} \e_4^* + \frac{J}{2st} \e_3^*, &
 \v_4^* & = \frac{s}{J} \v^* \i \j = \frac{s}{J} \e_4 - \frac{st}{J} \e_3.
\end{aligned}
\end{equation}
Then $\v_1, \dots, \v_4$ and $\v_1^*, \dots, \v_4^*$ are bases of $\X'$ and $\Y'$, respectively, such that $\llangle \v_i, \v_j^* \rrangle = \delta_{ij}$.

\subsection{Weil representations}
\label{ssec:weilrep}

Recall that we have the Weil representation $\omega_{\psi}$ of $\G(\U(V)^0 \times \U(W))$ on $\SS(\X)$ obtained from the map $s: \GU(V)^0 \times \GU(W) \rightarrow \C^1$ such that $z_{\Y} = \partial s$ given in Appendix \ref{sec:spl-B}.
This Weil representation is unitary with respect to the hermitian inner product $\langle \cdot, \cdot \rangle$ on $\SS(\X)$ given by
\[
 \langle \varphi_1, \varphi_2 \rangle = \int_{\X} \varphi_1(x) \overline{\varphi_2(x)} \, dx, 
\]
where $dx = dx_1 \cdots dx_4$ for $x = x_1 \e_1 + \dots + x_4 \e_4$ with the self-dual Haar measure $dx_i$ on $F$ with respect to $\psi$.
The map $s$ is defined in terms of another map $s': \GU(V)^0 \times \GU(W) \rightarrow \C^1$ such that $z_{\Y'} = \partial s'$ given in Appendix \ref{sec:spl-B}, based on \cite{kudla-splitting}.
Thus we obtain the Weil representation $\omega_{\psi}$ of $\G(\U(V)^0 \times \U(W))$ on $\SS(\X')$ from $s'$, as in \cite[\S 5]{kudla-splitting}, \cite[\S 5]{hk-duke}.
This Weil representation is unitary with respect to the hermitian inner product $\langle \cdot, \cdot \rangle$ on $\SS(\X')$ given in terms of certain Haar measure on $\X'$.
In this subsection, we define this Haar measure on $\X'$ and give explicit formulas for the Weil representation on $\SS(\X')$.

\subsubsection{The case $u \in (F^\times)^2$}
\label{sssec:wru}

Recall that we identified $\X'$ with $V^\dagger$.
We take the self-dual Haar measure on $V^\dagger$ with respect to the pairing $(x,y) \mapsto \psi(\langle x, y \rangle^\dagger)$.
More explicitly, this measure is given by
\[
 dx = dx_1 \cdots dx_4
\]
for $x = x_1 \v_1 + \cdots + x_4 \v_4 \in \X'$, where $\v_1,\dots,\v_4$ is the basis of $\X'$ given by \eqref{eq:basisX'Y'-u} and $dx_i$ is the self-dual Haar measure on $F$ with respect to $\psi$.

We identity $\GU(W) \cong B^{\times}$ with $\GL_2(F)$ via the isomorphism $\ii$ given by \eqref{eq:isomB-u}.
Then $\U(W) \cong \SL_2(F)$ acts on $\SS(\X')$ by 
\begin{align*}
 \omega_\psi \begin{pmatrix} a & \\ & a^{-1} \end{pmatrix} \varphi(x) & = |a|^2 \varphi(ax), & a & \in F^\times, \\
 \omega_\psi \begin{pmatrix} 1 & b \\ & 1 \end{pmatrix} \varphi(x)
 & = \psi\left(\frac{1}{2} b \langle x,x \rangle^\dagger\right) \varphi(x), & b & \in F, \\
 \omega_\psi \begin{pmatrix} & -1 \\ 1 & \end{pmatrix} \varphi (x)
 & = \int_{\X'} \varphi(y) \psi(-\langle x,y \rangle ^\dagger) \, dy.
\end{align*}
This action extends to an action of $\G(\U(V)^0 \times \U(W))$ by 
\[
 \omega_\psi(g,h) = \omega_\psi(g \cdot d(\nu)^{-1}) \circ L(h) = L(h) \circ \omega_\psi(d(\nu)^{-1} \cdot g)
\]
for $g \in \GU(W) \cong \GL_2(F)$ and $h \in \GU(V)^0 \cong \GO(V^\dagger)^0$ such that $\nu(g) = \nu(h) =: \nu$,
where $d(\nu) = \smat{1}{}{}{\nu}$ and
\[
 L(h) \varphi(x) = |\nu|^{-1} \varphi(h^{-1} x).
\]

\subsubsection{The case $J \in (F^\times)^2$}
\label{sssec:wrJ}

Recall that we identified $\X'$ with $V^\dagger$.
We take the self-dual Haar measure on $V^\dagger$ with respect to the pairing $(x,y) \mapsto \psi(\langle x, y \rangle^\dagger)$.
More explicitly, according the coordinate system, this measure is given as follows:
\begin{enumerate}
\item
\[
 dx = \left| \frac{u J_1}{4s^2} \right| dx_1 \cdots dx_4
\]
for $x = x_1 \v_1 + \cdots + x_4 \v_4 \in \X'$, where $\v_1,\dots,\v_4$ is the basis of $\X'$ given by \eqref{eq:basisX'Y'-J-i} and $dx_i$ is the self-dual Haar measure on $F$ with respect to $\psi$.
\item
\[
 dx = dx_1 \cdots dx_4
\]
for $x = x_1 \v_1 + \cdots + x_4 \v_4 \in \X'$, where $\v_1,\dots,\v_4$ is the basis of $\X'$ given by \eqref{eq:basisX'Y'-J-ii} and $dx_i$ is the self-dual Haar measure on $F$ with respect to $\psi$.
\end{enumerate}

We identity $\GU(W) \cong B^{\times}$ with $\GL_2(F)$ via the isomorphism $\ii$ given by \eqref{eq:isomB-J}.
Then $\U(W) \cong \SL_2(F)$ acts on $\SS(\X')$ by 
\begin{align*}
 \omega_\psi \begin{pmatrix} a & \\ & a^{-1} \end{pmatrix} \varphi(x) & = |a|^2 \varphi(ax), & a & \in F^\times, \\
 \omega_\psi \begin{pmatrix} 1 & b \\ & 1 \end{pmatrix} \varphi(x)
 & = \psi\left(\frac{1}{2} b \langle x,x \rangle^\dagger\right) \varphi(x), & b & \in F, \\
 \omega_\psi \begin{pmatrix} & -1 \\ 1 & \end{pmatrix} \varphi (x)
 & = \gamma_{B_1} \int_{\X'} \varphi(y) \psi(-\langle x,y \rangle ^\dagger) \, dy,
\end{align*}
where
\[
 \gamma_{B_1} =
 \begin{cases}
  1 & \text{if $B_1$ is split,} \\
  -1 & \text{if $B_1$ is ramified.}
 \end{cases}
\]
This action extends to an action of $\G(\U(V)^0 \times \U(W))$ by 
\[
 \omega_\psi(g,h) = \omega_\psi(g \cdot d(\nu)^{-1}) \circ L(h) = L(h) \circ \omega_\psi(d(\nu)^{-1} \cdot g)
\]
for $g \in \GU(W) \cong \GL_2(F)$ and $h \in \GU(V)^0 \cong \GO(V^\dagger)^0$ such that $\nu(g) = \nu(h) =: \nu$,
where $d(\nu) = \smat{1}{}{}{\nu}$ and
\[
 L(h) \varphi(x) = |\nu|^{-1} \varphi(h^{-1} x).
\]

\subsubsection{The case $J_1 \in (F^\times)^2$ or $J_2 \in (F^\times)^2$}
\label{sssec:wrJ1}

We only consider the case $J_1 \in (F^{\times})^2$; we switch the roles of $B_1$ and $B_2$ in the other case.
Recall that we identified $\X'$ with $W$.
We take the self-dual Haar measure on $W$ with respect to the pairing $(x,y) \mapsto \psi(\frac{1}{2} \tr_{B/F}\langle x, y \rangle)$.
More explicitly, this measure is given by
\[
 dx = \left| \frac{J}{s^2u} \right| dx_1 \cdots dx_4
\]
for $x = x_1 \v_1 + \cdots + x_4 \v_4 \in \X'$, where $\v_1,\dots,\v_4$ is the basis of $\X'$ given by \eqref{eq:basisX'Y'-J1} and $dx_i$ is the self-dual Haar measure on $F$ with respect to $\psi$.

We identity $\GU(V)^0 \cong (B_1^\times \times B_2^\times)/F^\times$ with the group
\[
 \left\{ g \in \GL_2 (B) \, \left| \, {}^t g ^* \begin{pmatrix} & 1 \\  -1  & \end{pmatrix} g = \nu (g) \begin{pmatrix} & 1 \\  -1  & \end{pmatrix} \right. \right\}
\]
via the map $(\ba_1, \ba_2) \mapsto \ii_1(\ba_1) \ii_2(\ba_2)$, where $\ii_1$ and $\ii_2$ are the isomorphisms given by \eqref{eq:isomB1-J1}, \eqref{eq:isomB2-J1}.
Then $\U(V)^0$ acts on $\SS(\X')$ via the identification $\U(V)^0 \cong \U(V')^0$ followed by the Weil representation of $\U(V')^0$ on $\SS(W' \otimes_B X')$ given in \cite[\S 5]{kudla-splitting}.
Hence $\U(V)^0$ acts on $\SS(\X')$ by
\begin{align*}
 \omega_{\psi} \begin{pmatrix} a & \\ & (a^{-1})^* \end{pmatrix} \varphi(x)
 & = |\nu(a)|^{-1} \varphi(a^{-1} x), & a & \in B^\times, \\
 \omega_\psi \begin{pmatrix} 1 & \\ b & 1 \end{pmatrix} \varphi(x)
 & = \psi \left(-\frac{1}{2} b \langle x,x \rangle \right) \varphi(x), & b & \in F, \\
 \omega_\psi \begin{pmatrix} & -1 \\ 1 & \end{pmatrix} \varphi(x)
 & = \gamma_B \int_{\X'} \varphi(y) \psi\left(- \frac{1}{2} \tr_{B/F}\langle x,y \rangle \right) dy,
\end{align*}
where
\[
 \gamma_B =
 \begin{cases}
  1 & \text{if $B$ is split,} \\
  -1 & \text{if $B$ is ramified.}
 \end{cases}
\]
This action extends to an action of $\G(\U(V)^0 \times \U(W))$ by 
\[
 \omega_\psi(g,h) = \omega_\psi(h \cdot d(\nu)^{-1}) \circ R(g) = R(g) \circ \omega_\psi(d(\nu)^{-1} \cdot h)
\]
for $g \in \GU(W)$ and $h \in \GU(V)^0$ such that $\nu(g) = \nu(h) =: \nu$, where $d(\nu) = \smat{1}{}{}{\nu}$ and
\[
  R(g) \varphi(x)= |\nu| \varphi(xg).
\]

\subsection{Partial Fourier transforms}
\label{ssec:pft}

Recall that the partial Fourier transform $\varphi \in \SS(\X)$ of $\varphi' \in \SS(\X')$ is given by
\[
 \varphi(x) = \int_{\Y/\Y\cap \Y'} \varphi'(x')
 \psi \left(\frac{1}{2} \left(\llangle x',y' \rrangle - \llangle x,y \rrangle \right)\right)
 d\mu_{\Y/\Y\cap \Y'}(y),
\]
where for $x \in \X$ and $y \in \Y$, we write $x+y = x'+y'$ with $x' = x'(x,y) \in \X'$ and $y' = y'(x,y) \in \Y'$, and we take the Haar measure $\mu_{\Y/\Y\cap \Y'}$ on $\Y/\Y\cap \Y'$ so that the map
\begin{align*}
 \SS(\X') & \longrightarrow \SS(\X) \\
 \varphi' & \longmapsto \varphi
\end{align*}
respects the hermitian inner products (given in terms of the Haar measures on $\X$ and $\X'$ given in \S \ref{ssec:weilrep}).
By construction, this partial Fourier transform is a unitary equivalence between the Weil representations of $\G(\U(V)^0 \times \U(W))$ on $\SS(\X)$ and $\SS(\X')$.
In this subsection, we explicate the Haar measure $\mu_{\Y/\Y\cap \Y'}$ and the partial Fourier transform $\SS(\X') \rightarrow \SS(\X)$.

We write
\begin{align*}
 x & = x_1 \e_1 + \dots + x_4 \e_4 \in \X, & 
 y & = y_1 \e_1^* + \dots + y_4 \e_4^* \in \Y, \\
 x' & = x'_1 \v_1 + \dots + x'_4 \v_4 \in \X', &
 y' & = y'_1 \v_1^* + \dots + y'_4 \v_4^* \in \Y',
\end{align*}
where $\v_1, \dots, \v_4$ and $\v_1^*, \dots, \v_4^*$ are the bases of $\X'$ and $\Y'$, respectively, given in \S \ref{ssec:complete-pol}.
Let $dx_i, dy_j, dx_i', dy_j'$ be the self-dual Haar measures on $F$ with respect to $\psi$.

\subsubsection{The case $u \in (F^\times)^2$}
\label{sssec:pftu}

Recall that $\v_i, \v_j^*$ are given by \eqref{eq:basisX'Y'-u}.
Note that $\Y \cap \Y' = \{ 0 \}$.
We define a Haar measure $\mu_{\Y/\Y\cap \Y'}$ on $\Y$ by
\[
 d\mu_{\Y/\Y\cap \Y'}(y) = |4u|^{-\frac{1}{2}} \, dy_1 \cdots dy_4
\]
for $y = y_1 \e_1^* + \dots + y_4 \e_4^*$.
We will see below that the partial Fourier transform with respect to this Haar measure is an isometry.

If $x + y = x' + y'$, then we have
\begin{align*}
 x'_1 & = \frac{1}{2t} (y_1 + tx_1), &
 y'_1 & = y_1 - tx_1, \\
 x'_2 & = - \frac{1}{2 tJ_1} (y_2 - tJ_1 x_2), &
 y'_2 & = y_2 + tJ_1 x_2, \\
 x'_3 & = J_1 (y_3 - t J_2 x_3), &
 y'_3 & = - \frac{1}{2tJ}(y_3 + t J_2 x_3), \\
 x'_4 & = -(y_4+ tJx_4), &
 y'_4 & = -\frac{1}{2tJ}(y_4 - tJx_4).
\end{align*}
Namely, putting
\[
 a_1 = t, \quad 
 a_2 = - tJ_1, \quad 
 a_3 = - tJ_2, \quad 
 a_4 = tJ, \quad 
 b_1 = b_2 = 1, \quad
 b_3 = b_4 = - 2tJ,
\]
we have
\[
 x_i' = \frac{b_i}{2 a_i} (y_i + a_i x_i), \qquad
 y_i' = \frac{1}{b_i}(y_i - a_i x_i),
\]
so that
\[
 x_i' y_i' - x_i y_i
 = x_i' \left( \frac{2a_i}{b_i^2} x_i' -\frac{2a_i}{b_i} x_i \right)
 - x_i \left( \frac{2a_i}{b_i} x_i' - a_i x_i \right)
 = \frac{2 a_i}{b_i^2} (x_i')^2 - \frac{4a_i}{b_i} x_i x_i' + a_i x_i^2. 
\]
Hence, if $\varphi'(x') = \prod_{i=1}^4 \varphi'_i (x_i')$ with $\varphi_i' \in \SS(F)$, then we have
\[
 \varphi(x) = |4u|^{-\frac{1}{2}} \prod_{i=1}^4 \varphi_i (x_i),
\]
where
\begin{align*}
 \varphi_i(x_i)
 & = \int_F \varphi'_i(x_i') \psi \left( \frac{1}{2} (x_i' y_i' - x_i y_i) \right) dy_i \\
 & = \left| \frac{2a_i}{b_i} \right| \psi \left( \frac{a_i}{2} x_i^2 \right)
 \int_F \varphi'_i(x_i') \psi \left( \frac{a_i}{b_i^2} (x_i')^2 - \frac{2a_i}{b_i} x_i x_i' \right) dx_i'.
\end{align*}
Since 
\[
 \prod_{i=1}^4 \frac{2a_i}{b_i} = 4u, 
\]
the partial Fourier transform with respect to $\mu_{\Y/\Y\cap \Y'}$ is an isometry.

\subsubsection{The case $J \in (F^\times)^2$}
\label{sssec:pftJ}

\paragraph{The case (i)}
\label{par:pftJi}

Recall that $\v_i, \v_j^*$ are given by \eqref{eq:basisX'Y'-J-i}.
Note that $\Y \cap \Y' = F \v_1^* + F \v_3^*$.
Let $\mu_\Y$ and $\mu_{\Y \cap \Y'}$ be the Haar measures on $\Y$ and $\Y \cap \Y'$, respectively, defined by 
\[
 d \mu_\Y(y) = dy_1 \cdots dy_4, \qquad
 d \mu_{\Y \cap \Y'}(y') = dy_1' \, dy_3'
\]
for $y = y_1 \e_1^* + \dots + y_4 \e_4^*$ and $y' = y_1' \v_1^* + y_3' \v_3^*$.
We define a Haar measure $\mu_{\Y/\Y\cap \Y'}$ on $\Y/\Y\cap \Y'$ by 
\[
 \mu_{\Y/\Y \cap \Y'} = \left| \frac{uJ_1}{4s^2J} \right|^{\frac{1}{2}} \frac{\mu_{\Y}}{\mu_{\Y \cap \Y'}}.
\]
We will see below that the partial Fourier transform with respect to this Haar measure is an isometry.

If $x + y = x' + y'$, then we have
\begin{align*}
 x'_1 & = x_1 + tx_4, &
 y'_1 & = \frac{1}{2} \left( y_1 + \frac{1}{t} y_4 \right), \\
 x'_2 & = y_1 - \frac{1}{t} y_4, &
 y'_2 & = - \frac{1}{2}(x_1 - t x_4), \\
 x'_3 & = s \left( x_2 + \frac{t}{J_1} x_3 \right), &
 y'_3 & = \frac{1}{2s} \left( y_2 + \frac{J_1}{t} y_3 \right), \\
 x'_4 & = - \frac{s}{J_1} \left( y_2 - \frac{J_1}{t} y_3 \right), &
 y'_4 & = \frac{J_1}{2s} \left( x_2 - \frac{t}{J_1} x_3 \right),
\end{align*}
so that 
\[
 x_1' y_1' - x_2' y_2' = x_1 y_1 + x_4 y_4, \qquad 
 x_3' y_3' - x_4' y_4' = x_2 y_2 + x_3 y_3.
\]
Also, we have
\[
 dx'_1 \, dx'_3 \, dy'_2 \, dy'_4  = |J| \, dx_1 \cdots dx_4, \qquad 
 dx'_2 \, dx'_4 \, dy'_1 \, dy'_3  = |J|^{-1} \, dy_1 \cdots dy_4.
\]
Hence, if $\varphi'(x') = \prod_{i=1}^4 \varphi'_i (x_i')$ with $\varphi_i' \in \SS(F)$, then we have
\begin{align*}
 \varphi(x) & = \left| \frac{uJJ_1}{4s^2} \right|^{\frac{1}{2}} \varphi_1'(x_1') \varphi_3'(x_3')
 \int_F \int_F \varphi_2'(x_2') \varphi_4'(x_4') \psi(x_2'y_2' + x_4' y_4') \, dx_2' \, dx_4' \\
 & = \left| \frac{uJJ_1}{4s^2} \right|^{\frac{1}{2}} \varphi_1'(x_1') \hat{\varphi}_2'(y_2') \varphi_3'(x_3') \hat{\varphi}_4'(y_4').
\end{align*}
In particular, the partial Fourier transform with respect to $\mu_{\Y/\Y\cap \Y'}$ is an isometry.

\paragraph{The case (ii)}
\label{par:pftJii}

Recall that $\v_i, \v_j^*$ are given by \eqref{eq:basisX'Y'-J-ii}.
Note that $\Y \cap \Y' = F \v_1^* + F \v_4^*$.
Let $\mu_\Y$ and $\mu_{\Y \cap \Y'}$ be the Haar measures on $\Y$ and $\Y \cap \Y'$, respectively, defined by 
\[
 d \mu_\Y(y) = dy_1 \cdots dy_4, \qquad
 d \mu_{\Y \cap \Y'}(y') = dy_1' \, dy_4'
\]
for $y = y_1 \e_1^* + \dots + y_4 \e_4^*$ and $y' = y_1' \v_1^* + y_4' \v_4^*$.
We define a Haar measure $\mu_{\Y/\Y\cap \Y'}$ on $\Y/\Y\cap \Y'$ by 
\[
 \mu_{\Y/\Y \cap \Y'} = |uJ|^{-\frac{1}{2}} \frac{\mu_{\Y}}{\mu_{\Y \cap \Y'}}.
\]
We will see below that the partial Fourier transform with respect to this Haar measure is an isometry.

If $x + y = x' + y'$, then we have
\begin{align*}
 x'_1 & = \frac{1}{2} \left( x_1 + t_1 x_2 + \frac{t}{t_1} x_3 + t x_4 \right), &
 y'_1 & = \frac{1}{2} \left( y_1 + \frac{1}{t_1} y_2 + \frac{t_1}{t} y_3 + \frac{1}{t} y_4 \right), \\
 x'_2 & = \frac{1}{2} \left( y_1 - \frac{1}{t_1} y_2 + \frac{t_1}{t} y_3 - \frac{1}{t} y_4 \right), &
 y'_2 & = - \frac{1}{2} \left( x_1 - t_1 x_2 + \frac{t}{t_1} x_3 - t x_4 \right), \\
 x'_3 & = \frac{1}{2} \left( y_1 + \frac{1}{t_1} y_2 - \frac{t_1}{t} y_3 - \frac{1}{t} y_4 \right), &
 y'_3 & = - \frac{1}{2} \left( x_1 + t_1 x_2 - \frac{t}{t_1} x_3 - t x_4 \right), \\
 x'_4 & = \frac{u}{2} \left( x_1 - t_1 x_2 - \frac{t}{t_1} x_3 + t x_4 \right), &
 y'_4 & = \frac{1}{2u} \left( y_1 - \frac{1}{t_1} y_2 - \frac{t_1}{t} y_3 + \frac{1}{t} y_4 \right),
\end{align*}
so that
\[
 x_1' y_1' - x_2' y_2' - x_3' y_3' + x_4' y_4' = x_1 y_1 + x_2 y_2 + x_3 y_3 + x_4 y_4.
\]
Also, we have
\[
 dx'_1 \, dx'_4 \, dy'_2 \, dy'_3  = |uJ| \, dx_1 \cdots dx_4, \qquad 
 dx'_2 \, dx'_3 \, dy'_1 \, dy'_4  = |uJ|^{-1} \, dy_1 \cdots dy_4.
\]
Hence, if $\varphi'(x') = \prod_{i=1}^4 \varphi'_i (x_i')$ with $\varphi_i' \in \SS(F)$, then we have
\begin{align*}
 \varphi(x) & = |uJ|^{\frac{1}{2}} \varphi_1'(x_1') \varphi_4'(x_4')
 \int_F \int_F \varphi_2'(x_2') \varphi_3'(x_3') \psi(x_2'y_2' + x_3' y_3') \, dx_2' \, dx_3' \\
 & = |uJ|^{\frac{1}{2}} \varphi_1'(x_1') \hat{\varphi}_2'(y_2') \hat{\varphi}_3'(y_3') \varphi_4'(x_4').
\end{align*}
In particular, the partial Fourier transform with respect to $\mu_{\Y/\Y\cap \Y'}$ is an isometry.

\subsubsection{The case $J_1 \in (F^\times)^2$ or $J_2 \in (F^\times)^2$}
\label{sssec:pftJ1}

We only consider the case $J_1 \in (F^{\times})^2$; we switch the roles of $B_1$ and $B_2$ in the other case.
Recall that $\v_i, \v_j^*$ are given by \eqref{eq:basisX'Y'-J1}.
Note that $\Y \cap \Y' = F \v_1^* + F \v_3^*$.
Let $\mu_\Y$ and $\mu_{\Y \cap \Y'}$ be the Haar measures on $\Y$ and $\Y \cap \Y'$, respectively, defined by 
\[
 d \mu_\Y(y) = dy_1 \cdots dy_4, \qquad
 d \mu_{\Y \cap \Y'}(y') = dy_1' \, dy_3'
\]
for $y = y_1 \e_1^* + \dots + y_4 \e_4^*$ and $y' = y_1' \v_1^* + y_3' \v_3^*$.
We define a Haar measure $\mu_{\Y/\Y\cap \Y'}$ on $\Y/\Y\cap \Y'$ by 
\[
 \mu_{\Y/\Y \cap \Y'} = |s^2u|^{-\frac{1}{2}} \frac{\mu_{\Y}}{\mu_{\Y \cap \Y'}}.
\]
We will see below that the partial Fourier transform with respect to this Haar measure is an isometry.

If $x + y = x' + y'$, then we have
\begin{align*}
 x'_1 & = x_1 + t x_2, & 
 y'_1 & = \frac{1}{2} \left( y_1 +\frac{1}{t} y_2 \right), \\
 x'_2 & = y_1 - \frac{1}{t} y_2, &
 y'_2 & = - \frac{1}{2}(x_1 - t x_2), \\
 x'_3 & = s \left( x_4 + \frac{1}{t} x_3 \right), &
 y'_3 & = \frac{1}{2s} (y_4 + t y_3), \\
 x'_4 & = - \frac{s}{J} (y_4 - t y_3), &
 y'_4 & = \frac{J}{2s} \left( x_4 - \frac{1}{t} x_3 \right),
\end{align*}
so that 
\[
 x_1' y_1' - x_2' y_2' = x_1 y_1 + x_2 y_2, \qquad 
 x_3' y_3' - x_4' y_4' = x_3 y_3 + x_4 y_4.
\]
Also, we have
\[
 dx'_1 \, dx'_3 \, dy'_2 \, dy'_4  = |J| \, dx_1 \cdots dx_4, \qquad
 dx'_2 \, dx'_4 \, dy'_1 \, dy'_3  = |J|^{-1} \, dy_1 \cdots dy_4.
\]
Hence, if $\varphi'(x') = \prod_{i=1}^4 \varphi'_i (x_i')$ with $\varphi_i' \in \SS(F)$, then we have
\begin{align*}
 \varphi(x) & = \left| \frac{J^2}{s^2u} \right|^{\frac{1}{2}} \varphi_1'(x_1') \varphi_3'(x_3')
 \int_F \int_F \varphi_2'(x_2') \varphi_4'(x_4') \psi(x_2'y_2' + x_4' y_4') \, dx_2' \, dx_4' \\
 & = \left| \frac{J^2}{s^2u} \right|^{\frac{1}{2}} \varphi_1'(x_1') \hat{\varphi}_2'(y_2') \varphi_3'(x_3') \hat{\varphi}_4'(y_4').
\end{align*}
In particular, the partial Fourier transform with respect to $\mu_{\Y/\Y\cap \Y'}$ is an isometry.

\subsection{Automorphic representations}
\label{ssec:autom-rep}

Suppose that $F$ is a totally real number field.
Let $\pi_B \cong \otimes_v \pi_{B,v}$ be an irreducible unitary cuspidal automorphic representation of $B^\times(\A)$ satisfying the following conditions:
\begin{itemize}
\item For $v \in \Sigma_{\fin} \smallsetminus \Sigma_{B,\fin}$,
\begin{itemize}
\item[(ur)] $\pi_{B,v} = \Ind(\chi_v \otimes \mu_v)$ is a principal series representation, where $\chi_v$ and $\mu_v$ are unitary unramified; or
\item[(rps)] $\pi_{B,v} = \Ind(\chi_v \otimes \mu_v)$ is a principal series representation, where $\chi_v$ is unitary unramified and $\mu_v$ is unitary ramified; or
\item[(st)] $\pi_{B,v} = \St \otimes \chi_v$ is a twist of the Steinberg representation, where $\chi_v$ is unitary unramified.
\end{itemize}
\item For $v \in \Sigma_{B,\fin}$, 
\begin{itemize}
\item[(1d)] $\pi_{B,v} = \chi_v \circ \nu_v$ is a $1$-dimensional representation, where $\chi_v$ is unitary unramified.
\end{itemize}
\item For $v \in \Sigma_{\infty} \smallsetminus \Sigma_{B,\infty}$,
\begin{itemize}
\item[(ds)] $\pi_{B,v} = \DS_{k_v}$ is the irreducible unitary (limit of) discrete series representation of weight $k_v$.
\end{itemize}
\item For $v \in \Sigma_{B,\infty}$,
\begin{itemize}
\item[(fd)] $\pi_{B,v} = \Sym^{k_v}$ is the irreducible unitary $(k_v+1)$-dimensional representation.
\end{itemize}
\end{itemize}
We assume that $\pi_{B,v}$ is unramified for all finite places $v$ of $F$ such that $F_v$ is ramified or of residual characteristic $2$.
By Proposition \ref{prop:choices1}, we may assume that the following conditions (which are relevant to the choice of the polarization $\V_v = \X_v' \oplus \Y_v'$) are satisfied:
\begin{itemize}
\item If $v \notin \Sigma_B$, then $J \in (F_v^\times)^2$ except in the case (ur).
\item If $v \in \Sigma_B$, then either $J_1 \in (F_v^\times)^2$ or $J_2 \in (F_v^\times)^2$.
\end{itemize}
In fact, Proposition \ref{prop:choices1} enables us to impose more precise ramification conditions as described in the following table:
\begin{flushleft}
{\renewcommand{\arraystretch}{1.2}
\renewcommand{\thefootnote}{\fnsymbol{footnote}}
\begin{tabular}{|c|c|c|c|c|c|c|c|c|} \hline
 & $\pi$ & $B$ & $B_1, B_2$ & $E$ & $u$ & $F$ & $J$ & $J_1, J_2$ \\ \hline
 ur & ur.p.s. & split & spl, spl & split & sq of unit & ur & integer & integers \\
 & & & & split & sq of unit & ram & sq of unit & sqs of units \\
 & & & & inert & nonsq unit & ur & unit $\cdot$ sq of int & units $\cdot$ sqs of ints \\
 & & & & ramified & uniformizer &ur & sq of unit & sqs of units \\ \hline
 rps & r.p.s. & split & spl, spl & split & sq of unit & ur & sq of unit & sqs of units \\ \hline 
 st & St & split & spl, spl & split & sq of unit & ur & sq of unit & sqs of units \\
 & & & ram, ram & inert & nonsq unit & ur & sq of uniform & uniforms\footnotemark \\ \hline
 1d & St & ramified & spl, ram & inert & nonsq unit & ur & uniform & sq of unit, uniform \\
 & & & ram, spl & inert & nonsq unit & ur & uniform & uniform, sq of unit \\ \hline \hline
 ds & d.s. & split & spl, spl & $\C$ & negative & $\R$ & positive & positive, positive \\
 & & & ram, ram & $\C$ & negative & $\R$ & positive & negative, negative \\ \hline
 fd & d.s. & ramified & spl, ram & $\C$ & negative & $\R$ & negative & positive, negative \\
 & & & ram, spl & $\C$ & negative & $\R$ & negative & negative, positive \\ \hline 
\end{tabular}
\footnotetext{The ratio $J_1/J_2$ is a square of a unit.}
\renewcommand{\thefootnote}{\arabic{footnote}}
}
\end{flushleft}
\begin{itemize}[leftmargin=\parindent]
\item[$\diamond$]
All places above $2$ fall into the case (ur) with $E$ being split.
\item[$\diamond$]
In the case (ur) with $E$ being inert, we need to consider separately the case $J \in (F^{\times})^2$ and the case $J_1 \in (F^{\times})^2$ or $J_2 \in (F^{\times})^2$.
\end{itemize}
Here $\pi \cong \otimes_v \pi_v$ is the Jacquet--Langlands transfer of $\pi_B$ to $\GL_2(\A)$.
These conditions will be very useful in the computation of the partial Fourier transform.
From now on, we fix a place $v$ of $F$ and suppress the subscript $v$ from the notation.

\subsection{Schwartz functions on $\X'$}
\label{ssec:schwartzX'}

In this subsection, we pick a Schwartz function $\varphi' \in \SS(\X')$ such that $\langle \varphi', \varphi' \rangle = 1$,
together with maximal compact subgroups $\K, \K_1, \K_2$ of $B^{\times}, B_1^{\times}, B_2^{\times}$, respectively.
Also, we study equivariance properties of $\varphi'$ under the action of $\K$ and $\K_1 \times \K_2$, regarded as subgroups of $\GU(W) \cong B^{\times}$ and $\GU(V)^0 \cong (B_1^{\times} \times B_2^{\times}) / F^{\times}$, respectively.

We need to introduce some notation.
For any set $A$, let $\I_A$ denote the characteristic function of $A$.
If $F$ is non-archimedean, then for any positive integer $n$, we define a subalgebra $R_n$ of $\M_2(F)$ by
\[
 R_n = \left\{ \left. \begin{pmatrix} a & b \\ c & d \end{pmatrix} \in \M_2(\o) \, \right| \, c \in \varpi^n \o \right\}.
\]
Note that $R_1$ is an Iwahori subalgebra of $\M_2(F)$.
If $F = \R$, then we choose an isomorphism $E \cong \C$ such that
\[
 \frac{\i}{\sqrt{-1}} > 0,
\]
i.e., $\i = |u|^{\frac{1}{2}} \sqrt{-1}$.
Put
\[
 \frac{\partial}{\partial z} = \frac{1}{2} \left( \frac{\partial}{\partial x} + \frac{1}{\i} \frac{\partial}{\partial y} \right), \qquad 
 \frac{\partial}{\partial z^\rho} = \frac{1}{2} \left( \frac{\partial}{\partial x} - \frac{1}{\i} \frac{\partial}{\partial y} \right)
\]
for $z = x + y \i$.
For any integer $k$, we define a character $\chi_k$ of $\C^{\times}$ by 
\[
 \chi_k(\alpha) = \left( \frac{\alpha}{\sqrt{\alpha \alpha^\rho}} \right)^k.
\]
Put
\[
 \mathtt{H} = \begin{pmatrix} 1 & 0 \\ 0 & -1 \end{pmatrix}, \qquad
 \mathtt{X} = \begin{pmatrix} 0 & 1 \\ 0 & 0 \end{pmatrix}, \qquad
 \mathtt{Y} = \begin{pmatrix} 0 & 0 \\ 1 & 0 \end{pmatrix}.
\]

\subsubsection{The case (ur)}
\label{sssec:schwartzX'-ur}

\paragraph{The case when $E$ is split and $F$ is unramified}
\label{par:schwartzX'-ur-Espl-Fur}

In this case, we have:
\begin{itemize}
 \item $F$ is non-archimedean,
 \item $\psi$ is of order zero,
 \item $u = t^2$ for some $t \in \o^\times$,
 \item $J, J_1, J_2 \in \o$.
\end{itemize}
We define maximal orders $\o_B, \o_{B_1}, \o_{B_2}$ in $B, B_1, B_2$, respectively, by
\[
 \o_B = \ii^{-1}(\M_2(\o)), \qquad 
 \o_{B_1} = \ii_1^{-1}(\M_2(\o)), \qquad
 \o_{B_2} = \ii_2^{-1}(\M_2(\o)),
\]
where $\ii,\ii_1,\ii_2$ are the isomorphisms given by \eqref{eq:isomB-u}, \eqref{eq:isomB1B2-u}.
Put
\[
 \K = \o_B^\times, \qquad
 \K_1 = \o_{B_1}^\times, \qquad
 \K_2 = \o_{B_2}^\times.
\]
We take the complete polarization $\V = \X' \oplus \Y'$ and identify $\X'$ with $V^\dagger \cong \M_2(F)$ as in \S \ref{sssec:cpu}.
We define $\varphi' \in \SS(\X')$ by $\varphi' = \I_{\M_2(\o)}$, i.e.,
\[
 \varphi'(x) = \I_{\o}(x_1) \I_{\o}(x_2) \I_{\o}(x_3) \I_{\o}(x_4)
\]
for $x = x_1 \v_1 + \dots + x_4 \v_4$, where $\v_1,\dots,\v_4$ is the basis of $\X'$ given by \eqref{eq:basisX'Y'-u}.
Then we have
\[
 \omega_{\psi}(k, (k_1, k_2)) \varphi' = \varphi'
\]
for $k \in \K$, $k_1 \in \K_1$, $k_2 \in \K_2$ such that $\nu(k) = \nu(k_1) \nu(k_2)$.

\paragraph{The case when $E$ is split and $F$ is ramified}
\label{par:schwartzX'-ur-Espl-Fram}

In this case, we have:
\begin{itemize}
 \item $F$ is non-archimedean,
 \item $u = t^2$ for some $t \in \o^\times$,
 \item $J, J_1, J_2 \in (\o^\times)^2$.
\end{itemize}
Let $d$ be the non-negative integer such that $\psi$ is trivial on $\varpi^{-d} \o$ but non-trivial on $\varpi^{-d-1} \o$.
We define maximal orders $\o_B, \o_{B_1}, \o_{B_2}$ in $B, B_1, B_2$, respectively, by 
\[
 \o_B = \ii^{-1}\left( \begin{pmatrix} 1 & \\ & \varpi^d \end{pmatrix} \M_2(\o) \begin{pmatrix} 1 & \\ & \varpi^{-d} \end{pmatrix} \right), \qquad
 \o_{B_1} = \ii_1^{-1}(\M_2(\o)), \qquad
 \o_{B_2} = \ii_2^{-1}(\M_2(\o)),
\]
where $\ii,\ii_1,\ii_2$ are the isomorphisms given by \eqref{eq:isomB-u}, \eqref{eq:isomB1B2-u}.
Put
\[
 \K = \o_B^\times, \qquad
 \K_1 = \o_{B_1}^\times, \qquad
 \K_2 = \o_{B_2}^\times.
\]
We take the complete polarization $\V = \X' \oplus \Y'$ and identify $\X'$ with $V^\dagger \cong \M_2(F)$ as in \S \ref{sssec:cpu}.
We define $\varphi' \in \SS(\X')$ by $\varphi' = q^d \cdot \I_{\M_2(\o)}$, i.e.,
\[
 \varphi'(x) = q^d \cdot \I_{\o}(x_1) \I_{\o}(x_2) \I_{\o}(x_3) \I_{\o}(x_4)
\]
for $x = x_1 \v_1 + \dots + x_4 \v_4$, where $\v_1,\dots,\v_4$ is the basis of $\X'$ given by \eqref{eq:basisX'Y'-u}.
Then we have
\[
 \omega_{\psi}(k, (k_1, k_2)) \varphi' = \varphi'
\]
for $k \in \K$, $k_1 \in \K_1$, $k_2 \in \K_2$ such that $\nu(k) = \nu(k_1) \nu(k_2)$.

\paragraph{The case when $E$ is inert and $J \in (F^{\times})^2$}
\label{par:schwartzX'-ur-Einert-J}

In this case, we have:
\begin{itemize}
 \item $F$ is non-archimedean,
 \item $\psi$ is of order zero,
 \item $2 \in \o^\times$,
 \item $u \in \o^\times \smallsetminus (\o^\times)^2$,
 \item $J = t^2$ for some $t \in \o$,
 \item $J_1 \in s^2 \o^{\times}$ for some $s \in \o$,
 \item $J_2 \in \o$.
\end{itemize}
We define maximal orders $\o_B, \o_{B_1}, \o_{B_2}$ in $B, B_1, B_2$, respectively, by 
\[
 \o_B = \ii^{-1}(\M_2(\o)), \qquad 
 \o_{B_1} = \o + \o \i + \o \frac{\j_1}{s} + \o \frac{\i\j_1}{s}, \qquad
 \o_{B_2} = \ii_2^{-1}(\o_{B_1}),
\]
where $\ii,\ii_2$ are the isomorphisms given by \eqref{eq:isomB-J}, \eqref{eq:isomB2-J-i}.
Put
\[
 \K = \o_B^\times, \qquad
 \K_1 = \o_{B_1}^\times, \qquad
 \K_2 = \o_{B_2}^\times.
\]
We take the complete polarization $\V = \X' \oplus \Y'$ as in \S \ref{sssec:cpJ} and identify $\X'$ with $V^\dagger \cong B_1$ as in \S \ref{par:cpJ-i}.
We define $\varphi' \in \SS(\X')$ by $\varphi' = \I_{\o_{B_1}}$, i.e.,
\[
 \varphi'(x) = \I_{\o}(x_1) \I_{\o}(x_2) \I_{\o}(x_3) \I_{\o}(x_4)
\]
for $x = x_1 \v_1 + \dots + x_4 \v_4$, where $\v_1,\dots,\v_4$ is the basis of $\X'$ given by \eqref{eq:basisX'Y'-J-i}.
Then we have
\[
 \omega_{\psi}(k, (k_1, k_2)) \varphi' = \varphi'
\]
for $k \in \K$, $k_1 \in \K_1$, $k_2 \in \K_2$ such that $\nu(k) = \nu(k_1) \nu(k_2)$.

\paragraph{The case when $E$ is inert, and $J_1 \in (F^{\times})^2$ or $J_2 \in (F^{\times})^2$}
\label{par:schwartzX'-ur-Einert-J1}

We only consider the case $J_1 \in (F^{\times})^2$; we switch the roles of $B_1$ and $B_2$ in the other case.
In this case, we have:
\begin{itemize}
 \item $F$ is non-archimedean,
 \item $\psi$ is of order zero,
 \item $2 \in \o^\times$,
 \item $u \in \o^\times \smallsetminus (\o^\times)^2$,
 \item $J_1 = t^2$ for some $t \in \o$,
 \item $J \in s^2 \o^{\times}$ for some $s \in \o$,
 \item $J_2 \in \o$.
\end{itemize}
We define maximal orders $\o_B, \o_{B_1}, \o_{B_2}$ in $B, B_1, B_2$, respectively, by
\[
 \o_B = \o + \o \i + \o \frac{\j}{s} + \o \frac{\i\j}{s}, \qquad 
 \o_{B_1} = \ii_1^{-1} (\M_2 (\o)), \qquad
 \o_{B_2} = \ii_2^{-1}(\o_B),
\]
where $\ii_1,\ii_2$ are the isomorphisms given by \eqref{eq:isomB1-J1}, \eqref{eq:isomB2-J1}.
Put
\[
 \K = \o_B^\times, \qquad
 \K_1 = \o_{B_1}^\times, \qquad
 \K_2 = \o_{B_2}^\times.
\]
We take the complete polarization $\V = \X' \oplus \Y'$ and identify $\X'$ with $W = B$ as in \S \ref{sssec:cpJ1}.
We define $\varphi' \in \SS(\X')$ by $\varphi' = \I_{\o_B}$, i.e.,
\[
 \varphi'(x) = \I_{\o}(x_1) \I_{\o}(x_2) \I_{\o}(x_3) \I_{\o}(x_4)
\]
for $x = x_1 \v_1 + \dots + x_4 \v_4$, where $\v_1,\dots,\v_4$ is the basis of $\X'$ given by \eqref{eq:basisX'Y'-J1}.
Then we have
\[
 \omega_{\psi}(k, (k_1, k_2)) \varphi' = \varphi'
\]
for $k \in \K$, $k_1 \in \K_1$, $k_2 \in \K_2$ such that $\nu(k) = \nu(k_1) \nu(k_2)$.

\paragraph{The case when $E$ is ramified}
\label{par:schwartzX'-ur-Eram}

In this case, we have:
\begin{itemize}
 \item $F$ is non-archimedean,
 \item $\psi$ is of order zero,
 \item $2 \in \o^\times$,
 \item $u \in \varpi \o^\times$,
 \item $J = t^2$ for some $t \in \o^\times$,
 \item $J_1 = t_1^2$ for some $t_1 \in \o^\times$,
 \item $J_2 \in (\o^\times)^2$.
\end{itemize}
We define maximal orders $\o_B, \o_{B_1}, \o_{B_2}$ in $B, B_1, B_2$, respectively, by 
\[
 \o_B = \ii^{-1}(\M_2(\o)), \qquad
 \o_{B_1} = \ii_1^{-1}(\M_2(\o)), \qquad 
 \o_{B_2} = \ii_2^{-1}(\M_2(\o)),
\]
where $\ii,\ii_1,\ii_2$ are the isomorphisms given by \eqref{eq:isomB-J}, \eqref{eq:isomB1B2-J-ii}.
Put
\[
 \K = \o_B^\times, \qquad
 \K_1 = \o_{B_1}^\times, \qquad
 \K_2 = \o_{B_2}^\times.
\]
We take the complete polarization $\V = \X' \oplus \Y'$ as in \S \ref{sssec:cpJ} and identify $\X'$ with $V^\dagger \cong \M_2(F)$ as in \S \ref{par:cpJ-ii}.
We define $\varphi' \in \SS(\X')$ by $\varphi' = \I_{\M_2(\o)}$, i.e.,
\[
 \varphi'(x) = \I_{\o}(x_1) \I_{\o}(x_2) \I_{\o}(x_3) \I_{\o}(x_4)
\]
for $x = x_1 \v_1 + \dots + x_4 \v_4$, where $\v_1,\dots,\v_4$ is the basis of $\X'$ given by \eqref{eq:basisX'Y'-J-ii}.
Then we have
\[
 \omega_{\psi}(k, (k_1, k_2)) \varphi' = \varphi'
\]
for $k \in \K$, $k_1 \in \K_1$, $k_2 \in \K_2$ such that $\nu(k) = \nu(k_1) \nu(k_2)$.

\subsubsection{The case (rps)}
\label{sssec:schwartzX'-tr}

In this case, we have:
\begin{itemize}
 \item $F$ is non-archimedean,
 \item $\psi$ is of order zero,
 \item $2 \in \o^\times$,
 \item $u = t^2$ for some $t \in \o^\times$,
 \item $J, J_1, J_2 \in (\o^\times)^2$.
\end{itemize}
We define maximal orders $\o_B, \o_{B_1}, \o_{B_2}$ in $B, B_1, B_2$ and subalgebras $\o_{B,n}, \o_{B_1,n}, \o_{B_2,n}$ of $B, B_1, B_2$, respectively, by 
\begin{align*}
 \o_B & = \ii^{-1}(\M_2(\o)), &
 \o_{B_1} & = \ii_1^{-1}(\M_2(\o)), &
 \o_{B_2} & = \ii_2^{-1}(\M_2(\o)), \\
 \o_{B,n} & = \ii^{-1}(R_n), &
 \o_{B_1,n} & = \ii_1^{-1}(R_n), &
 \o_{B_2,n} & = \ii_2^{-1}(R_n),
\end{align*}
where $\ii,\ii_1,\ii_2$ are the isomorphisms given by \eqref{eq:isomB-u}, \eqref{eq:isomB1B2-u}.
We define orientations
\[
 o_B : \o_{B,n} \longrightarrow \o/\varpi^n \o, \qquad 
 o_{B_1} : \o_{B_1,n} \longrightarrow \o/\varpi^n \o, \qquad 
 o_{B_2} : \o_{B_2,n} \longrightarrow \o/\varpi^n \o
\]
by 
\[
 o_B \! \left( \ii^{-1} \! \smat{a}{b}{c}{d} \right) = d \bmod {\varpi^n \o}, \qquad
 o_{B_1} \! \left( \ii_1^{-1} \! \smat{a}{b}{c}{d} \right) = d \bmod {\varpi^n \o}, \qquad
 o_{B_2} \! \left( \ii_2^{-1} \! \smat{a}{b}{c}{d} \right) = a \bmod {\varpi^n \o}.
\]
Put
\begin{align*}
 \K & = \o_B^\times, &
 \K_1 & = \o_{B_1}^\times, &
 \K_2 & = \o_{B_2}^\times, \\
 \K_n & = \o_{B,n}^\times, &
 \K_{1,n} & = \o_{B_1,n}^\times, &
 \K_{2,n} & = \o_{B_2,n}^\times.
\end{align*}
We take the complete polarization $\V = \X' \oplus \Y'$ and identify $\X'$ with $V^\dagger \cong \M_2(F)$ as in \S \ref{sssec:cpu}.
For a unitary ramified character $\mu$ of $F^{\times}$ of conductor $q^n$, i.e., trivial on $1+\varpi^n \o$ but non-trivial on $1+\varpi^{n-1} \o$ (resp.~$\o^\times$) if $n > 1$ (resp.~if $n=1$), we define $\varphi' = \varphi'_\mu \in \SS(\X')$ by
\[
 \varphi'(x) = q^{\frac{n+1}{2}} (q-1)^{-\frac{1}{2}} \cdot \I_\o(x_1) \I_\o(x_2) \I_{\varpi^n \o}(x_3) \I_{\o^{\times}}(x_4) \mu(x_4)
\]
for $x = x_1 \v_1 + \dots + x_4 \v_4$, where $\v_1,\dots,\v_4$ is the basis of $\X'$ given by \eqref{eq:basisX'Y'-u}.
Then we have
\[
 \omega_{\psi}(k, (k_1, k_2)) \varphi' = \Mu(k)^{-1} \Mu(k_1) \Mu(k_2)^{-1} \mu(\nu(k_2)) \varphi'
\]
for $k \in \K_n$, $k_1 \in \K_{1,n}$, $k_2 \in \K_{2,n}$ such that $\nu(k) = \nu(k_1) \nu(k_2)$, 
where $\Mu$ is the character of $R_n^\times$ (and those of $\K_n,\K_{1,n},\K_{2,n}$ via $\ii,\ii_1,\ii_2$) defined by
\[
 \Mu(k) := \mu(d)
\]
for $k = \smat{a}{b}{c}{d}$.

\subsubsection{The case (st)}
\label{sssec:schwartzX'-st}

\paragraph{The case when $B_1$ and $B_2$ are split}
\label{par:schwartzX'-st-B1spl}

In this case, we have:
\begin{itemize}
 \item $F$ is non-archimedean,
 \item $\psi$ is of order zero,
 \item $2 \in \o^\times$,
 \item $u = t^2$ for some $t \in \o^\times$,
 \item $J, J_1, J_2 \in (\o^\times)^2$.
\end{itemize}
We define maximal orders $\o_B, \o_{B_1}, \o_{B_2}$ in $B, B_1, B_2$ and Iwahori subalgebras $\mathfrak{I}, \mathfrak{I}_1, \mathfrak{I}_2$ of $B, B_1, B_2$, respectively, by
\begin{align*}
 \o_B & = \ii^{-1}(\M_2(\o)), &
 \o_{B_1} & = \ii_1^{-1}(\M_2(\o)), &
 \o_{B_2} & = \ii_2^{-1}(\M_2(\o)), \\
 \mathfrak{I} & = \ii^{-1}(R_1), &
 \mathfrak{I}_1 & = \ii_1^{-1}(R_1), &
 \mathfrak{I}_2 & = \ii_2^{-1}(R_1),
\end{align*}
where $\ii,\ii_1,\ii_2$ are the isomorphisms given by \eqref{eq:isomB-u}, \eqref{eq:isomB1B2-u}.
Put
\begin{align*}
 \K & = \o_B^\times, &
 \K_1 & = \o_{B_1}^\times, &
 \K_2 & = \o_{B_2}^\times, \\
 \II & = \mathfrak{I}^\times, &
 \II_1 & = \mathfrak{I}_1^\times, &
 \II_2 & = \mathfrak{I}_2^\times.
\end{align*}
We take the complete polarization $\V = \X' \oplus \Y'$ and identify $\X'$ with $V^\dagger \cong \M_2(F)$ as in \S \ref{sssec:cpu}.
We define $\varphi' \in \SS(\X')$ by
\[
 \varphi'(x) = q^{\frac{1}{2}} \cdot \I_\o(x_1) \I_\o(x_2) \I_\p(x_3) \I_\o(x_4)
\]
for $x = x_1 \v_1 + \dots + x_4 \v_4$, where $\v_1,\dots,\v_4$ is the basis of $\X'$ given by \eqref{eq:basisX'Y'-u}.
Then we have
\[
 \omega_{\psi}(k, (k_1, k_2)) \varphi' = \varphi'
\]
for $k \in \II$, $k_1 \in \II_1$, $k_2 \in \II_2$ such that $\nu(k) = \nu(k_1) \nu(k_2)$.

\paragraph{The case when $B_1$ and $B_2$ are ramified}
\label{par:schwartzX'-st-B1ram}

In this case, we have:
\begin{itemize}
 \item $F$ is non-archimedean,
 \item $\psi$ is of order zero,
 \item $2 \in \o^\times$,
 \item $u \in \o^\times \smallsetminus (\o^\times)^2$,
 \item $J = t^2$ for some $t \in \varpi \o^\times$,
 \item $J_1, J_2 \in \varpi \o^\times$.
\end{itemize}
We define a maximal order $\o_B$ in $B$ and an Iwahori subalgebra $\mathfrak{I}$ of $B$ by 
\[
 \o_B = \ii^{-1}(\M_2(\o)), \qquad
 \mathfrak{I} = \ii^{-1}(R_1),
\]
where $\ii$ is the isomorphism given by \eqref{eq:isomB-J}.
Let $\o_{B_1}$ and $\o_{B_2}$ be the unique maximal orders in $B_1$ and $B_2$, respectively.
Then we have
\[
 \o_{B_1} = \o + \o \i + \o \j_1 + \o \i\j_1, \qquad
 \o_{B_2} = \ii_2^{-1}(\o_{B_1}),
\]
where $\ii_2$ is the isomorphism given by \eqref{eq:isomB2-J-i}.
Put
\[
 \K = \o_B^\times, \qquad
 \II = \mathfrak{I}^\times, \qquad
 \K_1 = \o_{B_1}^{\times}, \qquad 
 \K_2 = \o_{B_2}^{\times}.
\]
Put $s=1$.
We take the complete polarization $\V = \X' \oplus \Y'$ as in \S \ref{sssec:cpJ} and identify $\X'$ with $V^\dagger \cong B_1$ as in \S \ref{par:cpJ-i}.
We define $\varphi' \in \SS(\X')$ by $\varphi' = q^{\frac{1}{2}} \cdot \I_{\o_{B_1}}$, i.e.,
\[
 \varphi'(x) = q^{\frac{1}{2}} \cdot \I_{\o}(x_1) \I_{\o}(x_2) \I_{\o}(x_3) \I_{\o}(x_4)
\]
for $x = x_1 \v_1 + \dots + x_4 \v_4$, where $\v_1,\dots,\v_4$ is the basis of $\X'$ given by \eqref{eq:basisX'Y'-J-i}.
Then we have
\[
 \omega_{\psi}(k, (k_1, k_2)) \varphi' = \varphi'
\]
for $k \in \II$, $k_1 \in \K_1$, $k_2 \in \K_2$ such that $\nu(k) = \nu(k_1) \nu(k_2)$.

\subsubsection{The case (1d)}
\label{sssec:schwartzX'-1d}

We only consider the case $J_1 \in (F^{\times})^2$; we switch the roles of $B_1$ and $B_2$ in the other case.
In this case, we have:
\begin{itemize}
 \item $F$ is non-archimedean,
 \item $\psi$ is of order zero,
 \item $2 \in \o^\times$,
 \item $u \in \o^\times \smallsetminus (\o^\times)^2$,
 \item $J_1 = t^2$ for some $t \in \o^\times$,
 \item $J, J_2 \in \varpi \o^\times$.
\end{itemize}
We define a maximal order $\o_{B_1}$ in $B_1$ and an Iwahori subalgebra $\mathfrak{I}_1$ of $B_1$ by 
\[ 
 \o_{B_1} = \ii_1^{-1}(\M_2(\o)), \qquad
 \mathfrak{I}_1 = \ii_1^{-1}\left( \begin{pmatrix} 1 & \\ & \varpi^{-1} \end{pmatrix} R_1 \begin{pmatrix} 1 & \\ & \varpi \end{pmatrix} \right),
\]
where $\ii_1$ is the isomorphism given by \eqref{eq:isomB1-J1}.
Let $\o_B$ and $\o_{B_2}$ be the unique maximal orders in $B$ and $B_2$, respectively.
Then we have
\[
 \o_B = \o + \o \i + \o \j + \o \i\j, \qquad
 \o_{B_2} = \ii_2^{-1}(\o_B),
\]
where $\ii_2$ is the isomorphism given by \eqref{eq:isomB2-J1}.
Put
\[
 \K = \o_B^{\times}, \qquad
 \K_1 = \o_{B_1}^{\times}, \qquad 
 \II_1 = \mathfrak{I}_1^\times, \qquad
 \K_2  = \o_{B_2}^{\times}.
\]
Put $s=1$.
We take the complete polarization $\V = \X' \oplus \Y'$ and identify $\X'$ with $W = B$ as in \S \ref{sssec:cpJ1}.
We define $\varphi' \in \SS(\X')$ by $\varphi' = q^{\frac{1}{2}} \cdot \I_{\o_B}$, i.e.,
\[
 \varphi'(x) = q^{\frac{1}{2}} \cdot \I_{\o}(x_1) \I_{\o}(x_2) \I_{\o}(x_3) \I_{\o}(x_4)
\]
for $x = x_1 \v_1 + \dots + x_4 \v_4$, where $\v_1,\dots,\v_4$ is the basis of $\X'$ given by \eqref{eq:basisX'Y'-J1}.
Then we have
\[
 \omega_{\psi}(k, (k_1, k_2)) \varphi' = \varphi'
\]
for $k \in \K$, $k_1 \in \II_1$, $k_2 \in \K_2$ such that $\nu(k) = \nu(k_1) \nu(k_2)$.

\subsubsection{The case (ds)}
\label{sssec:schwartzX'-ds}

\paragraph{The case when $B_1$ and $B_2$ are split}
\label{par:schwartzX'-ds-B1spl}

In this case, we have:
\begin{itemize}
 \item $F = \R$,
 \item $\psi(x) = e^{2 \pi \sqrt{-1} x}$,
 \item $u < 0$,
 \item $J = t^2$ for some $t \in F^\times$,
 \item $J_1 = s^2$ for some $s \in F^\times$,
 \item $J_2 > 0$.
\end{itemize}
Put $v = |u|^{\frac{1}{2}}$.
We take the complete polarization $\V = \X' \oplus \Y'$ as in \S \ref{sssec:cpJ} and identify $\X'$ with $V^\dagger \cong B_1$ as in \S \ref{par:cpJ-i}.
For a non-negative integer $k$, we define $\varphi' = \varphi'_k \in \SS(\X')$ by 
\[
 \varphi'(x) = c_k^{-\frac{1}{2}} \cdot (x_2 - x_1 \i)^k \cdot e^{- \frac{\pi}{2v}(x_2^2 - u x_1^2 + x_4^2 - u x_3^2)}
\]
for $x = x_1 \v_1 + \dots + x_4 \v_4$, where $\v_1,\dots,\v_4$ is the basis of $\X'$ given by \eqref{eq:basisX'Y'-J-i} and 
\[
 c_k = \frac{k! |u|^{\frac{k}{2}+1}}{4 \pi^k}.
\]

\begin{lem}
\label{lem:varphi'-ds-spl}
We have $\langle \varphi', \varphi' \rangle = 1$ and
\[
 \omega_\psi(\alpha, (\alpha_1,\alpha_2)) \varphi' = \chi_k(\alpha)^{-1} \chi_k(\alpha_1) \chi_k(\alpha_2) \varphi'
\]
for $\alpha,\alpha_1,\alpha_2 \in E^\times$ such that $\nu(\alpha) = \nu(\alpha_1) \nu(\alpha_2)$.
\end{lem}

\begin{proof}
Recall that the Haar measure on $\X'$ is given by $dx = \frac{|u|}{4} \, dx_1 \cdots dx_4$.
We have
\begin{align*}
 & \frac{|u|}{4} \int_{F^4} (x_2^2 - u x_1^2)^k e^{- \frac{\pi}{v}(x_2^2 - u x_1^2 + x_4^2 - u x_3^2)} \, dx_1 \cdots \, dx_4 \\
 & = \frac{|u|}{4 \pi^2} \left( \frac{v}{\pi} \right)^k
 \int_{F^4} (x_2^2 + x_1^2)^k e^{-(x_2^2 + x_1^2 + x_4^2 + x_3^2)} \, dx_1 \cdots \, dx_4 \\
 & = \frac{|u|^{\frac{k}{2}+1}}{4 \pi^{k+2}} \cdot (2 \pi)^2
 \int_0^\infty \int_0^\infty r_1^{2k} e^{-(r_1^2 + r_2^2)} \, r_1 \, dr_1 \, r_2 \, dr_2 \\
 & = \frac{|u|^{\frac{k}{2}+1}}{4 \pi^k}
 \int_0^\infty \int_0^\infty r_1^k e^{-(r_1 + r_2)} \, dr_1 \, dr_2 \\
 & = \frac{|u|^{\frac{k}{2}+1}}{4 \pi^k} \cdot \Gamma(k+1)
\end{align*}
and hence $\langle \varphi', \varphi' \rangle = 1$.
If we write $z_1 = x_2 + x_1 \i$ and $z_2 = x_4 + x_3 \i$, then 
\[
 \varphi'(x) = c_k^{-\frac{1}{2}} \cdot (z_1^\rho)^k \cdot e^{- \frac{\pi}{2v}(z_1 z_1^\rho + z_2 z_2^\rho)},
\]
and it is easy to see that 
\[
  \omega_\psi(\nu^\frac{1}{2}, (\alpha_1,\alpha_2)) \varphi' = \chi_k(\alpha_1) \chi_k(\alpha_2) \varphi'
\]
for $\alpha_1, \alpha_2 \in E^\times$ and $\nu = \nu(\alpha_1) \nu(\alpha_2)$.
On the other hand, we have
\begin{align*}
 \omega_\psi(\mathtt{H}) \varphi'(x) & 
 = \left( 2 + x_1 \frac{\partial}{\partial x_1} + \dots + x_4 \frac{\partial}{\partial x_4} \right)
 \varphi'(x), \\
 \omega_\psi(\mathtt{X}) \varphi'(x)
 & = \frac{\pi \sqrt{-1}}{2} (ux_1^2 - x_2^2 - ux_3^2 + x_4^2) \varphi'(x), \\
 \omega_\psi(\mathtt{Y}) \varphi'(x)
 & = -\frac{1}{2 \pi \sqrt{-1}} \left( \frac{1}{u} \frac{\partial^2}{\partial x_1^2} - \frac{\partial^2}{\partial x_2^2} - \frac{1}{u} \frac{\partial^2}{\partial x_3^2} + \frac{\partial^2}{\partial x_4^2} \right) \varphi'(x),
\end{align*}
where we identity $\GU(W) \cong B^{\times}$ with $\GL_2(F)$ via the isomorphism $\ii$ given by \eqref{eq:isomB-J}.
Thus, noting that 
\[
 \frac{\partial^2}{\partial z_1 \partial z_1^\rho} = \frac{1}{4} \left( \frac{\partial^2}{\partial x_2^2} - \frac{1}{u} \frac{\partial^2}{\partial x_1^2} \right), \qquad
 \frac{\partial^2}{\partial z_2 \partial z_2^\rho} = \frac{1}{4} \left( \frac{\partial^2}{\partial x_4^2} - \frac{1}{u} \frac{\partial^2}{\partial x_3^2} \right), 
\]
we see that
\[
 \omega_\psi(v^{-1} \mathtt{X} - v \mathtt{Y}) \varphi' = - \sqrt{-1} k \varphi'.
\]
This implies that
\[
 \omega_\psi(\alpha, (1,1)) \varphi' = \chi_k(\alpha)^{-1} \varphi'
\]
for $\alpha \in E^1$ since
\[
 \ii(\alpha) = 
 \begin{pmatrix}
  a & b \\
  bu & a
 \end{pmatrix}
 = 
 \begin{pmatrix}
  1 & \\
  & v
 \end{pmatrix}
 \begin{pmatrix}
  \cos \theta & \sin \theta \\
  - \sin \theta & \cos \theta
 \end{pmatrix}
 \begin{pmatrix}
  1 & \\
  & v^{-1}
 \end{pmatrix} 
 = \exp((v^{-1} \mathtt{X} - v \mathtt{Y}) \theta)
\]
if we write $\alpha = a + b \i = e^{\sqrt{-1} \theta}$.
This completes the proof.
\end{proof}

\paragraph{The case when $B_1$ and $B_2$ are ramified}
\label{par:schwartzX'-ds-B1ram}

In this case, we have:
\begin{itemize}
 \item $F = \R$,
 \item $\psi(x) = e^{2 \pi \sqrt{-1} x}$,
 \item $u < 0$,
 \item $J = t^2$ for some $t \in F^\times$,
 \item $J_1 = -s^2$ for some $s \in F^\times$,
 \item $J_2 < 0$.
\end{itemize}
Put $v = |u|^{\frac{1}{2}}$.
We take the complete polarization $\V = \X' \oplus \Y'$ as in \S \ref{sssec:cpJ} and identify $\X'$ with $V^\dagger \cong B_1$ as in \S \ref{par:cpJ-i}.
For a non-negative integer $k$, we define $\varphi' = \varphi'_k \in \SS(\X')$ by 
\[
 \varphi'(x) = c_k^{-\frac{1}{2}} \cdot (x_2 - x_1 \i)^k \cdot e^{- \frac{\pi}{2v}(x_2^2 - u x_1^2 + x_4^2 - u x_3^2)}
\]
for $x = x_1 \v_1 + \dots + x_4 \v_4$, where $\v_1,\dots,\v_4$ is the basis of $\X'$ given by \eqref{eq:basisX'Y'-J-i} and 
\[
 c_k = \frac{k! |u|^{\frac{k}{2}+1}}{4 \pi^k}.
\]

\begin{lem}
\label{lem:varphi'-ds-ram}
We have $\langle \varphi', \varphi' \rangle = 1$ and
\[
 \omega_\psi(\alpha, (\alpha_1,\alpha_2)) \varphi' = \chi_{k+2}(\alpha)^{-1} \chi_k(\alpha_1) \chi_k(\alpha_2) \varphi'
\]
for $\alpha,\alpha_1,\alpha_2 \in E^\times$ such that $\nu(\alpha) = \nu(\alpha_1) \nu(\alpha_2)$.
\end{lem}

\begin{proof}
The proof is the same as that of Lemma \ref{lem:varphi'-ds-spl} and we omit the details.
Note that, in this case, we have
\begin{align*}
 \omega_\psi(\mathtt{H}) \varphi'(x) & 
 = \left( 2 + x_1 \frac{\partial}{\partial x_1} + \dots + x_4 \frac{\partial}{\partial x_4} \right)
 \varphi'(x), \\
 \omega_\psi(\mathtt{X}) \varphi'(x)
 & = \frac{\pi \sqrt{-1}}{2} (ux_1^2 - x_2^2 + ux_3^2 - x_4^2) \varphi'(x), \\
 \omega_\psi(\mathtt{Y}) \varphi'(x)
 & = -\frac{1}{2 \pi \sqrt{-1}} \left( \frac{1}{u} \frac{\partial^2}{\partial x_1^2} - \frac{\partial^2}{\partial x_2^2} + \frac{1}{u} \frac{\partial^2}{\partial x_3^2} - \frac{\partial}{\partial x_4^2} \right) \varphi'(x).
\end{align*}
\end{proof}

\subsubsection{The case (fd)}
\label{sssec:schwartzX'-fd}

We only consider the case $J_1 \in (F^{\times})^2$; we switch the roles of $B_1$ and $B_2$ in the other case.
In this case, we have:
\begin{itemize}
 \item $F = \R$,
 \item $\psi(x) = e^{2 \pi \sqrt{-1} x}$,
 \item $u < 0$,
 \item $J_1 = t^2$ for some $t \in F^\times$,
 \item $J = -s^2$ for some $s \in F^\times$,
 \item $J_2 < 0$.
\end{itemize}
Put $v = |u|^{\frac{1}{2}}$.
We take the complete polarization $\V = \X' \oplus \Y'$ and identify $\X'$ with $W = B$ as in \S \ref{sssec:cpJ1}.
For a non-negative integer $k$, we define $\varphi' = \varphi'_k \in \SS(\X')$ by 
\[
 \varphi'(x) = c_k^{-\frac{1}{2}} \cdot \left( x_1 - x_2 \frac{\i}{u} \right)^k
 \cdot e^{- \frac{\pi v}{2} ( x_1^2 - \frac{1}{u}x_2^2 + x_3^2 - \frac{1}{u}x_4^2 )}
\]
for $x = x_1 \v_1 + \dots + x_4 \v_4$, where $\v_1,\dots,\v_4$ is the basis of $\X'$ given by \eqref{eq:basisX'Y'-J1} and 
\[
 c_k = \frac{k!}{\pi^k |u|^{\frac{k}{2}+1}}.
\]

\begin{lem}
\label{lem:varphi'-fd}
We have $\langle \varphi', \varphi' \rangle = 1$ and
\[
 \omega_\psi(\alpha, (\alpha_1,\alpha_2)) \varphi' = \chi_k(\alpha)^{-1} \chi_{k+2}(\alpha_1) \chi_k(\alpha_2) \varphi'
\]
for $\alpha,\alpha_1,\alpha_2 \in E^\times$ such that $\nu(\alpha) = \nu(\alpha_1) \nu(\alpha_2)$.
\end{lem}

\begin{proof}
Recall that the Haar measure on $\X'$ is given by $dx = \frac{1}{|u|} \, dx_1 \cdots dx_4$.
We have
\begin{align*}
 & \frac{1}{|u|} \int_{F^4} \left( x_1^2 - \frac{1}{u} x_2^2 \right)^k
 e^{- \pi v (x_1^2 - \frac{1}{u}x_2^2 + x_3^2 - \frac{1}{u}x_4^2)} \, dx_1 \cdots \, dx_4 \\
 & = \frac{1}{\pi^2 |u|} \left( \frac{1}{\pi v} \right)^k
 \int_{F^4} (x_1^2 + x_2^2)^k e^{-(x_1^2 + x_2^2 + x_3^2 + x_4^2)} \, dx_1 \cdots \, dx_4 \\
 & = \frac{1}{\pi^{k+2} |u|^{\frac{k}{2}+1}} \cdot (2 \pi)^2
 \int_0^\infty \int_0^\infty r_1^{2k} e^{-(r_1^2 + r_2^2)} \, r_1 \, dr_1 \, r_2 \, dr_2 \\
 & = \frac{1}{\pi^k |u|^{\frac{k}{2}+1}}
 \int_0^\infty \int_0^\infty r_1^k e^{-(r_1 + r_2)} \, dr_1 \, dr_2 \\
 & = \frac{1}{\pi^k |u|^{\frac{k}{2}+1}} \cdot \Gamma(k+1)
\end{align*}
and hence $\langle \varphi', \varphi' \rangle = 1$.
If we write $z_1 = x_1 + x_2 \frac{\i}{u}$ and $z_2 = x_3 + x_4 \frac{\i}{u}$, then 
\[
 \varphi'(x) = c_k^{-\frac{1}{2}} \cdot (z_1^\rho)^k \cdot e^{- \frac{\pi v}{2} (z_1 z_1^\rho + z_2 z_2^\rho)},
\]
and it is easy to see that 
\[
  \omega_\psi(\alpha, (\nu^\frac{1}{2},\alpha_2)) \varphi' = \chi_k(\alpha)^{-1} \chi_k(\alpha_2) \varphi'
\]
for $\alpha, \alpha_2 \in E^\times$ and $\nu = \nu(\alpha) \nu(\alpha_2)^{-1}$.
On the other hand, we have
\begin{align*}
 \omega_\psi(\mathtt{H}) \varphi'(x) & 
 = - \left( 2 + x_1 \frac{\partial}{\partial x_1} + \dots + x_4 \frac{\partial}{\partial x_4} \right)
 \varphi'(x), \\
 \omega_\psi(\mathtt{X}) \varphi'(x)
 & = \frac{1}{4 \pi \sqrt{-1}} \left( \frac{\partial^2}{\partial x_1^2} - u \frac{\partial^2}{\partial x_2^2} + \frac{\partial^2}{\partial x_3^2} - u \frac{\partial}{\partial x_4^2} \right) \varphi'(x), \\
 \omega_\psi(\mathtt{Y}) \varphi'(x)
 & = -\pi \sqrt{-1} \left( x_1^2 - \frac{1}{u} x_2^2 + x_3^2 - \frac{1}{u} x_4^2 \right) \varphi'(x),
\end{align*}
where we identity $\GU(V)^0 \cong (B_1^\times \times B_2^\times)/F^\times$ with a subgroup of $\GL_2(B)$ via the isomorphisms $\ii_1,\ii_2$ given by \eqref{eq:isomB1-J1}, \eqref{eq:isomB2-J1}.
Thus, noting that 
\[
 \frac{\partial^2}{\partial z_1 \partial z_1^\rho} = \frac{1}{4} \left( \frac{\partial^2}{\partial x_1^2} - u \frac{\partial^2}{\partial x_2^2} \right), \qquad
 \frac{\partial^2}{\partial z_2 \partial z_2^\rho} = \frac{1}{4} \left( \frac{\partial^2}{\partial x_3^2} - u \frac{\partial^2}{\partial x_4^2} \right),
\]
we see that
\[
 \omega_\psi(2 v^{-1} \mathtt{X} - 2^{-1} v \mathtt{Y}) \varphi' = \sqrt{-1} (k+2) \varphi'.
\]
This implies that
\[
 \omega_\psi(1, (\alpha,1)) \varphi' = \chi_{k+2}(\alpha) \varphi'
\]
for $\alpha \in E^1$ since
\[
 \ii_1(\alpha) = 
 \begin{pmatrix}
  a & 2b \\
  \frac{bu}{2} & a
 \end{pmatrix}
 = 
 \begin{pmatrix}
  1 & \\
  & 2^{-1}v
 \end{pmatrix}
 \begin{pmatrix}
  \cos \theta & \sin \theta \\
  - \sin \theta & \cos \theta
 \end{pmatrix}
 \begin{pmatrix}
  1 & \\
  & 2v^{-1}
 \end{pmatrix} 
 = \exp((2 v^{-1} \mathtt{X} - 2^{-1} v \mathtt{Y})\theta)
\]
if we write $\alpha = a + b \i = e^{\sqrt{-1} \theta}$.
This completes the proof.
\end{proof}

\subsection{Schwartz functions on $\X$}
\label{ssec:schwartzX}

Let $\varphi \in \SS(\X)$ be the partial Fourier transform of the Schwartz function $\varphi' \in \SS(\X')$ given in \S \ref{ssec:schwartzX'}.
(We also write $\varphi_\mu$ and $\varphi_k$ for the partial Fourier transforms of $\varphi'_\mu$ and $\varphi_k'$, respectively, to indicate the dependence on a unitary ramified character $\mu$ in the case (rps) and on a non-negative integer $k$ in the cases (ds), (fd).)
Then we have
\[
 \langle \varphi, \varphi \rangle = \langle \varphi', \varphi' \rangle = 1.
\]
Also, since the partial Fourier transform is a $\G(\U(V)^0 \times \U(W))$-equivariant map, $\varphi$ satisfies the same equivariance properties as $\varphi'$.
In this subsection, we compute $\varphi$ explicitly.

We need to introduce more notation.
Put $\kappa_1 = 1$ and $\kappa_2 = -J_1$.
We define a quadratic $F$-algebra $K$ by 
\[
 K = F + F \j.
\]
We write
\[
 x = \e_1 z_1 + \e_2 z_2 = x_1 \e_1 + x_2\e_2+x_3 \e_3+ x_4 \e_4 \in \X, \qquad
 z_i = \alpha_i + \beta_i \j \in K,
\]
so that
\[
 \alpha_1 = x_1, \qquad
 \beta_1 = x_4, \qquad
 \alpha_2 = x_2, \qquad
 \beta_2 = \frac{1}{J_1} x_3.
\]
Recall that the Weil index $\gamma_F(\psi)$ is an $8$th root of unity such that
\[
 \int_F \phi(x) \psi(x^2) \, dx
 = \gamma_F(\psi) |2|^{-\frac{1}{2}} \int_F \hat{\phi}(x) \psi\left( -\frac{x^2}{4} \right) dx
\]
for all $\phi \in \SS(F)$,
where $\hat{\phi}$ is the Fourier transform of $\phi$ with respect to $\psi$ and $dx$ is the self-dual Haar measure on $F$ with respect to $\psi$.
For any non-negative integer $n$, let $H_n(x)$ denote the Hermite polynomial defined by
\[
 H_n(x) = (-1)^n e^{x^2} \frac{d^n}{dx^n} \left( e^{-x^2} \right).
\]

\subsubsection{The case (ur)}

\paragraph{The case when $E$ is split and $F$ is unramified}

We use the notation of \S \ref{par:schwartzX'-ur-Espl-Fur}.
By the partial Fourier transform given in \S \ref{sssec:pftu}, we have $\varphi(x) = |2|^{-1} \prod_{i=1}^4 \varphi_i (x_i)$, where 
\[
 \varphi_i(x_i)
  = \left| \frac{2a_i}{b_i} \right| \psi \left( \frac{a_i}{2} x_i^2 \right)
 \int_\o \psi \left( \frac{a_i}{b_i^2} (x_i')^2 - \frac{2a_i}{b_i} x_i x_i' \right) dx_i'.
\]

\begin{lem}
\label{lem:main-fin}
Assume that $\psi$ is of order zero.
Put
\[
 I(a,b) = \int_{\o} \psi(ax^2 + bx) \, dx
\]
for $a, b \in F$.
\begin{enumerate}
\item 
We have
\[
 I(a, b) = 
 \begin{cases}
  \I_{\o}(b) & \text{if $a \in \o$,} \\
  \psi(-\frac{b^2}{4a}) \gamma_F(a \psi) |2a|^{-\frac{1}{2}} \I_{\o}(\frac{b}{2a}) & \text{if $4a \notin \p$,}
 \end{cases}
\]
where $a\psi$ is the non-trivial character of $F$ given by $(a\psi)(x) = \psi(ax)$.
\item
If $F = \Q_2$ and $2a \in \o^{\times}$, then we have
\[
 I(a,b) = \I_{\o^{\times}}(\tfrac{b}{a}).
\]
\end{enumerate}
\end{lem}

\begin{proof}
If $a \in \o$, then we have $I(a,b) = \hat{\I}_{\o}(b) = \I_{\o}(b)$.
If $2a \notin \p$, then we change the variable $x \mapsto x + \frac{y}{2a}$ with $y \in \o$ to get
\begin{align*}
 I(a,b)
 & = \int_{\o} \psi\left( a \left( x + \frac{y}{2a} \right)^2 + b \left( x + \frac{y}{2a} \right) \right) dx \\
 & = \psi\left( \frac{y^2}{4a} + \frac{by}{2a} \right) \int_{\o} \psi(ax^2 + xy+bx) \, dx \\
 & = \psi\left( \frac{y^2}{4a} + \frac{by}{2a} \right) I(a,b).
\end{align*}
Assume that $4a \notin \p$.
Then we have $I(a,b) = \psi(\frac{by}{2a}) I(a,b)$ for all $y \in \o$, so that $I(a,b) = 0$ unless $\frac{b}{2a} \in \o$.
If $\frac{b}{2a} \in \o$, then we have 
\[
 I(a,b) = \int_{\o} \psi\left( a \left( x + \frac{b}{2a} \right)^2 - \frac{b^2}{4a} \right) dx
 = \psi \left( - \frac{b^2}{4a} \right) \int_{\o} \psi(a x^2) \, dx.
\]
On the other hand, by definition, we have
\[
 \int_F \phi(x) \psi(a x^2) \, dx
 = \gamma_F(a \psi) |2a|^{-\frac{1}{2}} \int_F \hat{\phi}(x)
 \psi\left( -\frac{x^2}{4a} \right) dx
\]
for all $\phi \in \SS(F)$.
Hence we have
\[
 \int_{\o} \psi(a x^2) \, dx = \gamma_F(a \psi) |2a|^{-\frac{1}{2}} \int_{\o} \psi \left( - \frac{x^2}{4a} \right) dx
 = \gamma_F(a \psi) |2a|^{-\frac{1}{2}}.
\]
This proves (i).

Assume that $F = \Q_2$ and $2a \in \o^{\times}$.
As above, we have $I(a,b) = \psi(\frac{by}{2a}) I(a,b)$ for all $y \in 2\o$, so that $I(a,b) = 0$ unless $\frac{b}{a} \in \o$.
Also, we have
\[
 I(a,b) = \psi \left( - \frac{b^2}{4a} \right) \int_{\o} \psi(a x^2) \, dx
\]
if $\frac{b}{a} \in 2\o$.
Changing the variable $x \mapsto x+1$, we have
\[
 \int_{\o} \psi(a x^2) \, dx
 = \int_{\o} \psi(a x^2 + 2ax + a) \, dx 
 = \psi(a) \int_{\o} \psi(a x^2) \, dx.
\]
Since $F = \Q_2$, $\psi$ is of order zero, and $\psi(a)^2 = \psi(2a) = 1$, we must have $\psi(a) = -1$.
Hence we have
\[
 \int_{\o} \psi(a x^2) \, dx = 0,
\]
so that $I(a,b)=0$ if $\frac{b}{a} \in 2\o$.
Assume that $\frac{b}{a} \in \o^\times$.
Since $F = \Q_2$, we may write $a = y + \frac{1}{2}$ and $\frac{b}{2a} = z + \frac{1}{2}$ for some $y,z \in \o$.
Then we have
\[
 I(a,b) = \int_{\o}\psi \left( x^2 \left( y+\frac{1}{2} \right) + 2 x \left( y+\frac{1}{2} \right) \left( z+\frac{1}{2} \right) \right) dx = \int_{\o}\psi \left( \frac{1}{2} x^2 + \frac{1}{2} x \right) dx = 1
\]
since $\frac{1}{2} x(x+1) \in \o$ for all $x \in \o$.
This proves (ii).
\end{proof}

By Lemma \ref{lem:main-fin}, we have
\[
 I\left( \frac{a_i}{b_i^2}, -\frac{2a_i}{b_i} x_i \right) = 
 \begin{cases}
  \I_{\o}(\frac{2a_i}{b_i} x_i) & \text{if $\frac{a_i}{b_i^2} \in \o$,} \\
  \psi(-a_ix_i^2) \gamma_F(a_i \psi) \left| \frac{b_i^2}{2a_i} \right|^{\frac{1}{2}} \I_{\o}(b_i x_i) & \text{if $\frac{4a_i}{b_i^2} \notin \p$,}
 \end{cases}
\]
so that 
\begin{align*}
 \varphi_1(x_1) & = |2| \cdot \psi \left( \frac{t}{2} x_1^2 \right) \cdot \I_\o(2x_1), \\
 \varphi_2(x_2) & = |2J_1| \cdot \psi \left( - \frac{tJ_1}{2} x_2^2 \right) \cdot \I_\o(2J_1x_2), \\
 \varphi_3(x_3) & = \gamma_F(-tJ_2 \psi) \cdot |2J_2|^{\frac{1}{2}} \cdot \psi \left( \frac{tJ_2}{2} x_3^2 \right) \cdot \I_{\o}(2Jx_3), \\
 \varphi_4(x_4) & = \gamma_F(tJ \psi) \cdot |2J|^{\frac{1}{2}} \cdot \psi \left( -\frac{tJ}{2} x_4^2 \right) \cdot \I_{\o}(2Jx_4).
\end{align*}
We have 
\begin{align*}
 \gamma_F(-tJ_2 \psi) \cdot \gamma_F(tJ \psi)
 & = \gamma_F(-2tJ_2, \frac{1}{2} \psi) \cdot \gamma_F(2tJ, \frac{1}{2} \psi) \cdot \gamma_F(\frac{1}{2} \psi)^2 \\
 & = \gamma_F(-J_1, \frac{1}{2} \psi) \cdot (-2tJ_2, 2tJ)_F \cdot \gamma_F(-1, \frac{1}{2} \psi)^{-1} \\
 & = \gamma_F(J_1, \frac{1}{2} \psi) \cdot (2tJ_2, J_1)_F.
\end{align*}
Hence we have
\begin{align*}
 \varphi(x) & = \gamma_F(J_1, \frac{1}{2} \psi) \cdot (2tJ_2, J_1)_F \cdot |2|^2 |J_1|^{\frac{3}{2}} |J_2| \\
 & \quad \times \psi\left( \frac{t}{2} (x_1^2 -J_1 x_2^2 + J_2 x_3^2 - Jx_4^2) \right)
 \cdot \I_\o(2x_1) \I_\o(2J_1x_2) \I_{\o}(2Jx_3) \I_{\o}(2Jx_4) \\
 & = \gamma_F(J_1, \frac{1}{2} \psi) \cdot (2tJ_2, J_1)_F \cdot |2|^2 |J_1|^{\frac{3}{2}} |J_2|
 \cdot \tilde{\varphi}_{\kappa_1}(z_1) \tilde{\varphi}_{\kappa_2}(z_2),
\end{align*}
where 
\[
 \tilde{\varphi}_\kappa(z) = \psi\left( \frac{\kappa t}{2} \N_{K/F}(z) \right) 
 \cdot \I_{\o + \o \frac{\j}{J}}(2 \kappa z)
\]
for $z \in K$.

\paragraph{The case when $E$ is split and $F$ is ramified}

We use the notation of \S \ref{par:schwartzX'-ur-Espl-Fram}.
Note that the inverse different $\mathfrak{d}^{-1}$ is equal to $\varpi^{-d} \o$.
By the partial Fourier transform given in \S \ref{sssec:pftu},
we have $\varphi(x) = q^d |2|^{-1} \prod_{i=1}^4 \varphi_i (x_i)$, where 
\[
 \varphi_i(x_i)
  = \left| \frac{2a_i}{b_i} \right| \psi \left( \frac{a_i}{2} x_i^2 \right)
 \int_\o \psi \left( \frac{a_i}{b_i^2} (x_i')^2 - \frac{2a_i}{b_i} x_i x_i' \right) dx_i'.
\]
In particular, we have
\[
 \varphi_i(x_i) \in \Z[q^{-\frac{1}{2}}, \mu_{p^\infty}]
\]
for all $x_i \in F$, where $p$ is the residual characteristic of $F$ and $\mu_{p^\infty}$ is the group of $p$-power roots of unity.
If further $2 \in \o^\times$, then we have
\[
 \varphi_i(x_i) = \psi \left( \frac{a_i}{2} x_i^2 \right) \cdot \hat{\I}_\o\left( - \frac{2a_i}{b_i} x_i \right) 
 = q^{-\frac{d}{2}} \cdot \psi \left( \frac{a_i}{2} x_i^2 \right) \cdot \I_{\mathfrak{d}^{-1}}(x_i)
\]
and hence
\begin{align*}
 \varphi(x) & = q^{-d} \cdot \psi\left( \frac{t}{2}(x_1^2 - J_1 x_2^2 - J_2 x_3^2 + J x_4^2) \right)
 \cdot \I_{\mathfrak{d}^{-1}}(x_1) \I_{\mathfrak{d}^{-1}}(x_2) \I_{\mathfrak{d}^{-1}}(x_3) \I_{\mathfrak{d}^{-1}}(x_4) \\
 & = q^{-d} \cdot \tilde{\varphi}_{\kappa_1}(z_1) \tilde{\varphi}_{\kappa_2}(z_2),
\end{align*}
where
\[
 \tilde{\varphi}_\kappa(z) = \psi\left( \frac{\kappa t}{4} \tr_{K/F}(z^2) \right) \cdot \I_{\mathfrak{d}^{-1}+\mathfrak{d}^{-1}\j}(z)
\]
for $z \in K$.

\paragraph{The case when $E$ is inert and $J \in (F^{\times})^2$}

We use the notation of \S \ref{par:schwartzX'-ur-Einert-J}.
By the partial Fourier transform given in \S \ref{par:pftJi}, we have
\begin{align*}
 \varphi(x) & = |J|^{\frac{1}{2}} \cdot
 \I_\o(x_1 + tx_4) \hat{\I}_\o\left( - \frac{1}{2}(x_1 - t x_4) \right)
 \I_\o\left( s \left( x_2 + \frac{t}{J_1} x_3 \right) \right)
 \hat{\I}_\o\left( \frac{J_1}{2s} \left( x_2 - \frac{t}{J_1} x_3 \right) \right) \\
 & = |J|^{\frac{1}{2}} \cdot \I_\o(x_1) \I_\o(s x_2) \I_\o\left(\frac{st}{J_1} x_3\right) \I_\o(t x_4) \\
 & = |J|^{\frac{1}{2}} \cdot \I_{\o + \o \frac{\j}{t}}(z_1) \I_{\o + \o \frac{\j}{t}}(s z_2).
\end{align*}

\paragraph{The case when $E$ is inert, and $J_1 \in (F^{\times})^2$ or $J_2 \in (F^{\times})^2$}

We use the notation of \S \ref{par:schwartzX'-ur-Einert-J1}.
By the partial Fourier transform given in \S \ref{sssec:pftJ1}, we have
\begin{align*}
 \varphi(x) & = |J|^{\frac{1}{2}} \cdot
 \I_{\o}(x_1 + t x_2) \hat{\I}_{\o}\left( - \frac{1}{2}(x_1 - t x_2) \right)
 \I_\o\left( s \left( x_4 + \frac{1}{t} x_3 \right) \right)
 \hat{\I}_\o\left( \frac{J}{2s} \left( x_4 - \frac{1}{t} x_3 \right) \right) \\
 & = |J|^{\frac{1}{2}} \cdot \I_{\o}(x_1) \I_{\o}(tx_2) \I_{\o}\left( \frac{s}{t} x_3 \right) \I_{\o}(sx_4) \\
 & = |J|^{\frac{1}{2}} \cdot \I_{\o + \o \frac{\j}{s}}(z_1) \I_{\o + \o \frac{\j}{s}}(t z_2).
\end{align*}

\paragraph{The case when $E$ is ramified}

We use the notation of \S \ref{par:schwartzX'-ur-Eram}.
By the partial Fourier transform given in \S \ref{par:pftJii}, we have
\begin{align*}
 \varphi(x) & = q^{-\frac{1}{2}} \cdot
 \I_\o\left( \frac{1}{2} \left( x_1 + t_1 x_2 + \frac{t}{t_1} x_3 + t x_4 \right) \right)
 \hat{\I}_\o\left( - \frac{1}{2} \left( x_1 - t_1 x_2 + \frac{t}{t_1} x_3 - t x_4 \right) \right) \\
 & \quad \times \hat{\I}_\o\left( - \frac{1}{2} \left( x_1 + t_1 x_2 - \frac{t}{t_1} x_3 - t x_4 \right) \right)
 \I_\o\left( \frac{u}{2} \left( x_1 - t_1 x_2 - \frac{t}{t_1} x_3 + t x_4 \right) \right) \\
 & = q^{-\frac{1}{2}} \cdot \I_\o(x_1 - t x_4) \I_{\o}\left( x_2 - \frac{t}{J_1} x_3 \right)
 \I_\o\left( x_1 + t_1 x_2 + \frac{t}{t_1} x_3 + t x_4 \right)
 \I_{\varpi^{-1} \o}\left( x_1 - t_1 x_2 - \frac{t}{t_1} x_3 + t x_4 \right) \\
 & = q^{-\frac{1}{2}} \cdot \I_\o(\alpha_1 - t \beta_1) \I_{\o}(\alpha_2 - t \beta_2)
 \I_\o(\alpha_1 + t \beta_1 + t_1 \alpha_2 + tt_1 \beta_2)
 \I_{\varpi^{-1} \o}(\alpha_1 + t \beta_1 - t_1 \alpha_2 - tt_1 \beta_2).
\end{align*}

\subsubsection{The case (rps)}

We use the notation of \S \ref{sssec:schwartzX'-tr}.
By the partial Fourier transform given in \S \ref{sssec:pftu}, we have $\varphi(x) = q^{\frac{n+1}{2}}(q-1)^{-\frac{1}{2}}\prod_{i=1}^4 \varphi_i (x_i)$, where 
\[
 \varphi_i(x_i) = \psi \left( \frac{a_i}{2} x_i^2 \right) \cdot \hat{\I}_\o\left( - \frac{2a_i}{b_i} x_i \right)
 = \I_\o(x_i)
\]
for $i=1,2$,
\[
 \varphi_3(x_3) = \psi \left( \frac{a_3}{2} x_3^2 \right)
 \cdot \hat{\I}_{\varpi^n \o}\left( - \frac{2a_3}{b_3} x_3 \right) = q^{-n} \cdot 
 \psi\left( - \frac{tJ_2}{2} x_3^2 \right) \cdot \I_{\varpi^{-n} \o}(x_3), 
\]
and 
\[
 \varphi_4(x_4)
 = \psi \left( \frac{a_4}{2} x_4^2 \right) \cdot \widehat{\I_{\o^\times} \mu}
 \left( - \frac{2a_4}{b_4} x_4 \right)
 = \psi \left( \frac{tJ}{2} x_4^2 \right) \cdot \widehat{\I_{\o^\times} \mu} (x_4).
\]
Since $\mu$ is of conductor $q^n$, we have
\[
 \widehat{\I_{\o^\times}\mu} = q^{-n} \cdot \mathfrak{g}(\mu, \psi) \cdot \I_{\varpi^{-n} \o^\times} \mu^{-1},
\]
where 
\[
 \mathfrak{g}(\mu, \psi) = \int_{\varpi^{-n} \o^{\times}} \mu(y) \psi(y) \, dy.
\]
Note that $|\mathfrak{g}(\mu, \psi)| = q^{\frac{n}{2}}$.
Hence we have
\begin{align*}
 \varphi(x) & = q^{-\frac{3}{2}n + \frac{1}{2}}(q-1)^{-\frac{1}{2}} \cdot \mathfrak{g}(\mu, \psi) \\
 & \quad \times \psi\left( \frac{t}{2}(-J_2x_3^2 + J x_4^2) \right) \cdot 
 \I_\o(x_1) \I_\o(x_2) \I_{\varpi^{-n} \o}(x_3) \I_{\varpi^{-n} \o^\times}(x_4) \mu(x_4)^{-1} \\
 & = q^{-\frac{3}{2}n + \frac{1}{2}}(q-1)^{-\frac{1}{2}} \cdot \mathfrak{g}(\mu, \psi) \\
 & \quad \times \psi\left( \frac{\kappa_1 t J}{2} \beta_1^2 \right) \cdot \I_\o(\alpha_1) \I_{\varpi^{-n} \o^\times}(\beta_1) \mu(\beta_1)^{-1}
 \cdot \psi\left( \frac{\kappa_2 t J}{2} \beta_2^2 \right) \cdot \I_\o(\alpha_2) \I_{\varpi^{-n} \o}(\beta_2).
\end{align*}

\subsubsection{The case (st)}

\paragraph{The case when $B_1$ and $B_2$ are split}

We use the notation of \S \ref{par:schwartzX'-st-B1spl}.
By the partial Fourier transform given in \S \ref{sssec:pftu}, we have $\varphi(x) = q^{\frac{1}{2}} \prod_{i=1}^4 \varphi_i(x_i)$, where
\[
 \varphi_i(x_i) = \psi \left( \frac{a_i}{2} x_i^2 \right) \cdot \hat{\I}_\o\left( - \frac{2a_i}{b_i} x_i \right)
 = \I_\o(x_i)
\]
for $i=1,2,4$ and 
\[
 \varphi_3(x_3) = \psi \left( \frac{a_3}{2} x_3^2 \right) \cdot \hat{\I}_\p \left( - \frac{2a_3}{b_3} x_3 \right)
 = q^{-1} \cdot \psi \left( - \frac{tJ_2}{2} x_3^2 \right) \cdot \I_{\varpi^{-1} \o}(x_3).
\]
Hence we have
\begin{align*}
 \varphi(x) & = q^{-\frac{1}{2}} \cdot \psi\left( - \frac{tJ_2}{2} x_3^2 \right)
 \cdot \I_{\o}(x_1) \I_{\o}(x_2) \I_{\varpi^{-1} \o}(x_3) \I_{\o}(x_4) \\
 & = q^{-\frac{1}{2}} \cdot \I_\o(\alpha_1) \I_\o(\beta_1) \cdot
 \psi\left( \frac{\kappa_2 t J}{2} \beta_2^2 \right) \cdot \I_\o(\alpha_2) \I_{\varpi^{-1} \o}(\beta_2).
\end{align*}

\paragraph{The case when $B_1$ and $B_2$ are ramified}

We use the notation of \S \ref{par:schwartzX'-st-B1ram}.
By the partial Fourier transform given in \S \ref{par:pftJi}, we have
\begin{align*}
 \varphi(x) & = q^{-1} \cdot  \I_\o(x_1 + tx_4) \hat{\I}_\o\left( - \frac{1}{2}(x_1 - t x_4) \right)
 \I_\o\left( s \left( x_2 + \frac{t}{J_1} x_3 \right) \right)
 \hat{\I}_\o \left( \frac{J_1}{2s} \left( x_2 - \frac{t}{J_1} x_3 \right) \right) \\
 & = q^{-1} \cdot \I_\o(x_1) \I_{\varpi^{-1} \o}(x_4)
 \I_\o \left( x_2 + \frac{t}{J_1} x_3 \right) \I_{\varpi^{-1} \o} \left( x_2 - \frac{t}{J_1} x_3 \right) \\
 & = q^{-1} \cdot \I_\o(\alpha_1) \I_{\varpi^{-1} \o}(\beta_1)
 \I_\o(\alpha_2 + t \beta_2) \I_{\varpi^{-1} \o}(\alpha_2 - t \beta_2).
\end{align*}

\subsubsection{The case (1d)}

We use the notation of \S \ref{sssec:schwartzX'-1d}.
By the partial Fourier transform given in \S \ref{sssec:pftJ1}, we have
\begin{align*}
 \varphi(x) & = q^{-\frac{1}{2}} \cdot \I_\o(x_1 + t x_2) \hat{\I}_\o\left( - \frac{1}{2}(x_1 - t x_2) \right)
 \I_\o\left( s \left( x_4 + \frac{1}{t} x_3 \right) \right)
 \hat{\I}_\o\left( \frac{J}{2s} \left( x_4 - \frac{1}{t} x_3 \right) \right) \\
 & = q^{-\frac{1}{2}} \cdot \I_\o(x_1) \I_\o(x_2) 
 \I_\o\left( x_4 + \frac{1}{t} x_3 \right) \I_{\varpi^{-1} \o}\left( x_4 - \frac{1}{t} x_3 \right) \\
 & = q^{-\frac{1}{2}} \cdot \I_\o(\alpha_1) \I_\o(\alpha_2) 
 \I_\o(\beta_1 + t \beta_2) \I_{\varpi^{-1} \o}(\beta_1 - t \beta_2).
\end{align*}

\subsubsection{The case (ds)}

\paragraph{The case when $B_1$ and $B_2$ are split}

We use the notation of \S \ref{par:schwartzX'-ds-B1spl}.
By the partial Fourier transform given in \S \ref{par:pftJi}, we have
\[
 \varphi(x) = c_k^{-\frac{1}{2}} \left| \frac{uJ}{4} \right|^{\frac{1}{2}}
 \int_{-\infty}^\infty (x_2' - v \sqrt{-1} x_1')^k e^{- \frac{\pi}{2}(v x_1'^2 + \frac{1}{v} x_2'^2)} e^{2 \pi \sqrt{-1} x_2'y_2'} \, dx_2'
 \int_{-\infty}^\infty e^{- \frac{\pi}{2}(v x_3'^2 + \frac{1}{v} x_4'^2)} e^{2 \pi \sqrt{-1} x_4'y_4'} \, dx_4'.
\]

\begin{lem}
\label{lem:main-real}
Let $k$ be a non-negative integer and $v$ a positive real number.
Put
\[
 I(x,y) = \int_{-\infty}^\infty
 \left(x + \frac{\sqrt{-1}}{v} w \right)^k e^{-\frac{\pi}{2}(vx^2 + \frac{1}{v} w^2)} e^{2 \pi \sqrt{-1} wy} \, dw
\]
for $x, y \in \R$.
Then we have
\[
 I(x,y) = \frac{1}{\sqrt{2^{k-1} \pi^k v^{k-1}}}
 \cdot H_k\left( \sqrt{2 \pi v} \left( \frac{1}{2}x-y \right) \right)
 \cdot e^{-\pi v(\frac{1}{2} x^2 + 2 y^2)}.
\]
\end{lem}

\begin{proof}
We have
\begin{align*}
 I(x,y) & = e^{-\frac{\pi v}{2} x^2} \sqrt{\frac{2 v}{\pi}} \int_{-\infty}^\infty 
 \left(x + \sqrt{\frac{2}{\pi v}} \sqrt{-1} w \right)^k e^{-w^2} e^{2 \sqrt{2 \pi v} \sqrt{-1} wy} \, dw \\
 & = e^{-\frac{\pi v}{2} x^2} \sqrt{\frac{2^{k+1}}{\pi^{k+1} v^{k-1}}} \int_{-\infty}^\infty 
 \left(\sqrt{\frac{\pi v}{2}} x + \sqrt{-1} w \right)^k e^{-(w - \sqrt{2 \pi v} \sqrt{-1} y)^2 - 2 \pi v y^2} \, dw \\
 & = e^{-\frac{\pi v}{2} x^2 - 2 \pi v y^2} \sqrt{\frac{2^{k+1}}{\pi^{k+1} v^{k-1}}} \int_{-\infty}^\infty
 \left(\sqrt{\frac{\pi v}{2}} x + \sqrt{-1} w - \sqrt{2 \pi v} y \right)^k e^{-w^2} \, dw.
\end{align*}
Hence the assertion follows from the integral representation of the Hermite polynomial:
\[
 H_k(x) = \frac{2^k}{\sqrt{\pi}} \int_{-\infty}^\infty (x + \sqrt{-1} w)^k e^{-w^2} \, dw.
\]
\end{proof}

By Lemma \ref{lem:main-real}, we have
\begin{align*}
 \varphi(x) & = c_k^{-\frac{1}{2}} \left| \frac{uJ}{4} \right|^{\frac{1}{2}}
 \cdot \frac{(-v \sqrt{-1})^k}{\sqrt{2^{k-2} \pi^k v^{k-2}}}
 \cdot H_k\left( \sqrt{2 \pi v} \left( \frac{1}{2} x_1' - y_2' \right) \right)
 \cdot e^{-\pi v(\frac{1}{2} x_1'^2 + 2 y_2'^2 + \frac{1}{2} x_3'^2 + 2 y_4'^2)} \\
 & = \frac{|uJ|^{\frac{1}{2}}(-\sqrt{-1})^k}{2^{\frac{k}{2}-1} \sqrt{k!}}
 \cdot H_k(\sqrt{2 \pi v} x_1)
 \cdot e^{-\pi v(x_1^2 + J x_4^2 + J_1 x_2^2 + \frac{J}{J_1} x_3^2)} \\
 & = \frac{|uJ|^{\frac{1}{2}} (-\sqrt{-1})^k}{2^{\frac{k}{2}-1} \sqrt{k!}}
 \cdot H_k(\sqrt{2 \pi v} \alpha_1)
 \cdot e^{-\pi v(\alpha_1^2 + J \beta_1^2)} \cdot e^{-\pi v J_1 (\alpha_2^2 + J \beta_2^2)}.
\end{align*}

\paragraph{The case when $B_1$ and $B_2$ are ramified}

We use the notation of \S \ref{par:schwartzX'-ds-B1ram}.
By the partial Fourier transform given in \S \ref{par:pftJi}, we have
\[
 \varphi(x) = c_k^{-\frac{1}{2}} \left| \frac{uJ}{4} \right|^{\frac{1}{2}}
 \int_{-\infty}^\infty (x_2' - v \sqrt{-1} x_1')^k e^{- \frac{\pi}{2}(v x_1'^2 + \frac{1}{v} x_2'^2)} e^{2 \pi \sqrt{-1} x_2'y_2'} \, dx_2'
 \int_{-\infty}^\infty e^{- \frac{\pi}{2}(v x_3'^2 + \frac{1}{v} x_4'^2)} e^{2 \pi \sqrt{-1} x_4'y_4'} \, dx_4'.
\]
By Lemma \ref{lem:main-real}, we have
\begin{align*}
 \varphi(x) & = c_k^{-\frac{1}{2}} \left| \frac{uJ}{4} \right|^{\frac{1}{2}}
 \cdot \frac{(-v \sqrt{-1})^k}{\sqrt{2^{k-2} \pi^k v^{k-2}}}
 \cdot H_k\left( \sqrt{2 \pi v} \left( \frac{1}{2} x_1' - y_2' \right) \right)
 \cdot e^{-\pi v(\frac{1}{2} x_1'^2 + 2 y_2'^2 + \frac{1}{2} x_3'^2 + 2 y_4'^2)} \\
 & = \frac{|uJ|^{\frac{1}{2}}(-\sqrt{-1})^k}{2^{\frac{k}{2}-1} \sqrt{k!}}
 \cdot H_k(\sqrt{2 \pi v} x_1)
 \cdot e^{-\pi v(x_1^2 + J x_4^2 - J_1 x_2^2 - \frac{J}{J_1} x_3^2)} \\
 & = \frac{|uJ|^{\frac{1}{2}} (-\sqrt{-1})^k}{2^{\frac{k}{2}-1} \sqrt{k!}}
 \cdot H_k(\sqrt{2 \pi v} \alpha_1)
 \cdot e^{-\pi v (\alpha_1^2 + J \beta_1^2)} \cdot e^{\pi v J_1 (\alpha_2^2 + J \beta_2^2)}.
\end{align*}

\subsubsection{The case (fd)}

We use the notation of \S \ref{sssec:schwartzX'-fd}.
By the partial Fourier transform given in \S \ref{sssec:pftJ1}, we have
\[
 \varphi(x) = c_k^{-\frac{1}{2}} \left| \frac{J}{u} \right|^{\frac{1}{2}}
 \int_{-\infty}^\infty \left(x_1' + \frac{\sqrt{-1}}{v} x_2' \right)^k
 e^{- \frac{\pi}{2}(v x_1'^2 + \frac{1}{v} x_2'^2)} e^{2 \pi \sqrt{-1} x_2'y_2'} \, dx_2'
 \int_{-\infty}^\infty e^{- \frac{\pi}{2}(v x_3'^2 + \frac{1}{v} x_4'^2)} e^{2 \pi \sqrt{-1} x_4'y_4'} \, dx_4'.
\]
By Lemma \ref{lem:main-real}, we have
\begin{align*}
 \varphi(x) & = c_k^{-\frac{1}{2}} \left| \frac{J}{u} \right|^{\frac{1}{2}} \cdot \frac{1}{\sqrt{2^{k-2} \pi^k v^{k-2}}}
 \cdot H_k\left( \sqrt{2 \pi v} \left( \frac{1}{2} x_1' - y_2' \right) \right)
 \cdot e^{-\pi v(\frac{1}{2} x_1'^2 + 2 y_2'^2 + \frac{1}{2} x_3'^2 + 2 y_4'^2)} \\
 & = \frac{|uJ|^\frac{1}{2}}{2^{\frac{k}{2}-1} \sqrt{k!}} \cdot 
 H_k(\sqrt{2 \pi v} x_1) \cdot e^{-\pi v (x_1^2 + J_1 x_2^2 - J x_4^2 - \frac{J}{J_1} x_3^2)} \\
 & = \frac{|uJ|^\frac{1}{2}}{2^{\frac{k}{2}-1} \sqrt{k!}} \cdot 
 H_k(\sqrt{2 \pi v} \alpha_1) \cdot e^{-\pi v (\alpha_1^2 - J \beta_1^2)} 
 \cdot e^{-\pi v J_1 (\alpha_2^2 - J \beta_2^2)}.
\end{align*}

%% file: rallis-B-computation.tex
\section{Explicit form of the Rallis inner product formula}
\label{sec:rallis-explicit}

In this section, we shall explicate the Rallis inner product formula (Proposition \ref{prop:rallis-B}).

\subsection{Measures}
\label{ssec:measures}

In \S \ref{sec:rallis-general}, for any connected reductive algebraic group $G$ over a number field $F$, we have always taken the Tamagawa measure on $G(\A)$, which is a product of Haar measures on $G_v$ defined in terms of a non-zero invariant differential form of top degree on $G$ over $F$. However, with respect to this Haar measure, the volume of a hyperspecial maximal compact subgroup of $G_v$ is not necessarily $1$ for almost all $v$.
For our applications, it is more convenient to take the ``standard'' measure on $G(\A)$, which is a product of Haar measures on $G_v$ such that the volume of a maximal compact subgroup of $G_v$ is $1$ for all $v$.
In this subsection, we give a precise definition of the standard measures on $\A^\times \backslash B^\times(\A)$ and $B^1(\A)$, where $B$ is a quaternion algebra over $F$, and compare them with the Tamagawa measures.

Let $F$ be a number field and $\psi$ the standard additive character of $\A/F$.
Let $D = D_F$ be the discriminant of $F$.
We have $|D| = \prod_{v \in \Sigma_\fin} q_v^{d_v}$, where $d_v$ is the non-negative integer such that $\psi_v$ is trivial on $\varpi_v^{-d_v} \o_v$ but non-trivial on $\varpi_v^{-d_v-1} \o_v$.
For each place $v$ of $F$, let $\zeta_v(s)$ be the local zeta function of $F_v$ defined by 
\[
 \zeta_v(s) =
 \begin{cases}
  (1 - q_v^{-s})^{-1} & \text{if $v$ is finite,} \\
  \pi^{-\frac{s}{2}} \Gamma(\frac{s}{2}) & \text{if $v$ is real,} \\
  2 (2 \pi)^{-s} \Gamma(s) & \text{if $v$ is complex.}
 \end{cases}
\]
Note that
\[
 \zeta_v(1) =
 \begin{cases}
  1 & \text{if $v$ is real,} \\
  \pi^{-1} & \text{if $v$ is complex.}
 \end{cases}
\]
Let $\zeta_F(s) = \prod_{v \in \Sigma_\fin} \zeta_v(s)$ be the Dedekind zeta function of $F$.
Put
\[
 \rho_F := \operatorname*{Res}_{s=1} \zeta_F(s) = \frac{2^{r_1} (2\pi)^{r_2} h R}{|D|^{\frac{1}{2}} w},
\]
where $r_1$ is the number of the real places of $F$, $r_2$ is the number of the complex places of $F$,
$h = h_F$ is the class number of $F$, $R = R_F$ is the regulator of $F$, and $w = w_F$ is the number of roots of unity in $F$.
For any connected reductive algebraic group $G$ over $F$, let $\tau(G)$ denote the Tamagawa number of $G$.

From now on, we assume that $F$ is totally real.

\subsubsection{Measures on $\A^\times$}
\label{sssec:measures-Gm}

For each place $v$ of $F$, we define a Haar measure $d^\times x_v^\Tam$ on $F_v^\times$ by 
\[
 d^\times x_v^\Tam := \zeta_v(1) \cdot \frac{dx_v}{|x_v|},
\]
where $dx_v$ is the self-dual Haar measure on $F_v$ with respect to $\psi_v$.
Note that:
\begin{itemize}
 \item $\vol(\o_v, dx_v) = q_v^{-\frac{d_v}{2}}$ if $v$ is finite,
 \item $dx_v$ is the Lebesgue measure if $v$ is real.
\end{itemize}
Then the Tamagawa measure on $\A^\times$ is given by
\[
 d^\times x^\Tam := \rho_F^{-1} \cdot \prod_v d^\times x_v^\Tam.
\]
We have $\tau(\mathbb{G}_m) = 1$.

On the other hand, we define the standard measure on $\A^\times$ as a product measure $d^\times x := \prod_v d^\times x_v$, where 
\begin{itemize}
 \item $d^\times x_v$ is the Haar measure on $F_v^\times$ such that $\vol(\o_v^\times, d^\times x_v) = 1$ if $v$ is finite,
 \item $d^\times x_v = \frac{dx_v}{|x_v|}$ if $v$ is real.
\end{itemize}
We have
\begin{equation}
\label{eq:measure-compare-Gm}
 d^\times x^\Tam = |D|^{-\frac{1}{2}} \rho_F^{-1} \cdot d^\times x.
\end{equation}

\subsubsection{Measures on $B^\times(\A)$}
\label{sssec:measures-B}

For each place $v$ of $F$, we define a Haar measure $d^\times \ba_v^\Tam$ on $B_v^\times$ by
\[
 d^\times \ba_v^\Tam := \zeta_v(1) \cdot \frac{d \ba_v}{|\nu(\ba_v)|^2},
\]
where $d \ba_v$ is the self-dual Haar measure on $B_v$ with respect to the pairing $(\ba_v, \bb_v) \mapsto \psi_v(\tr_{B_v/F_v}(\ba_v \bb_v))$.
Then the Tamagawa measure on $B^\times(\A)$ is given by 
\[
 d^\times \ba^\Tam := \rho_F^{-1} \cdot \prod_v d^\times \ba_v^\Tam.
\]
Also, the Tamagawa measure on $(B^\times / \mathbb{G}_m)(\A) = B^\times(\A) / \A^\times$ is given by the quotient measure $d^\times\ba^\Tam / d^\times x^\Tam$.
We have $\tau(B^\times/\mathbb{G}_m) = 2$.

On the other hand, we define the standard measure on $B^\times(\A)$ as a product measure $d^\times \ba := \prod_v d^\times \ba_v$, where $d^\times \ba_v$ is given as follows:
\begin{itemize}
\item 
For $v \in \Sigma_\fin \smallsetminus \Sigma_{B,\fin}$, fix an isomorphism $i_v : B_v \rightarrow \M_2(F_v)$ of quaternion $F_v$-algebras and let $d^\times \ba_v$ be the Haar measure on $B_v^\times$ such that $\vol(i_v^{-1}(\GL_2(\o_v)), d^\times \ba_v) = 1$.
Since $i_v$ is unique up to inner automorphisms, $d^\times \ba_v$ is independent of the choice of $i_v$.
\item
For $v \in \Sigma_{B,\fin}$, let $d^\times \ba_v$ be the Haar measure on $B_v^\times$ such that $\vol(\o_{B_v}^\times, d^\times \ba_v) = 1$, where $\o_{B_v}$ is the unique maximal order in $B_v$.
\item
For $v \in \Sigma_\infty \smallsetminus \Sigma_{B,\infty}$, fix an isomorphism $i_v : B_v \rightarrow \M_2(F_v)$ of quaternion $F_v$-algebras and define a Haar measure $d^\times \ba_v$ on $B_v^\times$ by
\[
 d^\times \ba_v = \frac{dx_v \, dy_v}{|y_v|^2} \, \frac{dz_v}{z_v} \, d \kappa_v
\]
for $\ba_v = i_v^{-1} \left( \smat{1}{x_v}{}{1} \big( \begin{smallmatrix} y_v & \\ & 1 \end{smallmatrix} \big) z_v \kappa_v \right)$ with $x_v \in \R$, $y_v \in \R^\times$, $z_v \in \R^\times_+$, $\kappa_v \in \SO(2)$, where $dx_v$, $dy_v$, $dz_v$ are the Lebesgue measures and $d\kappa_v$ is the Haar measure on $\SO(2)$ such that $\vol(\SO(2), d\kappa_v) = 1$.
Since $i_v$ is unique up to inner automorphisms, $d^\times \ba_v$ is independent of the choice of $i_v$.
\item
For $v \in \Sigma_{B,\infty}$, let $d^\times \ba_v$ be the Haar measure on $B_v^\times$ such that $\vol(B_v^\times / F_v^\times, d^\times \ba_v / d^\times x_v) = 1$.
\end{itemize}
Also, we define the standard measure on $B^\times(\A) / \A^\times$ as the quotient measure $d^\times \ba / d^\times x$.

\begin{lem}
\label{lem:measure-compare-B}
We have
\[
 d^\times \ba^\Tam
 = (2 \pi)^{|\Sigma_\infty \smallsetminus \Sigma_{B,\infty}|}
 \cdot (4 \pi^2)^{|\Sigma_{B,\infty}|}
 \cdot \prod_{v \in \Sigma_{B, \fin}} (q_v-1)^{-1}
 \cdot |D|^{-2} \cdot \rho_F^{-1} \cdot \zeta_F(2)^{-1} \cdot d^\times \ba.
\]
\end{lem}

\begin{proof}
For each place $v$ of $F$, let $C_v$ be the constant such that $d^\times \ba_v^\Tam = C_v \cdot d^\times \ba_v$.
If $v \in \Sigma_\fin \smallsetminus \Sigma_{B,\fin}$, we identify $B_v$ with $\M_2(F_v)$.
Then we have $\vol(\M_2(\o_v), d\ba_v) = q_v^{-2d_v}$ and hence
\begin{align*}
 C_v & = \vol(\GL_2(\o_v), d^\times \ba_v^\Tam) \\
 & = \zeta_v(1) \cdot |\GL_2(\Ff_{q_v})| \cdot \vol(1 + \M_2(\p_v), d \ba_v) \\
 & = q_v^{-2d_v} \cdot \zeta_v(2)^{-1}.
\end{align*}
If $v \in \Sigma_{B,\fin}$, then we have $\vol(\o_{B_v}, d\ba_v) = q_v^{-2d_v-1}$ and hence 
\begin{align*}
 C_v & = \vol(\o_{B_v}^\times, d^\times \ba_v^\Tam) \\
 & = \zeta_v(1) \cdot |\Ff_{q_v^2}^\times| \cdot \vol(1 + \p_{B_v}, d \ba_v) \\
 & = q_v^{-2d_v} \cdot (q_v-1)^{-1} \cdot \zeta_v(2)^{-1}.
\end{align*}
If $v \in \Sigma_\infty \smallsetminus \Sigma_{B,\infty}$, we identify $B_v$ with $\M_2(\R)$.
Then $d^\times \ba_v^\Tam$ arises from the gauge form on $\GL_2(\R)$ determined (up to sign) by the lattice $\M_2(\Z)$ in $\Lie \GL_2(\R) = \M_2(\R)$.
Also, the measures $\frac{dx_v \, dy_v}{|y_v|^2}$, $\frac{dz_v}{z_v}$, $d \kappa_v$ in the definition of $d^\times \ba_v$ arise from the (left invariant) gauge forms determined by the lattices
\[
 \Z \begin{pmatrix} 1 & 0 \\ 0 & 0 \end{pmatrix} +
 \Z \begin{pmatrix} 0 & 1 \\ 0 & 0 \end{pmatrix}, \qquad
 \Z \begin{pmatrix} 1 & 0 \\ 0 & 1 \end{pmatrix}, \qquad
 2 \pi \Z \begin{pmatrix} 0 & 1 \\ -1 & 0 \end{pmatrix},
\]
respectively.
Hence we have
\[
 C_v = 2 \pi.
\]
If $v \in \Sigma_{B,\infty}$, we identify $B_v$ with
\[
 \H := \left\{
 \begin{pmatrix} \alpha & \beta \\ - \bar{\beta} & \bar{\alpha} \end{pmatrix}
 \, \middle| \, \alpha, \beta \in \C \right\}.
\]
Then $d^\times \ba_v^\Tam$ arises from the gauge form on $\H^\times$ determined by the lattice spanned by
\[
 \frac{1}{\sqrt{2}} \begin{pmatrix} 1 & 0 \\ 0 & 1 \end{pmatrix}, \qquad 
 \frac{1}{\sqrt{2}} \begin{pmatrix} \sqrt{-1} & 0 \\ 0 & -\sqrt{-1} \end{pmatrix}, \qquad
 \frac{1}{\sqrt{2}} \begin{pmatrix} 0 & 1 \\ -1 & 0 \end{pmatrix}, \qquad
 \frac{1}{\sqrt{2}} \begin{pmatrix} 0 & \sqrt{-1} \\ \sqrt{-1} & 0 \end{pmatrix}
\]
in $\Lie \H^\times = \H$.
Let $d^\times \dot{\ba}_v$ be the Haar measure on $\H^\times/\R^\times$ which arises from the gauge form determined by the lattice spanned by
\[
 \frac{1}{2} \begin{pmatrix} \sqrt{-1} & 0 \\ 0 & -\sqrt{-1} \end{pmatrix}, \qquad
 \begin{pmatrix} 0 & 1 \\ -1 & 0 \end{pmatrix}, \qquad
 \begin{pmatrix} 0 & \sqrt{-1} \\ \sqrt{-1} & 0 \end{pmatrix},
\]
so that we have $d^\times \ba_v^\Tam/d^\times x_v = 2 \cdot d^\times \dot{\ba}_v$.
Note that this lattice is an integral lattice in $\Lie \H^\times/\R^\times$ constructed in \cite{macdonald-volume}.
It follows from \cite{macdonald-volume} that $\vol(\H^\times/\R^\times, d^\times \dot{\ba}_v) = 2 \pi^2$ and hence
\[
 C_v = \vol(\H^\times/\R^\times, d^\times \ba_v^\Tam/d^\times x_v) = 4 \pi^2.
\]
This completes the proof.
\end{proof}

\begin{example}[Eichler's mass formula \cite{eichler}, \cite{vigneras-book}, \cite{yu-chia-fu}]
Suppose that $B$ is totally definite.
Put $B_\infty = B \otimes_\Q \R$ and fix a maximal compact subgroup $\K$ of $B^\times(\A_\fin)$.
We can write
\[
 B^\times(\A) = \bigsqcup_{i=1}^n \A^\times B^\times(F) B_\infty^\times \ba_i \K,
\]
where $\{ \ba_i \in B^\times(\A_\fin) \, | \, 1 \le i \le n \}$ is a (finite) set of representatives for $\A^\times B^\times(F) B_\infty^\times \backslash B^\times(\A) / \K$.
Put
\[
 \mathfrak{mass} := \vol \left( \A^\times B^\times(F) \backslash B^\times(\A), \frac{d^\times \ba}{d^\times x} \right)
 = \sum_{i=1}^n \frac{1}{|\Gamma_i|},
\]
where $\Gamma_i = F^\times \backslash (B^\times(F) \cap \A_\fin^\times \ba_i \K \ba_i^{-1})$.
Then it follows from \eqref{eq:measure-compare-Gm} and Lemma \ref{lem:measure-compare-B} that
\begin{align*}
 \mathfrak{mass} & = \tau(B^\times / \mathbb{G}_m)
 \cdot (4 \pi^2)^{-d} \cdot \prod_{v \in \Sigma_{B,\fin}} (q_v - 1)
 \cdot |D|^{\frac{3}{2}} \cdot \zeta_F(2) \\
 & = (-1)^d \cdot 2^{-d+1} \cdot \prod_{v \in \Sigma_{B,\fin}} (q_v-1) \cdot  \zeta_F(-1),
\end{align*}
where $d = [F:\Q]$.
\end{example}

Finally, we compare the standard measure on $\A^\times \backslash B^\times(\A)$ with the measure on the Shimura variety.
Let $(G,X) = (\Res_{F/\Q}(B^\times), X_B)$ be the Shimura datum given in \S \ref{ssec:avs-on-qvs}.
If $v \in \Sigma_\infty \smallsetminus \Sigma_{B,\infty}$, we identify $B_v$ with $\M_2(\R)$.
As explained in \S \ref{ssec:avs-on-qvs}, this gives rise to an identification 
\begin{equation}
\label{eq:X-identification}
 X = \prod_{v \in \Sigma_\infty \smallsetminus \Sigma_{B,\infty}}\fh^\pm.
\end{equation}
Fix an open compact subgroup $K$ of $B^\times(\A_\fin)$ such that 
\[
 \hat{\o}^\times \subset K, 
\]
where $\hat{\o}^\times := \prod_{v \in \Sigma_\fin} \o_v^\times \subset \A_\fin^\times$.
Let $\Sh_K(G,X)$ be the associated Shimura variety:
\[
 \Sh_K(G,X) = B^\times(F) \backslash X \times B^\times(\A_\fin) / K.
\]
Since $\fh^\pm = \GL_2(\R) / \R_+^\times \cdot \SO(2)$, we have a natural surjective map
\[
 p: B^\times(\A) \longrightarrow \Sh_K(G,X).
\]
Recall that in Definition \ref{defn:pet-norm}, we have taken the measure on $\Sh_K(G,X)$ given as follows:
\begin{itemize}
 \item On $X$, we take the product over $v \in \Sigma_\infty \smallsetminus \Sigma_{B,\infty}$ of the $\GL_2(\R)$-invariant measure
 \[
  \frac{dx_v \, dy_v}{|y_v|^2}
 \]
 for $x_v + \sqrt{-1} y_v \in \fh^\pm$, where $dx_v$, $dy_v$ are the Lebesgue measures.
 This measure is independent of the choice of identification \eqref{eq:X-identification}.
 \item On $B^\times(\A_\fin)/ K$, we take the counting measure.
 \item If $B$ is not totally definite, then $\o^\times \backslash B^\times(F)$ acts on $X \times B^\times(\A_\fin)/ K$ properly discontinuously, and we take a natural measure $d \mu_x$ on $\Sh_K(G,X)$ induced by the product of the above measures.
 \item If $B$ is totally definite, then $\Sh_K(G,X)$ is a finite set, and for any $x \in \Sh_K(G,X)$, its stabilizer $\Gamma_x$ in $\o^\times \backslash B^\times(F)$ is a finite group. We take a measure $d \mu_x$ on $\Sh_K(G,X)$ given by
 \[
  \int_{\Sh_K(G,X)} \phi(x) \, d \mu_x = \sum_{x \in \Sh_K(G,X)} |\Gamma_x|^{-1} \phi(x).
 \]
\end{itemize}

\begin{lem}
\label{lem:meas-shim-st}
Let $\phi$ be an integrable function on $\Sh_K(G,X)$ such that $\phi(x \cdot z) = \phi(x)$ for all $x \in \Sh_K(G,X)$ and $z \in \A^\times$.
Then we have
\begin{equation}
\label{eq:meas-shim-st}
 \int_{\Sh_K(G,X)} \phi(x) \, d\mu_x = 
 2^{|\Sigma_\infty \smallsetminus \Sigma_{B,\infty}|} \cdot [K_0:K] \cdot h_F \cdot 
 \int_{\A^\times B^\times(F) \backslash B^\times(\A)} p^*\phi(\ba) \, d^\times \ba,
\end{equation}
where $K_0$ is any maximal compact subgroup of $B^\times(\A_\fin)$ containing $K$.
\end{lem}

\begin{proof}
Put $F_\infty = F \otimes_\Q \R$ and $B_\infty = B \otimes_\Q \R$.
We can write
\[
 B^\times(\A) = \bigsqcup_{i=1}^n B^\times(F) B_\infty^\times \ba_i K,
\]
where $\{ \ba_i \in B^\times(\A_\fin) \, | \, 1 \le i \le n \}$ is a (finite) set of representatives for $B^\times(F) B_\infty^\times \backslash B^\times(\A) / K$.
Then we have
\[
 \Sh_K(G,X) = \bigsqcup_{i=1}^n \Gamma_i \backslash X,
\]
where $\Gamma_i = \o^\times \backslash (B^\times(F) \cap \ba_i K \ba_i^{-1})$.
For each $i$, we have a natural commutative diagram
\[
 \xymatrix{
 B_\infty^\times \ba_i K \ar[r]^(0.6){p_i} \ar[d] & X \ar[d] \\
 B^\times(F) B_\infty^\times \ba_i K \ar[r]^(0.65){p} & \Gamma_i \backslash X
 }.
\]

First we assume that $B$ is not totally definite.
Since both sides of \eqref{eq:meas-shim-st} are proportional, we may assume that for each $i$, the restriction of $\phi$ to $\Gamma_i \backslash X$ is of the form
\[
 \phi(x) = \sum_{\gamma \in \Gamma_i} \varphi_i(\gamma x)
\]
for some continuous compactly supported function $\varphi_i$ on $X$.
Then, noting that $\Gamma_i$ acts on $X$ faithfully, we have
\[
 \int_{\Gamma_i \backslash X} \phi(x) \, d \mu_x = \int_X \varphi_i(x) \, d \mu_x,
\]
where the measure $d \mu_x$ on $X$ on the right-hand side is as defined above.
By the definition of the standard measure, we have
\[
 \int_X \varphi_i(x) \, d \mu_x
 = 2^{|\Sigma_\infty \smallsetminus \Sigma_{B,\infty}|} \cdot \vol(K)^{-1} \cdot
 \int_{B_\infty^\times \ba_i K / F_\infty^\times \hat{\o}^\times} p_i^* \varphi_i(\ba) \, d^\times \ba.
\]
(Here the factor $2$ arises from $[\R^\times : \R_+^\times]$.)
Since
\[
 p^*\phi(\ba) = \phi(p(\ba))
 = \sum_{\gamma \in \Gamma_i} \varphi_i(\gamma p_i(\ba))
 = \sum_{\gamma \in \Gamma_i} \varphi_i(p_i(\gamma \ba))
 = \sum_{\gamma \in \Gamma_i} p_i^* \varphi_i(\gamma \ba)
\]
for $\ba \in B_\infty^\times \ba_i K$, we have
\[
 \int_{B_\infty^\times \ba_i K / F_\infty^\times \hat{\o}^\times} p_i^* \varphi_i(\ba) \, d^\times \ba
 = \int_{\Gamma_i \backslash B_\infty^\times \ba_i K / F_\infty^\times \hat{\o}^\times} p^* \phi(\ba) \, d^\times \ba.
\]
Thus, noting that 
\[
 \Gamma_i \backslash B_\infty^\times \ba_i K / F_\infty^\times \hat{\o}^\times
  = B^\times(F) \backslash B^\times(F) B_\infty^\times \ba_i K / F_\infty^\times \hat{\o}^\times,
\]
we have
\[
 \int_{\Gamma_i \backslash X} \phi(x) \, d \mu_x
 = 2^{|\Sigma_\infty \smallsetminus \Sigma_{B,\infty}|} \cdot \vol(K)^{-1} \cdot
 \int_{B^\times(F) \backslash B^\times(F) B_\infty^\times \ba_i K / F_\infty^\times \hat{\o}^\times} p^*\phi(\ba) \, d^\times \ba.
\]
Summing over $i$, we obtain 
\begin{align*}
 \int_{\Sh_K(G,X)} \phi(x) \, d \mu_x
 & = 2^{|\Sigma_\infty \smallsetminus \Sigma_{B,\infty}|} \cdot \vol(K)^{-1} \cdot  
 \int_{B^\times(F) \backslash B^\times(\A) / F_\infty^\times \hat{\o}^\times} p^*\phi(\ba) \, d^\times \ba \\
 & = 2^{|\Sigma_\infty \smallsetminus \Sigma_{B,\infty}|} \cdot \vol(K)^{-1} \cdot  
 \vol(F^\times \backslash \A^\times / F_\infty^\times \hat{\o}^\times) \cdot 
 \int_{\A^\times B^\times(F) \backslash B^\times(\A)} p^*\phi(\ba) \, d^\times \ba.
\end{align*}
On the other hand, we have $\vol(K) = [K_0:K]^{-1}$ for any maximal compact subgroup $K_0$ of $B^\times(\A_\fin)$ containing $K$, and $\vol(F^\times \backslash \A^\times / F_\infty^\times \hat{\o}^\times) = h_F$ since the standard measure on $\A^\times/ F_\infty^\times \hat{\o}^\times$ is the counting measure.
This proves \eqref{eq:meas-shim-st}.

Next we assume that $B$ is totally definite.
Since 
\[
 \vol(B^\times(F) \backslash B^\times(F) B_\infty^\times \ba_i K / F_\infty^\times \hat{\o}^\times)
 = |\Gamma_i|^{-1} \cdot \vol(K),
\]
we have
\begin{align*}
 \int_{B^\times(F) \backslash B^\times(\A) / F_\infty^\times \hat{\o}^\times} p^*\phi(\ba) \, d^\times \ba
 & = \vol(K) \cdot \sum_{i=1}^n |\Gamma_i|^{-1} p^*\phi(\ba_i) \\
 & = \vol(K) \cdot \int_{\Sh_K(G,X)} \phi(x) \, d \mu_x.
\end{align*}
The rest of the proof is the same as before.
\end{proof}

\subsubsection{Measures on $B^1(\A)$}
\label{sssec:measures-B^1}

We recall the exact sequence
\[
 1 \longrightarrow B^1 \longrightarrow B^\times \overset{\nu}\longrightarrow \mathbb{G}_m \longrightarrow 1
\]
of algebraic groups over $F$.
For each place $v$ of $F$, this induces an exact sequence
\[
 1 \longrightarrow B^1_v \longrightarrow B^\times_v \overset{\nu}\longrightarrow F_v^\times.
\]
We define a Haar measure $dg_v^\Tam$ on $B^1_v$ by requiring that
\[
 \int_{B_v^\times} \phi(\ba_v) \, d^\times \ba_v^\Tam
 = \int_{\nu(B_v^\times)} \dot{\phi}(x_v) \, d^\times x_v^\Tam 
\]
for all $\phi \in L^1(B_v^\times)$, where
\[
 \dot{\phi}(\nu(\ba_v)) :=  \int_{B_v^1} \phi(g_v \ba_v) \, dg_v^\Tam.
\]
Note that $\nu(B^\times_v) = F_v^\times$ unless $v \in \Sigma_{B,\infty}$, in which case we have $\nu(B^\times_v) = \R^\times_+$.
Then the Tamagawa measure on $B^1(\A)$ is given by 
\[
 dg^\Tam := \prod_v dg_v^\Tam.
\]
We have $\tau(B^1) = 1$.

On the other hand, we define the standard measure on $B^1(\A)$ as a product measure $dg := \prod_v dg_v$, where $dg_v$ is given as follows:
\begin{itemize}
\item 
For $v \in \Sigma_\fin \smallsetminus \Sigma_{B,\fin}$, fix an isomorphism $i_v : B_v \rightarrow \M_2(F_v)$ of quaternion $F_v$-algebras, which is unique up to inner automorphisms by elements of $\GL_2(F_v)$, and let $dg_v$ be the Haar measure on $B_v^1$ such that $\vol(i_v^{-1}(\SL_2(\o_v)), dg_v) = 1$.
Noting that there are exactly $2$ conjugacy classes of maximal compact subgroups of $\SL_2(F_v)$, i.e., those of $\SL_2(\o_v)$ and $\big( \begin{smallmatrix} \varpi_v & \\ & 1 \end{smallmatrix} \big) \SL_2(\o_v) \big( \begin{smallmatrix} \varpi_v^{-1} & \\ & 1 \end{smallmatrix} \big)$, we have
\[
 \vol(i_v^{-1}(h_v \SL_2(\o_v) h_v^{-1}), dg_v) = \vol(i_v^{-1}(\SL_2(\o_v)), dg_v)
\]
for $h_v \in \GL_2(F_v)$.
Hence $dg_v$ is independent of the choice of $i_v$.
\item 
For $v \in \Sigma_{B,\fin}$, let $dg_v$ be the Haar measure on $B_v^1$ such that $\vol(B_v^1, dg_v) = 1$.
\item 
For $v \in \Sigma_\infty \smallsetminus \Sigma_{B,\infty}$, fix an isomorphism $i_v : B_v \rightarrow \M_2(F_v)$ of quaternion $F_v$-algebras, which is unique up to inner automorphisms by elements of $\GL_2(F_v)$, and define a Haar measure $dg_v$ on $B_v^1$ by
\[
 dg_v = \frac{dx_v \, dy_v}{y_v^2} \, d \kappa_v
\]
for $g_v = i_v^{-1} \left( \smat{1}{x_v}{}{1} \big( \begin{smallmatrix} \sqrt{y_v} & \\ & \sqrt{y_v}^{-1} \end{smallmatrix} \big) \kappa_v \right)$ with $x_v \in \R$, $y_v \in \R_+^\times$, $\kappa_v \in \SO(2)$, where $dx_v$, $dy_v$ are the Lebesgue measures and $d\kappa_v$ is the Haar measure on $\SO(2)$ such that $\vol(\SO(2), d\kappa_v) = 1$.
This measure $dg_v$ does not change if we replace $i_v$ by $\Ad(h_v) \circ i_v$ for $h_v \in \SL_2(F_v)$.
If we replace $i_v$ by $\Ad \smat{-1}{}{}{1} \circ i_v$, then $dg_v$ becomes $\frac{dx_v \, dy_v}{y_v^2} \, d \kappa_v$ for $g_v = i_v^{-1} \left( \smat{1}{-x_v}{}{1} \big( \begin{smallmatrix} \sqrt{y_v} & \\ & \sqrt{y_v}^{-1} \end{smallmatrix} \big) \kappa_v^{-1} \right)$, which is in fact equal to the original $dg_v$.
Hence $dg_v$ is independent of the choice of $i_v$.
\item 
For $v \in \Sigma_{B,\infty}$, let $dg_v$ be the Haar measure on $B_v^1$ such that $\vol(B_v^1, dg_v) = 1$.
\end{itemize}

\begin{lem}
\label{lem:measure-compare-B^1}
We have
\[
 dg^\Tam = \pi^{|\Sigma_\infty \smallsetminus \Sigma_{B,\infty}|}
 \cdot (4 \pi^2)^{|\Sigma_{B,\infty}|}
 \cdot \prod_{v \in \Sigma_{B, \fin}} (q_v-1)^{-1}
 \cdot |D|^{-\frac{3}{2}} \cdot \zeta_F(2)^{-1} \cdot dg.
\]
\end{lem}

\begin{proof}
For each place $v$ of $F$, let $C_v$ be the constant such that $dg_v^\Tam = C_v \cdot dg_v$.
If $v \in \Sigma_\fin \smallsetminus \Sigma_{B,\fin}$, we identify $B_v$ with $\M_2(F_v)$.
As in the proof of Lemma \ref{lem:measure-compare-B}, we have
\[
 C_v = \vol(\SL_2(\o_v), dg_v^\Tam) = \frac{\vol(\GL_2(\o_v), d^\times \ba_v^\Tam)}{\vol(\o_v^\times, d^\times x_v^\Tam)} = q_v^{- \frac{3 d_v}{2}} \cdot \zeta_v(2)^{-1}.
\]
If $v \in \Sigma_{B,\fin}$, then as in the proof of Lemma \ref{lem:measure-compare-B}, we have
\[
 C_v = \vol(B_v^1, dg_v^\Tam) = \frac{\vol(\o_{B_v}^\times, d^\times \ba_v^\Tam)}{\vol(\o_v^\times, d^\times x_v^\Tam)} = q_v^{- \frac{3 d_v}{2}} \cdot (q_v-1)^{-1} \cdot \zeta_v(2)^{-1}.
\]
If $v \in \Sigma_\infty \smallsetminus \Sigma_{B,\infty}$, we identify $B_v$ with $\M_2(\R)$.
For $\ba_v \in \GL_2(\R)^+$, we write $\ba_v = z_v \cdot g_v$ with $z_v \in \R_+^\times$ and $g_v \in \SL_2(\R)$.
Then we have
\[
 d^\times \ba_v^\Tam = 2 \cdot d^\times z_v \, dg_v^\Tam, \qquad 
 d^\times \ba_v = d^\times z_v \, dg_v
\]
on $\GL_2(\R)^+$.
Since $d^\times \ba_v^\Tam = 2 \pi \cdot d^\times \ba_v$ as in the proof of Lemma \ref{lem:measure-compare-B}, we have
\[
 C_v = \frac{1}{2} \cdot 2 \pi = \pi.
\]
If $v \in \Sigma_{B,\infty}$, then we have $d^\times \ba_v^\Tam = 2 \cdot d^\times z_v \, dg_v^\Tam$ for $\ba_v = z_v \cdot g_v$ with $z_v \in \R_+^\times$ and $g_v \in B_v^1$.
Hence we have
\begin{align*}
 C_v & = \vol(B_v^1, dg_v^\Tam) \\
 & = \frac{1}{2} \cdot \vol(B_v^\times / \R_+^\times, d^\times \ba_v^\Tam / d^\times z_v) \\
 & = \vol(B_v^\times / \R^\times, d^\times \ba_v^\Tam / d^\times z_v) \\
 & = 4 \pi^2
\end{align*}
as in the proof of Lemma \ref{lem:measure-compare-B}.
This completes the proof.
\end{proof}

\begin{example}[Siegel's formula \cite{siegel-tams}]
Suppose that $B = \M_2(F)$.
Put
\[
 \mathfrak{vol} := \vol \left( \SL_2(\o) \backslash \fh^d, \prod_{v \in \Sigma_\infty} \frac{dx_v \, dy_v}{y_v^2} \right),
\]
where $d = [F:\Q]$.
Since 
\[
 \SL_2(\o) \backslash \fh^d \cong \SL_2(F) \backslash \SL_2(\A) / K, 
\]
where $K = \prod_{v \in \Sigma_\infty} \SO(2) \times \prod_{v \in \Sigma_\fin} \SL_2(\o_v)$, we have
\[
 \vol(\SL_2(F) \backslash \SL_2(\A), dg) 
 = \mathfrak{vol} \cdot
 \vol \left( \{ \pm 1 \} \backslash K, \prod_{v \in \Sigma_\infty} d \kappa_v \cdot \prod_{v \in \Sigma_\fin} dg_v \right)
 = \mathfrak{vol} \cdot \frac{1}{2}.
\]
On the other hand, it follows from Lemma \ref{lem:measure-compare-B^1} that
\[
 \vol(\SL_2(F) \backslash \SL_2(\A), dg)
 = \tau(B^1) \cdot \pi^{-d} \cdot |D|^{\frac{3}{2}} \cdot \zeta_F(2)
 = (-2 \pi)^d \cdot \zeta_F(-1).
\]
Hence we have
\[
 \mathfrak{vol} = (-1)^d \cdot 2^{d+1} \cdot \pi^d \cdot \zeta_F(-1).
\]
\end{example}

\subsection{New vectors}
\label{ssec:new-vec}

In this subsection, we define a $1$-dimensional subspace of new vectors in the space of an irreducible representation of $B_v^\times$.
For the moment, we fix a place $v$ of $F$ and suppress the subscript $v$ from the notation.
We only consider representations $\pi$ of $B^\times$ listed below:

\begin{itemize}
\item If $F$ is non-archimedean and $B$ is split, then
\begin{itemize}
\item[(ur)] $\pi = \Ind(\chi \otimes \mu)$ is a principal series representation, where $\chi$ and $\mu$ are unitary unramified; or
\item[(rps)] $\pi = \Ind(\chi \otimes \mu)$ is a principal series representation, where $\chi$ is unitary unramified and $\mu$ is unitary ramified of conductor $q^n$; or
\item[(st)] $\pi = \St \otimes \chi$ is a twist of the Steinberg representation, where $\chi$ is unitary unramified.
\end{itemize}
\item If $F$ is non-archimedean and $B$ is ramified, then
\begin{itemize}
\item[(1d)] $\pi = \chi \circ \nu$ is a $1$-dimensional representation, where $\chi$ is unitary unramified.
\end{itemize}
\item If $F = \R$ and $B$ is split, then
\begin{itemize}
\item[(ds)] $\pi = \DS_k$ is the irreducible unitary (limit of) discrete series representation of weight $k$.
\end{itemize}
\item If $F = \R$ and $B$ is ramified, then
\begin{itemize}
\item[(fd)] $\pi = \Sym^k$ is the irreducible unitary $(k+1)$-dimensional representation.
\end{itemize}
\end{itemize}
If $F$ is non-archimedean, we define a compact subgroup $K_n$ of $\GL_2(F)$ by
\[
 K_n = \left\{ \left. \begin{pmatrix} a & b \\ c & d \end{pmatrix} \in \GL_2(\o) \, \right| \, c \in  \varpi^n \o \right\}.
\]
Note that $I := K_1$ is an Iwahori subgroup of $\GL_2(F)$.
If $F = \R$, we define a character $\chi_k$ of $\C^{\times}$ by 
\[
 \chi_k(\alpha) = \left( \frac{\alpha}{\sqrt{\alpha \alpha^\rho}} \right)^k.
\]

\subsubsection{The case (ur)}

Fix an isomorphism $i : B \rightarrow \M_2(F)$.
This determines a maximal compact subgroup $\K = i^{-1}(\GL_2(\o))$ of $B^\times$.
We say that $f \in \pi$ is a new vector with respect to $\K$ if
\[
 \pi(k) f = f
\]
for all $k \in \K$.

\subsubsection{The case (rps)}

Fix an isomorphism $i : B \rightarrow \M_2(F)$.
This determines a compact subgroup $\K_n = i^{-1}(K_n)$ of $B^\times$.
We define a character $\Mu$ of $\K_n$ by $\Mu(k) = \mu(d)$ for $k = i^{-1} \smat{a}{b}{c}{d}$.
We say that $f \in \pi$ is a new vector with respect to $(\K_n, \Mu)$ if
\[
 \pi(k) f = \Mu(k) f
\]
for all $k \in \K_n$.

\subsubsection{The case (st)}

Fix an isomorphism $i : B \rightarrow \M_2(F)$.
This determines an Iwahori subgroup $\II = i^{-1}(I)$ of $B^\times$.
We say that $f \in \pi$ is a new vector with respect to $\II$ if
\[
 \pi(k) f = f
\]
for all $k \in \II$.

\subsubsection{The case (1d)}

Let $\K = \o_B^\times$ be the unique maximal compact subgroup of $B^\times$.
Then we have
\[
 \pi(k) f = f
\]
for all $k \in \K$ and $f \in \pi$.
For uniformity, we call any $f \in \pi$ a new vector with respect to $\K$.

\subsubsection{The cases (ds), (fd)}

Fix an embedding $h : \C^\times \hookrightarrow B^\times$.
We say that $f \in \pi$ is a new vector with respect to $h$ if
\[
 \pi(h(z)) f = \chi_k(z) f
\]
for all $z \in \C^\times$.

\subsection{An explicit Rallis inner product formula}
\label{ssec:explicit-rallis}

Suppose that $F$ is global.
Let $\pi \cong \otimes_v \pi_v$ be an irreducible unitary cuspidal automorphic representation of $\GL_2(\A)$ such that for $v \in \Sigma_{\fin}$, 
\begin{itemize}
\item $\pi_v = \Ind(\chi_v \otimes \mu_v)$, where $\chi_v$ and $\mu_v$ are unitary unramified; or
\item $\pi_v = \Ind(\chi_v \otimes \mu_v)$, where $\chi_v$ is unitary unramified and $\mu_v$ is unitary ramified of conductor $q_v^{n_v}$; or
\item $\pi_v = \St \otimes \chi_v$, where $\chi_v$ is unitary unramified, 
\end{itemize}
and for $v \in \Sigma_\infty$,
\begin{itemize}
\item $\pi_v = \DS_{k_v}$, where $k_v \ge 1$.
\end{itemize}
We assume that $\pi_v$ is unramified for all finite places $v$ of $F$ such that $F_v$ is ramified or of residual characteristic $2$.
Put $\Sigma_\pi = \{ v \, | \, \text{$\pi_v$ is a discrete series} \}$, $\Sigma_{\pi,\fin} = \Sigma_\pi \cap \Sigma_\fin$, and
\[ 
 \Sigma'_{\pi,\fin}
 := \{ v \in \Sigma_\fin \, | \, \text{$\pi_v$ is a ramified principal series} \}.
\]
We consider a non-zero vector $f = \otimes_v f_v \in \pi$ such that:
\begin{itemize}
 \item for $v \in \Sigma_\fin \smallsetminus (\Sigma_{\pi,\fin} \cup \Sigma'_{\pi,\fin})$, $f_v$ is a new vector with respect to $\GL_2(\o_v)$;
 \item for $v \in \Sigma_{\pi,\fin}$, $f_v$ is a new vector with respect to the Iwahori subgroup $I$ of $\GL_2(F_v)$ given in \S \ref{ssec:new-vec};
 \item for $v \in \Sigma'_{\pi,\fin}$, $f_v$ is a new vector with respect to $(K_{n_v},\Mu_v)$, where $K_{n_v}$ is the compact subgroup of $\GL_2(F_v)$ given in \S \ref{ssec:new-vec} and $\Mu_v$ is the character of $K_{n_v}$ defined by $\Mu_v\!\smat{a}{b}{c}{d} = \mu_v(d)$;
 \item for $v \in \Sigma_\infty$, $f_v$ is a new vector with respect to the embedding $h_v : \C^\times \hookrightarrow \GL_2(\R)$ defined by $h_v(a + b \sqrt{-1}) = \smat{a}{b}{-b}{a}$.
\end{itemize}
We normalize such a vector $f$, which is unique up to scalars, so that
\[
 W_f \begin{pmatrix} \delta^{-1} & \\ & 1 \end{pmatrix} = e^{-2 \pi d},
\]
where $W_f$ is the Whittaker function of $f$ defined by
\[
 W_f(g) = \int_{F \backslash \A} f\left( 
 \begin{pmatrix} 1 & x \\ & 1 \end{pmatrix} g \right)
 \overline{\psi(x)} \, dx
\]
with the Tamagawa measure $dx$ on $\A$, $\delta = (\varpi_v^{d_v}) \in \A_\fin^\times$, and $d = [F:\Q]$ (see also Lemmas \ref{lem:whittaker-formula-ur}, \ref{lem:whittaker-formula-st}, \ref{lem:whittaker-formula-tr} and \eqref{eq:whittaker-formula-ds} below).
Let $\langle f, f \rangle$ be the Petersson norm of $f$ defined by
\[
 \langle f, f \rangle = \int_{\A^\times \GL_2(F) \backslash \GL_2(\A)} |f(g)|^2 \, dg,
\]
where $dg$ is the standard measure on $\A^\times \backslash \GL_2(\A)$.
In \S \ref{ssec:computation_of_<f,f>} below, we will prove:
\begin{prop}
\label{prop:<f,f>}
We have
\[
 \langle f, f \rangle = 2 \cdot \prod_{v \in \Sigma_\infty} \frac{(k_v - 1)!}{2^{2 k_v + 1} \pi^{k_v + 1}}
 \cdot \prod_{v \in \Sigma_{\pi,\fin} \cup \Sigma'_{\pi,\fin}} \frac{q_v}{q_v+1} \cdot |D| \cdot L(1, \pi, \ad),
\]
where $L(s, \pi, \ad) = \prod_{v \in \Sigma_\fin}L(s, \pi_v, \ad)$ is the adjoint $L$-function of $\pi$.
\end{prop}

Let $B$, $B_1$, $B_2$ be quaternion algebras over $F$ such that $B = B_1 \cdot B_2$ in the Brauer group.
We assume that $\Sigma_B \ne \varnothing$ and $\Sigma_B \cup \Sigma_{B_1} \cup \Sigma_{B_2} \subset \Sigma_\pi$, i.e., $B$ is division and the Jacquet--Langlands transfers $\pi_B$, $\pi_{B_1}$, $\pi_{B_2}$ of $\pi$ to $B^{\times}(\A)$, $B_1^{\times}(\A)$, $B_2^{\times}(\A)$ exist.
Now, we choose a totally imaginary quadratic extension $E$ of $F$ such that $E$ embeds into $B$, $B_1$, $B_2$, and write $E = F + F \i$, $B = E + E \j$, $B_1 = E + E \j_1$, $B_2 = E + E \j_2$.
We also impose the ramification conditions on $u,J,J_1,J_2$ in \S \ref{ssec:autom-rep}.
We consider non-zero vectors $f_B = \otimes_v f_{B,v} \in \pi_B$, $f_{B_1} = \otimes_v f_{B_1,v} \in \pi_{B_1}$, $f_{B_2} = \otimes_v f_{B_2,v} \in \pi_{B_2}$ such that:
\begin{itemize}
 \item for $v \in \Sigma_\fin \smallsetminus (\Sigma_{\pi,\fin} \cup \Sigma'_{\pi,\fin})$, $f_{B,v}$, $f_{B_1,v}$, $f_{B_2,v}$ are new vectors with respect to $\K$, $\K_1$, $\K_2$, respectively, given in \S \ref{sssec:schwartzX'-ur};
 \item for $v \in \Sigma_{\pi,\fin} \smallsetminus (\Sigma_{B, \fin} \cup \Sigma_{B_1, \fin} \cup \Sigma_{B_2, \fin})$, $f_{B,v}$, $f_{B_1,v}$, $f_{B_2,v}$ are new vectors with respect to $\II$, $\II_1$, $\II_2$, respectively, given in \S \ref{par:schwartzX'-st-B1spl};
 \item for $v \in \Sigma_{B_1, \fin} \cap \Sigma_{B_2, \fin}$, $f_{B,v}$, $f_{B_1,v}$, $f_{B_2,v}$ are new vectors with respect to $\II$, $\K_1$, $\K_2$, respectively, given in \S \ref{par:schwartzX'-st-B1ram};
 \item for $v \in \Sigma_{B, \fin} \cap \Sigma_{B_2, \fin}$, $f_{B,v}$, $f_{B_1,v}$, $f_{B_2,v}$ are new vectors with respect to $\K$, $\II_1$, $\K_2$, respectively, given in \S \ref{sssec:schwartzX'-1d}; we switch the roles of $B_1$ and $B_2$ for $v \in \Sigma_{B, \fin} \cap \Sigma_{B_1, \fin}$;
 \item for $v \in \Sigma'_{\pi,\fin}$, $f_{B,v}$, $f_{B_1,v}$, $f_{B_2,v}$ are new vectors with respect to $(\K_{n_v}, \Mu_v)$, $(\K_{1,n_v}, \Mu_v)$, $(\K_{2,n_v}, \Mu_v^{-1} \cdot \mu_v \circ \nu)$, respectively, given in \S \ref{sssec:schwartzX'-tr};
 \item for $v \in \Sigma_\infty$, $f_{B,v}$, $f_{B_1,v}$, $f_{B_2,v}$ are new vectors with respect to the embeddings
 $\C^\times \cong E_v^\times \hookrightarrow B_v^\times$, $\C^\times \cong E_v^\times \hookrightarrow B_{1,v}^\times$, $\C^\times \cong E_v^\times \hookrightarrow B_{2,v}^\times$, respectively, given in \S\S \ref{sssec:schwartzX'-ds}, \ref{sssec:schwartzX'-fd}.
\end{itemize}
We fix such vectors $f_B$, $f_{B_1}$, $f_{B_2}$, which are unique up to scalars.
We emphasize that the $1$-dimensional subspaces of $\pi_B$, $\pi_{B_1}$, $\pi_{B_2}$ spanned by $f_B$, $f_{B_1}$, $f_{B_2}$, respectively, depend on the choice of $E$, $\i$, $\j$, $\j_1$, $\j_2$.
Let $\langle f_B, f_B \rangle$ be the Petersson norm of $f_B$ defined by
\[
 \langle f_B, f_B \rangle = \int_{\A^\times B^\times(F) \backslash B^\times(\A)} |f_B(g)|^2 \, dg,
\]
where $dg$ is the standard measure on $\A^\times \backslash B^\times(\A)$.
We define $\langle f_{B_1}, f_{B_1} \rangle$ and $\langle f_{B_2}, f_{B_2} \rangle$ similarly.

Let $\varphi = \otimes_v \varphi_v \in \SS(\X(\A))$ be the Schwartz function given in \S \ref{ssec:schwartzX}, where
\begin{itemize}
 \item $\varphi_v = \varphi_{\mu_v}$ for $v \in \Sigma'_{\pi,\fin}$;
 \item $\varphi_v = \varphi_{k_v}$ for $v \in \Sigma_\infty \smallsetminus (\Sigma_{B,\infty} \cup \Sigma_{B_1,\infty} \cup \Sigma_{B_2,\infty})$;
 \item $\varphi_v = \varphi_{k_v-2}$ for $v \in \Sigma_{B_1, \infty} \cap \Sigma_{B_2, \infty}$;
 \item $\varphi_v = \varphi_{k_v-2}$ for $v \in \Sigma_{B,\infty} \cap \Sigma_{B_2,\infty}$; we switch the roles of $B_1$ and $B_2$ for $v \in \Sigma_{B,\infty} \cap \Sigma_{B_1,\infty}$.
\end{itemize}
In \S \ref{subsec:jls}, we have defined the theta lift $\theta_\varphi(f_B)$, but for our purposes, we slightly modify its definition: on the right-hand side of \eqref{eq:jls_theta_lift_def}, we take the standard measure on $B^1(\A)$ rather than the Tamagawa measure on $B^1(\A)$.
We regard $\theta_\varphi(f_B)$ as an automorphic form on $B_1^\times(\A) \times B_2^\times(\A)$.
Then it follows from the equivariance properties of $\varphi$ that there exists a constant $\alpha(B_1,B_2) \in \C$ (once we fix $f_B$, $f_{B_1}$, $f_{B_2}$) such that
\[
 \theta_\varphi(f_B) = \alpha(B_1,B_2) \cdot (f_{B_1} \times f_{B_2}).
\]

Now we state an explicit Rallis inner product formula.

\begin{thm}
\label{thm:rallis-B-explicit}
We have
\[
 |\alpha(B_1,B_2)|^2
 \cdot \langle f_{B_1}, f_{B_1} \rangle \cdot \langle f_{B_2}, f_{B_2} \rangle
 = C \cdot \langle f, f \rangle \cdot \langle f_B, f_B \rangle,
\]
where $C = |D|^2 \cdot \prod_v C_v$ with
\[
 C_v = 
 \begin{cases}
  1 & \text{if $v \in \Sigma_\fin \smallsetminus (\Sigma_{\pi,\fin} \cup \Sigma'_{\pi,\fin})$,} \\
  \dfrac{q_v}{(q_v+1)^2} & \text{if $v \in \Sigma_{\pi,\fin} \smallsetminus (\Sigma_{B, \fin} \cup \Sigma_{B_1, \fin} \cup \Sigma_{B_2, \fin})$,} \\
  q_v & \text{if $v \in \Sigma_{B_1, \fin} \cap \Sigma_{B_2, \fin}$,} \\
  q_v & \text{if $v \in \Sigma_{B, \fin}$,} \\
  \dfrac{1}{q_v^{n_v-3}(q_v-1)(q_v+1)^2} & \text{if $v \in \Sigma'_{\pi,\fin}$,} \\
  \dfrac{2^{2 k_v + 2} \pi^{k_v}}{k_v!} & \text{if $v \in \Sigma_\infty \smallsetminus (\Sigma_{B,\infty} \cup \Sigma_{B_1,\infty} \cup \Sigma_{B_2,\infty})$,} \\
  \dfrac{2^{2 k_v} \pi^{k_v-2}}{(k_v-1)^2 \cdot (k_v - 2)!} & \text{if $v \in \Sigma_{B_1, \infty} \cap \Sigma_{B_2, \infty}$,} \\
  \dfrac{2^{2 k_v -2} \pi^{k_v-2}}{(k_v-1)^2 \cdot (k_v - 2)!} & \text{if $v \in \Sigma_{B,\infty}$.}
 \end{cases}
\]
\end{thm}

\begin{proof}
By Proposition \ref{prop:rallis-B}, we have
\begin{multline*}
 C_{B_1} \cdot C_{B_2} \cdot (C_B^1)^2 \cdot |\alpha(B_1,B_2)|^2
 \cdot \langle f_{B_1}, f_{B_1} \rangle \cdot \langle f_{B_2}, f_{B_2} \rangle \\
 = 2 \cdot C_B \cdot C_B^1 \cdot \frac{L^S(1, \pi, \ad)}{\zeta_F^S(2)^2}
 \cdot \langle f_B, f_B \rangle \cdot  \prod_{v \in S} Z_v
\end{multline*}
for a sufficiently large finite set $S$ of places of $F$, where
\begin{itemize}
 \item $C_B$ is the constant such that $dg^\Tam = C_B \cdot dg$, where $dg^\Tam$ is the Tamagawa measure on $\A^\times \backslash B^\times(\A)$ and $dg$ is the standard measure on $\A^\times \backslash B^\times(\A)$; we define $C_{B_1}$ and $C_{B_2}$ similarly; 
 \item $C_B^1$ is the constant such that $dg_1^\Tam = C_B^1 \cdot dg_1$, where $dg_1^\Tam$ is the Tamagawa measure on $B^1(\A)$ and $dg$ is the standard measure on $B^1(\A)$;
 \item $Z_v$ is the integral defined by
\[
 Z_v = \int_{B^1_v} \langle \omega_{\psi}(g_{1,v}) \varphi_v, \varphi_v \rangle \langle \pi_{B,v}(g_{1,v}) f_{B,v}, f_{B,v} \rangle \, d g_{1,v}
\]
(cf.~\eqref{eq:local_zeta_integral}), where
\begin{itemize}
 \item the hermitian inner product $\langle \cdot, \cdot \rangle$ on $\SS(\X_v)$ is normalized as in \S \ref{ssec:weilrep};
 \item the invariant hermitian inner product $\langle \cdot, \cdot \rangle$ on $\pi_{B,v}$ is normalized so that $\langle f_{B,v}, f_{B,v} \rangle = 1$;
 \item $dg_{1,v}$ is the standard measure on $B_v^1$.
\end{itemize}
\end{itemize}
Hence, by \eqref{eq:measure-compare-Gm} and Lemmas \ref{lem:measure-compare-B}, \ref{lem:measure-compare-B^1}, we have
\[
 |\alpha(B_1,B_2)|^2 \cdot \langle f_{B_1}, f_{B_1} \rangle \cdot \langle f_{B_2}, f_{B_2} \rangle
 = C' \cdot L(1, \pi, \ad) \cdot \langle f_B, f_B \rangle
 \cdot \prod_{v \in S_\fin} \frac{\zeta_v(2)^2}{L(1, \pi_v, \ad)} \cdot \prod_{v \in S} Z_v, 
\]
where $S_\fin = S \cap \Sigma_\fin$ and
\begin{align*}
 C' & = \frac{2 \cdot C_B}{C_{B_1} \cdot C_{B_2} \cdot C_B^1 \cdot \zeta_F(2)^2} \\
 & = 2 \cdot 2^{|\Sigma_\infty \smallsetminus \Sigma_{B,\infty}|} \cdot
 (2 \pi)^{-|\Sigma_\infty \smallsetminus \Sigma_{B_1,\infty}| - |\Sigma_\infty \smallsetminus \Sigma_{B_2,\infty}|}
 \cdot (4 \pi^2)^{-|\Sigma_{B_1,\infty}|-|\Sigma_{B_2,\infty}|} \\
 & \quad \times \prod_{v \in \Sigma_{B_1, \fin}} (q_v-1) \cdot \prod_{v \in \Sigma_{B_2, \fin}} (q_v-1) \cdot |D|^3 \\
 & = 2 \cdot |D|^3 \cdot \prod_v C_v' 
\end{align*}
with 
\[
 C_v' = 
 \begin{cases}
  1 & \text{if $v \in \Sigma_\fin \smallsetminus (\Sigma_{B, \fin} \cup \Sigma_{B_1, \fin} \cup \Sigma_{B_2, \fin})$,} \\
  (q_v - 1)^2 & \text{if $v \in \Sigma_{B_1, \fin} \cap \Sigma_{B_2, \fin}$,} \\
  q_v - 1 & \text{if $v \in \Sigma_{B, \fin}$,} \\
  (2 \pi^2)^{-1} & \text{if $v \in \Sigma_\infty \smallsetminus (\Sigma_{B,\infty} \cup \Sigma_{B_1,\infty} \cup \Sigma_{B_2,\infty})$,} \\
  (8 \pi^4)^{-1} & \text{if $v \in \Sigma_{B_1, \infty} \cap \Sigma_{B_2, \infty}$,} \\
  (8 \pi^3)^{-1} & \text{if $v \in \Sigma_{B,\infty}$.}
 \end{cases}
\]
Moreover, by Proposition \ref{prop:<f,f>}, we have
\[
 |\alpha(B_1,B_2)|^2 \cdot \langle f_{B_1}, f_{B_1} \rangle \cdot \langle f_{B_2}, f_{B_2} \rangle 
 = C'' \cdot \langle f, f \rangle \cdot \langle f_B, f_B \rangle
 \cdot \prod_{v \in S_\fin} \frac{\zeta_v(2)^2}{L(1, \pi_v, \ad)} \cdot  \prod_{v \in S} Z_v,
\]
where $C'' = |D|^2 \cdot \prod_v C_v''$ with 
\[
 C_v'' = 
 \begin{cases}
  1 & \text{if $v \in \Sigma_\fin \smallsetminus (\Sigma_{\pi,\fin} \cup \Sigma'_{\pi,\fin})$,} \\
  \dfrac{q_v + 1}{q_v} & \text{if $v \in \Sigma_{\pi,\fin} \smallsetminus (\Sigma_{B, \fin} \cup \Sigma_{B_1, \fin} \cup \Sigma_{B_2, \fin})$,} \\
  \dfrac{(q_v - 1)^2(q_v + 1)}{q_v} & \text{if $v \in \Sigma_{B_1, \fin} \cap \Sigma_{B_2, \fin}$,} \\
  \dfrac{(q_v - 1)(q_v + 1)}{q_v} & \text{if $v \in \Sigma_{B, \fin}$,} \\
  \dfrac{q_v + 1}{q_v} & \text{if $v \in \Sigma'_{\pi,\fin}$,} \\
  \dfrac{2^{2 k_v} \pi^{k_v-1}}{(k_v - 1)!} & \text{if $v \in \Sigma_\infty \smallsetminus (\Sigma_{B,\infty} \cup \Sigma_{B_1,\infty} \cup \Sigma_{B_2,\infty})$,} \\
  \dfrac{2^{2 k_v -2} \pi^{k_v-3}}{(k_v - 1)!} & \text{if $v \in \Sigma_{B_1, \infty} \cap \Sigma_{B_2, \infty}$,} \\
  \dfrac{2^{2 k_v -2} \pi^{k_v-2}}{(k_v - 1)!} & \text{if $v \in \Sigma_{B,\infty}$.}
 \end{cases}
\]
Now Theorem \ref{thm:rallis-B-explicit} follows from this and Lemmas \ref{lem:zeta-integral-ur}, \ref{lem:zeta-integral-tr}, \ref{lem:zeta-integral-st}, \ref{lem:zeta-integral-1d}, \ref{lem:zeta-integral-ds}, \ref{lem:zeta-integral-fd} in \S \ref{ssec:computation_of_Zv} below, where we compute the integral $Z_v$ explicitly.
\end{proof}

\subsection{Computation of $\langle f, f \rangle$}
\label{ssec:computation_of_<f,f>}

Proposition \ref{prop:<f,f>} follows from a standard computation of the Rankin--Selberg integral, but we give the details of the proof for the convenience of the reader.
We retain the notation of \S \ref{ssec:explicit-rallis}.
Put 
\[
 \n(x) = \begin{pmatrix} 1 & x \\ & 1 \end{pmatrix}, \qquad 
 \t(y) = \begin{pmatrix} y & \\ & 1 \end{pmatrix}, \qquad
 w = \begin{pmatrix} & -1 \\ 1 & \end{pmatrix}.
\]

By \cite[\S 4]{jacquet-shalika}, \cite[\S 2.2]{lapid-offen}, \cite[Proposition 3.1]{zhang-wei-jams}, we have
\[
 C \cdot \langle f, f \rangle = \frac{2}{\rho_F} \cdot \frac{\Res_{s=1} L^S(s, \pi \times \pi^\vee)}{\zeta^S_F(2)} \cdot |D|^{-\frac{1}{2}} \cdot \prod_{v \in S} \| W_v \|^2
\]
for a sufficiently large finite set $S$ of places of $F$, where
\begin{itemize}
 \item $C$ is the constant such that $dg^\Tam = C \cdot dg$, where $dg^\Tam$ is the Tamagawa measure on $\A^\times \backslash \GL_2(\A)$ and $dg$ is the standard measure on $\A^\times \backslash \GL_2(\A)$;
 \item the Whittaker function $W_f$ of $f$ is decomposed as a product $W_f = \prod_v W_v$, where $W_v$ is the Whittaker function of $\pi_v$ with respect to $\psi_v$ normalized so that
\begin{itemize}
 \item $W_v(\t(\varpi_v^{-d_v})) = 1$ for $v \in \Sigma_\fin \smallsetminus (\Sigma_{\pi,\fin} \cup \Sigma'_{\pi,\fin})$;
 \item $W_v(1) = 1$ for $v \in \Sigma_{\pi,\fin}$;
 \item $W_v(1) = 1$ for $v \in \Sigma'_{\pi,\fin}$;
 \item $W_v(1) = e^{-2\pi}$ for $v \in \Sigma_\infty$;
\end{itemize}
 \item $\| W_v \|^2$ is the integral defined by
 \[
  \| W_v \|^2 = \int_{F_v^\times} |W_v(\t(y_v))|^2 \, d^\times y_v,
 \]
 where $d^\times y_v$ is the standard measure on $F_v^\times$.
\end{itemize}
We remark that:
\begin{itemize}
 \item the volume of $F^\times \backslash \A^1$ given in \cite[Proposition 3.1]{zhang-wei-jams} is equal to $\rho_F$,
 \item $\prod_v d^\times y_v^\Tam = |D|^{-\frac{1}{2}} \cdot \prod_v d^\times y_v$.
\end{itemize}
Hence, by \eqref{eq:measure-compare-Gm} and Lemma \ref{lem:measure-compare-B}, we have
\[
 \langle f, f \rangle = 2 \cdot (2 \pi)^{-[F:\Q]} \cdot |D| \cdot L(1, \pi, \ad) \cdot 
 \prod_{v \in S_\fin} \frac{\zeta_v(2)}{\zeta_v(1) \cdot L(1, \pi_v, \ad)} \cdot \prod_{v \in S} \| W_v \|^2,
\]
where $S_\fin = S \cap \Sigma_\fin$.
Now Proposition \ref{prop:<f,f>} follows from this and Lemmas \ref{lem:|W|^2-ur}, \ref{lem:|W|^2-st}, \ref{lem:|W|^2-tr}, \ref{lem:|W|^2-ds} below, where we compute the integral $\| W_v \|^2$ explicitly.

For the rest of this subsection, we fix a place $v$ of $F$ and suppress the subscript $v$ from the notation.

\subsubsection{The case $v \in \Sigma_\fin \smallsetminus (\Sigma_{\pi,\fin} \cup \Sigma'_{\pi,\fin})$}

Recall that $\pi = \Ind(\chi \otimes \mu)$, where $\chi$ and $\mu$ are unitary unramified.
Put $\alpha = \chi(\varpi)$ and $\beta = \mu(\varpi)$.
We have
\[
 L(s, \pi, \ad) = \frac{1}{(1-q^{-s})(1 - \alpha \beta^{-1} q^{-s})(1 - \alpha^{-1} \beta q^{-s})}.
\]

\begin{lem}
\label{lem:whittaker-formula-ur}
We have
\[
 W(\t(\varpi^{i-d})) =
 \begin{cases}
  q^{-\frac{i}{2}} \cdot \dfrac{\alpha^{i+1} - \beta^{i+1}}{\alpha - \beta}
  & \text{if $i \ge 0$,} \\
  0 & \text{if $i < 0$.}
 \end{cases}
\]
\end{lem}

\begin{proof}
Recall that $d$ is the non-negative integer such that $\psi$ is trivial on $\varpi^{-d} \o$ but non-trivial on $\varpi^{-d-1} \o$.
We define a non-trivial character $\psi^0$ of $F$ of order zero by $\psi^0(x) = \psi(\varpi^{-d}x)$.
Let $W^0$ be the Whittaker function of $\pi$ with respect to $\psi^0$ such that
\begin{itemize}
 \item $W^0(gk) = W^0(g)$ for all $g \in \GL_2(F)$ and $k \in \GL_2(\o)$, 
 \item $W^0(1) = 1$.
\end{itemize}
Then we have $W(g) = W^0(\t(\varpi^d)) g)$, so that the assertion follows from the Casselman--Shalika formula \cite{casselman-shalika}.
\end{proof}

\begin{lem}
\label{lem:|W|^2-ur}
We have
\[
 \| W \|^2 = \frac{\zeta(1) \cdot L(1, \pi, \ad)}{\zeta(2)}.
\]
\end{lem}

\begin{proof}
By Lemma \ref{lem:whittaker-formula-ur}, we have
\begin{align*}
 \| W \|^2 & = \sum_{i=0}^{\infty} |W(\t(\varpi^{i-d}))|^2 \\ 
 & = \frac{1}{(\alpha - \beta)(\alpha^{-1} - \beta^{-1})} \cdot 
 \sum_{i=0}^\infty q^{-i}(\alpha^{i+1} - \beta^{i+1}) (\alpha^{-i-1} - \beta^{-i-1}) \\
 & = \frac{1}{(\alpha - \beta)(\alpha^{-1} - \beta^{-1})} \cdot 
 \left( \frac{1}{1 - q^{-1}}
 - \frac{\alpha \beta^{-1}}{1 - \alpha \beta^{-1} q^{-1}}
 - \frac{\alpha^{-1} \beta}{1 - \alpha^{-1} \beta q^{-1}}
 + \frac{1}{1 -q^{-1}} \right) \\
 & = \frac{1+q^{-1}}{(1 - q^{-1})(1 - \alpha \beta^{-1} q^{-1})(1 - \alpha^{-1} \beta q^{-1})}.
\end{align*}
\end{proof}

\subsubsection{The case $v \in \Sigma_{\pi,\fin}$}

Recall that $\pi = \St \otimes \chi$, where $\chi$ is unitary unramified.
Put $\alpha = \chi(\varpi)$.
We have $L(s, \pi, \ad) = \zeta(s+1)$.

\begin{lem}
\label{lem:whittaker-formula-st}
We have
\[
 W(\t(\varpi^i)) = 
 \begin{cases}
  q^{-i} \cdot \alpha^i & \text{if $i \ge 0$,} \\
  0 & \text{if $i < 0$.}
 \end{cases}
\]
\end{lem}

\begin{proof}
We may assume that $\chi = 1$.
We recall the exact sequence
\[
 0 \longrightarrow \St \longrightarrow \Ind(|\cdot|^{\frac{1}{2}} \otimes |\cdot|^{-\frac{1}{2}})
 \overset{M}\longrightarrow \1 \longrightarrow 0,
\]
where $M : \Ind(|\cdot|^{\frac{1}{2}} \otimes |\cdot|^{-\frac{1}{2}}) \rightarrow \Ind(|\cdot|^{-\frac{1}{2}} \otimes |\cdot|^{\frac{1}{2}})$ is the intertwining operator defined by 
\[
 M(f)(g) = \int_F f(w \n(x) g) \, dx
\]
with the Haar measure $dx$ on $F$ such that $\vol(\o, dx) = 1$.
In particular, we have
\[
 \St = \{ f \in \Ind(|\cdot|^{\frac{1}{2}} \otimes |\cdot|^{-\frac{1}{2}}) \, | \, M(f)(1) = 0 \}.
\]
Also, we have
\[
 \dim_\C \St^I = 1, \qquad 
 \dim_\C \Ind(|\cdot|^{\frac{1}{2}} \otimes |\cdot|^{-\frac{1}{2}})^I = 2.
\]
Let $f_1, f_w$ be the basis of $\Ind(|\cdot|^{\frac{1}{2}} \otimes |\cdot|^{-\frac{1}{2}})^I$ determined by
\[
 f_1|_{\GL_2(\o)} = \I_I, \qquad f_w|_{\GL_2(\o)} = \I_{IwI}.
\]
Then $f_1 - q^{-1} f_w$ is a basis of $\St^I$.
Indeed, noting that 
\[
 w \n(x) = \begin{pmatrix} & -1 \\ 1 & x \end{pmatrix}
 = \begin{pmatrix} x^{-1} & -1 \\ & x \end{pmatrix} \begin{pmatrix} 1 & \\ x^{-1} & 1 \end{pmatrix},
\]
we have
\[
 M(f_1)(1) = \sum_{j=1}^{\infty} \int_{\varpi^{-j} \o^{\times}} |x|^{-2} \, dx = \sum_{j=1}^{\infty} q^{-j} (1-q^{-1}) = q^{-1}
\]
and 
\[
 M(f_w)(1) = \int_{\o} dx = 1.
\]

We consider the Jacquet integral
\[
 \Wc_k(g) := \int_F f_k(w \n(x) g) \overline{\psi(x)} \, dx
\]
for $k = 1, w$, where we recall that $\psi$ is assumed to be of order zero.
We have 
\[
 \Wc_k(\t(y)) = |y|^{-1} \int_F f_k(w \n(xy^{-1})) \overline{\psi(x)} \, dx
 = \int_F f_k(w \n(x)) \overline{\psi(xy)} \, dx.
\]
If $k=1$, then we have
\[
 \Wc_1(\t(y)) = \sum_{j=1}^{\infty} \int_{\varpi^{-j} \o^{\times}} |x|^{-2} \overline{\psi(xy)} \, dx
 = \sum_{j=1}^\infty q^{-j} \cdot \hat{\I}_{\o^\times}(\varpi^{-j} y).
\]
Since 
\[
 \hat{\I}_{\o^\times}(x) = 
 \begin{cases}
  1-q^{-1} & \text{if $x \in \o$,} \\
  -q^{-1} & \text{if $x \in \varpi^{-1} \o^\times$,} \\
  0 & \text{otherwise,}
 \end{cases}
\]
we have
\[
 \Wc_1(\t(\varpi^i)) =
 \begin{cases}
 \sum_{j=1}^i q^{-j} (1-q^{-1}) + q^{-(i+1)} \cdot (-q^{-1}) = q^{-1} - q^{-i-1} - q^{-i-2} & \text{if $i>0$,} \\
 q^{-1} \cdot (-q^{-1}) = -q^{-2} & \text{if $i=0$,} \\
 0 & \text{if $i<0$.}
 \end{cases}
\]
If $k=w$, then we have
\[
 \Wc_w(\t(y)) = \int_\o \overline{\psi(xy)} \, dx = \I_\o(y).
\]
Hence, if we put $\Wc = \Wc_1 - q^{-1}\Wc_w$, then we have
\[
 \Wc(\t(\varpi^i)) = 
 \begin{cases}
 - q^{-i-1}(1+q^{-1}) & \text{if $i \ge 0$,} \\
 0 & \text{if $i<0$.}
 \end{cases}
\]
Thus $W = \Wc(1)^{-1} \cdot \Wc$ and the assertion follows.
\end{proof}

\begin{lem}
\label{lem:|W|^2-st}
We have
\[
 \| W \|^2 = L(1, \pi, \ad).
\]
\end{lem}

\begin{proof}
By Lemma \ref{lem:whittaker-formula-st}, we have
\[
 \| W \|^2 = \sum_{i=0}^\infty |W(\t(\varpi^i))|^2 = \sum_{i=0}^\infty q^{-2i} = \frac{1}{1-q^{-2}}.
\]
\end{proof}

\subsubsection{The case $v \in \Sigma'_{\pi,\fin}$}

Recall that $\pi = \Ind(\chi \otimes \mu)$, where $\chi$ is unitary unramified and $\mu$ is unitary ramified of conductor $q^n$.
Put $\alpha = \chi(\varpi)$.
We have $L(s, \pi, \ad) = \zeta(s)$.

\begin{lem}
\label{lem:whittaker-formula-tr}
We have
\[
 W(\t(\varpi^i)) = 
 \begin{cases}
  q^{-\frac{i}{2}} \cdot \alpha^i & \text{if $i \ge 0$,} \\
  0 & \text{if $i < 0$.}
 \end{cases}
\]
\end{lem}

\begin{proof}
Let $f \in \Ind(\chi \otimes \mu)$ be the new vector with respect to $(K_n, \Mu)$ determined by
\[
 f|_{\GL_2(\o)} = \I_{K_n} \Mu.
\]
We consider the Jacquet integral
\[
 \Wc(g) := \int_F f(w \n(x) g) \overline{\psi(x)} \, dx,
\]
where we recall that $\psi$ is assumed to be of order zero.
We have 
\[
 \Wc(\t(y)) = \mu(y) |y|^{-\frac{1}{2}} \int_F f(w \n(xy^{-1})) \overline{\psi(x)} \, dx
 = \mu(y) |y|^{\frac{1}{2}} \int_F f(w \n(x)) \overline{\psi(xy)} \, dx.
\]
Noting that 
\[
 w \n(x) = \begin{pmatrix} & -1 \\ 1 & x \end{pmatrix}
 = \begin{pmatrix} x^{-1} & -1 \\ & x \end{pmatrix} \begin{pmatrix} 1 & \\ x^{-1} & 1 \end{pmatrix},
\]
we have
\[
 \int_F f(w \n(x)) \overline{\psi(xy)} \, dx 
 = \sum_{j=n}^\infty \int_{\varpi^{-j} \o^{\times}} \chi(x)^{-1} \mu(x) |x|^{-1} \overline{\psi(xy)} \, dx
 = \sum_{j=n}^\infty \alpha^j \mu(\varpi)^{-j} \cdot \widehat{\I_{\o^\times} \mu}(\varpi^{-j} y).
\]
Since $\widehat{\I_{\o^\times}\mu} = q^{-n} \cdot \mathfrak{g}(\mu, \psi) \cdot \I_{\varpi^{-n} \o^\times} \mu^{-1}$, where
\[
 \mathfrak{g}(\mu, \psi) = \int_{\varpi^{-n} \o^{\times}} \mu(x) \psi(x) \, dx,
\]
we have
\begin{align*}
 \Wc(\t(\varpi^i)) & = \mu(\varpi)^i q^{-\frac{i}{2}} \cdot \alpha^{i+n} \mu(\varpi)^{-(i+n)}
 \cdot q^{-n} \cdot \mathfrak{g}(\mu, \psi) \cdot \mu(\varpi)^n \\
 & = q^{-\frac{i}{2}-n} \cdot \alpha^{i+n} \cdot \mathfrak{g}(\mu, \psi)
\end{align*}
if $i \ge 0$, and $\Wc(\t(\varpi^i)) = 0$ if $i < 0$.
Thus $W = \Wc(1)^{-1} \cdot \Wc$ and the assertion follows.
\end{proof}

\begin{lem}
\label{lem:|W|^2-tr}
We have
\[
 \| W \|^2 = L(1, \pi, \ad).
\]
\end{lem}

\begin{proof}
By Lemma \ref{lem:whittaker-formula-tr}, we have
\[
 \| W \|^2 = \sum_{i=0}^\infty |W(\t(\varpi^i))|^2 = \sum_{i=0}^\infty q^{-i} = \frac{1}{1-q^{-1}}.
\]
\end{proof}

\subsubsection{The case $v \in \Sigma_\infty$}

Recall that $\pi = \DS_k$ and $\psi(x) = e^{2 \pi \sqrt{-1} x}$.
It is known that
\begin{equation}
\label{eq:whittaker-formula-ds}  
 W(\t(y)) = 
 \begin{cases}
  y^{\frac{k}{2}} e^{-2 \pi y} & \text{if $y > 0$,} \\
  0 & \text{if $y < 0$.}
 \end{cases}
\end{equation}

\begin{lem}
\label{lem:|W|^2-ds}
We have
\[
 \| W \|^2 = \frac{(k-1)!}{(4\pi)^k}.
\]
\end{lem}

\begin{proof}
By \eqref{eq:whittaker-formula-ds}, we have
\[
 \| W \|^2 = \int_0^\infty |W(\t(y))|^2 \, \frac{dy}{y}
 = \int_0^\infty y^{k-1} e^{-4 \pi y} \, dy = \frac{\Gamma(k)}{(4\pi)^k}.
\]
\end{proof}

\subsection{Matrix coefficients of the Weil representation}
\label{ssec:mat-coeff-weil}

Suppose that $F$ is local.
In this subsection, we compute the function
\[
 \Phi(g) := \langle \omega_\psi(g) \varphi, \varphi \rangle
\]
on $\U(W) \cong B^1$ explicitly, where $\varphi \in \SS(\X)$ is the Schwartz function given in \S \ref{ssec:schwartzX}.
Since $\varphi$ is the partial Fourier transform of the Schwartz function $\varphi' \in \SS(\X')$ given in \S \ref{ssec:schwartzX'}, we have 
\[
 \Phi(g) = \langle \omega_\psi(g) \varphi', \varphi' \rangle.
\]
Put
\[
 \m(a) = \begin{pmatrix} a & \\ & a^{-1} \end{pmatrix}, \qquad
 \n(b) = \begin{pmatrix} 1 & b \\ & 1 \end{pmatrix}
\]
for $a \in F^\times$, $b \in F$.

\subsubsection{The case (ur)}
\label{sssec:mat-coeff-weil-ur}

We identify $B^\times$ with $\GL_2(F)$ via:
\begin{itemize}
 \item the isomorphism $\ii$ given by \eqref{eq:isomB-u} if $E$ is split and $F$ is unramified;
 \item the isomorphism $\Ad \smat{1}{}{}{\varpi^{-d}} \circ \ii$, where $\ii$ is the isomorphism given by \eqref{eq:isomB-u} and $\varpi^{-d} \o$ is the inverse different, if $E$ is split and $F$ is ramified; 
 \item the isomorphism $\ii$ given by \eqref{eq:isomB-J} if $E$ is inert and $J \in (F^{\times})^2$;
 \item any fixed isomorphism $\ii:B \rightarrow \M_2(F)$ such that $\ii(\o_B) = \M_2(\o)$, where $\o_B$ is the maximal order in $B$ given in \S \ref{par:schwartzX'-ur-Einert-J1}, if $E$ is inert, and $J_1 \in (F^{\times})^2$ or $J_2 \in (F^{\times})^2$;
 \item the isomorphism $\ii$ given by \eqref{eq:isomB-J} if $E$ is ramified.
\end{itemize}
Under this identification, we have $\K = \GL_2(\o)$, where $\K$ is the maximal compact subgroup of $B^\times$ given in \S \ref{sssec:schwartzX'-ur}.

\begin{lem}
\label{lem:mat-coeff-weil-ur}
We have $\Phi(\m(a)) = |a|^2$ for $a \in \o \smallsetminus \{ 0 \}$.
\end{lem}

\begin{proof}
Put
\[
 \phi(a) := \int_F \I_\o(a x) \I_\o(x) \, dx = q^{-\frac{d}{2}} \times 
 \begin{cases}
  1 & \text{if $a \in \o$,} \\
  |a|^{-1} & \text{otherwise,}
 \end{cases}
\]
where $dx$ is the self-dual Haar measure on $F$ with respect to $\psi$.
Note that $d=0$ unless $E$ is split and $F$ is ramified.

Assume that $E$ is split and $F$ is unramified.
We use the notation of \S \ref{par:schwartzX'-ur-Espl-Fur}.
Then the Weil representation $\omega_\psi$ on $\SS(\X')$ is given in \S \ref{sssec:wru}.
We have
\[
 \Phi(\m(a)) = |a|^2 \cdot \int_{\X'} \varphi'(a x) \overline{\varphi'(x)} \, dx
 = |a|^2 \cdot \prod_{i=1}^4 \int_F \I_\o(a x_i) \I_\o(x_i) \, dx_i = |a|^2 \cdot \phi(a)^4.
\]
This yields the desired identity.

Assume that $E$ is split and $F$ is ramified.
We use the notation of \S \ref{par:schwartzX'-ur-Espl-Fram}.
Then the Weil representation $\omega_\psi$ on $\SS(\X')$ is given in \S \ref{sssec:wru}, where $B^\times$ is identified with $\GL_2(F)$ via $\ii$ rather than $\Ad \smat{1}{}{}{\varpi^{-d}} \circ \ii$.
We have
\[
 \Phi(\m(a)) 
 = |a|^2 \cdot \int_{\X'} \varphi'(a x) \overline{\varphi'(x)} \, dx 
 = q^{2d} \cdot |a|^2 \cdot \prod_{i=1}^4 \int_{F} \I_\o(a x_i) \I_\o(x_i) \, dx_i
 = q^{2d} \cdot |a|^2 \cdot \phi(a)^4.
\]
This yields the desired identity.

Assume that $E$ is inert and $J \in (F^{\times})^2$.
We use the notation of \S \ref{par:schwartzX'-ur-Einert-J}.
Then the Weil representation $\omega_\psi$ on $\SS(\X')$ is given in \S \ref{sssec:wrJ}.
We have
\[
 \Phi(\m(a)) 
 = |a|^2 \cdot \int_{\X'} \varphi'(a x) \overline{\varphi'(x)} \, dx
 = |a|^2 \cdot \prod_{i=1}^4 \int_F \I_\o(a x_i) \I_\o(x_i) \, dx_i
 = |a|^2 \cdot \phi(a)^4.
\]
This yields the desired identity.

Assume that $E$ is inert and $J_1 \in (F^{\times})^2$; the case when $E$ is inert and $J_2 \in (F^{\times})^2$ is similar.
We use the notation of \S \ref{par:schwartzX'-ur-Einert-J1}.
Then the Weil representation $\omega_\psi$ on $\SS(\X')$ is given in \S \ref{sssec:wrJ1}.
We have
\[
 \Phi(\m(a)) 
 = \int_{\X'} \varphi'(x \m(a)) \overline{\varphi'(x)} \, dx
 = \int_{\M_2(F)} \varphi'(x \m(a)) \overline{\varphi'(x)} \, dx
 = \phi(a)^2 \cdot \phi(a^{-1})^2,
\]
where we identify $\X' \cong W$ with $\M_2(F)$ via the fixed isomorphism $\ii$ and normalize the Haar measure on $\M_2(F)$ so that $\vol(\M_2(\o)) = 1$.
This yields the desired identity.

Assume that $E$ is ramified.
We use the notation of \S \ref{par:schwartzX'-ur-Eram}.
Then the Weil representation $\omega_\psi$ on $\SS(\X')$ is given in \S \ref{sssec:wrJ}.
We have
\[
 \Phi(\m(a)) 
 = |a|^2 \cdot \int_{\X'} \varphi'(a x) \overline{\varphi'(x)} \, dx
 = |a|^2 \cdot \prod_{i=1}^4 \int_{F} \I_\o(a x_i) \I_\o(x_i) \, dx_i
 = |a|^2 \cdot \phi(a)^4.
\]
This yields the desired identity.
\end{proof}

\subsubsection{The case (rps)}
\label{sssec:mat-coeff-weil-tr}

We identify $B^\times$ with $\GL_2(F)$ via the isomorphism $\ii$ given by \eqref{eq:isomB-u}.
Under this identification, we have $\K = \GL_2(\o)$, where $\K$ is the maximal compact subgroup of $B^\times$ given in \S \ref{sssec:schwartzX'-tr}.
We write $\Phi = \Phi_\mu$ to indicate the dependence of $\varphi = \varphi_\mu$ on a unitary ramified character $\mu$ of conductor $q^n$.

\begin{lem}
\label{lem:mat-coeff-weil-tr}
We have
\[
 \Phi_\mu(\n(b) \m(a)) = 
 \begin{cases}
  \mu(a) & \text{if $a \in \o^\times$ and $b \in \o$,} \\
  0 & \text{otherwise.}
 \end{cases}
\]
\end{lem}

\begin{proof}
We use the notation of \S \ref{sssec:schwartzX'-tr}.
Then the Weil representation $\omega_\psi$ on $\SS(\X')$ is given in \S \ref{sssec:wru}.
We have
\begin{align*}
 & \Phi_\mu(\n(b) \m(a)) =
 |a|^2 \cdot \int_{\X'} \varphi'(a x) \overline{\varphi'(x)}
 \psi\left( \frac{1}{2} b \langle x, x \rangle^\dagger \right) dx \\
 & = q^{n+1} (q-1)^{-1} \cdot \mu(a) \cdot |a|^2 \\
 & \quad \times \int_{F^4} \I_\o(a x_1) \I_\o(x_1) \I_\o(a x_2) \I_\o(x_2)
 \I_{\varpi^n \o}(a x_3) \I_{\varpi^n \o}(x_3) \I_{\o^{\times}}(a x_4) \I_{\o^{\times}}(x_4) 
 \psi( b(x_1 x_4 - x_2 x_3)) \, dx_1 \, dx_2 \, dx_3 \, dx_4.
\end{align*}
Since $\I_{\o^{\times}}(a x_4) \I_{\o^{\times}}(x_4) = \I_{\o^{\times}}(a) \I_{\o^{\times}}(x_4)$, the above integral is zero unless $a \in \o^\times$, in which case it is equal to 
\begin{align*}
 & \int_{F^4} \I_\o(x_1) \I_\o(x_2) \I_{\varpi^n \o}(x_3) \I_{\o^{\times}}(x_4) 
 \psi( b(x_1 x_4 - x_2 x_3)) \, dx_1 \, dx_2 \, dx_3 \, dx_4 \\
 & = \int_F \I_{\o}(b x_3) \I_{\varpi^n \o}(x_3) \, dx_3 \cdot 
 \int_F \I_\o(b x_4) \I_{\o^{\times}}(x_4) \, dx_4 \\
 & = 
 \begin{cases}
 q^{-n} (1-q^{-1}) & \text{if $b \in \o$,} \\
 0 & \text{otherwise.}
 \end{cases}
\end{align*}
This yields the lemma.
\end{proof}

\subsubsection{The case (st)}
\label{sssec:mat-coeff-weil-st}

We identify $B^\times$ with $\GL_2(F)$ via:
\begin{itemize}
 \item the isomorphism $\ii$ given by \eqref{eq:isomB-u} if $B_1$ and $B_2$ are split;
 \item the isomorphism $\ii$ given by \eqref{eq:isomB-J} if $B_1$ and $B_2$ are ramified.
\end{itemize}
Under this identification, we have $\K = \GL_2(\o)$, where $\K$ is the maximal compact subgroup of $B^\times$ given in \S \ref{sssec:schwartzX'-st}.
Put
\[
 w = \begin{pmatrix} & -1 \\ 1 & \end{pmatrix}.
\]

\begin{lem}
\label{lem:mat-coeff-weil-st}
We have
\begin{align*}
 \Phi(\m(\varpi^i)) & =
 \begin{cases}
  q^{-2i} & \text{if $i \ge 0$,} \\
  q^{2i} & \text{if $i \le 0$,}
 \end{cases} \\
 \Phi(\m(\varpi^i) w) & =
 \begin{cases}
  \gamma_{B_1} \cdot q^{-2i-1} & \text{if $i \ge 0$,} \\
  \gamma_{B_1} \cdot q^{2i+1} & \text{if $i < 0$,}
 \end{cases}
\end{align*}
where  
\[
 \gamma_{B_1} = 
 \begin{cases}
  1 & \text{if $B_1$ is split,} \\
  -1 & \text{if $B_1$ is ramified.}
 \end{cases}
\]
\end{lem}

\begin{proof}
For convenience, we write $\p^i = \varpi^i \o$ for $i \in \Z$.
Put
\[
 \phi(j,k) := \int_F \I_{\p^j}(x) \I_{\p^k}(x) \, dx = 
 \begin{cases}
  q^{-j} & \text{if $j \ge k$,} \\
  q^{-k} & \text{if $j \le k$}
 \end{cases}
\]
for $j,k \in \Z$, where $dx$ is the self-dual Haar measure on $F$ with respect to $\psi$.

Assume that $B_1$ and $B_2$ are split.
We use the notation of \S \ref{par:schwartzX'-st-B1spl}.
Then the Weil representation $\omega_\psi$ on $\SS(\X')$ is given in \S \ref{sssec:wru}.
We have
\begin{align*}
 \varphi'(x) & = q^{\frac{1}{2}} \cdot \I_{\o}(x_1) \I_{\o}(x_2) \I_{\p}(x_3) \I_{\o}(x_4), \\
 \omega_\psi(w) \varphi'(x) & = q^{-\frac{1}{2}} \cdot \I_{\o}(x_1) \I_{\p^{-1}}(x_2) \I_{\o}(x_3) \I_{\o}(x_4),
\end{align*}
so that 
\begin{align*}
 \Phi(\m(\varpi^i)) & = q^{-2i} \cdot \int_{\X'} \varphi'(\varpi^i x) \overline{\varphi'(x)} \, dx = q^{-2i+1} \cdot \phi(-i, 0)^3 \cdot \phi(-i+1, 1), \\
 \Phi(\m(\varpi^i) w) & = q^{-2i} \cdot \int_{\X'} \omega_\psi(w) \varphi'(\varpi^i x) \overline{\varphi'(x)} \, dx = q^{-2i} \cdot \phi(-i,0)^2 \cdot \phi(-i-1,0) \cdot \phi(-i,1).
\end{align*}
This yields the desired identity.

Assume that $B_1$ and $B_2$ are ramified.
We use the notation of \S \ref{par:schwartzX'-st-B1ram}.
Then the Weil representation $\omega_\psi$ on $\SS(\X')$ is given in \S \ref{sssec:wrJ}.
We have
\begin{align*}
 \varphi'(x) & = q^{\frac{1}{2}} \cdot \I_{\o}(x_1) \I_{\o}(x_2) \I_{\o}(x_3) \I_{\o}(x_4), \\
 \omega_\psi(w) \varphi'(x) & = - q^{-\frac{1}{2}} \cdot \I_{\o}(x_1) \I_{\o}(x_2) \I_{\p^{-1}}(x_3) \I_{\p^{-1}}(x_4),
\end{align*}
so that 
\begin{align*}
 \Phi(\m(\varpi^i)) & = q^{-2i} \cdot \int_{\X'} \varphi'(\varpi^i x) \overline{\varphi'(x)} \, dx = q^{-2i} \cdot \phi(-i,0)^4, \\
 \Phi(\m(\varpi^i) w) & = q^{-2i} \cdot \int_{\X'} \omega_\psi(w) \varphi'(\varpi^i x) \overline{\varphi'(x)} \, dx = - q^{-2i-1} \cdot \phi(-i,0)^2 \cdot \phi(-i-1,0)^2.
\end{align*}
This yields the desired identity.
\end{proof}

\subsubsection{The case (1d)}
\label{sssec:mat-coeff-weil-1d}

Let $\K = \o_B^\times$ be the unique maximal compact subgroup of $B^\times$.
We have $B^1 \subset \K$.

\begin{lem}
\label{lem:mat-coeff-weil-1d}
We have $\Phi(g) = 1$ for $g \in B^1$.
\end{lem}

\begin{proof}
We use the notation of \S \ref{sssec:schwartzX'-1d}.
We have $\omega_\psi(g) \varphi' = \varphi'$ and hence $\Phi(g) = \langle \varphi', \varphi' \rangle = 1$ for all $g \in B^1$.
\end{proof}

\subsubsection{The case (ds)}
\label{sssec:mat-coeff-weil-ds}

We identify $B^\times$ with $\GL_2(F)$ via the isomorphism $\ii$ given by \eqref{eq:isomB-J}.
We write $\Phi = \Phi_k$ to indicate the dependence of $\varphi = \varphi_k$ on a non-negative integer $k$.

\begin{lem}
\label{lem:mat-coeff-weil-ds}
We have
\[
 \Phi_k(\m(a)) = \left( \frac{a+a^{-1}}{2} \right)^{-k-2}
\]
for $a > 0$.
\end{lem}

\begin{proof}
Assume that $B_1$ and $B_2$ are split.
We use the notation of \S \ref{par:schwartzX'-ds-B1spl}.
Then the Weil representation $\omega_\psi$ on $\SS(\X')$ is given in \S \ref{sssec:wrJ}.
We have
\begin{align*}
 \Phi_k(\m(a)) 
 & = a^2 \cdot \int_{\X'} \varphi'(a x) \overline{\varphi'(x)} \, dx \\
 & = \frac{|u|}{4} \cdot c_k^{-1} \cdot a^{k+2} \cdot
 \int_{F^4} (x_2^2 - u x_1^2)^k \cdot e^{-\frac{\pi}{2v} (a^2+1) (x_2^2 - ux_1^2 + x_4^2 - u x_3^2)} \, dx_1 \cdots dx_4 \\
 & = \frac{|u|}{4} \cdot c_k^{-1} \cdot a^{k+2} \cdot v^{-2} \cdot 
 \left(\frac{\pi}{2v} (a^2+1)\right)^{-k-2} \cdot \phi(k) \cdot \phi(0), 
\end{align*}
where 
\[
 \phi(k) := \int_{-\infty}^\infty \int_{-\infty}^\infty (x^2 + y^2)^k e^{-(x^2+y^2)} \, dx \, dy
\]
with the Lebesgue measures $dx$, $dy$ on $\R$.
Since 
\[
 \phi(k) = \int_{0}^{2 \pi} \int_{0}^{\infty} r^{2k} e^{-r^2} r \, dr \, d \theta 
 = \pi \cdot \int_{0}^{\infty} r^k e^{-r} \, dr
 = \pi \cdot k!,
\]
we have
\[
 \Phi_k(\m(a)) = \frac{|u|}{4} \cdot \frac{4 \pi^k}{k! |u|^{\frac{k}{2}+1}}
 \cdot a^{k+2} \cdot v^{-2} \cdot \left(\frac{\pi}{2v} (a^2+1)\right)^{-k-2} \cdot \pi^2 \cdot k!
 = \left( \frac{a+a^{-1}}{2} \right)^{-k-2}.
\]

Assume that $B_1$ and $B_2$ are ramified.
We use the notation of \S \ref{par:schwartzX'-ds-B1ram}.
Then the Weil representation $\omega_\psi$ on $\SS(\X')$ is given in \S \ref{sssec:wrJ} and the computation is the same as in the case when $B_1$ and $B_2$ are split.
\end{proof}

\subsubsection{The case (fd)}
\label{sssec:mat-coeff-weil-fd}

We identify $\C^\times$ with a subgroup of $B^\times$ via the isomorphism $E \cong \C$ such that $\i/\sqrt{-1} > 0$ and the fixed embedding $E \hookrightarrow B$.
Let $\phi_k$ be the matrix coefficient of $\Sym^k$ such that 
\begin{itemize}
 \item $\phi_k(\alpha g \beta) = \chi_k(\alpha) \chi_k(\beta) \phi_k(g)$ for $\alpha, \beta \in \C^\times$ and $g \in B^\times$,
 \item $\phi_k(1) = 1$.
\end{itemize}
We write $\Phi = \Phi_k$ to indicate the dependence of $\varphi = \varphi_k$ on a non-negative integer $k$.

\begin{lem}
\label{lem:mat-coeff-weil-fd}
We have $\Phi_k(g) = \overline{\phi_k(g)}$ for $g \in B^1$.
\end{lem}

\begin{proof}
We use the notation of \S \ref{sssec:schwartzX'-fd}.
Then the Weil representation $\omega_\psi$ on $\SS(\X')$ is given in \S \ref{sssec:wrJ1}.
If we write $x = z_1 + z_2 \frac{\j}{s} \in \X' \cong W = B$ with $z_1, z_2 \in E$, then 
\[
 \varphi'(x) = c_k^{-\frac{1}{2}} \cdot (z_1^\rho)^k \cdot e^{- \frac{\pi v}{2} (z_1 z_1^\rho + z_2 z_2^\rho)}.
\]
Let $\SS_k$ be the subspace of $\SS(\X')$ generated by $\omega_\psi(g) \varphi'$ for all $g \in B^1$.
Since
\[
 \omega_\psi(g) \varphi'(z_1, z_2)
 = \varphi'(z_1 \alpha_1 - z_2 \alpha_2^\rho, z_1 \alpha_2 + z_2 \alpha_1^\rho)
\]
for $g = \alpha_1 + \alpha_2 \frac{\j}{s} \in B^1$ with $\alpha_1, \alpha_2 \in E$, $\SS_k$ is generated by
\[
 (z_1^\rho)^i \cdot (z_2^\rho)^{k-i} \cdot e^{-\frac{\pi v}{2}(z_1 z_1 \rho + z_2 z_2^\rho)}
\]
for all $0 \le i \le k$.
Moreover, the representation of $B^1$ on $\SS_k$ is isomorphic to the unique irreducible $(k+1)$-dimensional representation $\Sym^k|_{B^1}$, so that $\Phi_k$ is a matrix coefficient of $\Sym^k|_{B^1}$.
On the other hand, by Lemma \ref{lem:varphi'-fd}, we have $\Phi_k(\alpha g \beta) = \chi_k(\alpha)^{-1} \chi_k(\beta)^{-1} \Phi_k(g)$ for $\alpha, \beta \in \C^1$ and $\Phi_k(1) = \langle \varphi', \varphi' \rangle = 1$. 
Hence we must have $\Phi_k = \bar{\phi}_k|_{B^1}$.
\end{proof}

\subsection{Computation of $Z_v$}
\label{ssec:computation_of_Zv}

To finish the proof of Theorem \ref{thm:rallis-B-explicit}, it remains to compute the integral $Z_v$.
We fix a place $v$ of $F$ and suppress the subscript $v$ from the notation.
Recall that
\[
 Z = \int_{B^1} \Phi(g) \Psi(g) \, dg,
\]
where 
\begin{itemize}
 \item $\Phi$ is the function on $B^1$ given in \S \ref{ssec:mat-coeff-weil};
 \item $\Psi$ is the function on $B^1$ defined by 
 \[
  \Psi(g) = \langle \pi_B(g) f_B, f_B \rangle, 
 \]
 where $f_B \in \pi_B$ is the new vector as in \S \ref{ssec:explicit-rallis} and $\langle \cdot, \cdot \rangle$ is the invariant hermitian inner product on $\pi_B$ normalized so that $\langle f_B, f_B \rangle = 1$;
 \item $dg$ is the standard measure on $B^1$.
\end{itemize}

\subsubsection{The case (ur)}

In this case, $\pi_B = \Ind(\chi \otimes \mu)$, where $\chi$ and $\mu$ are unitary unramified.
We have 
\[
 L(s, \pi, \ad) = \frac{1}{(1-q^{-s})(1 - \gamma q^{-s})(1 - \gamma^{-1} q^{-s})},
\]
where $\gamma = \chi(\varpi) \cdot \mu(\varpi)^{-1}$.

\begin{lem}
\label{lem:zeta-integral-ur}
We have
\[
 Z = \frac{L(1, \pi, \ad)}{\zeta(2)^2}.
\]
\end{lem}

\begin{proof}
We retain the notation of \S\S \ref{sssec:schwartzX'-ur}, \ref{sssec:mat-coeff-weil-ur}.
Put $\K' = \SL_2(\o)$.
Then we have
\[
 Z = \sum_{i=0}^\infty \Phi(\m(\varpi^i)) \Psi(\m(\varpi^i)) \vol(\K' \m(\varpi^i) \K').
\]
By Macdonald's formula \cite{macdonald-spherical}, \cite{casselman}, we have
\[
 \Psi(\m(\varpi^i)) = \frac{q^{-i}}{1 + q^{-1}} \cdot \left( \gamma^i \cdot \frac{1 - \gamma^{-1} q^{-1}}{1 - \gamma^{-1}} + \gamma^{-i} \cdot \frac{1 - \gamma q^{-1}}{1 - \gamma} \right).
\]
Also, we see that 
\[
 \vol(\K' \m(\varpi^i) \K') = 
 \begin{cases}
  1 & \text{if $i = 0$,} \\
  q^{2i}(1 + q^{-1}) & \text{if $i \ge 1$.}
 \end{cases}
\]
Combining these with Lemma \ref{lem:mat-coeff-weil-ur}, we obtain
\begin{align*}
 Z & = 1 + \sum_{i=1}^\infty q^{-i} \cdot \left( \gamma^i \cdot \frac{1 - \gamma^{-1} q^{-1}}{1 - \gamma^{-1}} + \gamma^{-i} \cdot \frac{1 - \gamma q^{-1}}{1 - \gamma} \right) \\
 & = 1 + \frac{\gamma q^{-1}}{1 - \gamma q^{-1}} \cdot \frac{1 - \gamma^{-1} q^{-1}}{1 - \gamma^{-1}} + \frac{\gamma^{-1} q^{-1}}{1 - \gamma^{-1} q^{-1}} \cdot \frac{1 - \gamma q^{-1}}{1 - \gamma} \\
 & = \frac{(1 + q^{-1})(1 - q^{-2})}{(1 - \gamma q^{-1})(1 - \gamma^{-1} q^{-1})}.
\end{align*}
\end{proof}

\subsubsection{The case (rps)}

In this case, $\pi_B = \Ind(\chi \otimes \mu)$ and $\Phi = \Phi_\mu$, where $\chi$ is unitary unramified and $\mu$ is unitary ramified of conductor $q^n$.
We have $L(s, \pi, \ad) = \zeta(s)$.

\begin{lem}
\label{lem:zeta-integral-tr}
We have
\[
 Z = \frac{1}{q^{n-4}(q-1)(q+1)^3} \cdot \frac{L(1,\pi,\ad)}{\zeta(2)^2}.
\]
\end{lem}

\begin{proof}
Following \cite[Chapter VIII]{kry}, we shall compute $Z$ explicitly.
We retain the notation of \S\S \ref{sssec:schwartzX'-tr}, \ref{sssec:mat-coeff-weil-tr}.
Put $\K' = \SL_2(\o)$ and $\K_n' = \K_n \cap \SL_2(\o)$.
We take the invariant hermitian inner product $\langle \cdot, \cdot \rangle$ on $\pi_B$ defined by 
\[
 \langle f_1, f_2 \rangle = \int_\K f_1(k) \overline{f_2(k)} \, dk,
\]
where $dk$ is the Haar measure on $\K$ such that $\vol(\K) = 1$.
Then $f_B$ is determined by
\[
 f_B|_\K = \vol(\K_n)^{-\frac{1}{2}} \cdot \I_{\K_n} \Mu.
\]
We can define a new vector $\tilde{f}_B \in \pi_B$ with respect to $(\K_n, \Mu)$ by
\[
 \tilde{f}_B(h) = \int_{B^1} \Phi(g) f_B(hg) \, dg 
\]
for $h \in B^\times$.
Since $\tilde{f}_B = \frac{\tilde{f}_B(1)}{f_B(1)} \cdot f_B$ and $\langle f_B, f_B \rangle = 1$, we have
\[
 Z = \langle \tilde{f}_B, f_B \rangle = \vol(\K_n)^{\frac{1}{2}} \cdot \tilde{f}_B(1).
\]
We have
\begin{align*}
 \tilde{f}_B(1) & = \int_{B^1} \Phi(g) f_B(g) \, dg \\
 & = \int_{\K'} \int_{F^\times} \int_F \Phi(\n(b) \m(a) k) \cdot f_B(\n(b) \m(a) k) \cdot |a|^{-2} \, db \, da \, dk \\
 & = \vol(\K_n)^{-\frac{1}{2}} \cdot \int_{\K_n'} \int_{F^\times} \int_F \Phi(\n(b) \m(a)) \Mu(k)^{-1} \cdot \chi(a) \mu(a)^{-1} |a| \Mu(k) \cdot |a|^{-2} \, db \, da \, dk \\
 & = \vol(\K_n)^{-\frac{1}{2}} \cdot \vol(\K_n') \cdot \int_{F^\times} \int_F \Phi(\n(b) \m(a)) \cdot \chi(a) \mu(a)^{-1} |a|^{-1} \, db \, da,
\end{align*}
where 
\begin{itemize}
 \item $db$ is the Haar measure on $F$ such that $\vol(\o) = 1$;
 \item $da$ is the Haar measure on $F^{\times}$ such that $\vol(\o^{\times}) = 1$;
 \item $dk$ is the Haar measure on $\K'$ such that $\vol(\K') = 1$.
\end{itemize}
By Lemma \ref{lem:mat-coeff-weil-tr}, we have
\[
 \int_{F^\times} \int_F \Phi(\n(b) \m(a)) \cdot \chi(a) \mu(a)^{-1} |a|^{-1} \, db \, da = \int_{\o^{\times}} \int_\o \chi(a) |a|^{-1} \, db \, da = 1.
\]
Hence we have
\[
 Z = \vol(\K_n') = \frac{1}{q^{n-1}(q+1)}.
\]
\end{proof}

\subsubsection{The case (st)}

In this case, $\pi_B = \St \otimes \chi$, where $\chi$ is unitary unramified.
We have $L(s, \pi, \ad) = \zeta(s+1)$.

\begin{lem}
\label{lem:zeta-integral-st}
\begin{enumerate}
\item If $B_1$ and $B_2$ are split, then we have
\[
 Z = \frac{q^2}{(q+1)^3} \cdot \frac{L(1,\pi,\ad)}{\zeta(2)^2}.
\]
\item If $B_1$ and $B_2$ are ramified, then we have
\[
 Z = \frac{q^2}{(q-1)^2(q+1)} \cdot \frac{L(1,\pi,\ad)}{\zeta(2)^2}.
\]
\end{enumerate}
\end{lem}

\begin{proof}
We retain the notation of \S\S \ref{sssec:schwartzX'-st}, \ref{sssec:mat-coeff-weil-st}.
Put $\II' = \II \cap \SL_2(\o)$.
Let $\tilde{W} = N(T_0)/T_0$ be the extended affine Weyl group of $\GL_2(F)$, where $T_0 = \{ \smat{a}{}{}{d} \, | \, a, d \in \o^\times \}$ and $N(T_0)$ is the normalizer of $T_0$ in $\GL_2(F)$.
Then we have
\[
 \GL_2(F) = \bigsqcup_{\tilde{w} \in \tilde{W}} \II \tilde{w} \II.
\]
We can write $\tilde{W} = \Omega \ltimes W_a$ with $\Omega = \langle \omega \rangle$ and $W_a = \langle w_1, w_2 \rangle$, where
\[
 \omega =
 \begin{pmatrix}
  & 1 \\
  \varpi &  
 \end{pmatrix}, \qquad
 w_1 =
 \begin{pmatrix}
  & 1 \\
  1 &
 \end{pmatrix}, \qquad
 w_2 =
 \begin{pmatrix}
  & \varpi^{-1} \\
  \varpi &
 \end{pmatrix}.
\]
Noting that $w_1^2 = w_2^2 = 1$ and $w_1 w_2 = \m(\varpi)$, we have
\[
 \SL_2(F) = \bigsqcup_{j=0}^1 \bigsqcup_{i=-\infty}^{\infty} \II' \m(\varpi^i) w^j \II'
\]
and hence
\[
 Z = \sum_{j=0}^1 \sum_{i=-\infty}^{\infty} \Phi(\m(\varpi^i) w^j) \Psi(\m(\varpi^i) w^j) \vol(\II' \m(\varpi^i) w^j \II').
\]
Let $\ell$ be the length function on $\tilde{W}$, so that $\ell(\omega) = 0$ and $\ell(w_1) = \ell(w_2) = 1$.
By \cite[\S 7]{gj}, we have
\[
 \Psi(\omega^k \tilde{w}) = (-\chi(\varpi))^k \cdot (-q)^{-\ell(\tilde{w})}
\]
for $k \in \Z$ and $\tilde{w} \in W_a$.
Also, we see that $|\II \tilde{w} \II/\II| = q^{\ell(\tilde{w})}$ for $\tilde{w} \in \tilde{W}$.
Hence we have
\[
 \Psi(\m(\varpi^i) w^j) \vol(\II' \m(\varpi^i) w^j \II') 
 = (-1)^{\ell(\m(\varpi^i) w^j)} \cdot \vol(\II')
 = \frac{1}{q+1} \times 
 \begin{cases}
  1 & \text{if $j=0$,} \\
  -1 & \text{if $j=1$,}
 \end{cases}
\]
so that
\[
 Z = \frac{1}{q+1} \cdot \left( \sum_{i=-\infty}^{\infty} \Phi(\m(\varpi^i)) - \sum_{i=-\infty}^{\infty} \Phi(\m(\varpi^i) w) \right).
\]
Combining this with Lemma \ref{lem:mat-coeff-weil-st}, we obtain
\begin{align*}
 Z & = \frac{1}{q+1} \cdot \left( \sum_{i=0}^{\infty} q^{-2i} + \sum_{i=1}^{\infty} q^{-2i} - \sum_{i=0}^{\infty} q^{-2i-1} - \sum_{i=1}^{\infty} q^{-2i+1} \right) \\
 & = \frac{1}{q+1} \cdot \frac{1 + q^{-2} - q^{-1} - q^{-1}}{1 - q^{-2}} \\
 & = \frac{q-1}{(q+1)^2}
\end{align*}
if $B_1$ and $B_2$ are split, and
\begin{align*}
 Z & = \frac{1}{q+1} \cdot \left( \sum_{i=0}^{\infty} q^{-2i} + \sum_{i=1}^{\infty} q^{-2i} + \sum_{i=0}^{\infty} q^{-2i-1} + \sum_{i=1}^{\infty} q^{-2i+1} \right) \\ 
 & = \frac{1}{q+1} \cdot \frac{1 + q^{-2} + q^{-1} + q^{-1}}{1 - q^{-2}} \\
 & = \frac{1}{q-1}
\end{align*}
if $B_1$ and $B_2$ are ramified.
\end{proof}

\subsubsection{The case (1d)}

In this case, $\pi_B = \chi \circ \nu$, where $\chi$ is unitary unramified.
We have $L(s, \pi, \ad) = \zeta(s+1)$.

\begin{lem}
\label{lem:zeta-integral-1d} 
We have
\[
 Z = \frac{q^2}{(q-1)(q+1)} \cdot \frac{L(1, \pi, \ad)}{\zeta(2)^2}.
\]
\end{lem}

\begin{proof}
We retain the notation of \S\S \ref{sssec:schwartzX'-1d}, \ref{sssec:mat-coeff-weil-1d}.
Then by Lemma \ref{lem:mat-coeff-weil-1d}, we have
\[
 Z = \int_{B^1} dg = 1.
\]
\end{proof}

\subsubsection{The case (ds)}

In this case, $\pi_B = \DS_k$ and $\Phi = \Phi_l$, where
\[
 l = 
 \begin{cases}
  k & \text{if $B_1$ and $B_2$ are split,} \\
  k-2 & \text{if $B_1$ and $B_2$ are ramified.}
 \end{cases}
\]

\begin{lem}
\label{lem:zeta-integral-ds}
\begin{enumerate}
\item If $B_1$ and $B_2$ are split, then we have
\[
 Z = \frac{4 \pi}{k}.
\]
\item If $B_1$ and $B_2$ are ramified, then we have
\[
 Z = \frac{4 \pi}{k-1}.
\]
\end{enumerate}
\end{lem}

\begin{proof}
We retain the notation of \S\S \ref{sssec:schwartzX'-ds}, \ref{sssec:mat-coeff-weil-ds}.
In particular, we identify $\C^\times$ with a subgroup of $B^\times$.
Then we have
\begin{align*}
 Z & = 4 \pi \cdot \int_{\C^1} \int_0^\infty \int_{\C^1} \Phi(\kappa_1 \m(e^t) \kappa_2) \Psi(\kappa_1 \m(e^t) \kappa_2) \sinh(2t) \, d \kappa_1 \, dt \, d \kappa_2 \\ 
 & = 4 \pi \cdot \int_0^\infty \Phi(\m(e^t)) \Psi(\m(e^t)) \sinh(2t) \, dt,
\end{align*}
where
\begin{itemize}
 \item $dt$ is the Lebesgue measure;
 \item $d \kappa_1$ and $d \kappa_2$ are the Haar measures on $\C^1$ such that $\vol(\C^1) = 1$.
\end{itemize}
It is known that
\[
 \Psi(\m(e^t)) = \cosh(t)^{-k}.
\]
Combining these with Lemma \ref{lem:mat-coeff-weil-ds}, we obtain
\begin{align*}
 Z & = 4 \pi \cdot \int_0^\infty \cosh(t)^{-k-l-2} \sinh(2t) \, dt \\ 
 & = 8 \pi \cdot \int_0^{\infty} \cosh(t)^{-k-l-1} \sinh(t) \, dt \\
 & = 8 \pi \cdot \int_1^{\infty} t^{-k-l-1} \, dt \\
 & = \frac{8 \pi}{k+l}.
\end{align*}
\end{proof}

\subsubsection{The case (fd)}

In this case, $\pi_B = \Sym^{k-2}$ and $\Phi = \Phi_{k-2}$.

\begin{lem}
\label{lem:zeta-integral-fd}
We have
\[
 Z = \frac{1}{k-1}.
\]
\end{lem}

\begin{proof}
We retain the notation of \S\S \ref{sssec:schwartzX'-fd}, \ref{sssec:mat-coeff-weil-fd}.
Then by Lemma \ref{lem:mat-coeff-weil-fd} and the Schur orthogonality relations, we have
\[
 Z = \int_{B^1} |\phi_{k-2}(g)|^2 \, dg = \frac{1}{k-1}.
\]
\end{proof}

%% file: main-conj.tex
\section{The main conjecture on the arithmetic of theta lifts}
\label{sec:main-conj}

\subsection{On the choices of $u$, $J_1$ and $J_2$}
\label{ssec:choices}

We suppose now that we are given a totally real number field $F$ and two quaternion algebras 
$B_1$ and $B_2$ over $F$. 
Let us define for convenience:
\begin{align*}
\fd_{B_1 \smallsetminus B_2} &= \prod_{\fq \mid \fd_{B_1}, \, \fq \nmid \fd_{B_2}} \fq, &  \hspace{-25mm} \fd_{B_2 \smallsetminus B_1} &= \prod_{\fq \mid \fd_{B_2}, \, \fq \nmid \fd_{B_1}} \fq, \\
\quad \fd_{B_1 \cup B_2} &= \prod_{\fq \mid \fd_{B_1} \fd_{B_2}} \fq, &  \hspace{-30mm} \fd_{B_1 \cap B_2} &= \prod_{\fq \mid (\fd_{B_1}, \fd_{B_2}) } \fq
\end{align*}
and 
\begin{align*}
\Sigma_{B_1 \smallsetminus B_2} &= \Sigma_{B_1} \smallsetminus \Sigma_{B_2},  &  \hspace{-25mm} \Sigma_{B_2\smallsetminus B_1}&=\Sigma_{B_2} \smallsetminus \Sigma_{B_1}, \\
 \Sigma_{B_1 \cup B_2} &= \Sigma_{B_1} \cup \Sigma_{B_2},  &  \hspace{-25mm}  \Sigma_{B_1 \cap B_2} &= \Sigma_{B_1} \cap \Sigma_{B_2}.
\end{align*}

For the constructions so far (especially the constructions of splittings), the only condition needed is:
\begin{equation}
\label{eqn:condition-atleast-one-is-a-square}
\quad \text{ At every place $v$ of $F$, at least one of $u$, $J_1$, $J_2$, $J$ is a square.}
\end{equation}
However, to formulate the main conjecture 
we will need to make a more careful choice. In this section, we show that we can make such a choice 
that satisfies a number of useful auxiliary conditions. 

\begin{prop}
\label{prop:choices1}
Suppose that $\ell$ is a rational prime that is coprime to $\fd_{B_1\cup B_2}$ and $\{ \ff_1, \cdots, \ff_n \}$ is a collection of primes of $\cO_F$ (possibly empty) that
are coprime to $\ell \fd_{B_1 \cup B_2}$. 
Then we can find elements $u,J_1,J_2 \in F$ such that the following hold:
\begin{enumerate}
\item \label{choices:are-integral} $u,J_1, J_2$ lie in $\cO_F$. 
\item \label{choices:atleast-one-is-a-square} At every place $v$ of $F$, at least one of $u$, $J_1$, $J_2$, $J$ is a square.
\item \label{choices:<<0} $u \ll 0$, so that $E:=F+F\i$, $\i^2=u$ is a CM field. 
\item \label{choices:unit-at-unramified} $u$ is a unit at any prime $\fq$ that is unramified in $E$. 
\item \label{choices:at-2} If $\fq$ is a prime of $F$ dividing $2$, then 
$E_\fq$ is the unique unramified quadratic extension of $F_\fq$ 
if $\fq \mid \fd_{B_1\cup B_2}$ and $E_\fq/F_\fq$ is split otherwise. 
\item \label{choices:j1j2} 
\begin{itemize}[labelindent=0pt]
\item $B_1 \simeq E+ E\j_1$, with $\j_1^2 = J_1$ and $\i\j_1=-\j_1 \i$. 
\item $B_2 \simeq E+ E\j_2$, with $\j_2^2 = J_2$ and $\i\j_2=-\j_2 \i$.
\item $B\simeq E+E\j$ with $\j^2=J=J_1 J_2$ and $\i \j = - \j \i$.
\end{itemize}

\item \label{choices:squares-at-bad-primes} 
\begin{itemize}[labelindent=0pt]
\item If $\fq \mid \fd_{B_1\smallsetminus B_2}$, then $J_1$ is a uniformizer at $\fq$ and $J_2$ is the square of a unit. 
\item If $\fq \mid \fd_{B_2\smallsetminus B_1}$, then $J_2$ is a uniformizer at $\fq$ and $J_1$ is the square of a unit. 
\item If $\fq \mid \fd_{B_1 \cap B_2}$, then $J_1$ and $J_2$ are both uniformizers at $\fq$ such that $J_1/J_2$ is the square of a unit. 
\end{itemize}

\item \label{choices:sq-of-units-at-pf} $u$, $J_1$, $J_2$ and $J$ are squares of units at the primes in $\{ \ff_1, \ldots, \ff_n \}$ and at all primes $\fl$ of $F$ above $\ell$. 

\end{enumerate}
\end{prop}

Let $K$ denote the quadratic extension of $F$ given by
\begin{equation}
\label{eqn:defn-of-K}
K=F+F\j.
\end{equation}
Note that the condition \eqref{choices:sq-of-units-at-pf} above implies that 
both $E$ and $K$ are split at the primes in $\{ \fl \mid \ell\} \cup \{ \ff_1, \ldots, \ff_n \}$.

Prop. \ref{prop:choices1} will suffice for the current paper. The following
enhancement of it will be useful in \cite{periods2}, \cite{periods3}.

\begin{prop}
\label{prop:choices2} 
Let $\ell$ and $\ff_1,\ldots, \ff_n$ be as in the previous proposition. 
Suppose that the prime $\ell$ satisfies the following conditions:
\begin{itemize}
\item $\ell$ is unramified in $F$. 
\item $\ell >5$ and for any $\fq \mid \fd_{B_1 \smallsetminus B_2} \cdot \fd_{B_2 \smallsetminus B_1}$, we have 
\[
\N \fq \not \equiv 0, \pm 1 \pmod \ell.
\]
\end{itemize}
Then we can choose $u$, $J_1$, $J_2$ such that in addition to  \eqref{choices:are-integral} through \eqref{choices:sq-of-units-at-pf} above, we have:
\vspace{2mm}
\begin{enumerate}
\setcounter{enumi}{8}
\item \label{choices:modell} If $E$ or $K$ is ramified at a prime $\fp$, then $\N \fp \not \equiv 0,\pm 1 \pmod \ell$.
\end{enumerate}
\end{prop}

The following Lemmas \ref{lem:choosing-q-one} and \ref{lem:choosing-q-two} will be useful in the proofs of 
Prop. \ref{prop:choices1} and Prop. \ref{prop:choices2} respectively. 

\begin{lem}
\label{lem:choosing-q-one}
Let $F$ be a number field, $\Xi_f$ a finite subset of $\Sigma_{\fin}$ and $\Xi_\infty \subseteq \Sigma_\infty$ a set of real infinite places. Let $I$ be an ideal in $\cO_F$ prime to the 
primes in $\Xi_f$. Then there exists a prime ideal $\fq \subset \cO_F$ such that $I \cdot \fq = (\alpha)$ is principal with $\alpha$ satisfying:
\begin{enumerate}[label=$(\mathrm{\alph*})$, ref=\alph*]
\item $\alpha$ is a square of a unit at the primes in $\Xi_f$.
\item $\sigma_v(\alpha) <0$ for $v$ in $\Xi_\infty$ and $\sigma_v(\alpha) >0$ for any real place $v$ of $F$ not in $\Xi_\infty$. 
\end{enumerate}
Further, $\fq$ can be picked to avoid any finite set of primes.
\end{lem}

\begin{lem}
\label{lem:choosing-q-two}
Let $F$ be a number field and $\ell>5$ a rational prime unramified in $F$. Suppose that 
$\Xi_f$ is a finite subset of $\Sigma_\fin$ all whose elements are prime to $\ell$ and $\Xi_\infty \subseteq \Sigma_\infty$ is a set of real infinite places. Let $I$ be an ideal in $\cO_F$ prime to $\ell$ and the 
primes in $\Xi_f$.  Then there exists a prime ideal $\fq \subset \cO_F$ such that $I \cdot \fq = (\alpha)$ is principal with $\alpha$ satisfying:
\begin{enumerate}[label=$(\mathrm{\alph*})$, ref=\alph*]
\item \label{cond:alpha-square} $\alpha$ is a square of a unit at the primes in $\Xi_f$ and at all primes $\fl$ above $\ell$. 
\item \label{cond:alpha><} $\sigma_v(\alpha) <0$ for $v$ in $\Xi_\infty$ and $\sigma_v(\alpha) >0$ for $v$ any real place of $F$ not in $\Xi_\infty$. 
\item \label{cond:Nq} $\N \fq \not \equiv 0, \pm 1 \pmod \ell$.
\end{enumerate}
Further, $\fq$ can be picked to avoid any finite set of primes.
\end{lem}

We first prove Lemma \ref{lem:choosing-q-one} and then explain the modifications needed to prove
Lemma \ref{lem:choosing-q-two}.

\begin{proof}[Proof (of Lemma  \ref{lem:choosing-q-one})] 
Let $\fm$ be the product of all the real places of $F$ and the primes in $\Xi_f$, each raised to a sufficiently large power so that the local units congruent to $1 \pmod \fm$ are squares.  For $\alpha \in F^\times$, let $\iota(\alpha)$ denote the 
principal fractional ideal generated by $\alpha$. Also let $F_{\fm,1}$ denote the set of elements in $F^\times$ that are congruent to $1 \pmodx \fm$. If $U_F$ denotes the units in $F$ and $U_{F,\fm}$ the units congruent to $1 \pmodx \fm$, then there is an exact sequence:
\[
1 \rightarrow \frac{U_F}{U_{F,\fm}} \rightarrow \frac{F^\times}{F_{\fm, 1}}\rightarrow \frac{\iota(F^\times)}{\iota(F_{\fm,1})} \rightarrow 1.
\]
Let $H$ be the Hilbert class field of $F$ and $H_\fm$ the ray class field of $F$ of conductor $\fm$. Then $F \subset H \subset H_\fm$ and there is a canonical isomorphism
\[ \Gal (H_\fm / H) \simeq \frac{\iota(F^\times)}{\iota(F_{\fm,1})}. \]
Pick an element $\beta \in F^\times$ such that $\beta \equiv 1 \pmod \fm$ and such that $\beta$ is negative at the real places  in $\Xi_\infty$ and positive at the real places not in $\Xi_\infty$. Let $\sigma_{(\beta)} \in \Gal(H_\fm /H)$ be the element corresponding to $[\iota(\beta)]$ via the isomorphism above. Let $\sigma_I$ denote the image of $I$ in $\Gal (H_\fm/F)$ under the Artin map. By Tchebotcharev, there exists a prime ideal 
$\fq$ in $\cO_F$ that is prime to $\fm$ and such that 
\[
\sigma_\fq = \sigma_I^{-1} \cdot \sigma_{(\beta)} \ \ \text{ in } \Gal(H_\fm /F).
\]
In particular this implies that $\sigma_\fq = \sigma_I^{-1}$ in $\Gal (H/F)$, so there exists $\alpha \in F^\times $ such that $\fq \cdot I = (\alpha)$. Then $\sigma_{(\alpha)} = \sigma_{(\beta)}$, which is the same as saying that 
\[
[\iota (\alpha)] = [\iota(\beta)] \ \ \text{ in } \ \frac{\iota(F^\times)}{\iota(F_{\fm,1})}. 
\]
The exact sequence above implies then that there is a unit $u\in U_F$ such that 
\[
[u \cdot \alpha] = [\beta] \ \ \text{ in } \ \frac{F^\times}{F_{\fm, 1}}.
\]
Replacing $\alpha$ by $u\cdot \alpha$, we see that it has the required properties.
\end{proof}

\begin{proof}[Proof (of Lemma \ref{lem:choosing-q-two})]
We modify the proof of Lemma \ref{lem:choosing-q-one}. 

Let $\{ \fl_1, \ldots, \fl_r \}$ be the primes of $F$ lying over $\ell$. 
 Let $\fm$ be the product of all the real places of $F$, the primes in $\Xi_f$ (each raised to a sufficiently large power so that the local units congruent to $1 \pmod \fm$ are squares) 
 and the primes $\fl_2, \ldots, \fl_r$. Fix for the moment an element $w \in (\o_F/\fl_1)^\times$. 
By the approximation theorem, we may pick $\beta \in F^\times$ such that 
\begin{itemize}
\item $\beta$ is negative at the places in $\Xi_\infty$ and positive at the real places not in $\Xi_\infty$.
\item $\beta \equiv 1 \pmod \fm$.
\item $\beta \equiv w^2 \pmod {\fl_1}$. 
\end{itemize}

Let $\sigma_I$ denote the image of $I$ in 
$\Gal(H_{ \fm\fl_1 }/F)$ under the Artin map. By Tchebotcharev, there exists a prime ideal $\fq$ in $\cO_F$ that is prime to $ \fm \cdot \fl_1$ and such that 
\[
\sigma_\fq = \sigma_I^{-1} \cdot \sigma_{(\beta)} \ \ \text{  in } \Gal(H_{ \fm \fl_1}/F).
\]
As before then, there exists $\alpha \in F^\times$ such that $\fq \cdot I = (\alpha)$ and a unit $u\in U_F$ such that 
\[
[u \cdot \alpha] = [\beta] \ \ \text{ in } \ \frac{F^\times}{F_{\fm \fl_1,1}}.
\]
Replacing $\alpha$ by $u \cdot \alpha$, we see that $\alpha$ satisfies the requirements \eqref{cond:alpha-square}, 
\eqref{cond:alpha><}
of the lemma. It remains to show that $w$ can be chosen so that $\fq$ satisfies \eqref{cond:Nq}.  Clearly $\fq$ is prime to $\ell$. 
But 
\[
\N \fq \cdot \N I = \pm \N (\alpha) = \pm \N(\beta) \equiv \pm \N_{\F_{\fl_1}/ \F_\ell} (w^2) \pmod \ell.
\]
Since $\N I$ is fixed and we only need $\N \fq \not \equiv \pm 1 \pmod \ell$, it suffices to show that the subgroup
\[
\left\{ \N_{\F_{\fl_1}/ \F_\ell} (w^2) : w \in \F_{\fl_1}^\times \right\} \subset \F_\ell^\times 
\]
contains at least $3$ elements. But this subgroup is just $(\F_\ell^\times)^2$ (since $\fl_1$ is unramified over $\ell$) and has 
cardinality $\frac{\ell-1}{2} > 2 $ since $\ell>5$ by assumption. \end{proof}

Now we prove Prop. \ref{prop:choices1} and then explain the modifications needed to 
prove Prop. \ref{prop:choices2}.

\begin{proof}[Proof (of Prop. \ref{prop:choices1})]
Let $\ff=\ff_1 \cdots \ff_n$ and 
$\fS = 2\ell\fd_{B_1\cup B_2} \ff$. 
We begin by using Lemma \ref{lem:choosing-q-one} above to pick:
\begin{itemize}
\item
A prime ideal $\fq_{B_1\smallsetminus B_2}$ (prime to $\fS$) such that $\fd_{B_1 \smallsetminus B_2} \cdot \fq_{B_1 \smallsetminus B_2} = (\alpha_{B_1\smallsetminus B_2})$, with $\alpha_{B_1 \smallsetminus B_2}$ satisfying the following conditions:

\begin{itemize}
 \item 
 $\alpha_{B_1 \smallsetminus B_2}$ is a square of a unit at the primes dividing $\ell  \fd_{B_2} \ff$ and 
 the primes above $2$ not dividing $\fd_{B_1 \smallsetminus B_2}$. 

 \item 
 For $v \in \Sigma_\infty$, 
\[
\sigma_v (\alpha_{B_1 \smallsetminus B_2} ) <0, \text{ if } v \in \Sigma_{B_1 \smallsetminus B_2}; \qquad 
\sigma_v (\alpha_{B_1 \smallsetminus B_2} ) > 0, \text{ if } v \not \in \Sigma_{B_1 \smallsetminus B_2}.
\]

\end{itemize}

\item A prime ideal $\fq_{B_2\smallsetminus B_1}$ (prime to $\fS$) such that $\fd_{B_2 \smallsetminus B_1} \cdot \fq_{B_2 \smallsetminus B_1} = (\alpha_{B_2\smallsetminus B_1})$, with $\alpha_{B_2 \smallsetminus B_1}$ satisfying the following conditions:

\begin{itemize}
 \item 
 $\alpha_{B_2 \smallsetminus B_1}$ is a square of a unit at  the primes dividing $\ell \fd_{B_1} \ff$ and the primes above $2$ not dividing $\fd_{B_2 \smallsetminus B_1}$.

 \item 
 For $v \in \Sigma_\infty$, 
\[
\sigma_v (\alpha_{B_2 \smallsetminus B_1} ) <0, \text{ if } v \in \Sigma_{B_2 \smallsetminus B_1}; \qquad 
\sigma_v (\alpha_{B_2 \smallsetminus B_1} ) > 0, \text{ if } v \not \in \Sigma_{B_2 \smallsetminus B_1}.
\]

\end{itemize}

\item A prime ideal $\fq_{B_1\cap B_2}$ (prime to $\fS$) such that $\fd_{B_1 \cap B_2} \cdot \fq_{B_1 \cap B_2} = (\alpha_{B_1\cap B_2})$, with $\alpha_{B_1 \cap B_2}$ satisfying the following conditions:

\begin{itemize}
 \item 
 $\alpha_{B_1 \cap B_2}$ is a square of a unit at  the primes dividing $\ell \fd_{B_1\smallsetminus B_2}  \fd_{B_2 \smallsetminus B_1}\ff$ and the primes above $2$ not dividing $\fd_{B_1 \cap B_2}$. 

 \item 
 For $v \in \Sigma_\infty$, 
\[
\sigma_v (\alpha_{B_1 \cap B_2} ) <0, \text{ if } v \in \Sigma_{B_1 \cap B_2}; \qquad 
\sigma_v (\alpha_{B_1 \cap B_2} ) > 0, \text{ if } v \not \in \Sigma_{B_1 \cap B_2}.
\]

\end{itemize}
\end{itemize}

\noindent Let $\fR$ denote the ideal
\[
\fR := \fq_{B_1\smallsetminus B_2}\cdot \fq_{B_2 \smallsetminus B_1}\cdot \fq_{B_1 \cap B_2}.
\]
Next, we use the approximation theorem to pick $\alpha \in F^\times $ satisfying the following properties:
\begin{enumerate}[label=(\Roman*), ref=\Roman*]
\item \label{cond:alpha>} $\alpha \gg 0$.
\item \label{cond:alphaR} $-\alpha$ is a square of a unit at the primes dividing $\ell \fR \ff$.
\item \label{cond:alphaB} If $\fq$ is a prime dividing $\fd_{B_1 \cup B_2}$, then $-\alpha$ is a unit at $\fq$ but not a square. 
If further $\fq$ divides $2$, then we also require that $\sqrt{-\alpha}$ generate the unique unramified extension of $F_\fq$. 
\item \label{cond:alpha2} If $\fq$ is a prime dividing $2$ but not dividing $\fd_{B_1\cup B_2}$, then $-\alpha$ is a square of a unit at $\fq$. 
 \end{enumerate}
Let 
\[
\fm := 2^a \ell \cdot \fd_{B_1\cup B_2} \ff \cdot \fR \cdot \prod_{v\in \Sigma_\infty} v,
\]
with the power $2^a$ being chosen large enough so that locally at any prime above $2$, the units congruent to $1$ modulo $2^a$ are squares. 
By Tchebotcharev, there exists a prime ideal $\fQ \subset \cO_F$ (prime to $\fm$) such that 
\[
\sigma_{(\alpha)} \cdot \sigma_{\fQ} = 1 \ \ \text{ in } \Gal(H_\fm/F).
\]
This implies that
\[
(\alpha) \cdot \fQ = (\beta),
\]
for some $\beta \equiv 1 \pmodx \fm$. Now, take 
\[
u:= - \alpha^{-1} \beta, \quad J_1:= \alpha_{B_1 \smallsetminus B_2} \cdot \alpha_{B_1 \cap B_2}, \quad J_2:= \alpha _{B_2 \smallsetminus B_1} \cdot \alpha_{B_1 \cap B_2}.
\]
Since 
\[
(u) = \fQ, \quad (J_1) = \fd_{B_1} \cdot \fq_{B_1\smallsetminus B_2} \cdot \fq_{B_1 \cap B_2}, \quad \quad (J_2) = \fd_{B_2} \cdot \fq_{B_2\smallsetminus B_1} \cdot \fq_{B_1 \cap B_2},
\]
 we see that $u,J_1, J_2 $ lie in $\cO_F$, which shows that \eqref{choices:are-integral} is satisfied. 
 Let $E/F$ be the quadratic extension $E=F+F\i$ with $\i^2=u$. 
 Since $\alpha\gg 0 $ and $\beta \gg 0$, we have $u\ll 0$, which shows \eqref{choices:<<0}, whence 
 $E$ is a CM quadratic extension of $F$.
The conditions \eqref{cond:alphaB} and \eqref{cond:alpha2} above imply that if $\fq$ is a prime above $2$, then $E_\fq$ is 
the unique unramified quadratic extension of $F_\fq$ if $\fq$ divides $\fd_{B_1\cup B_2}$ and otherwise is split, which 
shows \eqref{choices:at-2}.
Since $(u) = \fQ$, it follows that $E$ is ramified exactly at the prime $\fQ$, and in particular satisfies  \eqref{choices:unit-at-unramified}.
Now we check that
\[
B_1 \simeq E+E\j_1, \quad \j_1^2= J_1, \quad \i \j_1 = -\j_1 \i.
\]
To show this, it suffices to check that the Hilbert symbol $(u,J_1)_v$ equals $-1$ exactly for those $v$ at which $B_1$ is ramified. 
At the archimedean places this is clear since $u\ll 0$ and $J_1$ is negative exactly at the places at which $B_1$ is ramified. As for the finite places, we only need to check this for $v$ dividing $2uJ_1$, since outside of these primes $B_1$ is split and $(u,J_1)=1$ since both $u$ and $J_1$ are units at such places. At the primes dividing $\fq_{B_1\smallsetminus B_2} \cdot \fq_{B_1 \cap B_2}$, the algebra $B_1$ is split and $u$ is a square of a unit, so this is clear. For $\fq \mid \fd_{B_1}$, the algebra $B_1$ is ramified, $J_1$ is a uniformizer and by \eqref{cond:alphaB} above, we have $(u,J_1)_\fq=(-\alpha, J_1)_\fq= -1$. Next we consider the primes $\fq$ above $2$. If $\fq \mid \fd_{B_1}$, this is done already. If $\fq \mid \fd_{B_2\smallsetminus B_1}$, then $J_1$ is a square at $\fq$, so $(u,J_1)_\fq = 1$ as required. This leaves the primes $\fq$ above $2$ which do not divide $\fd_{B_1 \cup B_2}$. At such primes, $u$ is a square of a unit, so $(u,J_1)_\fq = 1$.
The only prime left is $\fQ$ at which the required equality follows from the product formula! 
The isomorphism $B_2 \simeq E+E\j_2$ follows similarly, and then the isomorphism $B \simeq E+ E\j$ follows from the equality $B=B_1 \cdot B_2$ in the Brauer group. This completes the 
proof of \eqref{choices:j1j2}. 
The conditions \eqref{choices:squares-at-bad-primes} and \eqref{choices:sq-of-units-at-pf} are easily verified,
which leaves \eqref{choices:atleast-one-is-a-square}.

Finally, we check that \eqref{choices:atleast-one-is-a-square} is satisfied, namely that at every 
place $v$ of $F$,
at least one of $u$, $J_1$, $J_2$ or $J$ is a square.
  At the archimedean places, this is obvious. At the primes dividing $\fd_{B_1 \cup B_2}$, this follows from \eqref{choices:squares-at-bad-primes}.  Let $\fq$ be a finite prime not dividing $\fd_{B_1 \cup B_2}$. 
If such a $\fq$ divides $2$, then all of $u,J_1,J_2,J$ are squares at $\fq$. So let $\fq$ be prime to $2 \fd_{B_1 \cup B_2}$. 
If $E$ is split at $\fq$, then $u$ is a square at $\fq$. If $E$ is inert at $\fq$, then $J_1, J_2 $ lie in $\N_{E_\fq/F_{\fq}} (E_\fq^\times)$
since $B_1$ and $B_2$ are split at $\fq$. 
If $J_1$ and $J_2$ are both not squares at such $\fq$, it must be the case that $J=J_1 J_2$ is a square. Finally, we deal with $\fq = \fQ$, the only ramified prime in $E$. 
At this prime, $J_1$ and $J_2$ are both units. Again, if both of them are non-squares, their product must be a square. This completes the proof. 
\end{proof}

We will now prove Prop \ref{prop:choices2}. 

\begin{proof}[Proof (of Prop. \ref{prop:choices2})]
We will show that we can pick $u, J_1, J_2$ such that \eqref{choices:modell} is satisfied in addition to  \eqref{choices:are-integral}-\eqref{choices:sq-of-units-at-pf}. 
The proof is almost the same as that of Prop. \ref{prop:choices1}
with some minor modifications which we now describe. 
First, we pick as before the prime ideals 
$\fq_{B_1 \smallsetminus B_2}$, $\fq_{B_2 \smallsetminus B_1}$,
$\fq_{B_1\cap B_2}$ and the elements $\alpha_{B_1 \smallsetminus B_2}$,
$\alpha_{B_2 \smallsetminus B_1}$, $\alpha_{B_1 \cap B_2}$.
Using Lemma  \ref{lem:choosing-q-two}, we can ensure that 
for $\fq = \fq_{B_1 \smallsetminus B_2}$ and $\fq_{B_2 \smallsetminus B_1}$, we have 
\[
\N\fq \not \equiv 0,\pm1 \pmod \ell.
\]
Fix a prime $\fl_1$ of $F$ above $\ell$. Let 
$\fR$ be as before and then 
 using the approximation theorem, pick $\alpha \in F^\times $ satisfying 
the properties \eqref{cond:alpha>} through \eqref{cond:alpha2} of the proof of Prop. \ref{prop:choices1}
and the following additional conditions:
\begin{enumerate}[label=(\Roman*), ref=\Roman*]
\setcounter{enumi}{4}
\item $-\alpha \equiv w^2 \pmod {\fl_1}$, where $w \in \F_{\fl_1}^\times$ is an element such that $\N_{\F_{\fl_1}/\F_\ell} (w^2) \neq \pm 1$. 
\item $-\alpha \equiv 1 \pmod {\fl}$ for all primes $\fl \neq \fl_1$ dividing $\ell$. 
\end{enumerate}
Note that this uses the assumption that $\ell$ is unramified in $F$ and $\ell >5$. 
Next, as before, we pick $\fm$, $\beta$ and $\fQ$ and set:
\[
u:= - \alpha^{-1} \beta, \quad J_1:= \alpha_{B_1 \smallsetminus B_2} \cdot \alpha_{B_1 \cap B_2}, \quad J_2:= \alpha _{B_2 \smallsetminus B_1} \cdot \alpha_{B_1 \cap B_2}.
\]
It is easy to see (using arguments similar to those of Prop. \ref{prop:choices1}) that  \eqref{choices:are-integral}-\eqref{choices:sq-of-units-at-pf} hold. We verify \eqref{choices:modell} now.

The field $E$ is only ramified at $\fQ$. Also,
\[ 
\N \fQ = \pm \N (\beta) \cdot \N (\alpha^{-1} )\equiv \pm \N (\alpha^{-1}) \equiv \pm \N_{\F_{\fl_1}/\F_\ell} (w^{-2})  \not \equiv 0,\pm 1 \pmod \ell,
\]
which proves what we need for the field $E$.
Now let us consider the field $K$. Since $J=J_1 J_2$, we have
\begin{align*}
(J)&=  \fd_{B_1} \cdot \fq_{B_1\smallsetminus B_2} \cdot \fq_{B_1 \cap B_2} \cdot \fd_{B_2} \cdot \fq_{B_2\smallsetminus B_1} \cdot \fq_{B_1 \cap B_2}\\
&= \fd_{B_1 \smallsetminus B_2} \cdot \fd_{B_2 \smallsetminus B_1} \cdot \fq_{B_1\smallsetminus B_2} \cdot \fq_{B_2 \smallsetminus B_1} \cdot \fd_{B_1 \cap B_2}^2 \cdot \fq_{B_1 \cap B_2}^2.
\end{align*}
We claim that $K$
 is ramified exactly at the primes dividing
\[
\fd_K:=\fd_{B_1 \smallsetminus B_2} \cdot \fd_{B_2 \smallsetminus B_1} \cdot \fq_{B_1\smallsetminus B_2} \cdot \fq_{B_2 \smallsetminus B_1}.
\]
This will follow if we show that $K/F$ is unramified 
at the primes $\fq$ over $2$ that do not divide $\fd_{B_1 \smallsetminus B_2} \cdot \fd_{B_2 \smallsetminus B_1}$.
We claim that $J$ is a square (and hence $K$ is split) at such primes $\fq$.  Indeed, if a prime $\fq$ above $2$ does not divide $\fd_{B_1 \cup B_2}$, then $\alpha_{B_1 \smallsetminus B_2}$, $\alpha_{B_2 \smallsetminus B_1}$ and $\alpha_{B_1 \cap B_2}$ are all squares of units at $\fq$, hence $J_1$, $J_2$ and $J$ are squares at $\fq$. On the other hand, if a prime $\fq$ above $2$ divides $\fd_{B_1 \cap B_2}$, then $\alpha_{B_1 \smallsetminus B_2}$, $\alpha_{B_2 \smallsetminus B_1}$
are squares of units at $\fq$ and thus
\[
J= \alpha_{B_1 \smallsetminus B_2} \cdot \alpha_{B_2 \smallsetminus B_1} \cdot \alpha_{B_1 \cap B_2}^2
\]
is a square at $\fq$. 

Thus it suffices to consider the primes $\fq$ dividing $\fd_K$. For 
$\fq$ dividing either $\fd_{B_1 \smallsetminus B_2}$ or $\fd_{B_2 \smallsetminus B_1}$, it follows 
from the assumptions in the statement of the proposition that $\N \fq \not \equiv 0, \pm 1 \pmod \ell$. 
For $\fq$ equal to either  $\fq_{B_1\smallsetminus B_2}$ or  $\fq_{B_2 \smallsetminus B_1}$, the same 
follows from the choice of these prime ideals. \end{proof}

\subsection{The main conjecture}
\label{ssec:mainconjecture}

Finally, in this section we come to the main conjecture. 
Our starting data will be the totally real field $F$, the automorphic representation $\Pi$ (from the introduction) of conductor $\fN= \fN_s \cdot \fN_{\ps}$ and the 
two quaternion algebras $B_1$ and $B_2$. We assume that $\Pi$ admits a Jacquet--Langlands transfer to
both $B_1$ and $B_2$. We also assume the following condition holds:
\begin{itemize}
\item $\fN$ is prime to $2 \fD_{F/\Q}$, where $\fD_{F/\Q}$ denotes the {\it different} of $F/\Q$. 
\end{itemize}

The conjecture will in addition depend on several
auxiliary choices which we now make completely explicit in the following series of steps.

 \begin{enumerate}
 \item Let $\ell$ be a rational prime such that $(\ell,N(\Pi))=1$. 
 \item
Pick $u$, $J_1$ and $J_2$  satisfying all the conditions 
of Prop. \ref{prop:choices1}, taking 
$\{ \ff_1, \ldots, \ff_n \}$ to be the set of primes of $F$ dividing $2\fD_{F/\Q}$.

\item
Set $E=F+F\i$ where $\i^2 = u$. An explicit model for $B_i$, $i=1,2,$ is $B_i = E+ E\j_i$ where 
$\j_i^2= J_i$ and $\alpha \j_i = \j_i \alpha^\rho$ for $\alpha\in E$. Likewise an explicit model for $B=B_1\cdot B_2 $ is 
$B= E+E\j$ where $\j^2 = J:=J_1J_2$. 

\item Let $B'$ denote any one of the quaternion algebras $B$, $B_1$ or $B_2$. Let $\fd_{B'}$ denote the discriminant of ${B'}$ and define $\fN_{B'}$ by 
\[
\fN = \fd_{B'} \cdot \fN_{B'}. 
\]
Thus $\fd_{B'}$ divides $\fN_\s$, and $\fN_{\ps}$ divides $\fN_{B'}$.

\item Given the choices of $u$, $J_1$ and $J_2$, in Sec. \ref{ssec:schwartzX'} we have picked local maximal orders and {\it oriented} Eichler orders of 
level $\fN_{B'}$.  (Here the 
orientation is only picked at places dividing $\fN_\ps$.)
Let $\cO_{B'}$ (resp. $\cO_{B'} (\fN_{B'})$) denote the unique maximal order 
(resp. Eichler order of level $\fN_{B'}$) corresponding to these choices. 
Also denote by
$o_{B'}$ the corresponding orientation on $\cO_{B'} (\fN_{B'})$.
This defines open compact subgroups $\cK^{B'} =\prod_v \cK_v^{B'}$ and 
$\cKt^{B'} =\prod_v \cKt_v^{B'}$  of ${B'}^\times(\A_f)$
with 
\[
\cK_v^{B'} = \ker \left[ o_{B',v}: ( \cO_{B'} (\fN_{B'}) \otimes_{\cO_F} \cO_{F,v})^\times \rightarrow (\cO_{F,v}/\fN_\ps \cO_{F,v})^\times\right]
\]
and
\[
\cKt_v^{B'}=( \cO_{B'} (\fN_{B'}) \otimes_{\cO_F} \cO_{F,v})^\times.
\]
For future use, we record that with our choices of $u$, $J_1$ and $J_2$ and local orders, we have
\[
\cO_E \otimes \Z_{(\ell)} \subseteq \cO_{B'} \otimes \Z_{(\ell)}.
\]

\item As in \S \ref{ssec:defn-period-invariant}, we pick a large enough number field $L$ and isomorphisms $\phi_B$, $\phi_{B_1}$ and $\phi_{B_2}$ satisfying \eqref{eqn:isomBinfty-integral1}.  
We recall that $\phi_{B'}$ gives an isomorphism:
\[
\cO_{B'} \otimes \cO_{L,(\ell)} \simeq \prod_{\sigma} \M_2 (\cO_{L,(\ell)})
\]
for $B'=B$, $B_1$ and $B_2$. 
As explained in  \S \ref{ssec:defn-period-invariant}, this choice defines 
automorphic vector bundles on $\Sh_{\cK^B}(G_B,X_B)$, $\Sh_{\cK^{B_1}}(G_{B_1},X_{B_1})$ and $\Sh_{ \cK^{B_2}}(G_{B_2},X_{B_2})$
as well as sections 
 $s_B$, $s_{B_1}$ and $s_{B_2}$ of these bundles that are $\ell$-normalized. 
For $B'=B,B_1$ or $B_2$ let us denote the corresponding bundle by 
$\cV_{\uk_{B'},r}^{B'}$ to indicate that is a bundle on $X_{B'}$.

\item So far, we have not had to pick a base point of $X_{B'}$ but we will now need to do so. 
Fix isomorphisms $E\otimes_{F,\sigma} \R \simeq \C$ for all infinite places $\sigma$ of $F$ as in  \S \ref{ssec:schwartzX'}.
For $B'=B,B_1,B_2$ we define $h_{B'}$ as follows: take the composite maps
\[
\C \rightarrow \prod_{\sigma \in\Sigma_\infty} \C \simeq E\otimes \R \rightarrow B'\otimes \R
\]
where the first map sends
\[
z \mapsto  (z_\sigma)_\sigma 
\]
with $z_\sigma=z$ if $\sigma$ is split in $B'$ and $z_\sigma=1$ is $\sigma$ is ramified in $B'$. 

\item For $B'=B,B_1$ or $B_2$, 
let $F_{B'} = \Lift_{h_{B'}} (s_{B'})$.

\item For each $\sigma$, we pick a vector $v^{B'}_{\sigma,k_{B',\sigma}} \in V_{\sigma,k_{B',\sigma},r}$ 
satisfying \eqref{eqn:choice-of-vksigma} 
and that is integrally normalized with respect to the $\ell$-integral structure given by 
\[
\Sym^{k_{B',\sigma}} \cO_{L,(\ell)}^2 
\otimes \det(\cO_{L,(\ell)}^2)^{\frac{r-k_{B',\sigma}}{2}}.
\]

\item Let $v^{B'}_{\uk_{B'}} = \otimes_\sigma v^{B'}_{\sigma, k_{B',\sigma}}$. 
Now we can define $\phi_{F_{B'}} = (F_{B'}(g), v^{B'}_{\uk_{B'}})$. 
Let $f_{B'}$ denote the corresponding element of $\pi_{B'}$:
\[
f_{B'} (g) =  \phi_{F_{B'}}  (g) \cdot \nu_{B'} (g)^{-r/2}.
\]
Then $f_{B'} $ is a new-vector as defined in \S \ref{ssec:explicit-rallis}, but 
is now moreover integrally normalized at $\ell$. 
\end{enumerate}

From Defn. \ref{defn:pet-norm}, Prop. \ref{prop:s-to-F} and Prop. \ref{prop:F-to-phi}, we find that 
\begin{align*}
\llangle s_{B'}, s_{B'} \rrangle_{\cKt^{B'}} & = 2^{d_{B'}} h_F [\cK_0^{B'} : \cKt^{B'}] \cdot \langle F_{B'},F_{B'} \rangle_{h_{B'}} \\
& = 2^{d_{B'}} h_F [\cK_0^{B'} : \cKt^{B'}] \cdot 
\frac{\rank \cV^{B'}_{\uk_{B'},r}}{\langle v^{B'}_{\uk_{B'}}, v^{B'}_{\uk_{B'}} \rangle_{h_{B'}} } \cdot \langle f_{B'},f_{B'} \rangle,
\end{align*}
where $d_{B'}$ is the number of infinite places of $F$ where $B'$ is split, and $\cK_0^{B'}$ is the maximal open compact subgroup of ${B'}^\times (\A_f)$ defined by $\cK_0^{B'} = \prod_v \cK_{0,v}^{B'}$, with $\cK_{0,v}^{B'} = (\cO_{B'} \otimes \cO_{F,v})^\times$. 

In order to state the main conjecture, we will need to renormalize the measure and the Schwartz function in the definition of the theta lift. 
First we renormalize the Schwartz function. Let 
$\varphi_v$ be the Schwartz function on $\SS (\X_v)$ defined in \S \ref{ssec:schwartzX}.
We set $\bv = \otimes_v \bv_v$ where 
\[
\bv_v = \sqrt{C_v^\varphi} \cdot \varphi_v
\]
with
\[
C_v^\varphi = \begin{cases}
1 & \text{if $\Pi_v$ is unramified principal series}, \\
\dfrac{q_v-1}{q_v^{2n_v+1}} & \text{if $\Pi_v$ is ramified principal series with conductor $q_v^{n_v}$},\\
q_v^{-2} & \text{if $\Pi_v$ is special},\\
\dfrac{k_v !}{2^{k_v}\pi^{k_v}} & \text{if } v\in \Sigma_\infty \smallsetminus (\Sigma_{B,\infty} \cup \Sigma_{B_1,\infty} \cup 
\Sigma_{B_2,\infty}),\\
\dfrac{(k_v-2)!}{2^{k_v-2} \pi^{k_v-2}} &  \text{if } v\in \Sigma_{B_1,\infty}\cap \Sigma_{B_2,\infty}\\
\dfrac{(k_v-2)!}{2^{k_v-4} \pi^{k_v-2}}  &  \text{if } v\in \Sigma_{B,\infty}.
\end{cases}
\]
As for the measure used in the theta lift \eqref{eq:jls_theta_lift_def} we renormalize the measure 
on $B^{(1)}(\A)$ to 
\[
 [\cK^{B}_0 : \cKt^B] \cdot \rank \cV^{B}_{\uk_{B},r} \cdot \text{standard measure on }  B^{(1)}(\A).
\]
With this choice of measure, let us define $\ba (B_1,B_2)$ by 
\begin{equation}
\label{eqn:alpha-defn}
\theta_{\bv} (f_B) = \ba (B_1,B_2) \cdot (f_{B_1} \times f_{B_2}).
\end{equation}

\begin{thm} Suppose $B\neq \M_2 (F)$. Then 
\begin{equation}
\label{eqn:exp-rallis-cons}
|\ba (B_1,B_2)|^2 \cdot \llangle s_{B_1}, s_{B_1} \rrangle_{\cKt^{B_1}} \cdot \llangle s_{B_2}, s_{B_2} \rrangle_{\cKt^{B_2}} =
\Lambda(1,\Pi,\ad) \cdot \llangle s_B, s_B \rrangle_{\cKt^{B}} \quad \text{ in } \ \C^\times / R_{(\ell)}^\times.
\end{equation}
\end{thm}

\begin{proof}
Recall that $\Lambda(s, \Pi, \ad) = \prod_{v \in \Sigma_\infty}L(s, \Pi_v, \ad) \cdot L(s, \Pi, \ad)$ with 
\[
 L(s, \Pi_v, \ad) = \Gamma_\R(s+1) \Gamma_\C(s+k_v-1),
\]
where 
$\Gamma_\R(s) = \pi^{-\frac{s}{2}} \Gamma(\frac{s}{2})$ and $\Gamma_\C(s) = 2 (2\pi)^{-s} \Gamma(s)$.
Since 
\[
 L(1, \Pi_v, \ad) = \frac{(k_v-1)!}{2^{k_v-1} \pi^{k_v+1}}
\]
for $v \in \Sigma_\infty$, it follows from Proposition \ref{prop:<f,f>} that
\[
 \langle f, f \rangle = 2 |D_F| \cdot \prod_v C_v^\Lambda \cdot \Lambda(1, \Pi, \ad), 
\]
where the constant $C_v^\Lambda$ is defined by the table below.
We also define a constant $C_v^{B'}$ for $B' = B_1,B_2,B$ by the table, so that the following hold:
\begin{align*}
 \llangle s_{B'}, s_{B'} \rrangle_{\cKt^{B'}} & = 2^{d_{B'}} h_F \cdot \prod_v C_v^{B'} \cdot 
 \frac{\langle f_{B'}, f_{B'} \rangle}{\langle v^{B'}_{\uk_{B'}}, v^{B'}_{\uk_{B'}} \rangle_{h_{B'}}}, \\
 |\ba (B_1,B_2)|^2 & = \prod_v (C_v^B)^2 C_v^\varphi \cdot |\alpha(B_1, B_2)|^2.
\end{align*}
\[
{\renewcommand{\arraystretch}{1.2}
\begin{array}{|c|c|c|c||c|c|c|c|c|c|} \hline 
 \Pi_v & B_{1,v} & B_{2,v} & B_v & C_v^{B_1} & C_v^{B_2} & C_v^B & C_v^\Lambda & C_v^\varphi & C_v \\ \hline \hline
 \text{ur} & \text{spl} & \text{spl} & \text{spl} & 1 & 1 & 1 & 1 & 1 & 1 \\ \hline
 \text{rps} & \text{spl} & \text{spl} & \text{spl} & q_v^{n_v-1}(q_v+1) & q_v^{n_v-1}(q_v+1) & q_v^{n_v-1}(q_v+1) & \frac{q_v}{q_v+1} & \frac{q_v-1}{q_v^{2n_v+1}} & \frac{1}{q_v^{n_v-3}(q_v-1)(q_v+1)^2} \\ \hline 
 \text{st} & \text{spl} & \text{spl} & \text{spl} & q_v+1 & q_v+1 & q_v+1 & \frac{q_v}{q_v+1} & q_v^{-2} & \frac{q_v}{(q_v+1)^2} \\ \hline
 \text{st} & \text{ram} & \text{ram} & \text{spl} & 1 & 1 & q_v+1 & \frac{q_v}{q_v+1} & q_v^{-2} & q_v \\ \hline
 \text{st} & \text{ram} & \text{spl} & \text{ram} & 1 & q_v+1 & 1 & \frac{q_v}{q_v+1} & q_v^{-2} & q_v \\ \hline
 \text{st} & \text{spl} & \text{ram} & \text{ram} & q_v+1 & 1 & 1 & \frac{q_v}{q_v+1} & q_v^{-2} & q_v \\ \hline
 \text{ds} & \text{spl} & \text{spl} & \text{spl} & 1 & 1 & 1 & \frac{1}{2^{k_v+2}} & \frac{k_v !}{2^{k_v}\pi^{k_v}} & \frac{2^{2 k_v + 2} \pi^{k_v}}{k_v!} \\ \hline
 \text{ds} & \text{ram} & \text{ram} & \text{spl} &
 k_v-1 & k_v-1 & 1 & \frac{1}{2^{k_v+2}} & \frac{(k_v-2)!}{2^{k_v-2} \pi^{k_v-2}} & \frac{2^{2 k_v} \pi^{k_v-2}}{(k_v-1)^2 \cdot (k_v - 2)!} \\ \hline 
 \text{ds} & \text{ram} & \text{spl} & \text{ram} &
 k_v-1 & 1 & k_v-1 & \frac{1}{2^{k_v+2}} & \frac{(k_v-2)!}{2^{k_v-4} \pi^{k_v-2}} & \frac{2^{2 k_v -2} \pi^{k_v-2}}{(k_v-1)^2 \cdot (k_v - 2)!} \\ \hline
 \text{ds} & \text{spl} & \text{ram} & \text{ram} &
 1 & k_v-1 & k_v-1 & \frac{1}{2^{k_v+2}} & \frac{(k_v-2)!}{2^{k_v-4} \pi^{k_v-2}} & \frac{2^{2 k_v -2} \pi^{k_v-2}}{(k_v-1)^2 \cdot (k_v - 2)!} \\ \hline
\end{array}
} 
\]

\noindent
By Theorem \ref{thm:rallis-B-explicit}, we have
\[
 |\alpha(B_1,B_2)|^2
 \cdot \langle f_{B_1}, f_{B_1} \rangle \cdot \langle f_{B_2}, f_{B_2} \rangle
 = |D_F|^2 \cdot \prod_v C_v \cdot \langle f_B, f_B \rangle \cdot \langle f, f \rangle
\]
with the constant $C_v$ above, and hence 
\begin{multline*}
 \frac{|\ba (B_1,B_2)|^2}{\prod_v (C_v^B)^2 C_v^\varphi} \cdot 
 \frac{\langle v^{B_1}_{\uk_{B_1}}, v^{B_1}_{\uk_{B_1}} \rangle_{h_{B_1}} \cdot \llangle s_{B_1}, s_{B_1} \rrangle_{\cKt^{B_1}}}{2^{d_{B_1}} h_F \cdot \prod_v C_v^{B_1}} \cdot 
 \frac{\langle v^{B_2}_{\uk_{B_2}}, v^{B_2}_{\uk_{B_2}} \rangle_{h_{B_2}} \cdot \llangle s_{B_2}, s_{B_2} \rrangle_{\cKt^{B_2}}}{2^{d_{B_2}} h_F \cdot \prod_v C_v^{B_2}} \\
 = |D_F|^2 \cdot \prod_v C_v \cdot 
 \frac{\langle v^{B}_{\uk_{B}}, v^{B}_{\uk_{B}} \rangle_{h_{B}} \cdot \llangle s_{B}, s_{B} \rrangle_{\cKt^{B}}}{2^{d_B} h_F \cdot \prod_v C_v^{B}} 
 \cdot 2 |D_F| \cdot \prod_v C_v^\Lambda \cdot \Lambda(1, \Pi, \ad).
\end{multline*}
Now the theorem follows from this and the fact that
\[
 C_v^{B_1} C_v^{B_2} C_v^B C_v^\Lambda C_v^\varphi C_v = 1
\]
for all $v$.
\end{proof}

We now motivate the main conjecture of this paper. Let us set
\[
\Lambda(\Pi):= \Lambda(1,\Pi,\ad).
\]
Thus from \eqref{eqn:exp-rallis-cons}, we see that for $B_1\neq B_2$, we have
\begin{equation}
\label{eqn:me}
|\ba(B_1,B_2)|^2 \cdot q_{B_1} (\Pi,\ell) \cdot q_{B_2}(\Pi,\ell) = \Lambda(\Pi) \cdot q_{B} (\Pi,\ell) \quad  \text{in} \quad \C^\times/ R_{(\ell)}^\times,
\end{equation}
and consequently, 
\begin{equation}
\label{eqn:me-ell}
|\ba (B_1,B_2)|^2 \cdot q_{B_1} (\Pi) \cdot q_{B_2}(\Pi) = \Lambda (\Pi) \cdot q_{B} (\Pi) \quad \text{in} \quad \C^\times/ R_{(\ell)}^\times. 
\end{equation}

If we combine this with Conjecture \ref{conj:oldintro}(ii) of the introduction, we are lead to the 
following conjectural expression for $|\ba(B_1,B_2)|^2$:
\[
|\ba (B_1,B_2)|^2 \stackrel{?}{=} \frac{\Lambda(\Pi) \cdot \dfrac{\Lambda(\Pi)}{\prod_{v\in \Sigma_B} c_v (\Pi) }}{\dfrac{\Lambda(\Pi)}{\prod_{v\in \Sigma_{B_1}} c_v(\Pi) }\cdot \dfrac{\Lambda(\Pi)}{\prod_{v\in \Sigma_{B_2}} c_v(\Pi) }}
= \left[ \prod_{v\in \Sigma_{B_1} \cap \Sigma_{B_2}} c_v (\Pi) \right]^2  \text{in} \quad \C^\times/ R_{(\ell)}^\times. 
\]

Combining this last expression with Conjecture \ref{conj:oldintro}(i) suggests Conjecture \ref{conj:newintro} of 
the introduction on the arithmetic nature of the constants $\ba (B_1,B_2)$. We restate it below for the convenience of the reader. 

\begin{conjecture}
\label{conj:main}
Suppose that $B_1\neq B_2$ and $\Sigma_{B_1} \cap \Sigma_{B_2} \cap \Sigma_\infty = \varnothing$, that is $B_1$ and $B_2$ have no infinite places of ramification in common. Then
\begin{enumerate}
\item $\ba (B_1,B_2)$ lies in $\Qbar^\times$.
\item Moreover, $\ba (B_1,B_2)$ belongs to $R_{(\ell)}$. 
\item If in addition $B_1$ and $B_2$ have no finite places of ramification in common, then $\ba (B_1,B_2)$
lies in $R_{(\ell)}^\times$. 
\end{enumerate}
\end{conjecture}

Note that this conjecture makes absolutely no reference to the constants $c_v(\Pi)$. However, we shall show 
now that the truth of this conjecture (for all $\ell$ prime to $N(\Pi)$) implies the truth of Conj. \ref{conj:oldintro}.

\begin{thm}
\label{thm:conjA-equiv-conjD}
Suppose that Conj. \ref{conj:main}  is true for all $\ell$ prime to $N(\Pi)$. Then 
  Conj. \ref{conj:oldintro} is true. 
 \end{thm} 

\begin{rem}
The proof below will show that the validity of Conj. \ref{conj:main} for a single $\ell$ implies 
a version of Conj. \ref{conj:oldintro} with $R^\times$ replaced by $R_{(\ell)}^\times$. 
\end{rem}

\begin{proof}
Recall from Remark \ref{rem:conj-in-split-case} of the introduction that 
\begin{equation}
\label{eqn:conj-true-for-split-case}
q_{\M_2(F)} = \Lambda (\Pi) \quad \text{in} \ \C^\times/R^\times.
\end{equation}
Note that if $|\Sigma_\Pi| =0 $ or $1$, then 
$\Pi$ does not transfer to any non-split quaternion algebra,
so the conjecture follows from \eqref{eqn:conj-true-for-split-case}.

 If $|\Sigma_\Pi|=2$, say $\Sigma_\Pi=\{ v,w\}$, then there 
is a unique non-split quaternion algebra $B$ with $\Sigma_B \subseteq \Sigma_\Pi$,
given by $\Sigma_B = \Sigma_\Pi$. In this case, we 
need to pick two elements $c_v (\Pi)$ and $c_w (\Pi)$ in $\C^\times / R^\times$ such that the relation
\[
q_B (\Pi) = \frac{\Lambda(\Pi)}{c_v (\Pi) \cdot c_w (\Pi)}
\]
is satisfied (in addition to \eqref{eqn:conj-true-for-split-case}), and there are obviously many ways to do this. 
Since at most one of the places in $\Sigma_\Pi$ (say $v$) is a finite place, we can also make this choice 
so that $c_v(\Pi)$ lies in $R$. 
Note that 
in this case, the invariants $c_v (\Pi)$ and $c_w (\Pi)$ are not 
uniquely determined by the single relation above, so in order 
to get canonical invariants one would need 
to rigidify the choices by imposing other constraints on them. 
We do not pursue this here. 

Thus we may assume that $|\Sigma_\Pi| \ge 3$. We need to first define the constants $c_v (\Pi)$ in this case. 
First, for any subset $\Sigma \subseteq \Sigma_\Pi$ of {\it even} cardinality let us define 
$c_\Sigma (\Pi) \in \C^\times / R^\times$ by 
\[
c_\Sigma (\Pi)  := \frac{\Lambda (\Pi)}{q_{B_\Sigma} (\Pi)},
\]
where $B_\Sigma$ denotes the unique quaternion algebra ramified exactly at $\Sigma$. 
Note that from \eqref{eqn:conj-true-for-split-case}, we have
\begin{equation}
\label{eqn:c-empty=1}
c_{\varnothing} (\Pi) =1 \quad \text{in} \ \C^\times/R^\times.
\end{equation}
Now let $v$ be any element in $\Sigma_\Pi$. 
We will define $c_v (\Pi)$ as follows. 
 Pick any two other elements $u,w \in \Sigma_\Pi$ and 
define $c_v(\Pi) $ to be the unique element in $\C^\times/ R^\times$ such that 
\begin{equation}
\label{eqn:cv-defn}
c_v(\Pi) ^2 = \frac{c_{\{v,u\}  } (\Pi)  \cdot c_{\{v,w\}}(\Pi) }{c_{\{u,w\}}(\Pi) }.
\end{equation}

We will show that the truth of Conj.  \ref{conj:main} for a single $\ell$ implies that  
$c_v(\Pi)$ is well defined in $\C^\times / R_{(\ell)}^\times$, that it lies in $R_{(\ell)}$ 
if $v$ is a finite place and that 
the relation 
\begin{equation}
\label{eqn:mainl}
q_B (\Pi) = \frac{\Lambda (\Pi)}{\prod_{v \in \Sigma_B} c_v (\Pi)} \quad \text{ in } \quad \C^\times/ R_{(\ell)}^\times
\end{equation}
is satisfied. It follows from this that the truth of Conj.  \ref{conj:main} for {\it all} $\ell$ prime to $N(\Pi)$ implies that the 
$c_v(\Pi)$ is well defined in $\C^\times / R^\times$, that it lies in $R$ if $v$ is a finite place and that 
the relation 
\[
q_B (\Pi) = \frac{\Lambda (\Pi)}{\prod_{v \in \Sigma_B} c_v (\Pi)} \quad \text{ in } \quad \C^\times/ R^\times
\]
is satisfied, which would complete the proof of the theorem.

Thus let $\ell$ be any prime not dividing $N(\Pi)$ and let us assume the truth of Conj. \ref{conj:main}
for this fixed $\ell$. 
If $\Sigma_1$ and $\Sigma_2$ are two distinct subsets of $\Sigma_\Pi$ of even cardinality and if $B_1$ and $B_2$ 
are the corresponding quaternion algebras, the relation \eqref{eqn:me-ell}
gives 
\[
|\ba(B_1,B_2)|^2 \cdot c_{\Sigma} (\Pi) = c_{\Sigma_1} (\Pi) \cdot c_{\Sigma_2} (\Pi)   \quad \text{in} \quad  \C^\times /R_{(\ell)}^\times.
\]
If moreover $\Sigma_1$ and $\Sigma_2$ are disjoint, then Conj. \ref{conj:main} implies that 
$\ba (B_1,B_2)$ lies in $R_{(\ell)}^\times$. Thus we get the key multiplicative relation:
\begin{equation}
\label{eqn:cs-mult}
 c_{\Sigma_1} (\Pi) \cdot c_{\Sigma_2} (\Pi) = c_{\Sigma} (\Pi) \quad \text{in} \quad  \C^\times /R_{(\ell)}^\times, \quad \text{if} \quad \Sigma_1 \cap \Sigma_2 = \varnothing,
 \end{equation}
 including the case $\Sigma_1 = \Sigma_2 = \varnothing$ on account of \eqref{eqn:c-empty=1}.
We can use this to check that $c_v (\Pi) $ defined via \eqref{eqn:cv-defn} is independent of the choice 
of $u$ and $w$, viewed as an element in $\C^\times/R_{(\ell)}^\times$. Since the definition is symmetric in $u$ and $w$, it suffices to show that 
it remains invariant under changing $u$ to some other $u'$ distinct from $u$ and $w$. 
However, this follows from the equality
\[
c_{\{v,u\}} (\Pi) \cdot c_{\{u',w\}} (\Pi)  = c_{\{ v,u,u',w \}} (\Pi)  =  c_{\{v,u'\}} (\Pi) \cdot c_{\{u,w \}} (\Pi) \quad \text{in} \quad  \C^\times /R_{(\ell)}^\times,
\]
which is implied by \eqref{eqn:cs-mult}.

Next we check that if $v$ is a finite place, then $c_v (\Pi)$ lies in $R_{(\ell)}$. 
If $B_1$, $B_2$ and $B$ are the quaternion algebras 
with $\Sigma_{B_1} = \{ v,u \}$, $\Sigma_{B_2} = \{ v,w \}$ and $\Sigma_B = \{ u,w \}$, then 
$B=B_1 \cdot B_2$ and 
\[
c_v(\Pi)^2 = \frac{\dfrac{\Lambda(\Pi)}{q_{B_1}(\Pi)}\cdot \dfrac{\Lambda(\Pi)}{q_{B_2}(\Pi)}}{\dfrac{\Lambda(\Pi)}{q_{B}(\Pi)}}
=
\frac{\Lambda(\Pi) \cdot q_B (\Pi)}{q_{B_1}(\Pi) \cdot q_{B_2} (\Pi)} = |\ba(B_1,B_2)|^2 \quad \text{in} \  \C^\times /R_{(\ell)}^\times. 
\]
Since $B_1$ and $B_2$ have no infinite places of ramification in common, it follows 
from (i) and (ii) of Conj. \ref{conj:main} that 
$c_v (\Pi)$ lies in $R_{(\ell)}$.

Finally, let us check that $c_u (\Pi) \cdot c_v (\Pi) = c_{\{u,v \}} (\Pi) $ in $ \C^\times /R_{(\ell)}^\times$ if $u,v$ are distinct elements in $\Sigma_B$. Indeed, picking 
any $w$ distinct from $u$ and $v$, we have 
\[
c_u(\Pi)^2 \cdot c_v(\Pi)^2 =  \frac{c_{\{u,v\}  } (\Pi)  \cdot c_{\{u,w\}}(\Pi) }{c_{\{v,w\}}(\Pi) } \cdot  \frac{c_{\{v,u\}  } (\Pi)  \cdot c_{\{v,w\}}(\Pi) }{c_{\{u,w\}}(\Pi) } = c_{\{u,v\}} (\Pi)^2 \quad \text{in} \quad  \C^\times /R_{(\ell)}^\times,
\]
as claimed. From this, \eqref{eqn:c-empty=1}
 and \eqref{eqn:cs-mult} it follows immediately that for any subset $\Sigma \subseteq \Sigma_\Pi$ of 
even cardinality, we have 
\[
c_\Sigma (\Pi) = \prod_{v\in \Sigma} c_v (\Pi)   \quad \text{in} \quad \C^\times /R_{(\ell)}^\times,
\]
from which \eqref{eqn:mainl} follows immediately. \end{proof}

%% file: hodge-polarization.tex
\section{Polarized Hodge structures, abelian varieties and complex conjugation}
\label{sec:hodge}

In this section we discuss polarizations and the action of complex conjugation
on Hodge structures attached to abelian varieties. This material is completely standard,
so the purpose of this section is simply to carefully fix our conventions and motivate some 
of our constructions in Chapter \ref{sec:qvs}.

If $A$ is a complex abelian variety, there is a natural Hodge structure on 
$\Lambda = H_1(A,\Z)$. If $V= \Lambda \otimes \Q$, we have
\[
V_\C = H_1(A,\C) = V^{-1,0} \oplus V^{0,-1}
\]
where $V^{-1,0} = \Lie (A)$ and $V^{0,-1} =F^0 (V) = \overline{V^{-1,0}}$ is identified with $H^1(A, \cO_A)^\vee$. 
In fact, the exact sequence
\[
0 \rightarrow V^{0,-1} \rightarrow V_\C \rightarrow V^{-1,0} \rightarrow 0
\]
is dual to 
\[
0 \rightarrow H^0 (A,\Omega^1_A) \rightarrow H^1 (A,\C) \rightarrow H^1(A,\cO_A) \rightarrow 0
\]
which describes the Hodge filtration on $H^1 (A,\C)$. 
As a complex torus, $A$ is recovered as 
\[
A = V^{0,-1} \backslash V_\C /\Lambda \simeq V^{-1,0} / \Lambda.
\]
Let $h: \Ss \rightarrow \GL(V_\R)$ be the homomorphism of the Deligne torus into 
$\GL(V_\R)$ corresponding to the Hodge structure on $H_1 (A,\Z)$. Let $C=h(i)$. Recall that 
according to our conventions, the operator $C\otimes 1$ on $V_\R \otimes _{\R} \C$ acts on $V^{-1,0}$ as $i$ 
and on $V^{0,-1}$ as $-i$. We write $F$ for $F^0 V = V^{0,-1}$ so that $\bar{F}=V^{-1,0}=\Lie (A)$. Then the composite maps 
\begin{equation}
\label{eqn:LambdaF}
\Lambda  \otimes \R \rightarrow \Lambda \otimes \C  \rightarrow F, \quad  \Lambda  \otimes \R \rightarrow \Lambda \otimes \C  \rightarrow \bar{F}
\end{equation}
are $\R$-linear isomorphisms.

Let  $\Psi$ be  a skew-symmetric form
\[
\Psi: \Lambda \times \Lambda \rightarrow \Z(1)
\]
whose $\R$-linear extension $\Psi_\R: V_\R \times V_\R \rightarrow \R (1)$ satisfies 
\[
\Psi_\R (Cv,  Cw) = \Psi_\R (v,w).
\]
Define
\[
B:\Lambda \times \Lambda \rightarrow \Z, \quad B(v,w) = \frac{1}{2\pi i }  \Psi (v,w).
\]

\begin{rem}
Note that the discussion up to this point was in fact independent of a choice of $i$. 
However, in the definition of $B$ above and in the sequel, we need to fix such a choice. 
For any element $x+yi \in \C$ let us also set 
\[
\Im(x+yi) = yi, \quad \im (x+yi) = y.
\]
\end{rem}

Let $B_\R$ and $B_\C$ denote the $\R$-linear and $\C$-linear extensions of $B$ to 
$V_\R$ and $V_\C$ respectively. Let $B_C$ denote the hermitian form on $V_\C$ given by 
\[
B_C (v,w) := B_\C (v, C \bar{w}).
\]
Finally, we let $B_F$ and $B_{\bar{F}}$ denote the bilinear forms on $F$ and 
$\bar{F}$ obtained from $B_\R$ via the isomorphisms \eqref{eqn:LambdaF} above.

\begin{prop} The forms $B_\C$ and $B_C$ have the following properties:
\begin{enumerate}
\item The subspaces $F$ and $\bar{F}$ of $V_\C$ are isotropic for $B_\C$.
\item The form $B_C$ pairs $F \times \bar{F}$ to zero. 
\item $2\cdot \im(B_C)|_F = B_F$ and  $2\cdot \im(B_C)|_{\bar{F}} = -B_{\bar{F}}$. 
\end{enumerate}
\end{prop}

\begin{proof}
  For $v,w \in F$, we have 
\[
B_\C (v,w) = B_\C (h(i)v, h(i) w) = B_\C (-iv,-iw) = -B_\C (v,w),
\]
so $F$ is isotropic for $B_\C$. The argument for $\bar{F}$ is similar. Part (ii) follows immediately
from part (i).  
For part (iii), suppose $v,w \in F$. Then 
\begin{align*}
2 \Im (B_C) (v,w) &=  B_C (v,w) - \overline{B_C (v,w)} =  B_\C(v, C\bar{w}) - B_\C (\bar{v}, C w) \\
&= B_\C (v,i\bar{w}) - B_\C (\bar{v}, - iw) = i \left( B_\C (v,\bar{w}) + B_\C (\bar{v},w) \right).
\end{align*}
On the other hand, under the isomorphism $V_\R \simeq F$, the element $v\in F$ corresponds to $v+\bar{v} \in V_\R$. Thus 
\[
B_F (v,w) = B_\R (v+\bar{v}, w+ \bar{w}) = B_\C (v+\bar{v}, w+ \bar{w}) = B_\C (v,\bar{w})+ B_\C (\bar{v},w)
\]
from part (i). This shows that $2 \cdot \im (B_C)|_F = B_F$. The proof for $\bar{F}$ is similar.
\end{proof}

\begin{prop}
The following are equivalent:
\begin{enumerate}
\item The bilinear form $(v,w)\mapsto B_\R (v, Cw)$ on $V_\R$ is positive definite. 
\item The hermitian form $B_C$ on $V_\C$ is positive definite and induces by restriction 
positive definite hermitian forms on both $F$ and $\bar{F}$. 
\end{enumerate}
\end{prop}

\begin{proof}
 Let $v,w\in V_\C$. Suppose $v=v_1+v_2$ and $w=w_1+w_2$ with $v_1,w_1 \in F$ and 
$v_2,w_2 \in \bar{F}$. 
Then 
\begin{align*}
B_C (v,w) &= B_C (v_1+v_2, w_1+w_2) = B_C (v_1,w_1) + B_C (v_2,w_2) \\
&= B_\C (v_1, C\bar{w}_1) + B_\C (v_2, C \bar{w}_2) \\
\end{align*}
so in particular,
\[
B_C (v,v) = B_\C (v_1,C\bar{v}_1 ) + B_\C (v_2, C \bar{v}_2). 
\]
On the other hand,
\begin{align*}
B_\R (v_1+\bar{v}_1, C (v_1 +\bar{v}_1)) &= B_\C ( v_1, C\bar{v}_1 ) + B_\C (\bar{v}_1, Cv_1)  \\
&=  B_\C ( v_1, C\bar{v}_1 ) - B_\C (C^2 \bar{v}_1, Cv_1) \\
&= B_\C ( v_1, C\bar{v}_1 ) - B_\C (C \bar{v}_1, v_1) \\
&= 2 B_\C ( v_1, C\bar{v}_1 ).
\end{align*}
Likewise,
\[
B_\R (v_2 + \bar{v}_2, C(v_2 + \bar{v}_2) )= 2 B_\C ( v_2, C\bar{v}_2 ).
\]
The implication (i) $\iff $ (ii) is clear from this.
\end{proof}

\begin{defn}
We will say that $\Psi$ or $B$ is a polarization if either of the equivalent conditions of the proposition 
above are satisfied.
\end{defn}

\begin{rem}
In the classical theory of complex abelian varieties, one 
considers hermitian forms $H$ on $F$ or $\bar{F}$ whose imaginary part 
$\im \ H$
equals a given skew-symmetric form. A polarization 
corresponds to the choice of a skew-symmetric form such that $H$ is 
either positive or negative definite. 
This can lead to some confusion: note 
for example that the skew-symmetric form $B_F$ 
is the imaginary part of the positive definite form $2 \cdot B_C|_F$,
while the skew-symmetric form $B_{\bar{F}}$ is the imaginary part of 
the negative definite form $-2 \cdot B_C|_{\bar{F}}$. 
 We will always use the form $B_C$ which is positive definite on both $F$ and $\bar{F}$. 
\end{rem}

%% file: weil-GSp.tex
\section{Metaplectic covers of symplectic similitude groups}
\label{sec:weil-gsp}

\subsection{Setup}
\label{ss:weil-GMp-setup}

Let $F$ be a local field of characteristic zero.
Fix a nontrivial additive character $\psi$ of $F$.

Let $\V$ be a $2n$-dimensional symplectic space over $F$.
Let $\GSp(\V)$ and $\Sp(\V) := \ker \nu$ be the similitude group 
and the symplectic group of $\V$ respectively,
where $\nu: \GSp(\V) \rightarrow F^{\times}$ is the similitude character.

Fix a complete polarization $\V = \X \oplus \Y$.
Choose a basis $\e_1, \ldots, \e_n, \e_1^*, \ldots, \e_n^*$ of $\V$ such that
\[
 \X = F \e_1 + \cdots + F \e_n, \qquad
 \Y = F \e_1^* + \cdots + F \e_n^*, \qquad
 \llangle \e_i, \e_j^* \rrangle = \delta_{ij}.
\]
Using this basis, we may write
\[
 \GSp(\V) = \left\{ g \in \GL_{2n}(F) \, \left| \, g \begin{pmatrix} & \1_n \\ - \1_n & \end{pmatrix} {}^t g = \nu(g) \cdot \begin{pmatrix} & \1_n \\ - \1_n & \end{pmatrix} \right. \right\}.
\]
For $\nu \in F^{\times}$, we define $d(\nu) = d_{\Y}(\nu) \in \GSp(\V)$ by
\[
 d(\nu) := \begin{pmatrix} \1_n & \\ & \nu \cdot \1_n \end{pmatrix}.
\]
Let $P = P_{\Y}$ be the maximal parabolic subgroup of $\Sp(\V)$ stabilizing $\Y$:
\[
 P = \left\{ \left. \begin{pmatrix} a & * \\ & {}^t a^{-1} \end{pmatrix} \, \right| \, a \in \GL_n(F) \right\}.
\]
We have a Bruhat decomposition
\[
 \Sp(\V) = \coprod_{j=0}^n P \tau_j P,
\]
where
\[
 \tau_j := \begin{pmatrix} \1_{n-j} & & & \\ & & & -\1_j \\ & & \1_{n-j} & \\ & \1_j & & \end{pmatrix}.
\]
For $h \in \Sp(\V)$, put $j(h) := j$ if $h \in P \tau_j P$.
We define a map
\[
 x : \Sp(\V) \longrightarrow F^{\times} / (F^{\times})^2
\]
by
\[
 x(p_1 \tau_j p_2) := \det(a_1 a_2) \bmod (F^{\times})^2,
\]
where 
\[
 p_i = \begin{pmatrix} a_i & * \\ & {}^t a_i^{-1} \end{pmatrix} \in P.
\]
In particular, we have $x(p_1hp_2) = x(p_1) x(h) x(p_2)$
for $p_1, p_2 \in P$ and $h \in \Sp(\V)$.

Let $z_{\Y} = z_{\Y}^{\Sp}$ be the $2$-cocycle given by
\[
 z_{\Y}(h_1, h_2) := \gamma_F(\frac{1}{2} \psi \circ q(\Y, \Y h_2^{-1}, \Y h_1))
\]
for $h_1, h_2 \in \Sp(\V)$.

\begin{lem}
We have
\begin{itemize}
 \item $z_{\Y}(h, h^{-1}) = 1$ for $h \in \Sp(\V)$,
 \item $z_{\Y}(p_1 h_1 p, p^{-1} h_2 p_2) = z_{\Y}(h_1, h_2)$
 for $p, p_i \in P$ and $h_i \in \Sp(\V)$,
 \item $z_{\Y}(\tau_i, \tau_j) = 1$,
 \item $z_{\Y}(\tau_n, \n(\beta) \tau_n) = \gamma_F(\frac{1}{2} \psi \circ q_{\beta})$ for $\n(\beta) = \left( \begin{smallmatrix} \1_n & \beta \\ & \1_n \end{smallmatrix} \right)$ with $\beta \in \Hom(\X, \Y)$ if $q_{\beta}$ is non-degenerate, where $q_{\beta}$ is a symmetric bilinear form on $\X$ defined by $q_{\beta}(x,y) = \llangle x, y \beta \rrangle$.
\end{itemize}
\end{lem}

\begin{proof}
See \cite[Theorem 4.1, Corollary 4.2]{rangarao}.
\end{proof}

Suppose that $\V = \V_1 \oplus \V_2$, where each $\V_i$ is a non-degenerate symplectic subspace over $F$.
If $\V_i = \X_i \oplus \Y_i$ is a complete polarization and 
\[
 \X = \X_1 \oplus \X_2, \qquad \Y = \Y_1 \oplus \Y_2,
\]
then we have
\[
 z_{\Y_1}(h_1, h_1') \cdot  z_{\Y_2}(h_2, h_2') = z_{\Y}(h_1 h_2, h_1' h_2')
\]
for $h_i, h_i' \in \Sp(\V_i)$ (see Theorem 4.1 of \cite{rangarao}).

\subsection{Action of outer automorphisms on the $2$-cocycle}
\label{ss:weil-gsp-outer}

For $\nu \in F^{\times}$, let $\alpha_{\nu} = \alpha_{\Y, \nu}$ be the outer automorphism of $\Sp(\V)$ given by
\[
 \alpha_{\nu}(h) = d(\nu) \cdot h \cdot d(\nu)^{-1}
\]
for $h \in \Sp(\V)$.
This induces an action of $F^{\times}$ on $\Sp(\V)$ and thus we have an isomorphism
\begin{align*}
 \Sp(\V) \rtimes F^{\times} & \longrightarrow \GSp(\V). \\
 (h, \nu) & \longmapsto (h, \nu)_{\Y} := h \cdot d(\nu)
\end{align*}
Note that
\[
 (h, \nu)_{\Y} \cdot (h', \nu')_{\Y} = (h \cdot \alpha_{\nu}(h'), \nu \cdot \nu')_{\Y}.
\]

There exists a unique automorphism $\tilde{\alpha}_{\nu}$ of $\Mp(\V)$ such that $\tilde{\alpha}_{\nu}|_{\C^1} = \id_{\C^1}$ and the diagram
\[
 \xymatrix{
  \Mp(\V) \ar@{->}[r]^{\tilde{\alpha}_{\nu}} \ar@{->}[d] &
  \Mp(\V) \ar@{->}[d] \\
  \Sp(\V) \ar@{->}[r]^{\alpha_{\nu}} & \Sp(\V)
 }
\]
commutes.
This implies that there exists a unique function
\[
 v_{\Y} : \Sp(\V) \times F^{\times} \longrightarrow \C^1
\]
such that
\[
 \tilde{\alpha}_{\nu}(h, z) = (\alpha_{\nu}(h), z \cdot v_{\Y}(h,\nu))
\]
for $(h, z) \in \Mp(\V)_{\Y}$.
Since $\tilde{\alpha}_{\nu}$ is an automorphism, we have
\[
 z_{\Y}(\alpha_{\nu}(h), \alpha_{\nu}(h'))
 = z_{\Y}(h, h') \cdot v_{\Y}(hh', \nu)
 \cdot v_{\Y}(h, \nu)^{-1} \cdot v_{\Y}(h', \nu)^{-1}
\]
for $h, h' \in \Sp(\V)$ and $\nu \in F^{\times}$.

\begin{lem}
\label{lem:v_Y}
We have
\[
 v_{\Y}(h, \nu) = (x(h), \nu)_F \cdot \gamma_F(\nu, \frac{1}{2} \psi)^{-j(h)}
\]
for $h \in \Sp(\V)$ and $\nu \in F^{\times}$.
\end{lem}

\begin{proof}
See \cite[Proposition 1.2.A]{barthel}.
For convenience, we recall the proof of \cite[Proposition 1.2.A]{barthel}.
{\bf Warning: our convention differs from that in \cite{barthel}.}

Note that $z_{\Y}(p,h) = z_{\Y}(h, p) = 1$ for $p \in P$ and $h \in \Sp(\V)$.
This implies that
\[
 v_{\Y}(php', \nu) = v_{\Y}(p, \nu) \cdot v_{\Y}(h, \nu) \cdot v_{\Y}(p', \nu)
\]
for $p, p' \in P$ and $h \in \Sp(\V)$.
Moreover, there exist a character $\xi_{\nu}$ of $F^{\times}$ and 
an element $\gamma_{\nu} \in \C^1$ such that
\[
 v_{\Y}(p, \nu) = \xi_{\nu}(x(p)), \qquad v_{\Y}(\tau_j, \nu) = \gamma_{\nu}^j.
\]

To determine $\xi_{\nu}$ and $\gamma_{\nu}$, we may assume that $\dim \V = 2$
as explained in the proof of \cite[Proposition 1.2.A]{barthel}.
Put
\[
 \underline{\mathbf{n}}(x) := \begin{pmatrix} 1 & 0 \\ x & 1 \end{pmatrix}.
\]
If $x \ne 0$, then we have
\[
 \underline{\mathbf{n}}(x) = 
 \begin{pmatrix} 1 & x^{-1} \\ 0 & 1 \end{pmatrix}
 \begin{pmatrix} 0 & -1 \\ 1 & 0 \end{pmatrix}
 \begin{pmatrix} x & 1 \\ 0 & x^{-1} \end{pmatrix},
\]
so that
\[
 v_{\Y}(\underline{\mathbf{n}}(x), \nu) = \xi_{\nu}(x) \cdot \gamma_{\nu}.
\]
Let $x, y \in F$ such that $x \ne 0$, $y \ne 0$, $x+y \ne 0$.
Since $\alpha_{\nu}(\underline{\mathbf{n}}(x)) = \underline{\mathbf{n}}(\nu x)$, we have
\[
 \frac{z_{\Y}(\underline{\mathbf{n}}(\nu x), \underline{\mathbf{n}}(\nu y))}
 {z_{\Y}(\underline{\mathbf{n}}(x), \underline{\mathbf{n}}(y))}
 = \frac{v_{\Y}(\underline{\mathbf{n}}(x+y), \nu)}
 {v_{\Y}(\underline{\mathbf{n}}(x), \nu) \cdot v_{\Y}(\underline{\mathbf{n}}(y), \nu)}
 = \xi_{\nu}\left( \frac{x+y}{xy} \right) \cdot \gamma_{\nu}^{-1}.
\]
By \cite[Corollary 4.3]{rangarao}, we have
\[
 z_{\Y}(\underline{\mathbf{n}}(x), \underline{\mathbf{n}}(y))
 = \gamma_F(\frac{1}{2} xy(x+y) \cdot \psi)
\]
and hence 
\begin{align*}
 \frac{z_{\Y}(\underline{\mathbf{n}}(\nu x), \underline{\mathbf{n}}(\nu y))}
 {z_{\Y}(\underline{\mathbf{n}}(x), \underline{\mathbf{n}}(y))}
 & = \frac{\gamma_F(\frac{1}{2} \nu^3 xy(x+y) \cdot \psi)} 
 {\gamma_F(\frac{1}{2} xy(x+y) \cdot \psi)} \\
 & = \frac{\gamma_F(\nu^3 xy(x+y), \frac{1}{2} \psi)} 
 {\gamma_F(xy(x+y), \frac{1}{2} \psi)} \\
 & = \gamma_F(\nu^3, \frac{1}{2} \psi) \cdot ( xy(x+y), \nu^3 )_F \\
 & = \gamma_F(\nu, \frac{1}{2} \psi) \cdot ( xy(x+y), \nu )_F.
\end{align*}
Thus we obtain
\[
 \gamma_F(\nu, \frac{1}{2} \psi) \cdot \left( \frac{x+y}{xy}, \nu \right)_F 
 = \xi_{\nu}\left( \frac{x+y}{xy} \right) \cdot \gamma_{\nu}^{-1}.
\]
Taking $x = y = 2$, we have
\[
 \gamma_{\nu} = \gamma_F(\nu, \frac{1}{2} \psi)^{-1}
\]
and hence 
\[
 \xi_{\nu}(a) = ( a, \nu )_F
\]
for all $a \in F^{\times}$.
\end{proof}

\subsection{Metaplectic groups}
\label{ss:weil-gsp-metaplectic}

For each $\nu \in F^{\times}$, we have an automorphism $\tilde{\alpha}_{\nu}$ of $\Mp(\V)$.
This induces an action of $F^{\times}$ on $\Mp(\V)$ and thus we have a topological group
\[
 \Mp(\V) \rtimes F^{\times}.
\]
We define a bijection
\begin{align*}
 \Mp(\V)_{\Y} \rtimes F^{\times} & \longrightarrow \GMp(\V)_{\Y} := \GSp(\V) \times \C^1 \\
 ((h, z), \nu) & \longmapsto ((h, \nu)_{\Y}, z)
\end{align*}
as sets.
Via this bijection, we regard $\GMp(\V)_{\Y}$ as a group.
Note that the diagram
\[
 \xymatrix{
  \Mp(\V)_{\Y} \rtimes F^{\times} \ar@{->}[r] \ar@{->}[d] &
  \GMp(\V)_{\Y} \ar@{->}[d] \\
  \Sp(\V) \rtimes F^{\times} \ar@{->}[r] & \GSp(\V)
 }
\]
commutes.
Let $z_{\Y}^{\GSp}$ be the 2-cocycle associated to $\GMp(\V)_{\Y}$.
By definition, one can see that
\[
 z_{\Y}^{\GSp}(g, g')
 = z_{\Y}^{\Sp}(h, \alpha_{\nu}(h')) \cdot v_{\Y}(h', \nu)
\]
for $g = (h, \nu)_{\Y}, g' = (h', \nu')_{\Y} \in \GSp(\V)$.
In particular, the restriction of $z_{\Y}^{\GSp}$ to $\Sp(\V) \times \Sp(\V)$
is equal to $z_{\Y}^{\Sp}$.
Thus we omit the superscripts $\GSp$ and $\Sp$ from the notation.

We shall see that the isomorphism class of $\GMp(\V)_{\Y}$ does not depend on the choice of the complete polarization.
If there is no confusion, we write $\GMp(\V) = \GMp(\V)_{\Y}$.

\subsection{Change of polarizations}
\label{ss:weil-gsp-pol}

Let $\V = \X' + \Y'$ be another complete polarization.
Fix an element $h_0 \in \Sp(\V)$ such that $\X' = \X h_0$ and $\Y' = \Y h_0$.
Let $\alpha_0$ be the inner automorphism of $\GSp(\V)$ given by
\[
 \alpha_0(g) = h_0 \cdot g \cdot h_0^{-1}
\]
for $g \in \GSp(\V)$.
Note that $\alpha_0|_{\Sp(\V)}$ is an inner automorphism of $\Sp(\V)$.
We have
\[
 d_{\Y'}(\nu) = h_0^{-1} \cdot d_{\Y}(\nu) \cdot h_0, \qquad
 \alpha_{\Y', \nu} = \alpha_0^{-1} \circ \alpha_{\Y, \nu} \circ \alpha_0.
\]

By \cite[Lemma 4.2]{kudla-splitting}, we have
\[
 z_{\Y'}(h, h') = z_{\Y}(\alpha_0(h), \alpha_0(h'))
\]
for $h,h' \in \Sp(\V)$, and an isomorphism
\begin{align*}
 \Mp(\V)_{\Y} & \longrightarrow \Mp(\V)_{\Y'}, \\
 (h,z) & \longmapsto (h, z \cdot \mu(h))
\end{align*}
where
\[
 \mu(h) = z_{\Y}(h_0, h h_0^{-1}) \cdot z_{\Y}(h, h_0^{-1})
\]
for $h \in \Sp(\V)$.

\begin{lem}
We have
\[
 v_{\Y'}(h, \nu) = v_{\Y}(\alpha_0(h), \nu)
\]
for $h \in \Sp(\V)$ and $\nu \in F^{\times}$.
\end{lem}

\begin{proof}
We have
\begin{align*}
 z_{\Y'}(\alpha_{\Y', \nu}(h), \alpha_{\Y', \nu}(h')) 
 & = z_{\Y'}((\alpha_0^{-1} \circ \alpha_{\Y, \nu} \circ \alpha_0)(h), (\alpha_0^{-1} \circ \alpha_{\Y, \nu} \circ \alpha_0)(h')) \\
 & = z_{\Y}((\alpha_{\Y, \nu} \circ \alpha_0)(h), (\alpha_{\Y, \nu} \circ \alpha_0)(h')) \\
 & = z_{\Y}(\alpha_0(h), \alpha_0(h')) 
 \cdot v_{\Y}(\alpha_0(h) \cdot \alpha_0(h'), \nu)
 \cdot v_{\Y}(\alpha_0(h), \nu)^{-1} \cdot v_{\Y}(\alpha_0(h'), \nu)^{-1} \\
 & = z_{\Y'}(h, h') 
 \cdot v_{\Y}(\alpha_0(h h'), \nu)
 \cdot v_{\Y}(\alpha_0(h), \nu)^{-1} \cdot v_{\Y}(\alpha_0(h'), \nu)^{-1}.
\end{align*}
Thus the assertion follows from the characterization of $v_{\Y'}$.
\end{proof}

\begin{lem}
We have
\[
 z_{\Y'}(g, g') = z_{\Y}(\alpha_0(g), \alpha_0(g'))
\]
for $g, g' \in \GSp(\V)$.
\end{lem}

\begin{proof}
Let $g = (h, \nu)_{\Y'}, g' = (h', \nu')_{\Y'} \in \GSp(\V)$.
Then we have
\begin{align*}
 z_{\Y'}(g, g') & = z_{\Y'}(h, \alpha_{\Y', \nu}(h')) \cdot v_{\Y'}(h', \nu) \\
 & = z_{\Y}(\alpha_0(h), (\alpha_0 \circ \alpha_{\Y', \nu})(h'))
 \cdot v_{\Y}(\alpha_0(h'), \nu) \\
 & = z_{\Y}(\alpha_0(h), (\alpha_{\Y, \nu} \circ \alpha_0)(h'))
 \cdot v_{\Y}(\alpha_0(h'), \nu) \\
 & = z_{\Y}((\alpha_0(h), \nu)_{\Y}, (\alpha_0(h'), \nu')_{\Y}).
\end{align*}
Since 
\[
 \alpha_0(g) = h_0 \cdot h \cdot d_{\Y'}(\nu) \cdot h_0^{-1}
 = h_0 \cdot h \cdot h_0^{-1} \cdot d_{\Y}(\nu)
 = (\alpha_0(h), \nu)_{\Y},
\]
the assertion follows.
\end{proof}

Put
\[
 \mu(g) = z_{\Y}(g, h_0^{-1}) \cdot z_{\Y}(h_0, g h_0^{-1})
 = z_{\Y'}(h_0^{-1} g h_0, h_0^{-1}) \cdot z_{\Y'}(h_0^{-1}, g)^{-1}
\]
for $g \in \GSp(\V)$.
Note that $\mu$ depends on the choice of $h_0$.
By a direct calculation, one can see that
\[
 z_{\Y'}(g, g') =  z_{\Y}(g, g') \cdot \mu(gg') \cdot \mu(g)^{-1}
 \cdot \mu(g')^{-1}
\]
for $g, g' \in \GSp(\V)$.
Thus we obtain an isomorphism
\begin{align*}
 \GMp(\V)_{\Y} & \longrightarrow \GMp(\V)_{\Y'}. \\
 (g,z) & \longmapsto (g, z \cdot \mu(g))
\end{align*}

%% file: spl-B.tex
\section{Splittings: the case $\dim_B V = 2$ and $\dim_B W = 1$}
\label{sec:spl-B}

\subsection{Setup}
\label{ss:spl-setup}
Let $F$ be a number field.
Recall that
\begin{align*}
 E & = F + F \i, & B & = E + E \j, & B_1 & = E + E \j_1, & B_2 & = E + E \j_2, \\
 u & := \i^2, & J & := \j^2, & J_1 & := \j_1^2, & J_2 & := \j_2^2,
\end{align*}
where
\[
 J = J_1 J_2.
\]
Recall that
\[
 V = B_1 \otimes_E B_2 \qquad \text{and} \qquad W = B
\]
are a right skew-hermitian $B$-space and a left hermitian $B$-space respectively, and
\[
 \V = V \otimes_B W
\]
is an $F$-space with a symplectic form
\[
 \llangle \cdot, \cdot \rrangle = \frac{1}{2} \tr_{B/F}
 (\langle \cdot, \cdot \rangle \otimes \langle \cdot, \cdot \rangle^*).
\]
Recall that $\V = \X + \Y$ is a complete polarization, where
\[
 \X = F \e_1 + F \e_2  + F \e_3 + F \e_4, \qquad
 \Y = F \e_1^* + F \e_2^*  + F \e_3^*  + F \e_4^*.
\]
The actions of $B$, $B_1$, $B_2$ on $\V$ are given as follows:
\begin{itemize}
\item $B$-action
\begin{align*}
 \e_1 \i & = u \e_1^* &
 \e_2 \i & = -u J_1 \e_2^* &
 \e_3 \i & = -u J_2 \e_3^* &
 \e_4 \i & = u J \e_4^* \\
 \e^*_1 \i & = \e_1 &
 \e^*_2 \i & = - \frac{1}{J_1} \e_2 &
 \e^*_3 \i & = - \frac{1}{J_2} \e_3 &
 \e^*_4 \i & = \frac{1}{J} \e_4 \\
 \e_1 \j & = \e_4 &
 \e_2 \j & = J_1 \e_3 &
 \e_3 \j & = J_2 \e_2 &
 \e_4 \j & = J \e_1 \\
 \e^*_1 \j & = - J \e_4^* &
 \e^*_2 \j & = - J_2 \e_3^* &
 \e^*_3 \j & = - J_1 \e_2^* &
 \e^*_4 \j & = - \e_1^* \\
 \e_1 \i \j & = - uJ \e_4^* &
 \e_2 \i \j & = u J \e_3^* &
 \e_3 \i \j & = u J \e_2^* &
 \e_4 \i \j & = - u J \e_1^* \\
 \e^*_1 \i \j & = \e_4 &
 \e^*_2 \i \j & = - \e_3 &
 \e^*_3 \i \j & = - \e_2 &
 \e^*_4 \i \j & = \e_1
\end{align*}
\item $B_1$-action
\begin{align*}
 \i \e_1 & = u \e_1^* &
 \i \e_2  & = u J_1 \e_2^* &
 \i \e_3 & = -u J_2 \e_3^* &
 \i \e_4 & = - u J \e_4^* \\
 \i \e^*_1 & = \e_1 &
 \i \e^*_2 & = \frac{1}{J_1} \e_2 &
 \i \e^*_3 & = - \frac{1}{J_2} \e_3 &
 \i \e^*_4 & = - \frac{1}{J} \e_4 \\
 \j_1 \e_1 & = \e_2 &
 \j_1 \e_2 & = J_1 \e_1 &
 \j_1 \e_3 & = \e_4 &
 \j_1 \e_4 & = J_1 \e_3 \\
 \j_1 \e^*_1 & = - J_1 \e_2^* &
 \j_1 \e^*_2 & = - \e_1^* &
 \j_1 \e^*_3 & = - J_1 \e_4^* &
 \j_1 \e^*_4 & = - \e_3^* \\
 \i \j_1 \e_1 & = u J_1 \e_2^* &
 \i \j_1 \e_2 & = u J_1 \e_1^* &
 \i \j_1 \e_3 & = - uJ \e_4^* &
 \i \j_1 \e_4 & = -uJ \e_3^* \\
 \i \j_1 \e^*_1 & = - \e_2 &
 \i \j_1 \e^*_2 & = - \e_1 &
 \i \j_1 \e^*_3 & = \frac{1}{J_2} \e_4 &
 \i \j_1 \e^*_4 & = \frac{1}{J_2} \e_3
\end{align*}
\item $B_2$-action
\begin{align*}
 \i \e_1 & = u \e_1^* &
 \i \e_2  & = - uJ_1 \e_2^*&
 \i \e_3 & = u J_2 \e_3^* &
 \i \e_4 & = -uJ \e_4^* \\
 \i \e^*_1 & = \e_1 &
 \i \e^*_2 & = - \frac{1}{J_1} \e_2 &
 \i \e^*_3 & = \frac{1}{J_2} \e_3 &
 \i \e^*_4 & = - \frac{1}{J} \e_4 \\
 \j_2 \e_1 & = \e_3 &
 \j_2 \e_2 & = \e_4 &
 \j_2 \e_3 & = J_2 \e_1 &
 \j_2 \e_4 & = J_2 \e_2 \\
 \j_2 \e^*_1 & = - J_2 \e_3^* &
 \j_2 \e^*_2 & = - J_2 \e_4^* &
 \j_2 \e^*_3 & = - \e_1^* &
 \j_2 \e^*_4 & = - \e_2^* \\
 \i \j_2 \e_1 & = u J_2 \e_3^* &
 \i \j_2 \e_2 & = - uJ \e_4^* &
 \i \j_2 \e_3 & = u J_2 \e_1^* &
 \i \j_2 \e_4 & = - uJ \e_2^* \\
 \i \j_2 \e^*_1 & = - \e_3 &
 \i \j_2 \e^*_2 & = \frac{1}{J_1} \e_4 &
 \i \j_2 \e^*_3 & = - \e_1 &
 \i \j_2 \e^*_4 & = \frac{1}{J_1} \e_2
\end{align*}
\end{itemize}

Let $\ba_i \in B_i^{\times}$ and $\ba \in B^{\times}$.
We write
\[
 \ba_1 = a_1 + b_1 \i + c_1 \j_1 + d_1 \i \j_1, \qquad
 \ba_2 = a_2 + b_2 \i + c_2 \j_2 + d_2 \i \j_2, \qquad
 \ba = a + b \i + c \j + d \i \j,
\]
where $a_i, a, b_i, b, c_i, c, d_i, d \in F$.
Then we have
\begin{align*}
\begin{bmatrix}
 \ba_1 \e_1 \\
 \ba_1 \e_2 \\
 \ba_1 \e_3 \\
 \ba_1 \e_4 \\
 \ba_1 \e_1^* \\
 \ba_1 \e_2^* \\
 \ba_1 \e_3^* \\
 \ba_1 \e_4^* 
\end{bmatrix} 
& = \g_1
\begin{bmatrix}
 \e_1 \\
 \e_2 \\
 \e_3 \\
 \e_4 \\
 \e_1^* \\
 \e_2^* \\
 \e_3^* \\
 \e_4^* 
\end{bmatrix},
& \g_1 & =
\begin{pmatrix}
 a_1 & c_1 &  &  & b_1 u & d_1 u J_1 &  &  \\
 c_1 J_1 & a_1 &  &  & d_1 u J_1 & b_1 u J_1 &  &  \\
  &  & a_1 & c_1 &  &  & -b_1 u J_2 & -d_1 u J \\
  &  & c_1 J_1 & a_1 &  &  & -d_1 u J & -b_1 u J \\
 b_1 & -d_1 &  &  & a_1 & - c_1 J_1 &  &  \\
 -d_1 & \frac{b_1}{J_1} &  &  & - c_1 & a_1 &  &  \\
  &  & - \frac{b_1}{J_2} & \frac{d_1}{J_2} &  &  & a_1 & -c_1 J_1 \\
  &  & \frac{d_1}{J_2} & -\frac{b_1}{J} &  &  & -c_1 & a_1
\end{pmatrix},
 \\
\begin{bmatrix}
 \ba_2 \e_1 \\
 \ba_2 \e_2 \\
 \ba_2 \e_3 \\
 \ba_2 \e_4 \\
 \ba_2 \e_1^* \\
 \ba_2 \e_2^* \\
 \ba_2 \e_3^* \\
 \ba_2 \e_4^* 
\end{bmatrix} 
& = \g_2 
\begin{bmatrix}
 \e_1 \\
 \e_2 \\
 \e_3 \\
 \e_4 \\
 \e_1^* \\
 \e_2^* \\
 \e_3^* \\
 \e_4^* 
\end{bmatrix},
& \g_2 & =
\begin{pmatrix}
 a_2 &  & c_2 &  & b_2 u &  & d_2uJ_2 &  \\
  & a_2 &  & c_2 &  & -b_2 u J_1 &  & -d_2 u J \\
 c_2 J_2 &  & a_2 &  & d_2 u J_2 &  & b_2 u J_2 &  \\
  & c_2 J_2 &  & a_2 &  & -d_2 uJ &  & -b_2 uJ \\
 b_2 &  & -d_2 &  & a_2 &  & -c_2 J_2 &  \\
  & -\frac{b_2}{J_1} &  & \frac{d_2}{J_1} &  & a_2 &  & -c_2 J_2 \\
 -d_2 &  & \frac{b_2}{J_2} &  & -c_2 &  & a_2 &  \\
  & \frac{d_2}{J_1} &  & - \frac{b_2}{J} &  & -c_2 &  & a_2
\end{pmatrix},
 \\
\begin{bmatrix}
 \e_1 \ba \\
 \e_2 \ba \\
 \e_3 \ba \\
 \e_4 \ba \\
 \e_1^* \ba \\
 \e_2^* \ba \\
 \e_3^* \ba \\
 \e_4^* \ba
\end{bmatrix} 
& = \g
\begin{bmatrix}
 \e_1 \\
 \e_2 \\
 \e_3 \\
 \e_4 \\
 \e_1^* \\
 \e_2^* \\
 \e_3^* \\
 \e_4^* 
\end{bmatrix},
& \g & =
\begin{pmatrix}
  a &  &  & c & bu &  &  & -duJ \\
  & a & c J_1 &  &  & -buJ_1 & duJ &  \\
  & c J_2 & a &  &  & duJ & -buJ_2 &  \\
  c J &  &  & a & -duJ &  &  & buJ \\
  b &  &  & d & a &  &  & -cJ \\
  & - \frac{b}{J_1} & -d &  &  & a & -c J_2 &  \\
  & -d & -\frac{b}{J_2} &  &  & -cJ_1 & a &  \\
  d &  &  &  \frac{b}{J} & -c &  &  & a
\end{pmatrix}.
\end{align*}

For $\a \in \GL_4(F)$ and $\b \in \Sym_4(F)$, put
\[
 \m(\a) := \begin{pmatrix} \a & \\ & {}^t \a^{-1} \end{pmatrix},
 \qquad
 \n(\b) := \begin{pmatrix} \1_4 & \b \\ & \1_4 \end{pmatrix}.
\]

Fix a place $v$ of $F$.
In \S \S \ref{ss:B-spl}, \ref{ss:B1-spl},
we shall suppress the subscript $v$ from the notation.
Thus $F = F_v$ will be a local field of characteristic zero.

\subsection{The case $u \in (F_v^{\times})^2$ or $J \in (F_v^{\times})^2$}
\label{ss:B-spl}
First we explicate Morita theory.
Fix an isomorphism
\[
 \ii : B \longrightarrow \M_2(F)
\]
of $F$-algebras such that 
\[
 \ii(\ba^*) = \ii(\ba)^*
\]
for $\ba \in B$.
Put 
\[
 e := \ii^{-1} \begin{pmatrix} 1 & 0 \\ 0 & 0 \end{pmatrix}, \qquad
 e' := \ii^{-1} \begin{pmatrix} 0 & 1 \\ 0 & 0 \end{pmatrix}, \qquad
 e'' := \ii^{-1} \begin{pmatrix} 0 & 0 \\ 1 & 0 \end{pmatrix}.
\]
Then we have
\[
 e^* = \ii^{-1} \begin{pmatrix} 0 & 0 \\ 0 & 1 \end{pmatrix}
\]
and 
\begin{align*}
 e^2 & = e, & e e' & = e', & e e'' & = 0, & e e^* & = 0, \\
 e' e & = 0, & (e')^2 & = 0, & e' e'' & = e, & e' e^* & = e', \\
 e'' e & = e'', & e'' e' & = e^*, & (e'')^2 & = 0, & e'' e^* & = 0, \\
 e^* e & = 0, & e^* e' & = 0, & e^* e'' & = e'', & (e^*)^2 & = e^*.
\end{align*}
Thus we obtain
\[
 B = F e + F e' + F e'' + F e^*, \qquad e B = F e + F e', \qquad
 B e = F e + F e'', \qquad e B e = F e
\]
and
\[
 \begin{bmatrix} e \cdot \ba \\ e' \cdot \ba \end{bmatrix}
 = \ii(\ba) \cdot \begin{bmatrix} e \\ e' \end{bmatrix}
\]
for $\ba \in B$.

Now we consider an $F$-space $W^\dagger := e W$.
Since $e B e^* = F e'$ and $(e')^* = -e'$, we have
\[
 \langle x, y \rangle \in F e', \qquad
 \langle y, x \rangle = - \langle x, y \rangle
\]
for $x, y \in W^\dagger$.
Hence we can define a symplectic form
\[
 \langle \cdot, \cdot \rangle^{\dagger} : W^\dagger \times W^\dagger \longrightarrow F
\]
by
\[
 \langle x, y \rangle^* = \langle x, y \rangle^{\dagger} \cdot e'
\]
for $x, y \in W^\dagger$.
Moreover, we see that $\langle \cdot, \cdot \rangle^{\dagger}$ is non-degenerate.

We have $W^\dagger = F e + F e'$ and
\[
 \langle e, e \rangle^{\dagger} = \langle e', e' \rangle^{\dagger} = 0, \qquad
 \langle e, e' \rangle^{\dagger} = 1.
\]
Thus we may identify $W^\dagger$ with the space of row vectors $F^2$ so that
\[
 \langle x, y \rangle^{\dagger} = x_1 y_2 - x_2 y_1
\]
for $x = (x_1, x_2)$, $y = (y_1, y_2) \in W^\dagger$.

\begin{lem}

\label{lem:morita-w}
The restriction to $W^\dagger$ induces an isomorphism
\[
 \GU(W) \overset{\cong}\longrightarrow \GSp(W^\dagger).
\]
\end{lem}

\begin{proof}
One can see that the restriction to $W^\dagger$ induces a homomorphism $\GU(W) \rightarrow \GSp(W^\dagger)$.
Since
\[
 B \cdot W^\dagger = B \cdot e W = B e B \cdot W = B \cdot W = W,
\]
this homomorphism is injective.
Let $h \in \GSp(W^\dagger)$.
Since $W = W^\dagger \oplus e'' W$, we can define a map $\tilde{h}: W \rightarrow W$ by
\[
 \tilde{h}(x) = 
 \begin{cases}
  h(x) & \text{if $x \in W^\dagger$,} \\
  e'' h(e'x) & \text{if $x \in e'' W$.}
 \end{cases}
\]
Then one can check that $\tilde{h} \in \GU(W)$.
This yields the lemma.
\end{proof}

Thus we may identify $\GU(W)$ with $\GSp(W^\dagger)$.
Similarly, we consider an $F$-space $V^\dagger := V e$ with a non-degenerate symmetric bilinear form
\[
 \langle \cdot, \cdot \rangle^{\dagger} : V^\dagger \times V^\dagger \longrightarrow F
\]
defined by
\[
 \frac{1}{2} \cdot \langle x, y \rangle = 
 \langle x, y \rangle^{\dagger} \cdot e''
\]
for $x, y \in V^\dagger$.
As in Lemma \ref{lem:morita-w}, the restriction to $V^\dagger$ induces an isomorphism
\[
 \GU(V) \overset{\cong}\longrightarrow \GO(V^\dagger).
\]
Thus we may identify $\GU(V)$ with $\GO(V^\dagger)$.

One can see that the natural map
\[
 V^\dagger \otimes_F W^\dagger \longrightarrow V \otimes_B W
\]
is an isomorphism.
Thus we may identify $\V$ with $V^\dagger \otimes_F W^\dagger$.

\begin{lem}
We have
\[
 \llangle \cdot, \cdot \rrangle = 
 \langle \cdot, \cdot \rangle^{\dagger} \otimes \langle \cdot, \cdot \rangle^{\dagger}.
\]
\end{lem}

\begin{proof}
Let $a = \langle x, x' \rangle^{\dagger}$ and $b = \langle y, y' \rangle^{\dagger}$ for $x, x' \in V^\dagger$ and $y, y' \in W^\dagger$.
Then we have
\begin{align*}
 \llangle x \otimes y, x' \otimes y' \rrangle 
 & = \frac{1}{2} \cdot \tr_{B/F}
 (\langle x, x' \rangle \cdot \langle y, y' \rangle^*) \\
 & = \tr_{B/F}(a e'' \cdot b e') \\
 & = ab \cdot \tr_{B/F}(e^*) \\
 & = ab.
\end{align*}
\end{proof}

Thus we obtain a commutative diagram
\[
 \xymatrix{
  \GU(V) \times \GU(W) \ar@{->}[r] \ar@{->}[d] &
  \GSp(\V) \ar@{=}[d] \\
  \GO(V^\dagger) \times \GSp(W^\dagger) \ar@{->}[r] & \GSp(\V)
 }.
\]

Let $W^\dagger = X + Y$ be a complete polarization given by
\[
 X = F e, \qquad Y = F e'.
\]
Put
\[
 \X' = V^\dagger \otimes_F X, \qquad \Y' = V^\dagger \otimes_F Y.
\]
Then we have a complete polarization $\V = \X' + \Y'$.
Put
\[
 s'(h) := \gamma^{j(h)}
\]
for $h \in \GSp(W^\dagger)$, where
\[
 \gamma = 
 \begin{cases}
  1 & \text{if $B_1$ and $B_2$ are split,} \\
  -1 & \text{if $B_1$ and $B_2$ are ramified,}
 \end{cases}
\]
and 
\[
 j(h) = 
 \begin{cases}
  0 & \text{if $\ii(h) = \left( \begin{smallmatrix} * & * \\ 0 & * \end{smallmatrix} \right)$,} \\
  1 & \text{otherwise.}
 \end{cases}
\]

\begin{lem}
\label{lem:spl-GSp}
We have
\[
 z_{\Y'}(h, h') = s'(h h') \cdot s'(h)^{-1} \cdot s'(h')^{-1}
\]
for $h, h' \in \GSp(W^\dagger)$.
\end{lem}

\begin{proof}
The lemma follows from \cite[Theorem 3.1, case $1_+$]{kudla-splitting} and \cite[Proposition 2.1]{roberts}.
We shall give a proof for the sake of completeness.

Recall that $\dim_F V^\dagger = 4$ and $\det V^\dagger = 1$.
By \cite[Theorem 3.1, case $1_+$]{kudla-splitting}, we have
\begin{equation}
\label{eq:kudla-spl-1}
 z_{\Y'}(h, h') = s'(h h') \cdot s'(h)^{-1} \cdot s'(h')^{-1} 
\end{equation}
for $h, h' \in \Sp(W^\dagger)$.

Let $g, g' \in \GSp(W^\dagger)$.
For $\nu \in F^{\times}$, put
\[
 d(\nu) = \begin{pmatrix} 1 & \\ & \nu \end{pmatrix} \in \GSp(W^\dagger).
\]
We write
\[
 g = h \cdot d(\nu), \qquad g' = h' \cdot d(\nu')
\]
with $h, h' \in \Sp(W^\dagger)$ and $\nu, \nu' \in F^{\times}$.
Then we have
\[
 z_{\Y'}(g, g') = z_{\Y'}(h, h'') \cdot v_{\Y'}(h', \nu),
\]
where 
\[
 h'' = d(\nu) \cdot h' \cdot d(\nu)^{-1}.
\]
By \eqref{eq:kudla-spl-1}, we have
\[
 z_{\Y'}(h, h'') = s'(h h'') \cdot s'(h)^{-1} \cdot s'(h'')^{-1}.
\]
We have $s'(h) = s'(g)$, and since $j(h'') = j(h')$, we have $s'(h'') = s'(h') = s'(g')$.
Moreover, since $g g' = h h'' \cdot d(\nu \nu')$, we have $s'(h h'') = s'(g g')$.
Thus we obtain
\[
 z_{\Y'}(h, h'') = s'(g g') \cdot s'(g)^{-1} \cdot s'(g')^{-1}.
\]
By Lemma \ref{lem:v_Y}, we have
\[
 v_{\Y'}(h', \nu) =
 (x_{\Y'}(h'), \nu)_F \cdot \gamma_F(\nu, \frac{1}{2} \psi)^{-j_{\Y'}(h')},
\]
where $x_{\Y'}$ and $j_{\Y'}$ are as in \S \ref{ss:weil-GMp-setup}
with respect to the complete polarization $\V = \X' + \Y'$.
Since $\dim_F V^\dagger = 4$ and $\det V^\dagger = 1$, one can see that
$x_{\Y'}(h') \equiv 1 \bmod (F^{\times})^2$ and $j_{\Y'}(h') = 4 \cdot j(h')$.
Hence we have
\[
 v_{\Y'}(h', \nu) = 1.
\]
This completes the proof.
\end{proof}

\begin{lem}
\label{lem:spl-GSO}
We have
\[
 z_{\Y'}(g, g') = 1
\]
for $g, g' \in \GO(V^\dagger)^0$.
\end{lem}

\begin{proof}
For $g, g' \in \GO(V^\dagger)^0$, we have
\[
 z_{\Y'}(g, g') = z_{\Y'}(h, h'') \cdot v_{\Y'}(h', \nu),
\]
where 
\begin{align*}
 h & = g \cdot d_{\Y'}(\nu)^{-1}, & h' & = g' \cdot d_{\Y'}(\nu')^{-1}, &
 h'' & = d_{\Y'}(\nu) \cdot h' \cdot d_{\Y'}(\nu)^{-1}, \\
 \nu & = \nu(g), & \nu' & = \nu(g'). & &
\end{align*}
We have $h, h' \in P_{\Y'}$ and $z_{\Y'}(h, h'') = 1$.
Since $g' \in \GO(V^\dagger)^0$, we have
\[
 x_{\Y'}(h') \equiv \det g' \equiv 1 \bmod (F^{\times})^2,
\]
so that $v_{\Y'}(h', \nu) = 1$ by Lemma \ref{lem:v_Y}.
This completes the proof.
\end{proof}

\begin{lem}
\label{lem:GSp-GSO-commute}
We have
\[
 z_{\Y'}(g, h) = z_{\Y'}(h, g) = 1
\]
for $g \in \GO(V^\dagger)^0$ and $h \in \GSp(W^\dagger)$.
\end{lem}

\begin{proof}
See \cite[Proposition 2.2.A]{barthel}.
We shall give a proof for the sake of completeness.

For $g \in \GO(V^\dagger)^0$ and $h \in \GSp(W^\dagger)$, we have
\[
 z_{\Y'}(g, h) = z_{\Y'}(g', h'') \cdot v_{\Y'}(h', \nu), \qquad
 z_{\Y'}(h, g) = z_{\Y'}(h', g'') \cdot v_{\Y'}(g', \nu'),
\]
where
\begin{align*}
 g' & = g \cdot d_{\Y'}(\nu)^{-1}, & g'' &= d_{\Y'}(\nu') \cdot g' \cdot d_{\Y'}(\nu')^{-1}, & \nu & = \nu(g), \\
 h' & = h \cdot d(\nu')^{-1}, & h'' & = d(\nu) \cdot h' \cdot d(\nu)^{-1}, &  \nu' & = \nu(h).
\end{align*}
Since $g', g'' \in P_{\Y'}$, we have $z_{\Y'}(g', h'') = z_{\Y'}(h', g'') = 1$.
As in the proof of Lemma \ref{lem:spl-GSp}, we have $v_{\Y'}(h', \nu) = 1$.
As in the proof of Lemma \ref{lem:spl-GSO}, we have $v_{\Y'}(g', \nu') = 1$.
This completes the proof.
\end{proof}

We define a map $s': \GO(V^\dagger)^0 \times \GSp(W^\dagger) \rightarrow \C^1$ by
\[
 s'(\g) = \gamma^{j(h)}
\]
for $\g = (g, h) \in \GO(V^\dagger)^0 \times \GSp(W^\dagger)$.
By Lemmas \ref{lem:spl-GSp}, \ref{lem:spl-GSO}, \ref{lem:GSp-GSO-commute}, we see that 
\[
 z_{\Y'}(\g, \g') = s'(\g \g') \cdot s'(\g)^{-1} \cdot s'(\g')^{-1}
\]
for $\g, \g' \in \GO(V^\dagger)^0 \times \GSp(W^\dagger)$.

Recall that we have two complete polarizations $\V = \X + \Y = \X' + \Y'$, where
\begin{align*}
 \X & = F \e_1 + F \e_2  + F \e_3  + F \e_4, &
 \Y & = F \e_1^* + F \e_2^*  + F \e_3^*  + F \e_4^*, \\
 \X' & = F \e_1 e + F \e_1 e'' + F \e_2 e  + F \e_2 e'', &
 \Y' & = F \e_1 e' + F \e_1 e^*  + F \e_2 e'  + F \e_2 e^*.
\end{align*}
Fix $\h_0 \in \Sp(\V)$ such that
$\X' = \X \h_0$ and $\Y' = \Y \h_0$, and put
\[
 s(\g) := s'(\g) \cdot \mu(\g),
\]
where
\[
 \mu(\g) := z_{\Y}(\h_0 \g \h_0^{-1}, \h_0) \cdot z_{\Y}(\h_0, \g)^{-1}
\]
for $\g \in \GU(V)^0 \times \GU(W)$.
Then we have
\[
 z_{\Y}(\g, \g') = s(\g \g') \cdot s(\g)^{-1} \cdot s(\g')^{-1}
\]
for $\g, \g' \in \GU(V)^0 \times \GU(W)$.

\subsubsection{The case $u \in (F_v^{\times})^2$}
Choose $t \in F^{\times}$ such that $u = t^2$.
We take an isomorphism $\ii : B \rightarrow \M_2(F)$ determined by
\[
 \ii(1) = \begin{pmatrix} 1 & \\ & 1 \end{pmatrix}, \qquad
 \ii(\i) = \begin{pmatrix} t & \\ & -t \end{pmatrix}, \qquad
 \ii(\j) = \begin{pmatrix} & 1 \\ J & \end{pmatrix}, \qquad
 \ii(\i \j) = \begin{pmatrix} & t \\ - tJ & \end{pmatrix}.
\]
Then we have
\[
 e = \frac{1}{2} + \frac{1}{2t} \i, \qquad  
 e' = \frac{1}{2} \j + \frac{1}{2t} \i \j, \qquad 
 e'' = \frac{1}{2J} \j - \frac{1}{2tJ} \i \j, \qquad 
 e^* = \frac{1}{2} - \frac{1}{2t} \i.
\]
Put 
\[
 \h_0 =
 \begin{pmatrix}
  - \frac{1}{2t} &  &  &  & - \frac{1}{2} &  &  & \\
  & \frac{1}{2 t J_1} & &  &  & - \frac{1}{2} &  &  \\
  &  & \frac{1}{2 t J_2} &  &  &  & -\frac{1}{2} &  \\
  &  &  & - \frac{1}{2 t J} &  &  &  & - \frac{1}{2} \\
  1 &  &  &  & -t &  &  & \\
  & 1 &  &  &  & t J_1 & &  \\
  &  & 1 &  &  &  & t J_2 & \\
  &  &  & 1 &  &  &  & -tJ
 \end{pmatrix} \in \Sp(\V).
\]
Then we have
\[
\begin{bmatrix}
 - \frac{1}{t} \e_1 e \\
 \frac{1}{t J_1} \e_2 e \\
 \frac{1}{t} \e_2 e'' \\
 - \frac{1}{t} \e_1 e'' \\
 2 \e_1 e^* \\
 2 \e_2 e^* \\
 \frac{2}{J_1} \e_2 e' \\
 2 \e_1 e'
\end{bmatrix}
= \h_0
\begin{bmatrix}
 \e_1 \\
 \e_2 \\
 \e_3 \\
 \e_4 \\
 \e_1^* \\
 \e_2^* \\
 \e_3^* \\
 \e_4^* 
\end{bmatrix}, 
\]
and hence $\X' = \X \h_0$ and $\Y' = \Y \h_0$.

\begin{lem}
\label{lem:mu-GU(V)-u}
Let $\g_i := \ba_i^{-1} \in \GU(V)^0$ with $\ba_i = a_i + b_i \i + c_i \j_i + d_i \i \j_i \in B_i^{\times}$.
Then we have
\[
 \mu(\g_i) =
 \begin{cases}
  1 & \text{if $b_i = d_i = 0$,} \\
 \gamma_F(J_j, \frac{1}{2} \psi)
 \cdot ((a_i b_i + c_i d_i J_i) \nu_i J_i, J_j)_F
  & \text{if $(b_i, d_i) \ne (0,0)$ and $b_i^2 - d_i^2 J_i = 0$,} \\
 (-(b_i^2 - d_i^2 J_i) \nu_i J_i, J_j)_F
 & \text{if $(b_i, d_i) \ne (0,0)$ and $b_i^2 - d_i^2 J_i \ne 0$,}
 \end{cases}
\]
where $\nu_i = \nu(\ba_i)$ and $\{ i, j \} = \{ 1, 2 \}$.
\end{lem}

\begin{proof}
We only consider the case $i = 1$; the other case is similar.
Put $\d := d_{\Y}(\nu_1) \in \GSp(\V)$.
We have
\[
 z_{\Y}(\h_0 \g_1 \h_0^{-1}, \h_0) 
 = z_{\Y}(\h_0 \g_1 \h_0^{-1} \cdot \d^{-1}, \d \cdot \h_0 \cdot \d^{-1}) 
 \cdot v_{\Y}(\h_0, \nu_1).
\]
Since $\Y' \g_1 = \Y'$,
we have $\h_0 \g_1 \h_0^{-1} \cdot \d^{-1} \in P_{\Y}$ and hence
\[
 z_{\Y}(\h_0 \g_1 \h_0^{-1} \cdot \d^{-1}, \d \cdot \h_0 \cdot \d^{-1}) =1.
\]
We have $\h_0 = \n(\b_1) \cdot \tau_4 \cdot \n(\b_2)$, where
\[
 \b_1
 = \frac{1}{2tJ} \cdot
 \begin{pmatrix}
 -J & & & \\
 & J_2 & & \\
 & & J_1 & \\
 & & & -1
 \end{pmatrix}, \qquad
 \b_2
 = t \cdot 
 \begin{pmatrix}
 -1 & & & \\
 & J_1 & & \\
 & & J_2 & \\
 & & & -J
 \end{pmatrix},
\]
so that $x_{\Y}(\h_0) \equiv 1 \bmod (F^{\times})^2$ and $j_{\Y}(\h_0) = 4$.
Hence we have $v_{\Y}(\h_0, \nu_1) = 1$.
Thus we obtain 
\[
 z_{\Y}(\h_0 \g_1 \h_0^{-1}, \h_0) = 1.
\]

Now we compute $z_{\Y}(\h_0, \g_1)$.
We have
\[
 z_{\Y}(\h_0, \g_1) = z_{\Y}(\h_0, \g_1 \cdot \d^{-1}).
\]
First assume that $b_1 = d_1 = 0$. 
Then we have $\g_1 \cdot \d^{-1} \in P_{\Y}$ and hence
\[
 z_{\Y}(\h_0, \g_1 \cdot \d^{-1}) = 1.
\]
Next assume that $(b_1, d_1) \ne (0,0)$ and $b_1^2 - d_1^2 J_1 = 0$.
Then we have $b_1 \ne 0$, $d_1 \ne 0$, and $\nu_1 = a_1^2 - c_1^2 J_1 \ne 0$.
Since 
\begin{align*}
 (a_1 d_1 - b_1 c_1) \cdot (a_1 b_1 + c_1 d_1 J_1)
 & = a_1^2 b_1 d_1 + a_1 c_1 d_1^2 J_1 - a_1 b_1^2 c_1  - b_1 c_1^2 d_1 J_1 \\
 & = a_1^2 b_1 d_1 - b_1 c_1^2 d_1 J_1 \\
 & = \nu_1 b_1 d_1 \\
 & \ne 0,
\end{align*}
we have $a_1 d_1 - b_1 c_1 \ne 0$ and $a_1 b_1 + c_1 d_1 J_1 \ne 0$.
We have $\g_1 \cdot \d^{-1} \in \m(\a_1) \cdot \n(\b_3) \cdot \tau_2 \cdot P_{\Y}$, where
\[
 \a_1 = 
 \begin{pmatrix}
  b_1 & & & \\
  d_1 J_1 & & 1 & \\
  & b_1 &  & \\
  & d_1 J_1 & & 1
 \end{pmatrix}, \qquad
 \b_3 = \frac{a_1 d_1 - b_1 c_1}{b_1 d_1} \cdot 
 \begin{pmatrix}
  0 & & & \\
  & 0 & & \\
  & & J_1 & \\
  & & & -J
 \end{pmatrix}.
\]
Hence we have
\begin{align*}
 z_{\Y}(\h_0, \g_1 \cdot \d^{-1})
 & = z_{\Y}(\tau_4 \cdot \n(\b_2), \m(\a_1) \cdot \n(\b_3) \cdot \tau_2) \\
 & = z_{\Y}(\tau_4 \cdot \m(\a_1),  \m(\a_1)^{-1} \cdot \n(\b_2) \cdot \m(\a_1) \cdot \n(\b_3) \cdot \tau_2) \\
 & = z_{\Y}(\tau_4, \n(\b_4 + \b_3) \cdot \tau_2),
\end{align*}
where
\[
 \b_4 = \a_1^{-1} \cdot \b_2 \cdot {}^t \a_1^{-1}
 = \frac{t}{b_1^2} \cdot 
 \begin{pmatrix}
 -1 & & d_1 J_1 & \\
  & J_2 & & - d_1 J \\ 
  d_1 J_1 & & & \\
  & - d_1 J & &
 \end{pmatrix}.
\]
Since $\tau_2^{-1} \cdot \n(\b_4) \cdot \tau_2 \in P_{\Y}$, we have
\[
 z_{\Y}(\h_0, \g_1 \cdot \d^{-1}) = z_{\Y}(\tau_4, \n(\b_3) \cdot \tau_2)
 = \gamma_F(\frac{1}{2} \psi \circ q_1),
\]
where $q_1$ is a non-degenerate symmetric bilinear form associated to
\[
 \frac{a_1 d_1 - b_1 c_1}{b_1 d_1} \cdot 
 \begin{pmatrix}
  J_1 & \\
  & -J
 \end{pmatrix}.
\]
Since $\det q_1 \equiv - J_2 \bmod (F^{\times})^2$ and 
$h_F(q_1) = (\frac{a_1 d_1 - b_1 c_1}{b_1 d_1} \cdot J_1, J_2)_F$,
we have
\begin{align*}
 \gamma_F(\frac{1}{2} \psi \circ q_1)
 & = \gamma_F(\frac{1}{2} \psi)^2 \cdot \gamma_F(- J_2, \frac{1}{2} \psi) 
 \cdot (\frac{a_1 d_1 - b_1 c_1}{b_1 d_1} \cdot J_1, J_2)_F \\
 & = \gamma_F(J_2, \frac{1}{2} \psi)^{-1} \cdot (\frac{a_1 d_1 - b_1 c_1}{b_1 d_1} \cdot J_1, J_2)_F \\
 & = \gamma_F(J_2, \frac{1}{2} \psi)^{-1} \cdot (\frac{\nu_1}{a_1 b_1 + c_1 d_1 J_1} \cdot J_1, J_2)_F \\
 & = \gamma_F(J_2, \frac{1}{2} \psi)^{-1} \cdot ((a_1 b_1 + c_1 d_1 J_1) \nu_1 J_1, J_2)_F.
\end{align*}
Finally assume that $(b_1, d_1) \ne (0,0)$ and $b_1^2 - d_1^2 J_1 \ne 0$.
We have $\g_1 \cdot \d^{-1} \in \n(\b_5) \cdot \tau_4 \cdot P_{\Y}$, where
\[
 \b_5 = \frac{1}{b_1^2 - d_1^2 J_1} \cdot 
 \begin{pmatrix}
  a_1 b_1 + c_1 d_1 J_1 & (a_1 d_1 + b_1 c_1) J_1 & & \\
  (a_1 d_1 + b_1 c_1) J_1 & (a_1 b_1 + c_1 d_1 J_1) J_1 & & \\
  & & - (a_1 b_1 + c_1 d_1 J_1) J_2 & - \left(a_1 d_1 + b_1 c_1\right) J \\
  & & - (a_1 d_1 + b_1 c_1) J & - (a_1 b_1 + c_1 d_1 J_1) J
 \end{pmatrix}.
\]
Hence we have
\[
 z_{\Y}(\h_0, \g_1 \cdot \d^{-1}) = 
 z_{\Y}(\tau_4 \cdot \n(\b_2), \n(\b_5) \cdot \tau_4) =
 z_{\Y}(\tau_4, \n(\b_2) \cdot \n(\b_5) \cdot \tau_4) =
 \gamma_F(\frac{1}{2} \psi \circ q_2),
\]
where $q_2$ is a non-degenerate symmetric bilinear form associated to $\b_2 + \b_5$.
We have
\[
 \b_2 + \b_5 = 
 \begin{pmatrix}
  \b' & \\
  & -J_2 \cdot \b'
 \end{pmatrix},
\]
where 
\[
 \b' = t \cdot 
 \begin{pmatrix}
 -1 & \\
 & J_1
 \end{pmatrix}
 + \frac{1}{b_1^2 - d_1^2 J_1} \cdot 
 \begin{pmatrix}
  a_1 b_1 + c_1 d_1 J_1 & (a_1 d_1 + b_1 c_1) J_1 \\
  (a_1 d_1 + b_1 c_1) J_1 & (a_1 b_1 + c_1 d_1 J_1) J_1
 \end{pmatrix}.
\]
Since 
\[
 \det \b' = \frac{\nu_1 J_1}{b_1^2 - d_1^2 J_1} \ne 0,
\]
we have $\det q_2 \equiv 1 \bmod (F^{\times})^2$ and 
\[
 h_F(q_2) = (\det \b', J_2)_F \cdot (-1, -J_2)_F
 = (- \frac{\nu_1 J_1}{b_1^2 - d_1^2 J_1}, J_2)_F \cdot (-1, -1)_F.
\]
Hence we have
\[
 \gamma_F(\frac{1}{2} \psi \circ q_2) = 
 \gamma_F(\frac{1}{2} \psi)^4 \cdot 
 (- \frac{\nu_1 J_1}{b_1^2 - d_1^2 J_1}, J_2)_F \cdot (-1, -1)_F
 = (-(b_1^2 - d_1^2 J_1) \nu_1 J_1, J_2)_F.
\]
This completes the proof.
\end{proof}

\begin{lem}
\label{lem:mu-GU(W)-u}
Let $\g := \ba \in \GU(W)$ with $\ba = a + b \i + c \j + d \i \j \in B^{\times}$.
Then we have
\[
 \mu(\g) =
 \begin{cases}
  (\nu, J_1)_F & \text{if $b = d = 0$,} \\
  \gamma_F(J_1, \frac{1}{2} \psi) \cdot (ab-cdJ, J_1)_F
  & \text{if $(b,d) \ne (0,0)$ and $b^2 - d^2 J = 0$,} \\
  (- (b^2 - d^2J) J, J_1)_F 
  & \text{if $(b,d) \ne (0,0)$ and $b^2 - d^2 J \ne 0$,}
 \end{cases}
\]
where $\nu = \nu(\ba)$.
\end{lem}

\begin{proof}
Put $\d := d_{\Y}(\nu) \in \GSp(\V)$. 
We have
\[
 z_{\Y}(\h_0 \g \h_0^{-1}, \h_0) 
 = z_{\Y}(\h_0 \g \h_0^{-1} \cdot \d^{-1}, \d \cdot \h_0 \cdot \d^{-1}) 
 \cdot v_{\Y}(\h_0, \nu).
\]
As in the proof of Lemma \ref{lem:mu-GU(V)-u}, we have $v_{\Y}(\h_0, \nu) = 1$.
We have
\[
 \h_0 \g \h_0^{-1} = 
 \begin{pmatrix}
  a+bt & & & & & & & -\frac{c+dt}{2t} \\
  & a+bt & & & & & \frac{c+dt}{2t} & \\
  & & a+bt & & & \frac{c+dt}{2t} & & \\
  & & & a+bt & - \frac{c+dt}{2t} & & & \\
  & & & -2(c-dt)tJ & a-bt & & & \\
  & & 2 (c-dt)tJ & & & a-bt & & \\
  & 2 (c-dt)tJ & & & & & a-bt & \\
  - 2 (c-dt)tJ & & & & & & &  a-bt
 \end{pmatrix}.
\]
If $c-dt = 0$, then we have $\h_0 \g \h_0^{-1} \cdot \d^{-1} \in P_{\Y}$
and hence 
$z_{\Y}(\h_0 \g \h_0^{-1} \cdot \d^{-1}, \d \cdot \h_0 \cdot \d^{-1}) = 1$.
If $c-dt \ne 0$, then we have
$\h_0 \g \h_0^{-1} \cdot \d^{-1} \in P_{\Y} \cdot \tau_4 \cdot \n(\b_6)$,
where
\[
 \b_6 = \frac{a-bt}{2 \nu t J (c-dt)} \cdot 
 \begin{pmatrix}
  & & & - 1 \\
  & & 1 & \\
  & 1 & & \\
 -1 & & &
 \end{pmatrix}.
\]
We have $\d \cdot \h_0 \cdot \d^{-1} \in \n(\nu^{-1} \cdot \b_1) \cdot \tau_4 \cdot P_{\Y}$, where $\b_1$ is as in the proof of Lemma \ref{lem:mu-GU(V)-u}.
Hence we have
\[
 z_{\Y}(\h_0 \g \h_0^{-1} \cdot \d^{-1}, \d \cdot \h_0 \cdot \d^{-1})  
 = z_{\Y}(\tau_4 \cdot \n(\b_6), \n(\nu^{-1} \cdot \b_1) \cdot \tau_4)
 = z_{\Y}(\tau_4, \n(\b_7) \cdot \tau_4),
\]
where $\b_7 = \nu^{-1} \cdot \b_1 + \b_6$.
Put $r = \frac{a-bt}{c-dt}$.
We have
\[
 \b_7 = \frac{1}{2 \nu t J} \cdot
 \begin{pmatrix}
  -J & & & -r \\
  & J_2 & r & \\
  & r & J_1 & \\
  -r & & & -1
 \end{pmatrix}
 = \a_2 \cdot \b_8 \cdot {}^t \a_2,
\]
where
\[
 \a_2 = 
 \begin{pmatrix}
  1 & & & \\
  & & \frac{r}{J_1} & 1 \\
  & & 1 & \\
  \frac{r}{J} & 1 & &
 \end{pmatrix}, \qquad
 \b_8 = 
 \frac{1}{2 \nu t J} \cdot 
 \begin{pmatrix}
  -J & & & \\
  & \frac{r^2}{J} - 1 & & \\
  & & J_1 & \\
  & & & J_2 - \frac{r^2}{J_1}
 \end{pmatrix},
\]
and hence
\[
 z_{\Y}(\tau_4, \n(\b_7) \cdot \tau_4)
 = z_{\Y}(\tau_4, \m(\a_2) \cdot \n(\b_8) \cdot \m(\a_2^{-1}) \cdot \tau_4)
 = z_{\Y}(\tau_4, \n(\b_8) \cdot \tau_4).
\]
If $J - r^2 = 0$, then we have $z_{\Y}(\tau_4, \n(\b_8) \cdot \tau_4) = \gamma_F(\frac{1}{2} \psi \circ q_3)$,
where $q_3$ is a non-degenerate symmetric bilinear form associated to
\[
 \frac{1}{2 \nu t J} \cdot 
 \begin{pmatrix}
  -J & \\ 
  & J_1
 \end{pmatrix}.
\]
We have $\det q_3 \equiv -J_2 \bmod (F^{\times})^2$ and 
\[
 h_F(q_3) = (-\frac{1}{2 \nu t}, \frac{1}{2 \nu t J_2})_F = (-2 \nu t, J_2)_F.
\]
Hence we have 
\[
 \gamma_F(\frac{1}{2} \psi \circ q_3)
 = \gamma_F(\frac{1}{2} \psi)^2 \cdot \gamma_F(-J_2, \frac{1}{2} \psi)
 \cdot (-2 \nu t, J_2)_F
 = \gamma_F(J_2, \frac{1}{2} \psi) \cdot (2 \nu t, J_2)_F.
\]
Note that $\gamma_F(J_1, \frac{1}{2} \psi) = \gamma_F(J_2, \frac{1}{2} \psi)$
and $(2 \nu t, J_1)_F = (2 \nu t, J_2)_F$ since $J = r^2 \in (F^{\times})^2$.
If $J - r^2 \ne 0$, then we have $z_{\Y}(\tau_4, \n(\b_8) \cdot \tau_4) = \gamma_F(\frac{1}{2} \psi \circ q_4)$,
where $q_4$ is a non-degenerate symmetric bilinear form associated to $\b_8$.
We have $\det q_4 \equiv 1 \bmod (F^{\times})^2$ and 
\begin{align*}
 h_F(q_4) & = (\det q_4, \frac{1}{2 \nu t J})_F
 \cdot (-J, \frac{r^2}{J} - 1)_F \cdot (J_1, J_2 - \frac{r^2}{J_1})_F \cdot
 (-J (\frac{r^2}{J} - 1), J_1 (J_2 - \frac{r^2}{J_1}))_F \\
 & = (-J, J-r^2)_F \cdot (-J, -\frac{1}{J})_F \cdot 
 (J_1, J - r^2)_F \cdot (J_1, \frac{1}{J_1})_F \cdot (J-r^2, J-r^2)_F \\
 & = (-J, J-r^2)_F \cdot (-J, -1)_F \cdot 
 (J_1, J - r^2)_F \cdot (J_1, -1)_F \cdot (J-r^2, -1)_F \\
 & = (J J_1, J-r^2)_F \cdot (-JJ_1, -1)_F \\
 & = (J_2, J-r^2)_F \cdot (J_2, -1)_F \cdot (-1, -1)_F \\
 & = (J_2, r^2-J)_F \cdot (-1, -1)_F.
\end{align*}
Note that $(J_1, r^2-J)_F = (J_2, r^2-J)_F$ since $(J, r^2-J)_F = 1$.
Hence we have
\[
 \gamma_F(\frac{1}{2} \psi \circ q_4)
 = \gamma_F(\frac{1}{2} \psi)^4 \cdot (J_2, r^2-J)_F \cdot (-1, -1)_F
 = (J_2, r^2-J)_F.
\]
Thus we obtain
\[
 z_{\Y}(\h_0 \g \h_0^{-1}, \h_0) =
 \begin{cases}
  1 & \text{if $c - dt = 0$,} \\
 \gamma_F(J_1, \frac{1}{2} \psi) \cdot (2 \nu t, J_1)_F
  & \text{if $c - dt \ne 0$, $(a-bt)^2 - (c-dt)^2 J = 0$,} \\
  ((a-bt)^2 - (c-dt)^2 J, J_1) _F
  & \text{if $c - dt \ne 0$, $(a-bt)^2 - (c-dt)^2 J \ne 0$.}
 \end{cases}
\]

Now we compute $z_{\Y}(\h_0, \g)$.
We have
\[
 z_{\Y}(\h_0, \g) = z_{\Y}(\h_0, \g \cdot \d^{-1}).
\]
First assume that $b = d = 0$.
Then we have $\g \cdot \d^{-1} \in P_{\Y}$ and hence
\[
 z_{\Y}(\h_0, \g \cdot \d^{-1}) = 1.
\]
Next assume that $(b, d) \ne (0,0)$ and $b^2 - d^2 J = 0$.
Then we have $b \ne 0$ and $d \ne 0$.
We have $\g \cdot \d^{-1} \in \m(\a_3) \cdot \n(\b_9) \cdot \tau_2 \cdot P_{\Y}$,
where
\[
 \a_3 =
 \begin{pmatrix}
  b & & & \\
  & b & & \\
  & -dJ_2 & 1 & \\
  -d J & & & 1
 \end{pmatrix},
 \qquad
 \b_9 = \frac{ad+bc}{bd} \cdot 
 \begin{pmatrix}
  0 & & & \\
  & 0 & & \\
  & & -J_2 & \\
  & & & J
 \end{pmatrix}.
\]
Hence we have
\begin{align*}
 z_{\Y}(\h_0, \g \cdot \d^{-1}) & = 
 z_{\Y}(\tau_4 \cdot \n(\b_2), \m(\a_3) \cdot \n(\b_9) \cdot \tau_2) \\
 & = z_{\Y}(\tau_4 \cdot \m(\a_3), \m(\a_3)^{-1} \cdot \n(\b_2) \cdot \m(\a_3) \cdot \n(\b_9) \cdot \tau_2) \\
 & = z_{\Y}(\tau_4, \n(\b_{10} + \b_9) \cdot \tau_2),
\end{align*}
where $\b_2$ is as in the proof of Lemma \ref{lem:mu-GU(V)-u} and
\[
 \b_{10} = \a_3^{-1} \cdot \b_2 \cdot {}^t \a_3^{-1}
 = \frac{t}{b^2} \cdot 
 \begin{pmatrix}
  -1 & & & -dJ \\
  & J_1 & dJ & \\
  & dJ & 2 b^2 J_2 & \\
  -dJ & & & -2 b^2 J
 \end{pmatrix}.
\]
We write $\b_9 + \b_{10} = \b_{11} + \b_{12}$, where
\[
 \b_{11} = \frac{t}{b^2} \cdot 
 \begin{pmatrix}
  -1 & & & -dJ \\
  & J_1 & dJ & \\
  & dJ & & \\
  -dJ & & & 
 \end{pmatrix}, 
 \qquad \b_{12} = r' \cdot 
 \begin{pmatrix}
  0 & & & \\
  & 0 & & \\
  & & J_2 & \\
  & & & -J
 \end{pmatrix}, \qquad
 r' = 2t - \frac{ad+bc}{bd}.
\]
Since $\tau_2^{-1} \cdot \n(\b_{11}) \cdot \tau_2 \in P_{\Y}$, we have
\[
 z_{\Y}(\h_0, \g \cdot \d^{-1}) = z_{\Y}(\tau_4, \n(\b_{12}) \cdot \tau_2).
\]
If $r' = 0$, then we have
$z_{\Y}(\tau_4, \n(\b_{12}) \cdot \tau_2) = z_{\Y}(\tau_4, \tau_2) = 1$.
If $r' \ne 0$, then we have $z_{\Y}(\tau_4, \n(\b_{12}) \cdot \tau_2) = \gamma_F(\frac{1}{2} \psi \circ q_5)$,
where $q_5$ is a non-degenerate symmetric bilinear form associated to
\[
 r' \cdot 
 \begin{pmatrix}
  J_2 & \\
  & -J
 \end{pmatrix}.
\]
Since $\det q_5 \equiv - J_1 \bmod (F^{\times})^2$ and $h_F(q_5) = (r' \cdot J_2, J_1)_F$, we have
\[
 \gamma_F(\frac{1}{2} \psi \circ q_5) =
 \gamma_F(\frac{1}{2} \psi)^2 \cdot 
 \gamma_F(-J_1, \frac{1}{2} \psi) \cdot (r' \cdot J_2, J_1)_F
 = \gamma_F(J_1, \frac{1}{2} \psi)^{-1} \cdot (r' \cdot J_2, J_1)_F.
\]
Finally assume that $(b, d) \ne (0,0)$ and $b^2 - d^2 J \ne 0$.
We have $\g \cdot \d^{-1} \in \n(\b_{13}) \cdot \tau_4 \cdot P_{\Y}$, where
\[
 \b_{13} = \frac{1}{b^2 - d^2 J} \cdot 
 \begin{pmatrix}
  ab- cdJ & & & -(ad - bc)J \\
  & - (ab - cdJ) J_1 & (ad - bc)J & \\
  & (ad - bc)J & -(ab - cdJ)J_2 & \\
  -(ad - bc)J & & & (ab - cdJ)J
 \end{pmatrix}.
\]
Hence we have
\[
 z_{\Y}(\h_0, \g \cdot \d^{-1})
 = z_{\Y}(\tau_4 \cdot \n(\b_2), \n(\b_{13}) \cdot \tau_4)
 = z_{\Y}(\tau_4, \n(\b_2 + \b_{13}) \cdot \tau_4).
\]
Put
\[
 l = ab - cdJ - (b^2 - d^2J)t, \qquad
 l' = (ad-bc)J, \qquad
 r'' = \frac{l^2 J - l'^2}{(b^2 - d^2 J)^2}
 = \frac{ ((a-bt)^2 - (c-dt)^2 J)J}{b^2 - d^2 J}.
\]
We have
\[
 \b_2 + \b_{13} = \frac{1}{b^2 - d^2 J} \cdot
 \begin{pmatrix}
   l & & & -l' \\ 
   & -l J_1 & l' & \\ 
   & l' & -l J_2 & \\ 
  -l' & & & lJ
 \end{pmatrix},
\]
and if $l \ne 0$, then we have $\b_2 + \b_{13} = \a_4 \cdot \b_{14} \cdot {}^t \a_4$, where
\[
 \a_4 = \frac{1}{l} \cdot
 \begin{pmatrix}
  l & & & \\
  & & & l \\
  & & -\frac{1}{J_1} & -\frac{l'}{J_1} \\
  -l' & 1 & & 
 \end{pmatrix}, \qquad
 \b_{14} = \frac{1}{b^2 - d^2J} \cdot 
 \begin{pmatrix}
  l & & & \\
  & (l^2J-l'^2)l & & \\
  & & - (l^2J-l'^2) l J_1 & \\
  & & & -lJ_1
 \end{pmatrix}.
\]
If $l = l' = 0$, then we have $z_{\Y}(\tau_4, \n(\b_2 + \b_{13}) \cdot \tau_4) = z_{\Y}(\tau_4, \tau_4) = 1$.
If $(l,l') \ne (0,0)$ and $r'' = 0$, then we have $l \ne 0$ and $l' \ne 0$,
so that $z_{\Y}(\tau_4, \n(\b_2 + \b_{13}) \cdot \tau_4) = \gamma_F(\frac{1}{2} \psi \circ q_6)$,
where $q_6$ is a non-degenerate symmetric bilinear form associated to
\[
 \frac{l}{b^2 - d^2J} \cdot 
 \begin{pmatrix}
  1 & \\
  & -J_1
 \end{pmatrix}.
\]
We have $\det q_6 \equiv -J_1 \bmod (F^{\times})^2$ and
\[
 h_F(q_6) = (\frac{l}{b^2 - d^2J}, J_1)_F
 = (\frac{ab-cdJ}{b^2 - d^2J} - t, J_1)_F.
\]
Hence we have
\[
 \gamma_F(\frac{1}{2} \psi \circ q_6)
 = \gamma_F(\frac{1}{2} \psi)^2 \cdot \gamma_F(-J_1, \frac{1}{2} \psi)
 \cdot (\frac{ab-cdJ}{b^2 - d^2J} - t, J_1)_F
 = \gamma_F(J_1, \frac{1}{2} \psi)^{-1}
 \cdot (\frac{ab-cdJ}{b^2 - d^2J} - t, J_1)_F.
\]
Note that $\gamma_F(J_1, \frac{1}{2} \psi) = \gamma_F(J_2, \frac{1}{2} \psi)$
and $(\frac{ab-cdJ}{b^2 - d^2J} - t, J_1)_F = (\frac{ab-cdJ}{b^2 - d^2J} - t, J_2)_F$ since $r'' = 0$ and hence $J \in (F^{\times})^2$.
If $r'' \ne 0$, then we have
$z_{\Y}(\tau_4, \n(\b_2 + \b_{13}) \cdot \tau_4) = \gamma_F(\frac{1}{2} \psi \circ q_7)$,
where $q_7$ is a non-degenerate symmetric bilinear form associated to
$\b_2 + \b_{13}$.
We have $\det q_7 \equiv 1 \bmod (F^{\times})^2$.
Also, we have
\begin{align*}
 h_F(q_7) & =
 (\frac{l}{b^2 - d^2J}, \frac{(l^2J-l'^2)l}{b^2 - d^2J})_F \cdot
 (-\frac{lJ_1}{b^2 - d^2J}, -\frac{(l^2J-l'^2)lJ_1}{b^2 - d^2J})_F \cdot
 (\frac{(l^2J-l'^2)l^2}{(b^2 - d^2J)^2}, 
 \frac{(l^2J-l'^2)l^2 J_1^2}{(b^2 - d^2J)^2})_F \\
 & = (\frac{l}{b^2 - d^2J}, -(l^2J-l'^2))_F \cdot
 (-\frac{lJ_1}{b^2 - d^2J}, -(l^2J-l'^2))_F \cdot
 (l^2J-l'^2, l^2J-l'^2)_F \\
 & = (-J_1, -(l^2J-l'^2))_F \cdot (-1, l^2J-l'^2)_F \\
 & = (J_1, -(l^2J-l'^2))_F \cdot (-1, -1)_F \\
 & = (J_1, -r'') \cdot (-1, -1)_F 
\end{align*}
if $l \ne 0$, and 
\[
 h_F(q_7) = (-1, -1)_F = (J_1, -r'') \cdot (-1, -1)_F
\]
if $l = 0$.
Hence we have
\[
 \gamma_F(\frac{1}{2} \psi \circ q_7)
 = \gamma_F(\frac{1}{2} \psi)^4 \cdot (J_1, -r'') \cdot (-1, -1)_F
 = (J_1, -r'').
\]
Note that $(J_1, -r'') = (J_2, -r'')$ since $(J, -r'') = (J, l'^2 - l^2 J) = 1$.
Thus we obtain
\[
 z_{\Y}(\h_0, \g) = 
 \begin{cases}
  1 & \text{if $b = d = 0$,} \\
  1 & \text{if $(b,d) \ne (0,0)$, $b^2 - d^2 J = 0$, $ad + bc - 2bdt = 0$,} \\
 \gamma_F(J_1, \frac{1}{2} \psi)^{-1}
 \cdot ((2t- \frac{ad+bc}{bd}) \cdot J_2, J_1)_F
  & \text{if $(b,d) \ne (0,0)$, $b^2 - d^2 J = 0$, $ad + bc - 2bdt \ne 0$,} \\
  1 & \text{if $(b,d) \ne (0,0)$, $b^2 - d^2 J \ne 0$,} \\
  & \text{\qquad $ab - cdJ - (b^2-d^2 J)t = ad-bc = 0$,} \\
 \gamma_F(J_1, \frac{1}{2} \psi)^{-1}
 \cdot (\frac{ab-cdJ}{b^2 - d^2J} - t, J_1)_F.
  & \text{if $(b,d) \ne (0,0)$, $b^2 - d^2 J \ne 0$,} \\
  & \text{\qquad $(ab - cdJ - (b^2-d^2 J)t, ad-bc) \ne (0,0)$,} \\
  & \text{\qquad $(a-bt)^2 - (c-dt)^2J = 0$,} \\
 (- \frac{((a-bt)^2 - (c-dt)^2 J) J}{b^2 - d^2J}, J_1)_F
  & \text{if $(b,d) \ne (0,0)$, $b^2 - d^2 J \ne 0$}, \\
  & \text{\qquad $(a-bt)^2 - (c-dt)^2J \ne 0$.}
 \end{cases}
\]

Now we compute $\mu(\g) = z_{\Y}(\h_0 \g \h_0^{-1}, \h_0) \cdot z_{\Y}(\h_0, \g)^{-1}$.
Recall that $u = t^2$ and $\nu = a^2 - b^2 u - c^2 J + d^2 u J \ne 0$.
We have
\[
 (a-bt)^2 - (c-dt)^2 J = a^2 + b^2 u -c^2 J - d^2 uJ -2t(ab - cdJ).
\]
First assume that $b = d = 0$.
Then we have $c- dt = c$ and 
\[
 (a-bt)^2 - (c-dt)^2 J = a^2 - c^2 J = \nu \ne 0.
\]
Hence we have
\[
 \mu(\g) = 1 \cdot 1 = (\nu, J_1)_F
\]
if $c = 0$, and 
\[
 \mu(\g) = ((a-bt)^2 - (c-dt)^2 J, J_1)_F \cdot 1
 = (\nu, J_1)_F 
\]
if $c \ne 0$.
Next assume that $(b,d) \ne (0,0)$ and $b^2 - d^2 J = 0$.
Then we have $b \ne 0$, $d \ne 0$, $\nu = a^2 - c^2 J \ne 0$, and $J \in (F^{\times})^2$.
Since $(a-bt)^2 - (c-dt)^2 J = a^2 - c^2J - 2t(ab - cdJ)$ and
\begin{align*}
 & ( (a-bt)^2 - (c-dt)^2 J ) \cdot bd - (ad + bc - 2bdt) \cdot (ab - cd J) \\
 & = ( a^2 - c^2 J - 2abt + 2cdtJ ) \cdot bd
 - (a^2bd + ab^2c - 2ab^2dt - acd^2J - bc^2dJ + 2bcd^2tJ) \\
 & = -ab^2c + acd^2 J \\
 & = 0,
\end{align*}
we have
\[
 (a-bt)^2 - (c-dt)^2 J = 0 \Longleftrightarrow ad + bc - 2bdt = 0.
\]
If $c - dt = 0$, then we have $\nu = a^2 - b^2 u = (a+bt)(a-bt) \ne 0$,
so that $(a-bt)^2 - (c-dt)^2 J \ne 0$.
Hence we have
\begin{align*}
 \mu(\g) & = 1 \cdot \gamma_F(J_1, \frac{1}{2} \psi)
 \cdot ((2t- \frac{ad+bc}{bd}) \cdot J_2, J_1)_F \\
 & = \gamma_F(J_1, \frac{1}{2} \psi)
 \cdot ((\frac{2c}{d} - \frac{ad+bc}{bd}) \cdot J_2, J_1)_F \\
 & = \gamma_F(J_1, \frac{1}{2} \psi) \cdot (\frac{ad-bc}{bd}, J_1)_F \\
 & = \gamma_F(J_1, \frac{1}{2} \psi) \cdot (\frac{abd-b^2c}{d}, J_1)_F \\
 & = \gamma_F(J_1, \frac{1}{2} \psi) \cdot (ab-cdJ, J_1)_F.
\end{align*}
If $c - dt \ne 0$ and $(a-bt)^2 - (c-dt)^2 J = 0$, then we have
\[
 \nu - 2t(ab-cdJ) = (a-bt)^2 - (c-dt)^2 J = 0
\]
and hence 
\[
 \mu(\g)
 = \gamma_F(J_1, \frac{1}{2} \psi) \cdot (2 \nu t, J_1)_F \cdot 1
 = \gamma_F(J_1, \frac{1}{2} \psi) \cdot (ab-cdJ, J_1)_F.
\]
\begin{align*}
\end{align*}
If $c - dt \ne 0$ and $(a-bt)^2 - (c-dt)^2 J \ne 0$, then we have
\[
 (a-bt)^2 - (c-dt)^2 J = (\frac{ad+bd}{bd} -2t) \cdot (ab - cdJ)
\]
and hence
\begin{align*}
 \mu(\g) & = 
 ((a-bt)^2 - (c-dt)^2 J, J_1)_F \cdot 
 \gamma_F(J_1, \frac{1}{2} \psi)
 \cdot ((2t- \frac{ad+bc}{bd}) \cdot J_2, J_1)_F \\
 & = \gamma_F(J_1, \frac{1}{2} \psi) \cdot (-(ab-cdJ) \cdot J_2, J_1)_F \\
 & = \gamma_F(J_1, \frac{1}{2} \psi) \cdot (ab-cdJ, J_1)_F.
\end{align*}
Finally assume that $(b,d) \ne (0,0)$ and $b^2 - d^2 J \ne 0$.
Recall that
\[
 (ab-cdJ -(b^2-d^2J)t)^2 - (ad-bc)^2 J
 = ((a-bt)^2 - (c-dt)^2 J) \cdot (b^2 - d^2 J).
\]
If $c - dt = 0$, then we have $\nu = a^2 - b^2 u = (a+bt)(a-bt) \ne 0$
and 
\[
 (a-bt)^2 - (c-dt)^2 J = (a-bt)^2 \ne 0.
\]
Hence we have
\[
 \mu(\g) = 1 \cdot 
 (- \frac{((a-bt)^2 - (c-dt)^2 J) J}{b^2 - d^2J}, J_1)_F
 = (- \frac{(a-bt)^2 J}{b^2 - d^2J}, J_1)_F
 = (- (b^2 - d^2J) J, J_1)_F.
\]
If $c - dt \ne 0$ and $(a-bt)^2 - (c-dt)^2 J = 0$, then we have
\[
 \nu + 2(b^2 - d^2 J)u = a^2 + b^2 u - c^2 J - d^2 u J = 2(ab - cdJ)t,
\]
so that 
\[
 ab - cdJ - (b^2 - d^2 J) t \ne 0.
\]
Hence we have
\begin{align*}
 \mu(\g) & = 
 \gamma_F(J_1, \frac{1}{2} \psi) \cdot (2 \nu t, J_1)_F \cdot 
 \gamma_F(J_1, \frac{1}{2} \psi) \cdot (\frac{ab-cdJ}{b^2 - d^2J} - t, J_1)_F \\
 & = (-1, J_1)_F \cdot 
 (\frac{2 (ab-cdJ) t}{b^2 - d^2J} \cdot \nu - 2 \nu u, J_1)_F \\
 & = (-1, J_1)_F \cdot 
 (\frac{\nu + 2(b^2 - d^2 J)u}{b^2 - d^2J} \cdot \nu - 2 \nu u, J_1)_F \\
 & = (-1, J_1)_F \cdot (\frac{\nu^2}{b^2 - d^2J}, J_1)_F \\
 & = (- (b^2 - d^2J), J_1)_F \\
 & = (- (b^2 - d^2J) J, J_1)_F.
\end{align*}
If $c - dt \ne 0$ and $(a-bt)^2 - (c-dt)^2 J \ne 0$, then we have
\begin{align*}
 \mu(\g) & = ((a-bt)^2 - (c-dt)^2 J, J_1)_F
 \cdot (- \frac{((a-bt)^2 - (c-dt)^2 J) J}{b^2 - d^2J}, J_1)_F \\
 & = (- \frac{J}{b^2 - d^2J}, J_1)_F \\
 & = (- (b^2 - d^2J) J, J_1)_F.
\end{align*}
This completes the proof.
\end{proof}

\subsubsection{The case $J \in (F_v^{\times})^2$}

Choose $t \in F^{\times}$ such that $J = t^2$.
We take an isomorphism $\ii : B \rightarrow \M_2(F)$ determined by
\[
 \ii(1) = \begin{pmatrix} 1 & \\ & 1 \end{pmatrix}, \qquad
 \ii(\i) = \begin{pmatrix} & 1 \\ u & \end{pmatrix}, \qquad
 \ii(\j) = \begin{pmatrix} t & \\ & -t \end{pmatrix}, \qquad
 \ii(\i \j) = \begin{pmatrix} &  -t \\ tu & \end{pmatrix}.
\]
Then we have
\[
 e = \frac{1}{2} + \frac{1}{2t} \j, \qquad  
 e' = \frac{1}{2} \i - \frac{1}{2t} \i \j, \qquad 
 e'' = \frac{1}{2u} \i + \frac{1}{2tu} \i \j, \qquad 
 e^* = \frac{1}{2} - \frac{1}{2t} \j.
\]
Put 
\[
 \h_0 =
 \begin{pmatrix}
  \frac{1}{2} &  &  & \frac{t}{2J} &  &  &  & \\
  & \frac{1}{2} & \frac{t}{2 J_2} & &  &  & &  \\
  &  &  &  & -\frac{1}{2} &  &  & \frac{t}{2} \\
  &  &  &  &  & -\frac{1}{2} & \frac{t}{2 J_1} &  \\
  &  &  & & 1 &  &  & t \\
  &  & &  &  & 1 & \frac{t}{J_1} &  \\
  1 &  &  & -\frac{t}{J} &  &  &  & \\
  & 1 & -\frac{t}{J_2} &  &  &  &  & 
 \end{pmatrix} \in \Sp(\V).
\]
Then we have
\[
\begin{bmatrix}
 \e_1 e \\
 \e_2 e \\
 -\e_1 e'' \\
 \frac{1}{J_1} \e_2 e'' \\
 \frac{2}{u} \e_1 e' \\
 -\frac{2}{uJ_1} \e_2 e' \\
 2 \e_1 e^* \\
 2 \e_2 e^*
\end{bmatrix}
= \h_0
\begin{bmatrix}
 \e_1 \\
 \e_2 \\
 \e_3 \\
 \e_4 \\
 \e_1^* \\
 \e_2^* \\
 \e_3^* \\
 \e_4^* 
\end{bmatrix}, 
\]
and hence $\X' = \X \h_0$ and $\Y' = \Y \h_0$.

\begin{lem}
\label{lem:mu-GU(V)-J}
Let $\g_i := \ba_i^{-1} \in \GU(V)^0$ with $\ba_i = a_i + b_i \i + c_i \j_i + d_i \i \j_i \in B_i^{\times}$.
Then we have
\[
 \mu(\g_i) =
 \begin{cases}
 1 & \text{if $b_i = d_i = 0$,} \\
 \gamma_F(J_j, \frac{1}{2} \psi) \cdot ((a_i b_i + c_i d_i J_i) \nu_i J_i, J_j)_F
 & \text{if $(b_i, d_i) \ne (0,0)$ and $b_i^2 - d_i^2 J_i = 0$,} \\
 (-(b_i^2 - d_i^2 J_i) \nu_i J_i, J_j)_F
 & \text{if $(b_i, d_i) \ne (0,0)$ and $b_i^2 - d_i^2 J_i \ne 0$,}
 \end{cases}
\]
where $\nu_i = \nu(\ba_i)$ and $\{ i, j \} = \{ 1, 2 \}$.
\end{lem}

\begin{proof}
We only consider the case $i = 1$; the other case is similar.
Note that $J_1 \equiv J_2 \bmod (F^{\times})^2$ since $J \in (F^{\times})^2$.
Put $\d := d_{\Y}(\nu_1) \in \GSp(\V)$.
We have
\[
 z_{\Y}(\h_0 \g_1 \h_0^{-1}, \h_0) 
 = z_{\Y}(\h_0 \g_1 \h_0^{-1} \cdot \d^{-1}, \d \cdot \h_0 \cdot \d^{-1}) 
 \cdot v_{\Y}(\h_0, \nu_1).
\]
Since $\Y' \g_1 = \Y'$,
we have $\h_0 \g_1 \h_0^{-1} \cdot \d^{-1} \in P_{\Y}$ and hence
\[
 z_{\Y}(\h_0 \g_1 \h_0^{-1} \cdot \d^{-1}, \d \cdot \h_0 \cdot \d^{-1}) =1.
\]
We have $\h_0 = \m(\a_5) \cdot \n(\b_{15}) \cdot \tau_2 \cdot \m(\a_6)$, where
\[
 \a_5 = 
 \begin{pmatrix}
 1 & & & \\
 & 1 & & \\
 & & & -t \\
 & & - \frac{t}{J_1} &
 \end{pmatrix}, \qquad
 \b_{15} = \frac{1}{2t} \cdot 
 \begin{pmatrix}
 & & & 1 \\
 & & J_1 & \\
 & J_1 & & \\
 1 & & & 
 \end{pmatrix}, \qquad
 \a_6 = 
 \begin{pmatrix}
 1 & & & \\
 & 1 & & \\
 & - \frac{t}{J_1} & 1 & \\
 -t & & & 1
 \end{pmatrix},
\]
so that $x_{\Y}(\h_0) \equiv -J_1 \bmod (F^{\times})^2$ and $j_{\Y}(\h_0) = 2$.
Hence we have
\[
 v_{\Y}(\h_0, \nu_1) = (-J_1, \nu_1)_F \cdot \gamma_F(\nu_1, \frac{1}{2} \psi)^{-2} = (J_1, \nu_1)_F.
\]
Thus we obtain 
\[
 z_{\Y}(\h_0 \g_1 \h_0^{-1}, \h_0) = (J_1, \nu_1)_F = (J_2, \nu_1)_F.
\]
Moreover, if $b_1 = d_1 = 0$, then we have
\[
 (J_1, \nu_1)_F = (J_1, a_1^2 - c_1^2 J_1)_F = 1.
\]

Now we compute $z_{\Y}(\h_0, \g_1)$.
We have
\[
 z_{\Y}(\h_0, \g_1) = z_{\Y}(\h_0, \g_1 \cdot \d^{-1}).
\]
First assume that $b_1 = d_1 = 0$.
Then we have $\g_1 \cdot \d^{-1} \in P_{\Y}$ and hence
\[
 z_{\Y}(\h_0, \g_1 \cdot \d^{-1}) = 1.
\]
Next assume that $(b_1, d_1) \ne (0,0)$ and $b_1^2 - d_1^2 J_1 = 0$.
Then we have $b_1 \ne 0$ and $d_1 \ne 0$.
As in the proof of Lemma \ref{lem:mu-GU(V)-u},
we have $a_1 b_1 + c_1 d_1 J_1 \ne 0$.
We have $\g_1 \cdot \d^{-1} \in \m(\a_1) \cdot \n(\b_3) \cdot \tau_2 \cdot P_{\Y}$, where $\a_1$ and $\b_3$ are as in the proof of Lemma \ref{lem:mu-GU(V)-u}.
Hence we have
\[
 z_{\Y}(\h_0, \g_1 \cdot \d^{-1})
 = z_{\Y}(\tau_2 \cdot \m(\a_6), \m(\a_1) \cdot \n(\b_3) \cdot \tau_2)
 = z_{\Y}(\tau_2, \m(\a_6) \cdot \m(\a_1) \cdot \n(\b_3) \cdot \tau_2).
\]
We have
\[
 \m(\a_6) \cdot \m(\a_1) \cdot \n(\b_3) \cdot \tau_2 
 \in \m(\a_7) \cdot \n(\b_{16}) \cdot \tau' \cdot P_{\Y},
\]
where
\[
 \a_7 = 
 \begin{pmatrix}
  b_1 & & & \\
  d_1 J_1 & 1 & & \\
  & & b_1 & \\
  & & d_1 J_1 & 1
 \end{pmatrix}, \qquad
 \b_{16} = \frac{(a_1 d_1 - b_1 c_1) J_1}{b_1} \cdot 
 \begin{pmatrix}
  0 & & & \\
  & \frac{1}{d_1} & & \frac{t}{b_1} \\
  & & 0 & \\
  & \frac{t}{b_1} & & 0
 \end{pmatrix},
\]
and
\[
 \tau' = 
 \begin{pmatrix}
  1 & & & & & & & \\
  & 0 & & & & -1 & & \\
  & & 1 & & & & & \\
  & & & 0 & & & & -1 \\
  & & & & 1 & & & \\
  & 1 & & & & 0 & & \\
  & & & & & & 1 & \\
  & & & 1 & & & & 0
 \end{pmatrix}.
\]
Hence we have
\[
 z_{\Y}(\tau_2, \m(\a_6) \cdot \m(\a_1) \cdot \n(\b_3) \cdot \tau_2)
 = z_{\Y}(\tau_2, \m(\a_7) \cdot \n(\b_{16}) \cdot \tau')
 = z_{\Y}(\tau_2 \cdot \m(\a_7) \cdot \n(\b_{16}), \tau').
\]
Since $\tau_2 \cdot \m(\a_7) \cdot \n(\b_{16}) \cdot \tau_2^{-1} \in P_{\Y}$, we have
\[
 z_{\Y}(\tau_2 \cdot \m(\a_7) \cdot \n(\b_{16}), \tau')
 = z_{\Y}(\tau_2, \tau') = 1.
\]
On the other hand, since $J \in (F^{\times})^2$ and $J_1 \in (F^{\times})^2$,
we have $\gamma_F(J_2, \frac{1}{2} \psi) = 1$ and 
\[
 ((a_1 b_1 + c_1 d_1 J_1) J_1, J_2)_F = 1.
\]
Finally assume that $(b_1, d_1) \ne (0,0)$ and $b_1^2 - d_1^2 J_1 \ne 0$.
We have $\g_1 \cdot \d^{-1} \in \n(\b_5) \cdot \tau_4 \cdot P_{\Y}$,
where $\b_5$ is as in the proof of Lemma \ref{lem:mu-GU(V)-u}.
Hence we have
\begin{align*}
 z_{\Y}(\h_0, \g_1 \cdot \d^{-1})  
 & = z_{\Y}(\tau_2 \cdot \m(\a_6), \n(\b_5) \cdot \tau_4) \\
 & = z_{\Y}(\tau_2 \cdot \m(\a_6), \n(\b_5) \cdot \m(\a_6)^{-1} \cdot \tau_4) \\
 & = z_{\Y}(\tau_2 \cdot \m(\a_6) \cdot \n(\b_5) \cdot \m(\a_6)^{-1}, \tau_4).
\end{align*}
Since $\tau_2 \cdot \m(\a_6) \cdot \n(\b_5) \cdot \m(\a_6)^{-1} \cdot \tau_2^{-1} \in P_{\Y}$, we have
\[
 z_{\Y}(\h_0, \g_1 \cdot \d^{-1}) = z_{\Y}(\tau_2, \tau_4) = 1.
\]
On the other hand, we have
\[
 (-(b_1^2 - d_1^2 J_1) J_1, J_2)_F = (d_1^2 J_1^2 - b_1^2 J_1, J_1)_F = 1.
\]
This competes the proof.
\end{proof}

\begin{lem}
\label{lem:mu-GU(W)-J}
Let $\g := \ba \in \GU(W)$ with $\ba = a + b \i + c \j + d \i \j \in B^{\times}$.
Then we have
\begin{align*}
 \mu(\g) & =
 \begin{cases}
 (\nu, J_1)_F & \text{if $b = d = 0$,} \\
 \gamma_F(J_1, \frac{1}{2} \psi) \cdot (ab - cdJ, J_1)_F
  & \text{if $(b,d) \ne (0,0)$ and $b^2 - d^2 J = 0$,} \\
 (-(b^2 - d^2 J) J, J_1)_F
  & \text{if $(b,d) \ne (0,0)$ and $b^2 - d^2 J \ne 0$,}
 \end{cases} \\
 & \times 
 \begin{cases}
  1 & \text{if $b + dt = 0$,} \\
  (u, J_1)_F & \text{if $b + dt \ne 0$,} \\
 \end{cases} \\
\end{align*}
where $\nu = \nu(\ba)$.
\end{lem}

\begin{proof}
Put $\d := d_{\Y}(\nu) \in \GSp(\V)$.
We have
\[
 z_{\Y}(\h_0 \g \h_0^{-1}, \h_0) 
 = z_{\Y}(\h_0 \g \h_0^{-1} \cdot \d^{-1}, \d \cdot \h_0 \cdot \d^{-1}) 
 \cdot v_{\Y}(\h_0, \nu).
\]
As in the proof of Lemma \ref{lem:mu-GU(V)-J}, we have $v_{\Y}(\h_0, \nu) = (\nu, J_1)_F$.
We have
\[
 \h_0 \g \h_0^{-1} = 
 \begin{pmatrix}
 a+ct & & & & \frac{(b-dt)u}{2} & & & \\
 & a+ct & & & & - \frac{(b-dt) u J_1}{2} & & \\
 & & a+ct & & & & - \frac{b-dt}{2} & \\
 & & & a+ct & & & & \frac{b-dt}{2 J_1} \\
 2(b+dt) & & & & a-ct & & & \\
 & - \frac{2(b+dt)}{J_1} & & & & a-ct & & \\
 & & -2(b+dt)u & & & &  a-ct & \\
 & & & 2 (b+dt) u J_1 & & & & a-ct 
 \end{pmatrix}.
\]
If $b+dt = 0$, then we have $\h_0 \g \h_0^{-1} \cdot \d^{-1} \in P_{\Y}$
and hence 
$z_{\Y}(\h_0 \g \h_0^{-1} \cdot \d^{-1}, \d \cdot \h_0 \cdot \d^{-1}) = 1$.
If $b+dt \ne 0$, then we have
$\h_0 \g \h_0^{-1} \cdot \d^{-1} \in P_{\Y} \cdot \tau_4 \cdot \n(\b_{17})$, where
\[
 \b_{17} = \frac{a-ct}{2 \nu (b+dt)} \cdot 
 \begin{pmatrix}
  1 & & & \\
  & -J_1 & & \\
  & & - \frac{1}{u} & \\
  & & & \frac{1}{u J_1}
 \end{pmatrix}.
\]
We have $\d \cdot \h_0 \cdot \d^{-1} \in \m(\a_5) \cdot \n(\nu^{-1} \cdot \b_{15}) \cdot \tau_2 \cdot P_{\Y}$,
where $\a_5$ and $\b_{15}$ are as in the proof of Lemma \ref{lem:mu-GU(V)-J}.
Hence we have
\begin{align*}
 z_{\Y}(\h_0 \g \h_0^{-1} \cdot \d^{-1}, \d \cdot \h_0 \cdot \d^{-1})  
 & = z_{\Y}(\tau_4 \cdot \n(\b_{17}), \m(\a_5) \cdot \n(\nu^{-1} \cdot \b_{15}) \cdot \tau_2) \\
 & = z_{\Y}(\tau_4 \cdot \m(\a_5), \m(\a_5)^{-1} \cdot \n(\b_{17}) \cdot \m(\a_5) \cdot \n(\nu^{-1} \cdot \b_{15}) \cdot \tau_2) \\
 & = z_{\Y}(\tau_4, \n(\b_{18}) \cdot \tau_2),
\end{align*}
where $\b_{18} = \nu^{-1} \cdot \b_{15} + \a_5^{-1} \cdot \b_{17} \cdot {}^t \a_5^{-1}$.
Put $r = \frac{a-ct}{b+dt}$.
We have
\[
 \b_{18} = \frac{1}{2 \nu t} \cdot 
 \begin{pmatrix}
  rt & & & 1 \\
  & -rt J_1 & J_1 & \\
  & J_1 & \frac{rJ_1}{tu} & \\
  1 & & & - \frac{r}{tu}
 \end{pmatrix}.
\]
We write $\b_{18} = \b_{19} + \b_{20}$, where
\[
 \b_{19} = \frac{1}{2 \nu t} \cdot 
 \begin{pmatrix}
  rt & & & 1 \\
  & -rt J_1 & J_1 & \\
  & J_1 & & \\
  1 & & &
 \end{pmatrix}, \qquad
 \b_{20} = \frac{r}{2 \nu u J} \cdot 
 \begin{pmatrix}
  0 & & & \\
  & 0 & & \\
  & & J_1 & \\
  & & & -1
 \end{pmatrix}.
\]
Since $\tau_2 ^{-1} \cdot \n(\b_{19}) \cdot \tau_2 \in P_{\Y}$, we have
\[
 z_{\Y}(\tau_4, \n(\b_{18}) \cdot \tau_2)
 = z_{\Y}(\tau_4, \n(\b_{20}) \cdot \tau_2).
\]
If $r=0$, then we have $z_{\Y}(\tau_4, \n(\b_{20}) \cdot \tau_2) = z_{\Y}(\tau_4, \tau_2) = 1$.
If $r \ne 0$, then we have $z_{\Y}(\tau_4, \n(\b_{20}) \cdot \tau_2) = \gamma_F(\frac{1}{2} \psi \circ q_8)$,
where $q_8$ is a non-degenerate symmetric bilinear form associated to
\[
 \frac{r}{2 \nu u J} \cdot 
 \begin{pmatrix}
  J_1 & \\
  & -1
 \end{pmatrix}.
\]
We have $\det q_8 \equiv -J_1 \bmod (F^{\times})^2$ and
\[
 h_F(q_8) = (\frac{r J_1}{2 \nu u J}, - \frac{r}{2 \nu u J})_F
 = (J_1, - 2 \nu r u)_F.
\]
Hence we have
\[
 \gamma_F(\frac{1}{2} \psi \circ q_8)
 = \gamma_F(\frac{1}{2} \psi)^2 \cdot
 \gamma_F(-J_1, \frac{1}{2} \psi) \cdot (J_1, - 2 \nu r u)_F
 = \gamma_F(J_1, \frac{1}{2} \psi) \cdot (J_1, 2 \nu r u)_F.
\]
Thus we obtain
\[
 z_{\Y}(\h_0 \g \h_0^{-1}, \h_0) =
 \begin{cases}
  (\nu, J_1)_F & \text{if $b + dt = 0$,} \\
  (\nu, J_1)_F & \text{if $b + dt \ne 0$, $a - ct = 0$,} \\
 \gamma_F(J_1, \frac{1}{2} \psi) \cdot (2u \cdot \frac{a-ct}{b+dt}, J_1)_F
  & \text{if $b + dt \ne 0$, $a - ct \ne 0$.}
 \end{cases}
\]

Now we compute $z_{\Y}(\h_0, \g)$.
We have
\[
 z_{\Y}(\h_0, \g) = z_{\Y}(\h_0, \g \cdot \d^{-1}).
\]
First assume that $b = d = 0$.
Then we have $\g \cdot \d^{-1} \in P_{\Y}$ and hence
\[
 z_{\Y}(\h_0, \g \cdot \d^{-1}) = 1.
\]
Next assume that $(b, d) \ne (0,0)$ and $b^2 - d^2 J = 0$.
Then we have $b \ne 0$, $d \ne 0$, and $\nu = a^2 - c^2 J \ne 0$.
Since 
\begin{align*}
 (ad + bc) \cdot (ab - cdJ)
 & = a^2 bd - ac d^2 J + a b^2 c  - b c^2 d J \\
 & = a^2 bd - b c^2 d J \\
 & = \nu bd \\
 & \ne 0,
\end{align*}
we have $ad + bc \ne 0$ and $ab - cd J \ne 0$.
We have $\g \cdot \d^{-1} \in \m(\a_3) \cdot \n(\b_9) \cdot \tau_2 \cdot P_{\Y}$, where $\a_3$ and $\b_9$ are as in the proof of Lemma \ref{lem:mu-GU(W)-u}.
Hence we have
\[
 z_{\Y}(\h_0, \g \cdot \d^{-1}) = 
 z_{\Y}(\tau_2 \cdot \m(\a_6), \m(\a_3) \cdot \n(\b_9) \cdot \tau_2)
 = z_{\Y}(\tau_2 \cdot \m(\a_8), \n(\b_9) \cdot \tau_2),
\]
where $\a_6$ is as in the proof of Lemma \ref{lem:mu-GU(V)-J} and 
\[
 \a_8 = \a_6 \cdot \a_3 = 
 \begin{pmatrix}
  b & & & \\
  & b & & \\
  & -\frac{(b + dt) t}{J_1} & 1 & \\
 - (b + dt) t & & & 1
 \end{pmatrix}.
\]
If $b+dt = 0$, then we have $\tau_2 \cdot \m(\a_8) \cdot \tau_2^{-1} \in P_{\Y}$ and hence
\[
 z_{\Y}(\tau_2 \cdot \m(\a_8), \n(\b_9) \cdot \tau_2)
 = z_{\Y}(\tau_2, \n(\b_9) \cdot \tau_2)
 = \gamma_F(\frac{1}{2} \psi \circ q_9),
\]
where $q_9$ is a non-degenerate symmetric bilinear form associated to
\[
 \frac{ad+bc}{bd} \cdot 
 \begin{pmatrix}
  -J_2 & \\
  & J
 \end{pmatrix}.
\]
We have $\det q_9 \equiv - J_1 \bmod (F^{\times})^2$ and
\[
 h_F(q_9) = (- \frac{ad+bc}{bd} \cdot J_2, \frac{ad+bc}{bd} \cdot J)_F 
 = (J_2, \frac{ad+bc}{bd})_F = (J_2, \frac{\nu}{ab - cd J})_F.
\]
Hence we have
\[
 \gamma_F(\frac{1}{2} \psi \circ q_9) =
 \gamma_F(\frac{1}{2} \psi)^2 \cdot 
 \gamma_F(-J_1, \frac{1}{2} \psi) \cdot (J_2, \frac{\nu}{ab - cd J})_F
 = \gamma_F(J_1, \frac{1}{2} \psi)^{-1} \cdot (J_1, (ab - cd J) \nu)_F.
\]
If $b+dt \ne 0$ ,then we have
\[
 \m(\a_8) \cdot \n(\b_9) \cdot \tau_2 \in
 \m(\a_9) \cdot \n(\b_{21}) \cdot \tau'' \cdot P_{\Y},
\]
where
\[
 \a_9 =
 \begin{pmatrix}
  1 & & & -1 \\
  & 1 & -1 & \\
  & & \frac{(b+dt)t}{bJ_1} & \\
  & & & \frac{(b+dt)t}{b}
 \end{pmatrix}, \qquad
 \b_{21} = \frac{(ad+bc) b}{(b+dt)^2 d} \cdot 
 \begin{pmatrix}
  1 & & & 1 \\
  & -J_1 & & \\
  & & 0 & \\
  1 & & & 0
 \end{pmatrix},
\]
and
\[
 \tau'' =
 \begin{pmatrix}
  & & -\1_2 & \\
  & \1_2 & & \\
  \1_2 & & & \\
  & & & \1_2
 \end{pmatrix}.
\]
Since $\tau_2 \cdot \m(\a_9) \cdot \n(\b_{21}) \cdot \tau_2^{-1} \in P_{\Y}$, we have
\[
 z_{\Y}(\tau_2 \cdot \m(\a_8), \n(\b_9) \cdot \tau_2) 
 = z_{\Y}(\tau_2, \m(\a_9) \cdot \n(\b_{21}) \cdot \tau'')
 = z_{\Y}(\tau_2, \tau'')
 = 1.
\]
Finally assume that $(b, d) \ne (0,0)$ and $b^2 - d^2 J \ne 0$.
We have $\g \cdot \d^{-1} \in \n(\b_{13}) \cdot \tau_4 \cdot P_{\Y}$,
where $\b_{13}$ is as in the proof of Lemma \ref{lem:mu-GU(W)-u}.
Hence we have
\[
 z_{\Y}(\h_0, \g \cdot \d^{-1})
 = z_{\Y}(\tau_2 \cdot \m(\a_6), \n(\b_{13}) \cdot \tau_4)
 = z_{\Y}(\tau_2, \n(\b_{22}) \cdot \tau_4),
\]
where
\begin{align*}
 \b_{22} & = \a_6 \cdot \b_{13} \cdot {}^t \a_6 \\
 & = \frac{1}{b^2 - d^2 J} \cdot 
 \begin{pmatrix}
 ab -cdJ & & & - (a - ct) (b + dt) t \\
 & - (ab - cdJ) J_1 & (a - ct) (b + dt) t & \\
 & (a - ct) (b + dt)  t & - 2 (a - ct) (b + dt) J_2 & \\
 - (a - ct) (b + dt) t & & & 2 (a - ct) (b + dt) J
 \end{pmatrix}.
\end{align*}
We write $\b_{22} = \b_{23} + \b_{24}$, where
\begin{align*}
 \b_{23} & = 
 \frac{1}{b^2 - d^2 J} \cdot 
 \begin{pmatrix}
 ab -cdJ & & & - (a - ct) (b + dt) t \\
 & - (ab - cdJ) J_1 & (a - ct) (b + dt) t & \\
 & (a - ct) (b + dt)  t & & \\
 - (a - ct) (b + dt) t & & & 
 \end{pmatrix}, \\
 \b_{24} & = 
 \frac{2 (a-ct)}{b-dt} \cdot 
 \begin{pmatrix}
 0 & & & \\
 & 0 & & \\
 & & - J_2 & \\
 & & & J
 \end{pmatrix}.
\end{align*}
Since $\tau_2 \cdot \b_{23} \cdot \tau_2^{-1} \in P_{\Y}$,
we have $z_{\Y}(\tau_2, \n(\b_{22}) \cdot \tau_4) = z_{\Y}(\tau_2, \n(\b_{24}) \cdot \tau_4)$.
If $a - ct = 0$,
then we have $z_{\Y}(\tau_2, \n(\b_{24}) \cdot \tau_4) = z_{\Y}(\tau_2, \tau_4) = 1$.
If $a - ct \ne 0$, then we have
$z_{\Y}(\tau_2, \n(\b_{24}) \cdot \tau_4) = \gamma_F(\frac{1}{2} \psi \circ q_{10})$,
where $q_{10}$ is a non-degenerate symmetric bilinear form associated to
\[
 \frac{2 (a-ct)}{b-dt} \cdot 
 \begin{pmatrix}
 - J_2 & \\
 & J
 \end{pmatrix}.
\]
We have $\det q_{10} \equiv -J_1 \bmod (F^{\times})^2$ and 
\[
 h_F(q_{10}) = (- \frac{2 (a-ct)}{b-dt} \cdot J_2, \frac{2 (a-ct)}{b-dt} \cdot J)_F = (J_2, \frac{2 (a-ct)}{b-dt})_F.
\]
Hence we have
\[
 \gamma_F(\frac{1}{2} \psi \circ q_{10}) 
 = \gamma_F(\frac{1}{2} \psi)^2 \cdot \gamma_F(-J_1, \frac{1}{2} \psi)
 \cdot (J_2, \frac{2 (a-ct)}{b-dt})_F
 = \gamma_F(J_1, \frac{1}{2} \psi)^{-1} \cdot (J_1, \frac{2 (a-ct)}{b-dt})_F.
\]
Thus we obtain
\[
 z_{\Y}(\h_0, \g) = 
 \begin{cases}
  1 & \text{if $b = d = 0$,} \\
  \gamma_F(J_1, \frac{1}{2} \psi)^{-1} \cdot (J_1, (ab - cd J) \nu)_F 
  & \text{if $(b,d) \ne (0,0)$, $b^2 - d^2 J = 0$, $b+dt = 0$,} \\
  1 & \text{if $(b,d) \ne (0,0)$, $b^2 - d^2 J = 0$, $b+dt \ne 0$,} \\
  1 & \text{if $(b,d) \ne (0,0)$, $b^2 - d^2 J \ne 0$, $a-ct = 0$,} \\
  \gamma_F(J_1, \frac{1}{2} \psi)^{-1} \cdot (J_1, \frac{2 (a-ct)}{b-dt})_F
  & \text{if $(b,d) \ne (0,0)$, $b^2 - d^2 J \ne 0$, $a+ct \ne 0$.}
 \end{cases}
\]

Now we compute $\mu(\g) = z_{\Y}(\h_0 \g \h_0^{-1}, \h_0) \cdot z_{\Y}(\h_0, \g)^{-1}$.
Recall that $J = t^2$ and $\nu = a^2 - b^2 u - c^2 J + d^2 u J \ne 0$.
First assume that $b = d = 0$.
Then we have 
\[
 \mu(\g) = (\nu, J_1)_F \cdot 1 = (\nu, J_1)_F.
\]
Next assume that $(b,d) \ne (0,0)$ and $b^2 - d^2 J = 0$.
Then we have $\nu = a^2 - c^2 J = (a+ct)(a-ct) \ne 0$.
Since $b^2 - d^2 J = (b + dt)(b-dt)$, we have
\[
 b+dt = 0 \Longleftrightarrow b-dt \ne 0.
\]
If $b+dt = 0$, then we have
\[
 \mu(\g) = (\nu, J_1)_F \cdot
 \gamma_F(J_1, \frac{1}{2} \psi) \cdot (J_1, (ab - cd J) \nu)_F
 = \gamma_F(J_1, \frac{1}{2} \psi) \cdot (J_1, ab - cd J)_F.
\]
If $b+dt \ne 0$, then we have
\[
 (a-ct)(b+dt) = 2 (a-ct) dt = 2 (adt - cd J) = 2(ab - cd J).
\]
Hence we have
\begin{align*}
 \mu(\g) & =
 \gamma_F(J_1, \frac{1}{2} \psi) \cdot (2u \cdot \frac{a-ct}{b+dt}, J_1)_F
 \cdot 1 \\
 & = \gamma_F(J_1, \frac{1}{2} \psi)
 \cdot (2 (a-ct)(b+dt), J_1)_F \cdot (u, J_1)_F \\
 & = \gamma_F(J_1, \frac{1}{2} \psi)
 \cdot (ab - cdJ, J_1)_F \cdot (u, J_1)_F.
\end{align*}
Finally assume that $(b,d) \ne (0,0)$ and $b^2 - d^2 J \ne 0$.
Then we have $b + dt \ne 0$.
If $a - ct = 0$, then we have $\nu = -b^2 u + d^2 uJ$ and hence 
\[
 \mu(\g) = (\nu, J_1)_F \cdot 1 = (-b^2 + d^2 J, J_1)_F \cdot (u, J_1)_F
= (-(b^2 - d^2 J) J, J_1)_F \cdot (u, J_1)_F.
\]
If $a - ct \ne 0$, then we have
\begin{align*}
\mu(\g) & = 
 \gamma_F(J_1, \frac{1}{2} \psi) \cdot (2u \cdot \frac{a-ct}{b+dt}, J_1)_F
 \cdot
 \gamma_F(J_1, \frac{1}{2} \psi) \cdot (J_1, \frac{2 (a-ct)}{b-dt})_F \\
 & = \gamma_F(J_1, \frac{1}{2} \psi)^2 \cdot (u (b+dt)(b-dt), J_1)_F \\
 & = (-1, J_1)_F \cdot (b^2 - d^2 J, J_1)_F \cdot (u, J_1)_F \\
 & = (-(b^2 - d^2 J) J, J_1)_F \cdot (u, J_1)_F.
\end{align*}
This completes the proof.
\end{proof}

\subsection{The case $J_i \in (F_v^{\times})^2$}
\label{ss:B1-spl}
We only consider the case $i=1$; the other case is similar.
Choose $t \in F^{\times}$ such that $J_1 = t^2$.
We take an isomorphism
\[
 \ii_1 : B_1 \longrightarrow \M_2(F)
\]
of $F$-algebras determined by
\[
 \ii_1(1) = \begin{pmatrix} 1 & \\ & 1 \end{pmatrix}, \qquad
 \ii_1(\i) = \begin{pmatrix} & 2 \\ \frac{u}{2} & \end{pmatrix}, \qquad
 \ii_1(\j_1) = \begin{pmatrix} t & \\ & -t \end{pmatrix}, \qquad
 \ii_1(\i \j_1) = \begin{pmatrix} & -2t \\ \frac{tu}{2} & \end{pmatrix}.
\]
Note that
\[
 \ii_1(\ba_1^*) = \ii_1(\ba_1)^*
\]
for $\ba_1 \in B_1$.
Let
\[
 \v := \frac{1}{2} \e_1 + \frac{1}{2 t} \e_2, \qquad
 \v^* := \e_1^* + t \e_2^*
 = \frac{1}{u} \e_1 \i - \frac{1}{tu} \e_2 \i.
\]
Then we have
\[
 V = \v B + \v^* B
\]
and 
\[
 \langle \v, \v \rangle = \langle \v^*, \v^* \rangle = 0, \qquad
 \langle \v, \v^* \rangle = 1.
\]
Moreover, we see that
\[
 \begin{bmatrix} \ba_1 \cdot \v & \ba_1 \cdot \v^* \end{bmatrix}
 = \begin{bmatrix} \v & \v^* \end{bmatrix} \cdot \ii_1(\ba_1)
\]
for $\ba_1 \in B_1$, and
\[
 \begin{bmatrix} \ba_2 \cdot \v & \ba_2 \cdot \v^* \end{bmatrix}
 = \begin{bmatrix} \v & \v^* \end{bmatrix} \cdot (\alpha + \frac{\beta}{t} \j)
\]
for $\ba_2 = \alpha + \beta \j_2 \in B_2$ with $\alpha, \beta \in E$.

We regard $V' := V$ as a left $B$-space by putting
\[
 \ba \cdot x' := (x \cdot \ba^*)'
\]
for $\ba \in B$ and $x' \in V'$.
Here we let $x'$ denote the element in $V'$ corresponding to $x \in V$.
We let $\GL(V')$ act on $V'$ on the right.
We define a skew-hermitian form
\[
 \langle \cdot, \cdot \rangle' : V' \times V' \longrightarrow B
\]
by 
\[
 \langle x', y' \rangle' := \langle x, y \rangle.
\]
Note that
\[
 \langle \ba x', \bb y' \rangle' = \ba \langle x', y' \rangle' \bb^*
\]
for $\ba, \bb \in B$.
For $x' \in V'$ and $g \in \GL(V)$, put
\[
 x' \cdot g := (g^{-1} \cdot x)'.
\]
Then we have an isomorphism
\begin{align*}
 \GL(V) & \longrightarrow \GL(V'), \\
 g & \longmapsto \left[ x' \mapsto x' \cdot g \right]
\end{align*}
so that we may identify $\GU(V)$ with $\GU(V')$ via this isomorphism.
Let $V' = X' + Y'$ be a complete polarization given by
\[
 X' = B \cdot \v', \qquad Y' = B \cdot (\v^*)'.
\]
Note that
\[
 \begin{bmatrix} \v' \cdot \ba \\ (\v^*)' \cdot \ba \end{bmatrix}
 = {}^t \ii_1(\ba)^{-1} \cdot
 \begin{bmatrix} \v' \\ (\v^*)' \end{bmatrix} 
\]
for $\ba \in B_1$.
We may identify $V'$ with the space of row vectors $B^2$ so that
\[
 \langle x', y' \rangle' = x_1 y_2^* - x_2 y_1^*
\]
for $x' = (x_1, x_2)$, $y' = (y_1, y_2) \in V'$.
Then we may write
\[
 \GU(V') = \left\{ 
 g \in \GL_2(B) \, \left| \, g 
 \begin{pmatrix} & 1 \\ -1 & \end{pmatrix} {}^t g^*
 = \nu(g) \cdot \begin{pmatrix} & 1 \\ -1 & \end{pmatrix} \right. \right\}.
\]

Similarly, we have a right $B$-space $W' := W$ with a hermitian form
\[
 \langle \cdot, \cdot \rangle': W' \times W' \longrightarrow B.
\]
We let $\GL(W')$ act on $W'$ on the left.
Now we consider an $F$-space
\[
 \V' := W' \otimes_B V'
\]
with a symplectic form
\[
 \llangle \cdot, \cdot \rrangle' := 
 \frac{1}{2} \tr_{B/F} (\langle \cdot, \cdot \rangle' \otimes \langle \cdot, \cdot \rangle'^*).
\]
We let $\GL(\V')$ act on $\V'$ on the right.
For $\mathbf{x} = x \otimes y \in \V$ and $\g \in \GL(\V)$, put
\[
 \mathbf{x}' := y' \otimes x' \in \V'
\]
and 
\[
 \mathbf{x}' \cdot \g := (\mathbf{x} \cdot \g)'.
\]

\begin{lem}
\label{lem:GSp(V)-GSp(V')-identify}
We have an isomorphism
\begin{align*}
 \GSp(\V) & \longrightarrow \GSp(\V'). \\
 \g & \longmapsto \left[ \mathbf{x}' \mapsto \mathbf{x}' \cdot \g \right]
\end{align*}
Moreover, this isomorphism induces a commutative diagram
\[
 \xymatrix{
  \GU(V) \times \GU(W) \ar@{->}[r] \ar@{->}[d] &
  \GSp(\V) \ar@{->}[d] \\
  \GU(W') \times \GU(V') \ar@{->}[r] & \GSp(\V')
 }.
\]
\end{lem}

\begin{proof}
For $x_1, x_2 \in V$ and $y_1, y_2 \in W$, we have
\begin{align*}
 \llangle y_1' \otimes x_1', y_2' \otimes x_2' \rrangle'
 & =
 \frac{1}{2} \tr_{B/F} (\langle y_1', y_2' \rangle' \cdot \langle x_1', x_2' \rangle'^*) \\
 & = \frac{1}{2} \tr_{B/F} (\langle y_1, y_2 \rangle \cdot \langle x_1, x_2 \rangle^*) \\
 & = \frac{1}{2} \tr_{B/F} (\langle x_1, x_2 \rangle \cdot \langle y_1, y_2 \rangle^*) \\
 & = \llangle x_1 \otimes y_1, x_2 \otimes y_2 \rrangle.
\end{align*}
Also, for $\g = (g, h) \in \GL(V) \times \GL(W)$ and $\mathbf{x} = x \otimes y \in \V$, we have
\[
 \mathbf{x}' \cdot \g = ( (x \otimes y) \cdot (g, h) )'
 = (g^{-1} x \otimes y h)' = (yh)' \otimes (g^{-1} x)' = h^{-1} y' \otimes x' g.
\]
This completes the proof.
\end{proof}

Thus we may identify $\GSp(\V)$ with $\GSp(\V')$ and $\GU(V) \times \GU(W)$ with $\GU(W') \times \GU(V')$ respectively.

Let $\V' = \X' + \Y'$ be a complete polarization given by
\[
 \X' = W' \otimes_B X', \qquad \Y' = W' \otimes_B Y'.
\]
Put
\[
 s'(g) := \gamma^{j(g)}
\]
for $g \in \GU(V')^0$, where
\[
 \gamma = 
 \begin{cases}
  1 & \text{if $B$ and $B_2$ are split,} \\
  -1 & \text{if $B$ and $B_2$ are ramified,}
 \end{cases}
\]
and
\[
 j(g) = 
 \begin{cases}
  0 & \text{if $g = \left( \begin{smallmatrix} * & * \\ 0 & * \end{smallmatrix} \right)$,} \\
  1 & \text{othersiwe.}
 \end{cases}
\]

\begin{lem}
\label{lem:spl-GU(V)}
We have
\[
 z_{\Y'}(g, g') = s'(g g') \cdot s'(g)^{-1} \cdot s'(g')^{-1}
\]
for $g, g' \in \GU(V)^0$.
\end{lem}

\begin{proof}
The proof is similar to that of Lemma \ref{lem:spl-GSp}.
If $B$ is ramified, then we have
\begin{equation}
\label{eq:kudla-spl-2}
 z_{\Y'}(g, g') = s'(g g') \cdot s'(g)^{-1} \cdot s'(g')^{-1} 
\end{equation}
for $g, g' \in \U(V)^0$ by \cite[Theorem 3.1, case $2_+$]{kudla-splitting}.
If $B$ is split, then we see that \eqref{eq:kudla-spl-2} also holds by using Morita theory as in \S \ref{ss:B-spl} and \cite[Theorem 3.1, case $1_-$]{kudla-splitting}.

Let $g, g' \in \GU(V)^0$.
For $\nu \in F^{\times}$, put
\[
 d(\nu) = \begin{pmatrix} 1 & \\ & \nu \end{pmatrix} \in \GU(V)^0.
\]
We write
\[
 g = h \cdot d(\nu), \qquad g' = h' \cdot d(\nu')
\]
with $h, h' \in \U(V)^0$ and $\nu, \nu' \in F^{\times}$.
Then we have
\[
 z_{\Y'}(g, g') = z_{\Y'}(h, h'') \cdot v_{\Y'}(h', \nu),
\]
where
\[
 h'' = d(\nu) \cdot h' \cdot d(\nu)^{-1}.
\]
By \eqref{eq:kudla-spl-2}, we have
\[
 z_{\Y'}(h, h'') = s'(h h'') \cdot s'(h)^{-1} \cdot s'(h'')^{-1}.
\]
We have $s'(h) = s'(g)$, and since $j(h'') = j(h')$, we have $s'(h'') = s'(h') = s'(g')$.
Moreover, since $g g' = h h'' \cdot d(\nu \nu')$, we have $s'(h h'') = s'(g g')$.
Thus we obtain
\[
 z_{\Y'}(h, h'') = s'(g g') \cdot s'(g)^{-1} \cdot s'(g')^{-1}.
\]
By Lemma \ref{lem:v_Y}, we have
\[
 v_{\Y'}(h', \nu) =
 (x_{\Y'}(h'), \nu)_F \cdot \gamma_F(\nu, \frac{1}{2} \psi)^{-j_{\Y'}(h')},
\]
where $x_{\Y'}$ and $j_{\Y'}$ are as in \S \ref{ss:weil-GMp-setup}
with respect to the complete polarization $\V' = \X' + \Y'$.
Since the determinant over $F$ of the automorphism $x \mapsto x \cdot \ba$ of $B$ is $\nu(\ba)^2$ for $\ba \in B^{\times}$, we have $x_{\Y'}(h') \equiv 1 \bmod (F^{\times})^2$.
Noting that either $c = 0$ or $c \in B^{\times}$ for $h' = \left( \begin{smallmatrix} a & b \\ c & d \end{smallmatrix} \right)$, one can see that $j_{\Y'}(h') = 4 \cdot j(h')$.
Hence we have
\[
 v_{\Y'}(h', \nu) = 1.
\]
This completes the proof.
\end{proof}

\begin{lem}
\label{lem:spl-GU(W)}
We have
\[
 z_{\Y'}(h, h') = 1
\]
for $h, h' \in \GU(W)$.
\end{lem}

\begin{proof}
The proof is similar to that of Lemma \ref{lem:spl-GSO}. 

For $g, g' \in \GU(W)$, we have
\[
 z_{\Y'}(g, g') = z_{\Y'}(h, h'') \cdot v_{\Y'}(h', \nu),
\]
where 
\begin{align*}
 h & = g \cdot d_{\Y'}(\nu)^{-1}, & h' & = g' \cdot d_{\Y'}(\nu')^{-1}, &
 h'' & = d_{\Y'}(\nu) \cdot h' \cdot d_{\Y'}(\nu)^{-1}, \\
 \nu & = \nu(g), & \nu' & = \nu(g'). & &
\end{align*}
We have $h, h' \in P_{\Y'}$ and $z_{\Y'}(h, h'') = 1$.
Since the determinant over $F$ of the automorphism $x \mapsto \ba \cdot x$ of $B$ is $\nu(\ba)^2$ for $\ba \in B^{\times}$, we have $x_{\Y'}(h') \equiv 1 \bmod (F^{\times})^2$, 
so that $v_{\Y'}(h', \nu) = 1$ by Lemma \ref{lem:v_Y}.
This completes the proof.
\end{proof}

\begin{lem}
\label{lem:GU(V)-GU(W)-commute}
We have
\[
 z_{\Y'}(g, h) = z_{\Y'}(h, g) = 1
\]
for $g \in \GU(V)^0$ and $h \in \GU(W)$.
\end{lem}

\begin{proof}
The proof is similar to that of Lemma \ref{lem:GSp-GSO-commute}.

For $g \in \GU(V)^0$ and $h \in \GU(W)$, we have
\[
 z_{\Y'}(g, h) = z_{\Y'}(g', h'') \cdot v_{\Y'}(h', \nu), \qquad
 z_{\Y'}(h, g) = z_{\Y'}(h', g'') \cdot v_{\Y'}(g', \nu'),
\]
where
\begin{align*}
 g' & = g \cdot d(\nu)^{-1}, & g'' & = d(\nu') \cdot g' \cdot d(\nu')^{-1}, & \nu & = \nu(g), \\
 h' & = h \cdot d_{\Y'}(\nu')^{-1}, & h'' & = d_{\Y'}(\nu) \cdot h' \cdot d_{\Y'}(\nu)^{-1}, &  \nu' & = \nu(h).
\end{align*}
Since $h', h'' \in P_{\Y'}$, we have $z_{\Y'}(g', h'') = z_{\Y'}(h', g'') = 1$.
As in the proof of Lemma \ref{lem:spl-GU(W)}, we have $v_{\Y'}(h', \nu) = 1$.
As in the proof of Lemma \ref{lem:spl-GU(V)}, we have $v_{\Y'}(g', \nu') = 1$.
This completes the proof.
\end{proof}

We define a map $s': \GU(V)^0 \times \GU(W) \rightarrow \C^1$ by
\[
 s'(\g) = \gamma^{j(g)}
\]
for $\g = (g, h) \in \GU(V)^0 \times \GU(W)$.
By Lemmas \ref{lem:spl-GU(V)}, \ref{lem:spl-GU(W)}, \ref{lem:GU(V)-GU(W)-commute}, we see that 
\[
 z_{\Y'}(\g, \g') = s'(\g \g') \cdot s'(\g)^{-1} \cdot s'(\g')^{-1}
\]
for $\g, \g' \in \GU(V)^0 \times \GU(W)$.

Recall that we may identify $\V$ with $\V'$, and we have two complete polarizations $\V = \X + \Y = \X' + \Y'$, where
\begin{align*}
 \X & = F \e_1 + F \e_2 + F \e_3  + F \e_4, &
 \Y & = F \e_1^* + F \e_2^*  + F \e_3^*  + F \e_4^*, \\
 \X' & = \v \cdot B, &
 \Y' & = \v^* \cdot B.
\end{align*}
Put
\[
 \h_0 =
 \begin{pmatrix}
  \frac{1}{2} & \frac{1}{2t} & & & & & & \\
  & & \frac{1}{2} & \frac{1}{2t} & & & & \\
  & & & & -\frac{1}{2} & \frac{t}{2} & & \\
  & & & & & & -\frac{1}{2} & \frac{t}{2} \\
  & & & & 1 & t & & \\
  & & & & & & 1 & t \\
  1 & - \frac{1}{t} & & & & & & \\
  & & 1 & - \frac{1}{t} & & & &
 \end{pmatrix} \in \Sp(\V).
\]
Then we have
\[
\begin{bmatrix}
 \v \\
 \frac{1}{t} \v \j \\
 - \frac{1}{u} \v \i \\
 - \frac{t}{uJ} \v \i \j \\
 \v^* \\
 - \frac{t}{J} \v^* \j \\
 \v^* \i \\
 - \frac{1}{t} \v^* \i \j 
\end{bmatrix} 
= \h_0
\begin{bmatrix}
 \e_1 \\
 \e_2 \\
 \e_3 \\
 \e_4 \\
 \e_1^* \\
 \e_2^* \\
 \e_3^* \\
 \e_4^* 
\end{bmatrix},
\]
and hence $\X' = \X \h_0$ and $\Y' = \Y \h_0$.
Put
\[
 s(\g) := s'(\g) \cdot \mu(\g),
\]
where
\[
 \mu(\g) := z_{\Y}(\h_0 \g \h_0^{-1}, \h_0) \cdot z_{\Y}(\h_0, \g)^{-1}
\]
for $\g \in \GU(V)^0 \times \GU(W)$.
Then we have
\[
 z_{\Y}(\g, \g') = s(\g \g') \cdot s(\g)^{-1} \cdot s(\g')^{-1}
\]
for $\g, \g' \in \GU(V)^0 \times \GU(W)$.

\begin{lem}
\label{lem:mu-GU(V)-B1-J1}
Let $\g_1 := \ba_1^{-1} \in \GU(V)^0$ with $\ba_1 = a_1 + b_1 \i + c_1 \j_1 + d_1 \i \j_1 \in B_1^{\times}$.
Then we have
\begin{align*}
 \mu(\g_1) & =
 \begin{cases}
 1 & \text{if $b_1 = d_1 = 0$,} \\
 \gamma_F(J_2, \frac{1}{2} \psi)
 \cdot ( (a_1 b_1 + c_1 d_1 J_1) \nu_1 J_1, J_2 )_F
  & \text{if $(b_1,d_1) \ne (0,0)$ and $b_1^2 - d_1^2 J_1 = 0$,} \\
 (-(b_1^2 - d_1^2 J_1) \nu_1 J_1, J_2)_F
  & \text{if $(b_1,d_1) \ne (0,0)$ and $b_1^2 - d_1^2 J_1 \ne 0$,}
 \end{cases} \\
 & \times 
 \begin{cases}
  1 & \text{if $b_1 - d_1 t = 0$,} \\
  (u, J)_F & \text{if $b_1 - d_1 t \ne 0$,} \\
 \end{cases} \\
\end{align*}
where $\nu_1 = \nu(\ba_1)$.
\end{lem}

\begin{proof}
Put $\d := d_{\Y}(\nu_1) \in \GSp(\V)$.
We have
\[
 z_{\Y}(\h_0 \g_1 \h_0^{-1}, \h_0) 
 = z_{\Y}(\h_0 \g_1 \h_0^{-1} \cdot \d^{-1}, \d \cdot \h_0 \cdot \d^{-1}) 
 \cdot v_{\Y}(\h_0, \nu_1).
\]
We have $\h_0 = \tau_2 \cdot \m(\a_{10})$, where
\[
 \a_{10} = 
 \begin{pmatrix}
  \frac{1}{2} & \frac{1}{2t} & & \\
  & & \frac{1}{2} & \frac{1}{2t} \\
  1 & -\frac{1}{t} & & \\
  & & 1 & -\frac{1}{t}
 \end{pmatrix},
\]
so that $x_{\Y}(\h_0) \equiv -1 \bmod (F^{\times})^2$ and $j_{\Y}(\h_0) = 2$.
Hence we have
\[
 v_{\Y}(\h_0, \nu_1) = (-1, \nu_1)_F \cdot \gamma_F(\nu_1, \frac{1}{2} \psi)^{-2} = 1.
\]
We have
\begin{align*}
 & \h_0 \g_1 \h_0^{-1} \\
 & = 
 \begin{pmatrix}
 a_1 + c_1 t & & & & \frac{(b_1 + d_1 t)u}{2} & & & \\
 & a_1 + c_1 t & & & & - \frac{(b_1 + d_1 t)u J_2}{2} & & \\
 & & a_1 + c_1 t & & & & - \frac{b_1 + d_1 t}{2} & \\
 & & & a_1 + c_1 t & & & & \frac{(b_1 + d_1 t)}{2 J_2} \\
 2(b_1 - d_1 t) & & & & a_1 - c_1 t & & & \\
 & - \frac{2 (b_1 - d_1 t)}{J_2} & & & & a_1 - c_1 t & & \\
 & & -2 (b_1 - d_1 t) u & & & & a_1 - c_1 t & \\
 & & & 2 (b_1 - d_1 t) u J_2 & & & & a_1 - c_1 t
 \end{pmatrix}.
\end{align*}
If $b_1 - d_1 t = 0$,
then we have $\h_0 \g_1 \h_0^{-1} \cdot \d^{-1} \in P_{\Y}$ and hence 
$z_{\Y}(\h_0 \g_1 \h_0^{-1} \cdot \d^{-1}, \d \cdot \h_0 \cdot \d^{-1}) = 1$.
If $b_1 - d_1 t \ne 0$, then we have
$\h_0 \g_1 \h_0^{-1} \cdot \d^{-1} \in P_{\Y} \cdot \tau_4 \cdot \n(\b_{25})$,
where
\[
 \b_{25} = \frac{a_1 - c_1 t}{2 \nu_1 (b_1 - d_1 t)} \cdot 
 \begin{pmatrix}
  1 & & & \\
  & -J_2 & & \\
  & & - \frac{1}{u} & \\
  & & & \frac{1}{uJ_2}
 \end{pmatrix}.
\]
Since $\d \cdot \h_0 \cdot \d^{-1} \in \tau_2 \cdot P_{\Y}$, we have
\[
 z_{\Y}(\h_0 \g_1 \h_0^{-1} \cdot \d^{-1}, \d \cdot \h_0 \cdot \d^{-1}) =
 z_{\Y}(\tau_4 \cdot \n(\b_{25}), \tau_2).
\]
If $a_1 - c_1 t = 0$, then we have
$z_{\Y}(\tau_4 \cdot \n(\b_{25}), \tau_2) = z_{\Y}(\tau_4, \tau_2) = 1$.
If $a_1 - c_1 t \ne 0$, then we have
$z_{\Y}(\tau_4 \cdot \n(\b_{25}), \tau_2) = \gamma_F(\frac{1}{2} \psi \circ q_{11})$,
where $q_{11}$ is a non-degenerate symmetric bilinear form associated to
\[
 \b_{25} = \frac{a_1 - c_1 t}{2 \nu_1 (b_1 - d_1 t)} \cdot 
 \begin{pmatrix}
  - \frac{1}{u} & \\
  & \frac{1}{uJ_2}
 \end{pmatrix}.
\]
We have $\det q_{11} \equiv -J_2 \bmod (F^{\times})^2$ and 
\[
 h_F(q_{11}) 
 = ( - \frac{a_1 - c_1 t}{2 \nu_1 u (b_1 - d_1 t)},
 \frac{a_1 - c_1 t}{2 \nu_1 u J_2 (b_1 - d_1 t)} )_F
 = ( - \frac{a_1 - c_1 t}{2 \nu_1 u (b_1 - d_1 t)}, J_2 )_F.
\]
Hence we have 
\begin{align*}
 \gamma_F(\frac{1}{2} \psi \circ q_{11}) 
 & = \gamma_F(\frac{1}{2} \psi)^2 \cdot \gamma_F(-J_2, \frac{1}{2} \psi)
 \cdot ( - \frac{a_1 - c_1 t}{2 \nu_1 u (b_1 - d_1 t)}, J_2 )_F \\
 & = \gamma_F(J_2, \frac{1}{2} \psi)
 \cdot ( \frac{a_1 - c_1 t}{2 \nu_1 u (b_1 - d_1 t)}, J_2 )_F.
\end{align*}
Thus we obtain
\[
 z_{\Y}(\h_0 \g_1 \h_0^{-1}, \h_0) =
 \begin{cases}
  1 & \text{if $b_1 - d_1 t = 0$,} \\
  1 & \text{if $b_1 - d_1 t \ne 0$, $a_1 - c_1 t = 0$,} \\
  \gamma_F(J_2, \frac{1}{2} \psi)
  \cdot ( \frac{a_1 - c_1 t}{2 \nu_1 u (b_1 - d_1 t)}, J_2 )_F
   & \text{if $b_1 - d_1 t \ne 0$, $a_1 - c_1 t \ne 0$.}
 \end{cases}
\]

Now we compute $z_{\Y}(\h_0, \g_1)$.
We have
\[
 z_{\Y}(\h_0, \g_1) = z_{\Y}(\h_0, \g_1 \cdot \d^{-1}).
\]
First assume that $b_1 = d_1 = 0$.
Then we have $\g_1 \cdot \d^{-1} \in P_{\Y}$ and hence
\[
 z_{\Y}(\h_0, \g_1 \cdot \d^{-1}) = 1.
\]
Next assume that $(b_1, d_1) \ne (0,0)$ and $b_1^2 - d_1^2 J_1 = 0$.
Then we have $b_1 \ne 0$ and $d_1 \ne 0$.
As in the proof of Lemma \ref{lem:mu-GU(V)-u}, we have
\[
 (a_1 d_1 - b_1 c_1) \cdot (a_1 b_1 + c_1 d_1 J_1) = \nu_1 b_1 d_1 \ne 0.
\]
We have $\g_1 \cdot \d^{-1} \in \m(\a_1) \cdot \n(\b_3) \cdot \tau_2 \cdot P_{\Y}$, where $\a_1$ and $\b_3$ are as in the proof of Lemma \ref{lem:mu-GU(V)-u}.
Hence we have
\[
 z_{\Y}(\h_0, \g_1 \cdot \d^{-1}) = 
 z_{\Y}(\tau_2 \cdot \m(\a_{10}), \m(\a_1) \cdot \n(\b_3) \cdot \tau_2)
 = z_{\Y}(\tau_2 \cdot \m(\a_{11}), \n(\b_3) \cdot \tau_2),
\]
where 
\[
 \a_{11} = \a_{10} \cdot \a_1 = 
 \begin{pmatrix}
 \frac{b_1 + d_1 t}{2} & & \frac{1}{2t} & \\
 & \frac{b_1 + d_1 t}{2} & & \frac{1}{2t} \\
 b_1 - d_1 t  & & - \frac{1}{t} & \\
 & b_1 - d_1 t & & - \frac{1}{t}
 \end{pmatrix}.
\]
If $b_1 - d_1 t = 0$, then we have $\tau_2 \cdot \m(\a_{11}) \cdot \tau_2^{-1} \in P_{\Y}$ and hence
\[
 z_{\Y}(\tau_2 \cdot \m(\a_{11}), \n(\b_3) \cdot \tau_2)
 = z_{\Y}(\tau_2, \n(\b_3) \cdot \tau_2)
 = \gamma_F(\frac{1}{2} \psi \circ q_{12}),
\]
where $q_{12}$ is a non-degenerate symmetric bilinear form associated to
\[
 \frac{a_1 d_1 - b_1 c_1}{b_1 d_1} \cdot 
 \begin{pmatrix}
  J_1 & \\
  & -J
 \end{pmatrix}.
\]
We have $\det q_{12} \equiv - J_2 \bmod (F^{\times})^2$ and
\[
 h_F(q_{12}) =  (  \frac{a_1 d_1 - b_1 c_1}{b_1 d_1} \cdot J_1,
 -\frac{a_1 d_1 - b_1 c_1}{b_1 d_1} \cdot J )_F \\
 = ( \frac{a_1 d_1 - b_1 c_1}{b_1 d_1}, J )_F
 = ( \frac{\nu_1}{a_1 b_1 + c_1 d_1 J_1}, J )_F.
\]
Hence we have
\begin{align*}
 \gamma_F(\frac{1}{2} \psi \circ q_{12})
 & = \gamma_F(\frac{1}{2} \psi)^2 \cdot 
 \gamma_F(-J_2, \frac{1}{2} \psi) \cdot
 ( \frac{\nu_1}{a_1 b_1 + c_1 d_1 J_1}, J )_F \\
 & = \gamma_F(J_2, \frac{1}{2} \psi)^{-1} \cdot
 ( (a_1 b_1 + c_1 d_1 J_1) \nu_1, J_2 )_F.
\end{align*}
If $b_1 - d_1 t \ne 0$, then we have
\[
 \m(\a_{11}) \cdot \n(\b_3) \cdot \tau_2 \in
 \n(\b_{26}) \cdot \tau'' \cdot P_{\Y},
\]
where
\[
 \b_{26} =  \frac{b_1}{2 (b_1 - d_1 t)} \cdot 
 \begin{pmatrix}
  1 & & & \\
  & -J_2 & & \\
  & & 0 & \\
  & & & 0
 \end{pmatrix}
\]
and $\tau''$ is as in the proof of Lemma \ref{lem:mu-GU(W)-J}.
Since $\tau_2 \cdot \n(\b_{26}) \cdot \tau_2^{-1} \in P_{\Y}$, we have
\[
 z_{\Y}(\tau_2 \cdot \m(\a_{11}), \n(\b_3) \cdot \tau_2)
 = z_{\Y}(\tau_2, \n(\b_{26}) \cdot \tau'')
 = z_{\Y}(\tau_2, \tau'')
 = 1.
\]
Finally assume that $(b_1, d_1) \ne (0,0)$ and $b_1^2 - d_1^2 J_1 \ne 0$.
We have $\g_1 \cdot \d^{-1} \in \n(\b_5) \cdot \tau_4 \cdot P_{\Y}$,
where $\b_5$ is as in the proof of Lemma \ref{lem:mu-GU(V)-u}.
Hence we have
\[
 z_{\Y}(\h_0, \g \cdot \d^{-1}) 
 = z_{\Y}(\tau_2 \cdot \m(\a_{10}), \n(\b_5) \cdot \tau_4)
 = z_{\Y}(\tau_2, \n(\b_{27}) \cdot \tau_4),
\]
where
\begin{align*}
 \b_{27} & = \a_{10} \cdot \b_5 \cdot {}^t \a_{10} \\
 & = \frac{1}{b_1^2 - d_1^2 J_1} \cdot 
 \begin{pmatrix}       	
 \frac{(a_1 + c_1 t) (b_1 + d_1 t)}{2} & & & \\
 & -\frac{(a_1 + c_1 t) (b_1 + d_1 t) J_2}{2} & & \\
 & & 2 (a_1 - c_1 t) (b_1 - d_1 t) & \\
 & & & - 2 (a_1 - c_1 t) (b_1 - d_1 t) J_2
\end{pmatrix}.
\end{align*}
We write $\b_{27} = \b_{28} + \b_{29}$, where
\[
 \b_{28} = \frac{a_1 + c_1 t}{2 (b_1 - d_1 t)} \cdot 
 \begin{pmatrix}       	
 1 & & & \\
 & - J_2 & & \\
 & & 0 & \\
 & & & 0
\end{pmatrix}, \qquad
 \b_{29} = \frac{2 (a_1 - c_1 t)}{b_1 + d_1 t} \cdot 
 \begin{pmatrix}       	
 0 & & & \\
 & 0 & & \\
 & & 1 & \\
 & & & - J_2
\end{pmatrix}.
\]
Since $\tau_2 \cdot \b_{28} \cdot \tau_2^{-1} \in P_{\Y}$,
we have $z_{\Y}(\tau_2, \n(\b_{27}) \cdot \tau_4) = z_{\Y}(\tau_2, \n(\b_{29}) \cdot \tau_4)$.
If $a_1 - c_1t = 0$,
then we have $z_{\Y}(\tau_2, \n(\b_{29}) \cdot \tau_4) = z_{\Y}(\tau_2, \tau_4) = 1$.
If $a_1 - c_1t \ne 0$, then we have
$z_{\Y}(\tau_2, \n(\b_{29}) \cdot \tau_4) = \gamma_F(\frac{1}{2} \psi \circ q_{13})$,
where $q_{13}$ is a non-degenerate symmetric bilinear form associated to
\[
 \frac{2 (a_1 - c_1 t)}{b_1 + d_1 t} \cdot 
 \begin{pmatrix}       	
 1 & \\
 & - J_2
\end{pmatrix}.
\]
We have $\det q_{13} \equiv -J_2 \bmod (F^{\times})^2$ and
\[
 h_F(q_{13}) = (\frac{2 (a_1 - c_1 t)}{b_1 + d_1 t}, - \frac{2 (a_1 - c_1 t)}{b_1 + d_1 t} \cdot J_2)_F
= (\frac{2 (a_1 - c_1 t)}{b_1 + d_1 t}, J_2)_F.
\] 
Hence we have
\[
 \gamma_F(\frac{1}{2} \psi \circ q_{13}) 
 = \gamma_F(\frac{1}{2} \psi)^2 \cdot \gamma_F(-J_2, \frac{1}{2} \psi)
 \cdot (\frac{2 (a_1 - c_1 t)}{b_1 + d_1 t}, J_2)_F
 = \gamma_F(J_2, \frac{1}{2} \psi)^{-1}
 \cdot (\frac{2 (a_1 - c_1 t)}{b_1 + d_1 t}, J_2)_F.
\]
Thus we obtain
\[
 z_{\Y}(\h_0, \g_1) =  
 \begin{cases}
  1 & \text{if $b_1 = d_1 = 0$,} \\
 \gamma_F(J_2, \frac{1}{2} \psi)^{-1} \cdot
 ( (a_1 b_1 + c_1 d_1 J_1) \nu_1, J_2 )_F
  & \text{if $(b_1,d_1) \ne (0,0)$, $b_1^2 - d_1^2 J_1 = 0$, $b_1-d_1t = 0$,} \\
  1 & \text{if $(b_1,d_1) \ne (0,0)$, $b_1^2 - d_1^2 J_1 = 0$, $b_1-d_1t \ne 0$,} \\
  1 & \text{if $(b_1,d_1) \ne (0,0)$, $b_1^2 - d_1^2 J_1 \ne 0$, $a_1-c_1t = 0$,} \\
 \gamma_F(J_2, \frac{1}{2} \psi)^{-1}
 \cdot (\frac{2 (a_1 - c_1 t)}{b_1 + d_1 t}, J_2)_F
  & \text{if $(b_1,d_1) \ne (0,0)$, $b_1^2 - d_1^2 J_1 \ne 0$, $a_1-c_1t \ne 0$.}
 \end{cases}
\]

Now we compute $\mu(\g_1) = z_{\Y}(\h_0 \g_1 \h_0^{-1}, \h_0) \cdot z_{\Y}(\h_0, \g_1)^{-1}$.
Recall that $J_1 = t^2$ and $\nu_1 = a_1^2 - b_1^2 u - c_1^2 J_1 + d_1^2 u J_1 \ne 0$.
First assume that $b_1 = d_1 = 0$.
Then we have
\[
 \mu(\g_1) = 1 \cdot 1 = 1.
\]
Next assume that $(b_1,d_1) \ne (0,0)$ and $b_1^2 - d_1^2 J_1 = 0$.
Then we have $\nu_1 = a_1^2 - c_1^2 J_1 = (a_1+c_1t)(a_1-c_1t) \ne 0$.
Since $b_1^2 - d_1^2 J_1 = (b_1 + d_1t)(b_1-d_1t)$, we have
\[
 b_1-d_1t = 0 \Longleftrightarrow b_1+d_1t \ne 0.
\]
If $b_1-d_1t = 0$, then we have
\[
 \mu(\g_1) = 1 \cdot
 \gamma_F(J_2, \frac{1}{2} \psi) \cdot
 ( (a_1 b_1 + c_1 d_1 J_1) \nu_1, J_2 )_F
 = \gamma_F(J_2, \frac{1}{2} \psi) \cdot
 ( (a_1 b_1 + c_1 d_1 J_1) \nu_1 J_1, J_2 )_F.
\]
If $b_1-d_1t \ne 0$, then we have
\[
 (a_1-c_1t)(b_1-d_1t) = - 2 (a_1-c_1t) d_1 t
 = 2 (-a_1 d_1 t + c_1 d_1 J_1 )
 = 2 ( a_1 b_1 + c_1 d_1 J_1 ).
\]
Hence we have
\begin{align*}
 \mu(\g_1) & =
  \gamma_F(J_2, \frac{1}{2} \psi)
  \cdot ( \frac{a_1 - c_1 t}{2 \nu_1 u (b_1 - d_1 t)}, J_2 )_F \cdot 1 \\
 & = \gamma_F(J_2, \frac{1}{2} \psi)
 \cdot ( 2 \nu_1 (a_1 - c_1 t) (b_1 - d_1 t), J_2 )_F
 \cdot ( u, J_2 )_F \\
 & = \gamma_F(J_2, \frac{1}{2} \psi)
 \cdot ( (a_1 b_1 + c_1 d_1 J_1) \nu_1, J_2 )_F \cdot (u, J_2)_F \\
 & = \gamma_F(J_2, \frac{1}{2} \psi)
 \cdot ( (a_1 b_1 + c_1 d_1 J_1) \nu_1 J_1, J_2 )_F \cdot (u, J)_F.
\end{align*}
Finally assume that $(b_1,d_1) \ne (0,0)$ and $b_1^2 - d_1^2 J_1 \ne 0$.
Then we have $b_1 - d_1t \ne 0$.
If $a_1 - c_1t = 0$, then we have
\[
 \mu(\g_1) = 1 \cdot 1 = 1.
\]
On the other hand, since $\nu_1 = -b_1^2 u + d_1^2 uJ_1$, we have
\[
 (-(b_1^2 - d_1^2 J_1) \nu_1 J_1, J_2)_F \cdot (u, J)_F
 = (\frac{\nu_1}{u} \cdot \nu_1 J_1, J_2)_F \cdot (u, J_2)_F
 = (\nu_1^2 J_1, J_2)_F
 = 1.
\]
If $a_1 - c_1t \ne 0$, then we have
\begin{align*}
\mu(\g_1) & = 
 \gamma_F(J_2, \frac{1}{2} \psi)
 \cdot ( \frac{a_1 - c_1 t}{2 \nu_1 u (b_1 - d_1 t)}, J_2 )_F
 \cdot \gamma_F(J_2, \frac{1}{2} \psi)
 \cdot (\frac{2 (a_1 - c_1 t)}{b_1 + d_1 t}, J_2)_F \\
 & = \gamma_F(J_2, \frac{1}{2} \psi)^2
 \cdot ( \nu_1 u (b_1 + d_1 t) (b_1 - d_1 t) , J_2 )_F \\
 & = (-1, J_2) \cdot ((b_1^2 - d_1^2 J_1) \nu_1, J_2)_F \cdot (u, J_2)_F \\
 & = (-(b_1^2 - d_1^2 J_1) \nu_1, J_2)_F \cdot (u, J_2)_F \\
 & = (-(b_1^2 - d_1^2 J_1) \nu_1 J_1, J_2)_F \cdot (u, J)_F.
\end{align*}
This completes the proof.
\end{proof}

\begin{lem}
\label{lem:mu-GU(V)-B2-J1}
Let $\g_2 := \ba_2^{-1} \in \GU(V)^0$ with $\ba_2 = a_2 + b_2 \i + c_2 \j_2 + d_2 \i \j_2 \in B_2^{\times}$.
Then we have
\[
 \mu(\g_2) =
 \begin{cases}
  1 & \text{if $b_2 = d_2 = 0$,} \\
 \gamma_F(J_1, \frac{1}{2} \psi)
 \cdot ( (a_2 b_2 + c_2 d_2 J_2) \nu_2 J_2, J_1 )_F
  & \text{if $(b_2,d_2) \ne (0,0)$ and $b_2^2 - d_2^2 J_2 = 0$,} \\
 (-(b_2^2 - d_2^2 J_2) \nu_2 J_2, J_1)_F
  & \text{if $(b_2,d_2) \ne (0,0)$ and $b_2^2 - d_2^2 J_2 \ne 0$,}
 \end{cases}
\]
where $\nu_2 = \nu(\ba_2)$.
\end{lem}

\begin{proof}
Put $\d := d_{\Y}(\nu_2) \in \GSp(\V)$.
We have
\[
 z_{\Y}(\h_0 \g_2 \h_0^{-1}, \h_0) 
 = z_{\Y}(\h_0 \g_2 \h_0^{-1} \cdot \d^{-1}, \d \cdot \h_0 \cdot \d^{-1}) 
 \cdot v_{\Y}(\h_0, \nu_2).
\]
Since $\Y' \g_2 = \Y'$,
we have $\h_0 \g_2 \h_0^{-1} \cdot \d^{-1} \in P_{\Y}$ and hence
\[
 z_{\Y}(\h_0 \g_2 \h_0^{-1} \cdot \d^{-1}, \d \cdot \h_0 \cdot \d^{-1}) =1.
\]
As in the proof of Lemma \ref{lem:mu-GU(V)-B1-J1},
we have $v_{\Y}(\h_0, \nu_2) = 1$.
Thus we obtain
\[
 z_{\Y}(\h_0 \g_2 \h_0^{-1}, \h_0) = 1.
\]

Now we compute $z_{\Y}(\h_0, \g_2)$.
We have
\[
 z_{\Y}(\h_0, \g_2) = z_{\Y}(\h_0, \g_2 \cdot \d^{-1}).
\]
First assume that $b_2 = d_2 = 0$.
Then we have $\g_2 \cdot \d^{-1} \in P_{\Y}$ and hence
\[
 z_{\Y}(\h_0, \g_2 \cdot \d^{-1}) = 1.
\]
Next assume that $(b_2, d_2) \ne (0,0)$ and $b_2^2 - d_2^2 J_2 = 0$.
Then we have $b_2 \ne 0$ and $d_2 \ne 0$.
As in the proof of Lemma \ref{lem:mu-GU(V)-u}, we have $a_2 b_2 + c_2 d_2 J_2 \ne 0$.
We have $\g_2 \cdot \d^{-1} \in \m(\a_{12}) \cdot \n(\b_{30}) \cdot \tau_2 \cdot P_{\Y}$, where
\[
 \a_{12} = 
 \begin{pmatrix}
  d_2 & & & \\
  & d_2 & & \\
  b_2 & & 1 & \\
  & b_2 & & 1
 \end{pmatrix}, \qquad
 \b_{30} = \frac{a_2 d_2 - b_2 c_2}{b_2 d_2} \cdot 
 \begin{pmatrix}
  0 & & & \\
  & 0 & & \\
  & & J_2 & \\
  & & & -J
 \end{pmatrix}.
\]
Hence we have
\[
 z_{\Y}(\h_0, \g_2 \cdot \d^{-1})
 = z_{\Y}(\tau_2 \cdot \m(\a_{10}), \m(\a_{12}) \cdot \n(\b_{30}) \cdot \tau_2)
 = z_{\Y}(\tau_2, \m(\a_{10}) \cdot \m(\a_{12}) \cdot \n(\b_{30}) \cdot \tau_2),
\]
where $\a_{10}$ is as in the proof of Lemma \ref{lem:mu-GU(V)-B1-J1}.
We have
\[
 \m(\a_{10}) \cdot \m(\a_{12}) \cdot \n(\b_{30}) \cdot \tau_2 
 \in \m(\a_{13}) \cdot \n(\b_{31}) \cdot \tau' \cdot P_{\Y},
\]
where
\[
 \a_{13} = 
 \begin{pmatrix}
  d_2 & & & \\
  b_2 & 1 & & \\
  & & d_2 & \\
  & & b_2 & 1
 \end{pmatrix}, \qquad
 \b_{31} = \frac{(a_2 d_2 - b_2 c_2) J_2}{b_2 d_2} \cdot 
 \begin{pmatrix}
  0 & & & \\
  & 0 & & 1 \\
  & & 0 & \\
  & 1 & & 0
 \end{pmatrix},
\]
and $\tau'$ is as in the proof of Lemma \ref{lem:mu-GU(V)-J}.
Hence we have
\[
 z_{\Y}(\tau_2, \m(\a_{10}) \cdot \m(\a_{12}) \cdot \n(\b_{30}) \cdot \tau_2) 
 = z_{\Y}(\tau_2, \m(\a_{13}) \cdot \n(\b_{31}) \cdot \tau')
 = z_{\Y}(\tau_2 \cdot \m(\a_{13}) \cdot \n(\b_{31}), \tau').
\]
Since $\tau_2 \cdot \m(\a_{13}) \cdot \n(\b_{31}) \cdot \tau_2^{-1} \in P_{\Y}$, we have
\[
 z_{\Y}(\tau_2 \cdot \m(\a_{13}) \cdot \n(\b_{31}), \tau')
 = z_{\Y}(\tau_2, \tau') = 1.
\]
On the other hand, since $J_1 \in (F^{\times})^2$, 
we have $\gamma_F(J_1, \frac{1}{2} \psi) = 1$ and 
\[
 ( (a_2 b_2 + c_2 d_2 J_2) \nu_2 J_2, J_1 )_F = 1.
\]
Finally assume that $(b_2, d_2) \ne (0,0)$ and $b_2^2 - d_2^2 J_2 \ne 0$.
We have $\g_2 \cdot \d^{-1} \in \n(\b_{32}) \cdot \tau_4 \cdot P_{\Y}$, where
\[
 \b_{32} = \frac{1}{b_2^2 - d_2^2 J_2} \cdot
 \begin{pmatrix}
 a_2 b_2 + c_2 d_2 J_2 & & (a_2 d_2 + b_2 c_2) J_2 & \\
 & - (a_2 b_2 + c_2 d_2 J_2) J_1 & & -(a_2 d_2 + b_2 c_2) J \\
 (a_2 d_2 + b_2 c_2) J_2 & & (a_2 b_2 + c_2 d_2 J_2) J_2 & \\
 & -(a_2 d_2 + b_2 c_2) J & & -(a_2 b_2 + c_2 d_2 J_2) J
\end{pmatrix}.
\]
Hence we have
\begin{align*}
 z_{\Y}(\h_0, \g_2 \cdot \d^{-1})
 & = z_{\Y}(\tau_2 \cdot \m(\a_{10}), \n(\b_{32}) \cdot \tau_4) \\
 & = z_{\Y}(\tau_2 \cdot \m(\a_{10}), \n(\b_{32}) \cdot 
 \m(\a_{10})^{-1} \cdot \tau_4) \\
 & = z_{\Y}(\tau_2 \cdot \m(\a_{10}) \cdot \n(\b_{32}) \cdot 
 \m(\a_{10})^{-1}, \tau_4).
\end{align*}
Since $\tau_2 \cdot \m(\a_{10}) \cdot \n(\b_{32}) \cdot 
 \m(\a_{10})^{-1} \cdot \tau_2^{-1} \in P_{\Y}$, we have
\[
 z_{\Y}(\h_0, \g_2 \cdot \d^{-1}) = z_{\Y}(\tau_2, \tau_4) = 1.
\]
On the other hand, since $J_1 \in (F^{\times})^2$, we have
\[
 (-(b_2^2 - d_2^2 J_2) \nu_2 J_2, J_1)_F = 1.
\]
This completes the proof.
\end{proof}

\begin{lem}
\label{lem:mu-GU(W)-J1}
Let $\g := \ba \in \GU(W)$ with $\ba = a + b \i + c \j + d \i \j \in B^{\times}$.
Then, for any $i \in \{ 1, 2\}$, we have
\[
 \mu(\g) =
 \begin{cases}
  (\nu, J_i)_F & \text{if $b = d = 0$,} \\
 \gamma_F(J_i, \frac{1}{2} \psi) \cdot (ab - cdJ, J_i)_F
  & \text{if $(b,d) \ne (0,0)$ and $b^2 - d^2 J = 0$,} \\
  (- (b^2 - d^2J) J, J_i)_F
  & \text{if $(b,d) \ne (0,0)$ and $b^2 - d^2 J \ne 0$,}
 \end{cases}
\]
where $\nu = \nu(\ba)$.
\end{lem}

\begin{proof}
Put $\d := d_{\Y}(\nu) \in \GSp(\V)$.
We have
\[
 z_{\Y}(\h_0 \g \h_0^{-1}, \h_0) 
 = z_{\Y}(\h_0 \g \h_0^{-1} \cdot \d^{-1}, \d \cdot \h_0 \cdot \d^{-1}) 
 \cdot v_{\Y}(\h_0, \nu).
\]
Since $\Y' \g = \Y'$,
we have $\h_0 \g \h_0^{-1} \cdot \d^{-1} \in P_{\Y}$ and hence
\[
 z_{\Y}(\h_0 \g \h_0^{-1} \cdot \d^{-1}, \d \cdot \h_0 \cdot \d^{-1}) = 1.
\]
As in the proof of Lemma \ref{lem:mu-GU(V)-B1-J1},
we have $v_{\Y}(\h_0, \nu) = 1$.
Thus we obtain
\[
 z_{\Y}(\h_0 \g \h_0^{-1}, \h_0) = 1.
\]

Now we compute $z_{\Y}(\h_0, \g)$.
We have
\[
 z_{\Y}(\h_0, \g) = z_{\Y}(\h_0, \g \cdot \d^{-1}).
\]
First assume that $b = d = 0$.
Then we have $\g \cdot \d^{-1} \in P_{\Y}$ and hence
\[
 z_{\Y}(\h_0, \g \cdot \d^{-1}) = 1.
\]
On the other hand, since $J_1 \in (F^{\times})^2$, 
we have $(\nu, J_1)_F = 1$ and
\[
 (\nu, J_2)_F = (\nu, J)_F = (a^2 - c^2 J, J)_F = 1.
\]
Next assume that $(b, d) \ne (0,0)$ and $b^2 - d^2 J = 0$.
Then we have $b \ne 0$ and $d \ne 0$.
As in the proof of Lemma \ref{lem:mu-GU(W)-J}, we have $ab - cd J \ne 0$.
We have $\g \cdot \d^{-1} \in \m(\a_3) \cdot \n(\b_9) \cdot \tau_2 \cdot P_{\Y}$, where $\a_3$ and $\b_9$ are as in the proof of Lemma \ref{lem:mu-GU(W)-u}.
Hence we have
\[
 z_{\Y}(\h_0, \g \cdot \d^{-1})
 = z_{\Y}(\tau_2 \cdot \m(\a_{10}), \m(\a_3) \cdot \n(\b_9) \cdot \tau_2)
 = z_{\Y}(\tau_2, \m(\a_{10}) \cdot \m(\a_3) \cdot \n(\b_9) \cdot \tau_2),
\]
where $\a_{10}$ is as in the proof of Lemma \ref{lem:mu-GU(V)-B1-J1}.
We have
\[
 \m(\a_{10}) \cdot \m(\a_3) \cdot \n(\b_9) \cdot \tau_2
 \in \m(\a_{14}) \cdot \n(\b_{33}) \cdot \tau' \cdot P_{\Y},
\]
where
\[
 \a_{14} = 
 \begin{pmatrix}
  bt & & & \\
  -dJ & 1 & & \\
  & & bt & \\
  & & dJ & 1
 \end{pmatrix}, \qquad
 \b_{33} = - \frac{(ad + bc) J_2}{bd} \cdot 
 \begin{pmatrix}
  0 & & & \\
  & 0 & & 1 \\
  & & 0 & \\
  & 1 & & 0
 \end{pmatrix},
\]
and $\tau'$ is as in the proof of Lemma \ref{lem:mu-GU(V)-J}.
Hence we have
\[
 z_{\Y}(\tau_2, \m(\a_{10}) \cdot \m(\a_3) \cdot \n(\b_9) \cdot \tau_2)
 = z_{\Y}(\tau_2, \m(\a_{14}) \cdot \n(\b_{33}) \cdot \tau')
 = z_{\Y}(\tau_2 \cdot \m(\a_{14}) \cdot \n(\b_{33}), \tau').
\]
Since $\tau_2 \cdot \m(\a_{14}) \cdot \n(\b_{33}) \cdot \tau_2^{-1} \in P_{\Y}$, we have
\[
 z_{\Y}(\tau_2 \cdot \m(\a_{14}) \cdot \n(\b_{33}), \tau')
 = z_{\Y}(\tau_2, \tau') = 1.
\]
On the other hand, since $J_1 \in (F^{\times})^2$ and $J \in (F^{\times})^2$, we have $\gamma_F(J_1, \frac{1}{2} \psi) = \gamma_F(J_2, \frac{1}{2} \psi) = 1$ and
\[
 (ab - cdJ, J_1)_F =  (ab - cdJ, J_2)_F = 1.
\]
Finally assume that $(b, d) \ne (0,0)$ and $b^2 - d^2 J \ne 0$.
We have $\g \cdot \d^{-1} \in \n(\b_{13}) \cdot \tau_4 \cdot P_{\Y}$, where
$\b_{13}$ is as in the proof of Lemma \ref{lem:mu-GU(W)-u}.
Hence we have
\begin{align*}
 z_{\Y}(\h_0, \g \cdot \d^{-1})
 & = z_{\Y}(\tau_2 \cdot \m(\a_{10}), \n(\b_{13}) \cdot \tau_4) \\
 & = z_{\Y}(\tau_2 \cdot \m(\a_{10}), \n(\b_{13}) \cdot 
 \m(\a_{10})^{-1} \cdot \tau_4) \\
 & = z_{\Y}(\tau_2 \cdot \m(\a_{10}) \cdot \n(\b_{13}) \cdot 
 \m(\a_{10})^{-1}, \tau_4).
\end{align*}
Since $\tau_2 \cdot \m(\a_{10}) \cdot \n(\b_{13}) \cdot 
 \m(\a_{10})^{-1} \cdot \tau_2^{-1} \in P_{\Y}$, we have
\[
 z_{\Y}(\h_0, \g \cdot \d^{-1}) = z_{\Y}(\tau_2, \tau_4) = 1.
\]
On the other hand, since $J_1 \in (F^{\times})^2$, we have
$(- (b^2 - d^2J) J, J_1)_F = 1$ and 
\[
 (- (b^2 - d^2J) J, J_2)_F = (- (b^2 - d^2J) J, J)_F 
= (b^2 - d^2J, J)_F = 1.
\]
This completes the proof.
\end{proof}

\subsection{The product formula}
\label{ss:spl-B-prod}
Suppose that $F$ is a number field.
First we fix quaternion algebras $B_1$ and $B_2$ over $F$.
Next we fix a quadratic extension $E$ of $F$
such that $E$ embeds into $B_1$ and $B_2$.
Let $B$ be the quaternion algebra over $F$
which is the product of $B_1$ and $B_2$ in the Brauer group.
Then $E$ also embeds into $B$.

Fix a finite set $\Sigma$ of places of $F$ containing
\[
 \Sigma_{\infty} \cup \Sigma_2 \cup \Sigma_E
 \cup \Sigma_B \cup \Sigma_{B_1} \cup \Sigma_{B_2}.
\]
Here $\Sigma_{\infty}$ is the set of archimedean places of $F$,
$\Sigma_2$ is the set of places of $F$ lying above $2$,
and $\Sigma_{\bullet}$ is the set of places $v$ of $F$ such that $\bullet_v$ is ramified over $F_v$ for $\bullet = E$, $B$, $B_1$, $B_2$.

We write $B_i = E + E \j_i$.
Put $J_i = \j_i^2$ and $J = J_1 J_2$.
We may write $B = E + E \j$ such that $\j^2 = J$.
Then, for each place $v$ of $F$, we have
\begin{itemize}
 \item $J \in \N_{E_v/F_v}(E_v^{\times})$ if $v \notin \Sigma_B$,
 \item $J_1 \in \N_{E_v/F_v}(E_v^{\times})$ if $v \notin \Sigma_{B_1}$,
 \item $J_2 \in \N_{E_v/F_v}(E_v^{\times})$ if $v \notin \Sigma_{B_2}$.
\end{itemize}
By using the weak approximation theorem and replacing $\j_i$ by $\alpha_i \j_i$ with some $\alpha_i \in E^{\times}$ if necessary, we may assume that
\[
 J \in (F_v^{\times})^2 \quad \text{or} \quad
 J_1 \in (F_v^{\times})^2 \quad \text{or} \quad
 J_2 \in (F_v^{\times})^2
\]
for all $v \in \Sigma$.

\begin{lem}
We have 
\[
 u \in (F_v^{\times})^2 \quad \text{or} \quad
 J \in (F_v^{\times})^2 \quad \text{or} \quad
 J_1 \in (F_v^{\times})^2 \quad \text{or} \quad
 J_2 \in (F_v^{\times})^2
\]
for all $v \notin \Sigma$.
\end{lem}

\begin{proof}
Let $v \notin \Sigma$.
We may assume that $v$ is inert in $E$.
Assume that $J_i \notin (F_v^{\times})^2$ for $i=1,2$.
Since $J_i \in \N_{E_v/F_v}(E_v^{\times})$, we have $J_i \in \varepsilon \cdot (F_v^{\times})^2$ for $i=1,2$, where $\varepsilon \in \o_{F_v}^{\times}$ but $\varepsilon \notin (F_v^{\times})^2$.
Hence we have
\[
 J = J_1 J_2 \in (F_v^{\times})^2.
\]
This yields the lemma.
\end{proof}

Thus, for each place $v$ of $F$, we can define a map
\[
 s_v : \GU(V_v)^0 \times \GU(W_v) \longrightarrow \C^1
\]
by $s_v := s_v' \cdot \mu_v$, where $s_v'$ and $\mu_v$ are as in \S \S \ref{ss:B-spl}, \ref{ss:B1-spl}.
Here, for $\bullet = u$, $J$, $J_1$, $J_2$ with $\bullet \in (F_v^{\times})^2$, we have chosen $t \in F_v^{\times}$ such that $\bullet = t^2$.
Recall that
\[
 z_{\Y_v}(\g, \g') = s_v(\g \g') \cdot s_v(\g)^{-1} \cdot s_v(\g')^{-1}
\]
for $\g, \g' \in \GU(V_v)^0 \times \GU(W_v)$.

\begin{prop}
\label{prop:spl-B}

\noindent $\mathrm{(i)}$
Let $\g_i := \ba_i^{-1} \in \GU(V_v)^0$ with $\ba_i = a_i + b_i \i + c_i \j_i + d_i \i \j_i \in B_{i,v}^{\times}$.
Then we have
\[
 s_v(\g_i) =
 \begin{cases}
  1 & \text{if $b_i = d_i = 0$,} \\
 \gamma_{F_v}(J_j, \frac{1}{2} \psi_v)
 \cdot ((a_i b_i + c_i d_i J_i) \nu_i J_i, J_j)_{F_v}
  & \text{if $(b_i, d_i) \ne (0,0)$ and $b_i^2 - d_i^2 J_i = 0$,} \\
 (-(b_i^2 - d_i^2 J_i) \nu_i J_i, J_j)_{F_v}
 & \text{if $(b_i, d_i) \ne (0,0)$ and $b_i^2 - d_i^2 J_i \ne 0$,}
 \end{cases}
\]
where $\nu_i = \nu(\ba_i)$ and $\{ i, j \} = \{ 1, 2 \}$.

\noindent $\mathrm{(ii)}$
Let $\g := \ba \in \GU(W_v)$ with $\ba = a + b \i + c \j + d \i \j \in B_v^{\times}$.
Then we have
\[
 s_v(\g) =
 \begin{cases}
  (\nu, J_1)_{F_v} & \text{if $b = d = 0$,} \\
  \gamma_{F_v}(J_1, \frac{1}{2} \psi_v) \cdot (ab-cdJ, J_1)_{F_v}
  & \text{if $(b,d) \ne (0,0)$ and $b^2 - d^2 J = 0$,} \\
  (- (b^2 - d^2J) J, J_1)_{F_v}
  & \text{if $(b,d) \ne (0,0)$ and $b^2 - d^2 J \ne 0$,}
 \end{cases}
\]
where $\nu = \nu(\ba)$.
\end{prop}

\begin{proof}
If $u \in (F_v^{\times})^2$, then $B_{i,v}$ is split and the assertion follows from Lemmas \ref{lem:mu-GU(V)-u} and \ref{lem:mu-GU(W)-u}.

Assume that $J \in (F_v^{\times})^2$.
Let $\ii: B_v \rightarrow \M_2(F_v)$ be the isomorphism as in \S \ref{ss:B-spl}.
Since
\[
 \ii(\ba) = 
 \begin{pmatrix}
  a + ct & b - dt \\
  u(b + dt) & a - ct
 \end{pmatrix},
\]
we have
\[
 j(\g) = 
 \begin{cases}
  0 & \text{if $b + dt = 0$,} \\
  1 & \text{if $b + dt \ne 0$.}
 \end{cases}
\]
Since 
\[
 (u, J_1)_{F_v} =
 \begin{cases}
  1 & \text{if $B_{1,v}$ is split,} \\
  -1 & \text{if $B_{1,v}$ is ramified,}
 \end{cases}
\]
the assertion follows from Lemmas \ref{lem:mu-GU(V)-J} and \ref{lem:mu-GU(W)-J}.

Assume that $J_i \in (F_v^{\times})^2$.
We only consider the case $i = 1$; the other case is similar.
Let $\ii_1: B_{1,v} \rightarrow \M_2(F_v)$ be the isomorphism as in \S \ref{ss:B1-spl}.
Since
\[
 {}^t \ii_1(\ba_1) =
 \begin{pmatrix}
  a_1 + c_1 t & \frac{u}{2} (b_1 + d_1 t) \\
  2(b_1 - d_1 t) & a_1 - c_1 t
 \end{pmatrix},
\]
we have
\[
 j(\g_1) = 
 \begin{cases}
  0 & \text{if $b_1 - d_1 t = 0$,} \\
  1 & \text{if $b_1 - d_1 t \ne 0$.}
 \end{cases}
\]
Also, we have $j(\g_2) = 0$.
Since 
\[
 (u, J)_{F_v} =
 \begin{cases}
  1 & \text{if $B_v$ is split,} \\
  -1 & \text{if $B_v$ is ramified,}
 \end{cases}
\]
the assertion follows from Lemmas \ref{lem:mu-GU(V)-B1-J1}, \ref{lem:mu-GU(V)-B2-J1}, and \ref{lem:mu-GU(W)-J1}.
\end{proof}

Recall that, for almost all $v$, we have a maximal compact subgroup $K_v$ of $\Sp(\V_v)$ and a map $s_{\Y_v}:K_v \rightarrow \C^1$ such that 
\[
 z_{\Y_v}(k, k') = s_{\Y_v}(k k') \cdot s_{\Y_v}(k)^{-1} \cdot s_{\Y_v}(k')^{-1}
\]
for $k, k' \in K_v$.
Put
\[
 \KK_v := \G(\U(V_v) \times \U(W_v))^0 \cap K_v.
\]
Then $\KK_v$ is a maximal compact subgroup of $\G(\U(V_v) \times \U(W_v))^0$ for almost all $v$.

\begin{lem}
\label{lem:spl-B-K}
We have
\[
 s_v|_{\KK_v} = s_{\Y_v}|_{\KK_v}
\]
for almost all $v$.
\end{lem}

\begin{proof}
Recall that $s_v(\g) = s'_v(\g) \cdot \mu_v(\g)$ for $\g \in \GU(V_v)^0 \times \GU(W_v)$, where 
\[
 s'_v : \GU(V_v)^0 \times \GU(W_v) \longrightarrow \C^1
\]
is the map as in \S \S \ref{ss:B-spl}, \ref{ss:B1-spl} and
\[
 \mu_v(\g) = z_{\Y_v}(\h_0 \g \h_0^{-1}, \h_0) \cdot z_{\Y_v}(\h_0, \g)^{-1}
\]
for $\g \in \GSp(\V_v)$ with some $\h_0 \in \Sp(\V_v)$
such that $\X'_v = \X_v \h_0$ and $\Y'_v = \Y_v \h_0$.
By the uniqueness of the splitting, we have 
\[
 s_{\Y_v} =  s_{\Y'_v} \cdot \mu_v|_{K_v}
\]
for almost all $v$.
On the other hand, by definition, one can see that 
\[
 s'_v|_{\KK_v} = s_{\Y'_v}|_{\KK_v}
\]
for almost all $v$.
This yields the lemma.
\end{proof}

\begin{prop}
\label{prop:spl-B-prod}
Let $\gamma \in \GU(V)^0(F) \times \GU(W)(F)$.
Then we have $s_v(\gamma) = 1$ for almost all $v$ and
\[
 \prod_v s_v(\gamma) = 1.
\]
\end{prop}

\begin{proof}
Let $\gamma_1, \gamma_2 \in \GU(V)^0(F) \times \GU(W)(F)$.
Suppose that $s_v(\gamma_i) = 1$ for almost all $v$ and 
$\prod_v s_v(\gamma_i) = 1$ for $i = 1,2$.
Since $s_v(\gamma_1 \gamma_2) = s_v(\gamma_1) \cdot s_v(\gamma_2) \cdot 
 z_{\Y_v}(\gamma_1, \gamma_2)$,
the product formulas for the quadratic Hilbert symbol and the Weil index
imply that $s_v(\gamma_1 \gamma_2) = 1$ for almost all $v$
and $\prod_v s_v(\gamma_1 \gamma_2) = 1$.
Hence the assertion follows from Proposition \ref{prop:spl-B}.
\end{proof}

%% file: spl-double-B.tex
\section{Splittings for the doubling method: quaternionic unitary groups}
\label{sec:spl-double-B}

\subsection{Setup}
\label{ss:spl-double-B-setup}

Let $F$ be a number field and $B$ a quaternion algebra over $F$.
Recall that
\begin{itemize}
\item $V$ is a $2$-dimensional right skew-hermitian $B$-space with $\det V = 1$,
\item $W$ is a $1$-dimensional left hermitian $B$-space,
\item $\V := V \otimes_B W$ is an $8$-dimensional symplectic $F$-space,
\item $\V = \X \oplus \Y$ is a complete polarization over $F$.
\end{itemize} 
We consider a $2$-dimensional left $B$-space $W^{\square} := W \oplus W$
equipped with a hermitian form
\[
 \langle (x, x'), (y, y') \rangle
 := \langle x, y \rangle - \langle x', y' \rangle
\]
for $x,x',y,y' \in W$.
Put $W_+ := W \oplus \{ 0\}$ and $W_- := \{ 0 \} \oplus W$.
We regard $\GU(W_{\pm})$ as a subgroup of $\GL(W)$ and identify it with $\GU(W)$ via the identity map.
Note that the identity map $\GU(W_-) \rightarrow \GU(W)$ is an anti-isometry.
We have a natural map
\[
 \iota : \G(\U(W) \times \U(W)) \longrightarrow \GU(W^{\square})
\]
and seesaw dual pairs
\[
 \xymatrix{
  \GU(W^{\square}) \ar@{-}[dr] \ar@{-}[d] &
  \G(\U(V) \times \U(V)) \ar@{-}[dl] \ar@{-}[d] \\
  \G(\U(W) \times \U(W)) & \GU(V)}.
\]
Put
\[
 W^{\triangle} := \{ (x,x) \in W^{\square} \, | \, x \in W \}, \qquad
 W^{\bigtriangledown} := \{ (x,-x) \in W^{\square} \, | \, x \in W \}.
\]
Then $W^{\square} = W^{\bigtriangledown} \oplus W^{\triangle}$ is a complete polarization over $B$.
Choosing a basis $\w, \w^*$ of $W^{\square}$ such that
\[
 W^{\bigtriangledown} = B \w, \qquad W^{\triangle} = B \w^*, \qquad
 \langle \w, \w^* \rangle = 1,
\]
we may write
\[
 \GU(W^{\square}) = \left\{ g \in \GL_2(B) \, \left| \,
 g \begin{pmatrix} & 1 \\ 1 & \end{pmatrix} {}^t g^*
 = \nu(g) \cdot \begin{pmatrix} & 1 \\ 1 & \end{pmatrix} \right. \right\}.
\]
For $\nu \in F^{\times}$, put
\[
 d(\nu) = d_{W^{\triangle}}(\nu) := \begin{pmatrix} 1 & \\ & \nu \end{pmatrix} \in \GU(W^{\square}).
\]

Similarly, we consider a $16$-dimensional $F$-space $\V^{\square} := V \otimes_B W^{\square} = \V \oplus \V$ equipped with a symplectic form
\begin{equation}
\label{eq:doubled-form}
 \llangle (x, x'), (y, y') \rrangle
 := \llangle x, y \rrangle - \llangle x', y' \rrangle
\end{equation}
for $x,x',y,y' \in \V$.
Put $\V_+ := \V \oplus \{ 0\}$ and $\V_- := \{ 0 \} \oplus \V$.
We regard $\Sp(\V_{\pm})$ as a subgroup of $\GL(\V)$ and identify it with $\Sp(\V)$ via the identity map.
Note that the identity map $\Sp(\V_-) \rightarrow \Sp(\V)$ is an anti-isometry.
We have a natural map
\[
 \iota : \Sp(\V) \times \Sp(\V) \longrightarrow \Sp(\V^{\square}).
\]
Put
\begin{align*}
 \V^{\triangle} & := V \otimes_B W^{\triangle} = \{ (x,x) \in \V^{\square} \, | \, x \in \V \}, & \X^{\square} & := \X \oplus \X, \\
 \V^{\bigtriangledown} & :=  V \otimes_B W^{\bigtriangledown} = \{ (x,-x) \in \V^{\square} \, | \, x \in \V \}, & \Y^{\square} & := \Y \oplus \Y.
\end{align*}
Then $\V^{\square} = \V^{\bigtriangledown} \oplus \V^{\triangle}
= \X^{\square} \oplus \Y^{\square}$ are complete polarizations over $F$.

For the rest of this section, we fix a place $v$ of $F$ and suppress the subscript $v$ from the notation.
Thus $F = F_v$ will be a local field of characteristic zero.
We may lift the natural map $\iota : \Sp(\V) \times \Sp(\V) \rightarrow \Sp(\V^{\square})$ to a unique homomorphism
\[
 \tilde{\iota} : \Mp(\V) \times \Mp(\V) \longrightarrow \Mp(\V^{\square})
\]
such that $\tilde{\iota}(z_1, z_2) = z_1 z_2^{-1}$ for $z_1, z_2 \in \C^1$.

\subsection{Splitting $z_{\V^{\triangle}}$}

First assume that $B$ is split.
Fix an isomorphism $\ii : B \rightarrow \M_2(F)$.
Put $e = \ii^{-1} \left( \begin{smallmatrix} 1 & 0 \\ 0 & 0 \end{smallmatrix} \right)$ and $e' = \ii^{-1} \left( \begin{smallmatrix} 0 & 1 \\ 0 & 0 \end{smallmatrix} \right)$.
Then $W^{\square \dagger} := e W^{\square}$ is a $4$-dimensional symplectic $F$-space and the restriction $\GU(W^{\square}) \rightarrow \GSp(W^{\square \dagger})$ is an isomorphism.
Using a basis $e \w, e' \w, e' \w^*, -e \w^*$ of $W^{\square \dagger}$, we write
\[
 \GSp(W^{\square \dagger}) = \left\{ 
 h \in \GL_4(F) \, \left| \, h 
 \begin{pmatrix} & \1_2 \\ -\1_2 & \end{pmatrix} {}^t h
 = \nu(h) \cdot \begin{pmatrix} & \1_2 \\ -\1_2 & \end{pmatrix} \right. \right\}.
\]
Then the restriction $\GU(W^{\square}) \rightarrow \GSp(W^{\square \dagger})$ is given by
\[
 \begin{pmatrix} a & b \\ c & d \end{pmatrix} \longmapsto
 \begin{pmatrix} \1_2 & \\ & \tau^{-1} \end{pmatrix} \cdot 
 \begin{pmatrix} \ii(a) & \ii(b) \\ \ii(c) & \ii(d) \end{pmatrix} \cdot 
 \begin{pmatrix} \1_2 & \\ & \tau \end{pmatrix},
\]
where $\tau = \tau_1 = \left( \begin{smallmatrix} & -1 \\ 1 & \end{smallmatrix} \right)$.
Note that $x^* = \tau \cdot {}^t x \cdot \tau^{-1}$ for $x \in \M_2(F)$.
Also, $W^\dagger := e W$ is a $2$-dimensional symplectic $F$-space and $V^{\dagger} := V e$ is a $4$-dimensional quadratic $F$-space.
We define a map
\[
 \hat{s} : \G(\U(V) \times \U(W^{\square})) \longrightarrow \C^1
\]
by 
\[
 \hat{s}(\g) = \gamma^{\hat{\jmath}(h)}
\]
for $\g = (g,h) \in \G(\U(V) \times \U(W^{\square}))$, where
\[
 \gamma = 
 \begin{cases}
  1 & \text{if $V^\dagger$ is isotropic,} \\
  -1 & \text{if $V^\dagger$ is anisotropic,}
 \end{cases}
\]
and
\[
 \hat{\jmath}(h) =
 \begin{cases}
  0 & \text{if $c=0$,} \\
  1 & \text{if $c \ne 0$ and $\det c = 0$,} \\
  2 & \text{if $\det c \ne 0$,}
 \end{cases} \qquad
 h = \begin{pmatrix} a & b \\ c & d \end{pmatrix} \in \GSp(W^{\square \dagger}).
\]
Since $\dim_F V^\dagger = 4$ and $\det V^\dagger = 1$, we have
\[
 z_{\V^{\triangle}}(h, h')
 = \hat{s}(h h') \cdot \hat{s}(h)^{-1} \cdot \hat{s}(h')^{-1} 
\]
for $h, h' \in \U(W^{\square})$ by \cite[Theorem 3.1, case $1_+$]{kudla-splitting}.

Next assume that $B$ is ramified.
We define a map
\[
 \hat{s} : \G(\U(V) \times \U(W^{\square})) \longrightarrow \C^1
\]
by 
\[
 \hat{s}(\g) = 1
\]
for $\g \in \G(\U(V) \times \U(W^{\square}))$.
Since $\dim_B V = 2$ and $\det V = 1$, we have
\[
 z_{\V^{\triangle}}(h, h')
 = \hat{s}(h h') \cdot \hat{s}(h)^{-1} \cdot \hat{s}(h')^{-1} 
\]
for $h, h' \in \U(W^{\square})$ by \cite[Theorem 3.1, case $2_-$]{kudla-splitting}.

\begin{lem}
\label{lem:kudla-spl-double-B}
We have 
\[
 z_{\V^{\triangle}}(\g, \g') = \hat{s}(\g \g') \cdot 
 \hat{s}(\g)^{-1} \cdot \hat{s}(\g')^{-1}
\]
for $\g, \g' \in \G(\U(V) \times \U(W^{\square}))$.
\end{lem}

\begin{proof}
Let $\g_i = (g_i, h_i) \in \G(\U(V) \times \U(W^{\square}))$ and
put $h_i' = h_i \cdot d(\nu(h_i))^{-1} \in \U(W^{\square})$.
Then we have $h_1 h_2 = h_1' h_2'' \cdot d(\nu(h_1 h_2))$,
where $h_2'' = d(\nu(h_1)) \cdot h_2' \cdot d(\nu(h_1))^{-1}$.
Since
\[
 \V^{\triangle} \cdot g = \V^{\triangle}, \qquad
 \V^{\triangle} \cdot d(\nu) = \V^{\triangle}
\]
for $g \in \GU(V)$ and $\nu \in F^{\times}$, we have
\[
 \V^{\triangle} \cdot \g_1^{-1} = \V^{\triangle} \cdot h_1^{-1}
 = \V^{\triangle} \cdot h_1'^{-1}, \qquad
 \V^{\triangle} \cdot \g_2^{-1} \g_1^{-1}
 = \V^{\triangle} \cdot h_2^{-1} h_1^{-1}
 = \V^{\triangle} \cdot h_2''^{-1} h_1'^{-1}.
\]
Hence we have
\begin{align*}
 q(\V^{\triangle}, \V^{\triangle} \cdot \g_2^{-1}, \V^{\triangle} \cdot \g_1)
 & = q(\V^{\triangle} \cdot \g_1^{-1}, \V^{\triangle} \cdot \g_2^{-1} \g_1^{-1}, \V^{\triangle}) \\
 & = q(\V^{\triangle} \cdot h_1'^{-1}, \V^{\triangle} \cdot h_2''^{-1} h_1'^{-1}, \V^{\triangle}) \\
 & = q(\V^{\triangle}, \V^{\triangle} \cdot h_2''^{-1}, \V^{\triangle} \cdot h_1'),
\end{align*}
so that
\[
 z_{\V^{\triangle}}(\g_1, \g_2) = z_{\V^{\triangle}}(h_1', h_2'') 
 = \hat{s}(h'_1 h''_2) \cdot \hat{s}(h'_1)^{-1} \cdot \hat{s}(h''_2)^{-1}.
\]
By definition, we have $\hat{s}(h_1') = \hat{s}(\g_1)$,
$\hat{s}(h_2'') = \hat{s}(h_2') = \hat{s}(\g_2)$, and 
\[
 \hat{s}(h_1' h_2'') = \hat{s}(h_1 h_2 \cdot d(\nu(h_1 h_2))^{-1}) = \hat{s}(\g_1 \g_2). 
\]
This completes the proof.
\end{proof}

\subsection{Splitting $z_{\Y'^{\square}}$}
\label{ss:spl-double-B-Y'}

Let $\V = \X' \oplus \Y'$ be the complete polarization given in \S \ref{ss:B-spl}, \ref{ss:B1-spl}.
Put
\[
 \X'^{\square} := \X' \oplus \X', \qquad \Y'^{\square} := \Y' \oplus \Y'. 
\]
Then $\V^{\square} = \X'^{\square} \oplus \Y'^{\square}$ is a complete polarization.
Noting that the symplectic form on $\V^{\square} = \V \oplus \V$ is given by \eqref{eq:doubled-form}, we have
\[
 z_{\Y'^{\square}, \psi}(\iota(g_1, g_2), \iota(g_1',g_2'))
 = z_{\Y', \psi}(g_1, g_1') \cdot z_{\Y', \psi^{-1}}(g_2, g_2')
 = z_{\Y', \psi}(g_1, g_1') \cdot z_{\Y', \psi}(g_2, g_2')^{-1}
\]
for $g_i, g_i' \in \Sp(\V)$, where we write $z_{\Y'} = z_{\Y', \psi}$ to indicate the dependence of the $2$-cocycle on $\psi$.
The Weil representation $\omega_{\psi}^{\square}$ of $\Mp(\V^{\square})$ can be realized on the Schwartz space
\[
 \SS(\X'^{\square}) = \SS(\X') \otimes \SS(\X').
\]
As representations of $\Mp(\V)_{\Y'} \times \Mp(\V)_{\Y'}$, we have
\[
 \omega_{\psi}^{\square} \circ \tilde{\iota} = \omega_{\psi} \otimes (\omega_{\psi} \circ \tilde{\jj}_{\Y'}),
\]
where $\tilde{\jj}_{\Y'}$ is the automorphism of $\Mp(\V)_{\Y'} = \Sp(\V) \times \C^1$ defined by
\[
 \tilde{\jj}_{\Y'}(g, z) = (\jj_{\Y'}(g), z^{-1}), \qquad
 \jj_{\Y'}(g) = d_{\Y'}(-1) \cdot g \cdot d_{\Y'}(-1).
\]

Fix $\h_0' \in \Sp(\V^{\square})$ such that $\X'^{\square} = \V^{\bigtriangledown} \cdot \h_0'$ and $\Y'^{\square} = \V^{\triangle} \cdot \h_0'$.
Put
\[
 \mu'(g)
 = z_{\V^{\triangle}}(g, \h_0'^{-1})
 \cdot z_{\V^{\triangle}}(\h_0'^{-1}, \h_0' g \h_0'^{-1})^{-1}
\]
for $g \in \Sp(\V^{\square})$.
Then we have
\[
 z_{\Y'^{\square}}(g, g') = z_{\V^{\triangle}}(g,g') \cdot
 \mu'(g g') \cdot \mu'(g)^{-1} \cdot \mu'(g')^{-1}
\]
for $g, g' \in \Sp(\V^{\square})$.
Put
\[
 \Gc := \{ (g, h_1, h_2) \in \GU(V)^0 \times \GU(W) \times \GU(W)
 \, | \, \nu(g) = \nu(h_1) = \nu(h_2) \}.
\]
We have natural maps
\[
 \Gc \hookrightarrow \G(\U(V) \times \U(W^{\square})), \qquad
 \Gc \hookrightarrow \G(\U(V)^0 \times \U(W)) \times \G(\U(V)^0 \times \U(W)).
\]

\begin{lem}
\label{lem:spl-double-B-compare'}
We have
\[
 \hat{s} \cdot \mu' = s' \otimes (s' \circ \jj_{\Y'})
\]
on $\Gc$, where $s': \GU(V)^0 \times \GU(W) \rightarrow \C^1$ is the map defined in \S \ref{ss:B-spl}, \ref{ss:B1-spl}.
\end{lem}

The proof of this lemma will be given in the next two sections.

\subsubsection{The case $u \in (F^{\times})^2$ or $J \in (F^{\times})^2$}

Recall that $\X' = V^\dagger \otimes_F X$ and $\Y' = V^\dagger \otimes_F Y$, where $X = Fe$ and $Y = Fe'$, and $W^\dagger = X + Y$ is a complete polarization over $F$.
We have
\[
 \X'^{\square} = V^\dagger \otimes_F X^{\square}, \qquad
 \Y'^{\square} = V^\dagger \otimes_F Y^{\square},
\]
where $X^{\square} = X \oplus X$ and $Y^{\square} = Y \oplus Y$.
We have $d_{\Y'}(-1) = \id \otimes d_Y(-1)$ and $\jj_{\Y'} = \id \otimes \jj_Y$, where
\[
 d_Y(\nu) = \begin{pmatrix} 1 & \\ & \nu \end{pmatrix} \in \GSp(W^\dagger)
\]
and $\jj_Y(h) = d_Y(-1) \cdot h \cdot d_Y(-1)$ for $h \in \GSp(W^\dagger)$.
In particular, we have 
\[
 \jj_{\Y'}(\G(\U(V)^0 \times \U(W))) = \G(\U(V)^0 \times \U(W)).
\]

Let $\iota : \G(\Sp(W^\dagger) \times \Sp(W^\dagger)) \rightarrow \GSp(W^{\square \dagger})$ be the natural map.
We may take
\[
 \w = \frac{1}{2}(1,-1), \qquad \w^* = (1,1). 
\]
Since
\[
 \begin{bmatrix}
  (e,0) \\
  (0,e) \\
  (e',0) \\
  (0,-e')
 \end{bmatrix}
 = h_0 \cdot 
 \begin{bmatrix}
  e \w \\
  e' \w \\
  e' \w^* \\
  -e \w^*
 \end{bmatrix}, \qquad
 h_0 = 
 \begin{pmatrix}
 1 & & & -\frac{1}{2} \\
 -1 & & & -\frac{1}{2} \\
 & 1 & \frac{1}{2} & \\
 & 1 & -\frac{1}{2} & 
 \end{pmatrix}
 \in \Sp(W^{\square \dagger}),
\]
we have
\[
 \iota(h_1, h_2) = h_0^{-1} \cdot 
 \iota^{\natural}(h_1, \jj_Y(h_2)) \cdot h_0.
\]
Here, using a basis $e, e'$ of $W^\dagger$, we identify $\GSp(W^\dagger)$ with $\GL_2(F)$ and put
\[
 \iota^{\natural}(h_1, h_2) =
 \begin{pmatrix}
  a_1 & & b_1 & \\
  & a_2 & & b_2 \\  
  c_1 & & d_1 & \\
  & c_2 & & d_2 
 \end{pmatrix}, 
 \qquad 
 h_i =
 \begin{pmatrix}
  a_i & b_i \\
  c_i & d_i
 \end{pmatrix}.
\]
Since $X^{\square} = e W^{\bigtriangledown} \cdot h_0$ and $Y^{\square} = e W^{\triangle} \cdot h_0$, we may take $\h_0' = \id \otimes h_0$.

\begin{proof}[Proof of Lemma \ref{lem:spl-double-B-compare'}]
Let $\g = (g,h_1,h_2) \in \Gc$ and $\nu = \nu(\g)$.
Put $\g_i = (g, h_i) \in \G(\U(V)^0 \times \U(W))$.
By definition, we have 
\[
 s'(\g_1) \cdot s'(\jj_{\Y'}(\g_2))
 = \gamma^{j(h_1)} \cdot \gamma^{j(\jj_Y(h_2))},
\]
where $j$ is as in \S \ref{ss:B-spl}.

Put $h = \iota(h_1, h_2) \in \GU(W^{\square})$.
We identify $\g$ with $(g, h) \in \G(\U(V)^0 \times \U(W^{\square}))$.
Since $\hat{\jmath}(h) = \hat{\jmath}(d(\nu)^{-1} \cdot h)$, we have
\[
 \hat{s}(\g) = \hat{s}(d(\nu)^{-1} \cdot h).
\]
Put $\g' = (g, d(\nu)) \in \G(\U(V)^0 \times \U(W^{\square}))$.
Then we have $\V^{\triangle} \cdot \g' = \V^{\triangle}$ and
\[
 \g = \g' \cdot d(\nu)^{-1} \cdot h, \qquad
 \h_0' \g \h_0'^{-1} = h_0 h h_0^{-1} \cdot g
 = h_0 h h_0^{-1} \cdot d(\nu)^{-1} \cdot \g'.
\]
Hence, by Lemma \ref{lem:kudla-spl-double-B}, we have
\begin{align*}
 z_{\V^{\triangle}}(\g, \h_0'^{-1})
 & = z_{\V^{\triangle}}(d(\nu)^{-1} \cdot h, h_0^{-1}) \\
 & = \hat{s}(d(\nu)^{-1} \cdot h h_0^{-1})
 \cdot \hat{s}(d(\nu)^{-1} \cdot h)^{-1}
 \cdot \hat{s}(h_0^{-1})^{-1}, \\
 z_{\V^{\triangle}}(\h_0'^{-1}, \h_0' \g \h_0'^{-1})
 & = z_{\V^{\triangle}}(h_0^{-1}, h_0 h h_0^{-1} \cdot d(\nu)^{-1}) \\
 & = \hat{s}(h h_0^{-1} \cdot d(\nu)^{-1}) \cdot \hat{s}(h_0^{-1})^{-1} \cdot \hat{s}(h_0 h h_0^{-1} \cdot d(\nu)^{-1})^{-1}.
\end{align*}
Since $\hat{\jmath}(d(\nu)^{-1} \cdot h h_0^{-1}) = \hat{\jmath}(h h_0^{-1} \cdot d(\nu)^{-1})$, we have
\[
 \hat{s}(d(\nu)^{-1} \cdot h h_0^{-1}) = \hat{s}(h h_0^{-1} \cdot d(\nu)^{-1}).
\]
Since $h_0 h h_0^{-1} = \iota^{\natural}(h_1, \jj_Y(h_2))$, we have
$\hat{\jmath}(h_0 h h_0^{-1} \cdot d(\nu)^{-1}) = j(h_1) + j(\jj_Y(h_2))$ and
\[
 \hat{s}(h_0 h h_0^{-1} \cdot d(\nu)^{-1}) = \gamma^{j(h_1) + j(\jj_Y(h_2))}.
\]
Thus we obtain
\begin{align*}
 \hat{s}(\g) \cdot \mu'(\g) & =
 \hat{s}(d(\nu)^{-1} \cdot h) \cdot 
 z_{\V^{\triangle}}(\g, \h_0'^{-1}) \cdot
 z_{\V^{\triangle}}(\h_0'^{-1}, \h_0' \g \h_0'^{-1})^{-1} \\
 & = \hat{s}(h_0 h h_0^{-1} \cdot d(\nu)^{-1}) \\
 & = \gamma^{j(h_1) + j(\jj_Y(h_2))}.
\end{align*}
This completes the proof.
\end{proof}

\subsubsection{The case $J_i \in (F^{\times})^2$}
We only consider the case $i = 1$; the other case is similar.
As in \S \ref{ss:B1-spl}, we regard $V$ and $W$ as left and right $B$-spaces, respectively.
Recall that $\X' = W \otimes_B X$ and $\Y' = W \otimes_B Y$, where $X = B \v$ and $Y = B \v^*$, and $V = X + Y$ is a complete polarization over $B$.
As in \S \ref{ss:spl-double-B-setup}, we define a $4$-dimensional left skew-hermitian $B$-space $V^{\square} = V \oplus V$ and a complete polarization $V^{\square} = V^{\bigtriangledown} \oplus V^{\triangle}$ over $B$.
Using a basis
\[
 \v_1 := \frac{1}{2}(\v, -\v), \qquad 
 \v_2 := \frac{1}{2}(\v^*, -\v^*), \qquad
 \v_1^* := (\v^*, \v^*), \qquad
 \v_2^* := (-\v, -\v) 
\]
of $V^{\square}$, we write
\[
 \GU(V^{\square}) = \left\{ g \in \GL_4(B) \, \left| \, g
 \begin{pmatrix} & \1_2 \\ -\1_2 & \end{pmatrix} {}^t g^*
 = \nu(g) \cdot \begin{pmatrix} & \1_2 \\ -\1_2 & \end{pmatrix} \right. \right\}.
\]
We may identify $\V^{\square}$ with $W \otimes_B V^{\square}$.
Under this identification, we have
\begin{align*}
 \V^{\bigtriangledown} & = W \otimes_B V^{\bigtriangledown}, & 
 \X'^{\square} & = W \otimes_B X^{\square}, \\
 \V^{\triangle} & = W \otimes_B V^{\triangle}, &
 \Y'^{\square} & = W \otimes_B Y^{\square},
\end{align*}
where $X^{\square} = X \oplus X$ and $Y^{\square} = Y \oplus Y$.
We have $d_{\Y'}(-1) = \id \otimes d_Y(-1)$ and $\jj_{\Y'} = \id \otimes \jj_Y$, where
\[
 d_Y(\nu) = \begin{pmatrix} 1 & \\ & \nu \end{pmatrix} \in \GU(V)^0
\]
and $\jj_Y(g) = d_Y(-1) \cdot g \cdot d_Y(-1)$ for $g \in \GU(V)$.
In particular, we have
\[
 \jj_{\Y'}(\G(\U(V)^0 \times \U(W))) = \G(\U(V)^0 \times \U(W)).
\]

Let $\iota : \G(\U(V) \times \U(V)) \rightarrow \GU(V^{\square})$ be the natural map.
Since
\[
 \begin{bmatrix}
  (\v,0) \\
  (0,\v) \\
  (\v^*,0) \\
  (0,-\v^*)
 \end{bmatrix}
 = g_0 \cdot 
 \begin{bmatrix}
 \v_1 \\
 \v_2 \\
 \v_1^* \\
 \v_2^*
 \end{bmatrix}, \qquad
 g_0 = 
 \begin{pmatrix}
 1 & & & -\frac{1}{2} \\
 -1 & & & -\frac{1}{2} \\
 & 1 & \frac{1}{2} & \\
 & 1 & -\frac{1}{2} & 
 \end{pmatrix}
 \in \U(V^{\square}),
\]
we have
\[
 \iota(g_1, g_2) = g_0^{-1} \cdot \iota^{\natural}(g_1, \jj_Y(g_2)) \cdot g_0.
\]
Here, regarding $V$ as a left $B$-space and using a basis $\v, \v^*$ of $V$, we identify $\GU(V)$ with a subgroup of $\GL_2(B)$ and put
\[
 \iota^{\natural}(g_1, g_2) =
 \begin{pmatrix}
  a_1 & & b_1 & \\
  & a_2 & & b_2 \\  
  c_1 & & d_1 & \\
  & c_2 & & d_2 
 \end{pmatrix}, 
 \qquad 
 g_i =
 \begin{pmatrix}
  a_i & b_i \\
  c_i & d_i
 \end{pmatrix}.
\]
Since $X^{\square} = V^{\bigtriangledown} \cdot g_0$ and $Y^{\square} = V^{\triangle} \cdot g_0$, we may take $\h_0' = \id \otimes g_0$.

When $B$ is split, we define a map
\[
 \hat{s}' : \U(V^{\square}) \longrightarrow \C^1
\]
by
\[
 \hat{s}'(g) = 1
\]
for $g \in \U(V^{\square})$.
Then we have
\[
 z_{\V^{\triangle}}(g, g')
 = \hat{s}'(g g') \cdot \hat{s}'(g)^{-1} \cdot \hat{s}'(g')^{-1} 
\]
for $g, g' \in \U(W^{\square})$ by \cite[Theorem 3.1, case $1_-$]{kudla-splitting}.
When $B$ is ramified, we define a map
\[
 \hat{s}' : \U(V^{\square}) \longrightarrow \C^1
\]
by
\[
 \hat{s}'(g) = (-1)^{\hat{\jmath}'(g)}
\]
for $g \in \U(V^{\square})$, where
\[
 \hat{\jmath}'(g) =
 \begin{cases}
  0 & \text{if $c=0$,} \\
  1 & \text{if $c \ne 0$ and $\nu(c) = 0$,} \\
  2 & \text{if $\nu(c) \ne 0$,}
 \end{cases} \qquad
 g = \begin{pmatrix} a & b \\ c & d \end{pmatrix} \in \U(V^{\square}),
\]
and $\nu:\M_2(B) \rightarrow F$ is the reduced norm.
Since $\dim_B W = 1$, we have
\[
 z_{\V^{\triangle}}(g, g')
 = \hat{s}'(g g') \cdot \hat{s}'(g)^{-1} \cdot \hat{s}'(g')^{-1} 
\]
for $g, g' \in \U(W^{\square})$ by \cite[Theorem 3.1, case $2_+$]{kudla-splitting}.

\begin{proof}[Proof of Lemma \ref{lem:spl-double-B-compare'}]
Let $\g = (g,h_1,h_2) \in \Gc$ and $\nu = \nu(\g)$.
Put $\g_i = (g, h_i) \in \G(\U(V)^0 \times \U(W))$.
By definition, we have $j(\jj_Y(g)) = j(g)$, where $j$ is as in \S \ref{ss:B1-spl}.
Hence we have
\[
 s'(\g_1) \cdot s'(\jj_{\Y'}(\g_2)) = \gamma^{j(g)} \cdot \gamma^{j(\jj_Y(g))} = 1.
\]
Here 
\[
 \gamma = 
 \begin{cases}
  1 & \text{if $B$ is split,} \\
  -1 & \text{if $B$ is ramified.}
 \end{cases}
\]
Also, by definition, we have 
\[
 \hat{s}(\g) = 1.
\]

Now we compute $z_{\V^{\triangle}}(\g, \h_0'^{-1})$.
We identify $\g$ with $(g, \iota(h_1, h_2)) \in \G(\U(V)^0 \times \U(W^{\square}))$, where $\iota : \G(\U(W) \times \U(W)) \rightarrow \GU(W^{\square})$ is the natural map.
Put $\g' = (g, \iota(h_2, h_2)) \in \G(\U(V)^0 \times \U(W^{\square}))$ and $h = h_1^{-1} h_2 \in \U(W)$.
Then we have $\g = \g' \cdot \iota(h^{-1}, 1)$.
Via the identification $\V^{\square} = V \otimes_B W^{\square} = W \otimes_B V^{\square}$, we identify $\g'$ with $(h_2, \iota(g,g)) \in \G(\U(W) \times \U(V^{\square}))$.
Since $\V^{\triangle} \cdot \g' = \V^{\triangle}$, we have
\[
 z_{\V^{\triangle}}(\g, \h_0'^{-1}) = z_{\V^{\triangle}}(\iota(h^{-1},1), g_0^{-1}).
\]
Put
\[
 \tau := 
 \begin{pmatrix}
  1 & & & \\
  & & & -1 \\
  & & 1 & \\
  & 1 & &
 \end{pmatrix}
 \in \U(V^{\square}).
\]
Then we have
\[
 g_0 \tau = 
 \begin{pmatrix}
  1 & -\frac{1}{2} & & \\
  -1 & -\frac{1}{2} & & \\
  & & \frac{1}{2} & -1 \\
  & & -\frac{1}{2} & -1
 \end{pmatrix}
\]
and $\V^{\triangle} \cdot \tau^{-1} g_0^{-1} = \V^{\triangle}$, so that
\[
 z_{\V^{\triangle}}(\iota(h^{-1},1), g_0^{-1})
 = z_{\V^{\triangle}}(\iota(h^{-1},1), \tau).
\]
Under the identification $\V^{\square} = \V \oplus \V = W \otimes_B V^{\square}$, we have
\begin{align*}
 x \otimes (y, \pm y) \cdot \iota(h^{-1},1)
 & = (x \otimes y, \pm x \otimes y) \cdot \iota(h^{-1},1) \\
 & = (h x \otimes y, \pm x \otimes y) \\
 & = \frac{1}{2} ((h \pm 1) x \otimes y, (h \pm 1) x \otimes y)
 + \frac{1}{2} ((h \mp 1) x \otimes y, - (h \mp 1) x \otimes y) \\
 & = \frac{1}{2} (h \pm 1) x \otimes (y, y)
 + \frac{1}{2} (h \mp 1) x \otimes (y, -y)
\end{align*}
for $x \in W$ and $y \in V$.
Thus we obtain
\begin{align*}
 x \otimes \v_1 \cdot \iota(h^{-1}, 1) & =
 \frac{1}{2} (h+1) x \otimes \v_1 - \frac{1}{4} (h-1) x \otimes \v_2^*, \\
 x \otimes \v_2 \cdot \iota(h^{-1}, 1) & = 
 \frac{1}{2} (h+1) x \otimes \v_2 + \frac{1}{4} (h-1) x \otimes \v_1^*, \\
 x \otimes \v_1^* \cdot \iota(h^{-1}, 1) & =
 \frac{1}{2} (h+1) x \otimes \v_1^*
 + (h-1) x \otimes \v_2, \\
 x \otimes \v_2^* \cdot \iota(h^{-1}, 1) & = 
 \frac{1}{2} (h+1) x \otimes \v_2^*
 - (h-1) x \otimes \v_1.
\end{align*}

First assume that $B$ is split.
Fix an isomorphism $\ii : B \rightarrow \M_2(F)$.
Put $e = \ii^{-1} \left( \begin{smallmatrix} 1 & 0 \\ 0 & 0 \end{smallmatrix} \right)$, $e' = \ii^{-1} \left( \begin{smallmatrix} 0 & 1 \\ 0 & 0 \end{smallmatrix} \right)$, and $e'' = \ii^{-1} \left( \begin{smallmatrix} 0 & 0 \\ 1 & 0 \end{smallmatrix} \right)$.
Then $W^\dagger := We$ is a $2$-dimensional symplectic $F$-space and the restriction $\GU(W) \rightarrow \GSp(W^\dagger)$ is an isomorphism.
Also, $V^{\square \dagger} := e V^{\square}$ is an $8$-dimensional quadratic $F$-space.
We identify $\V^{\square}$ with $W^\dagger \otimes_F V^{\square \dagger}$.
Put $f := 2 e''$ and
\begin{align*}
 \x_1 & := e \v_1, & \x_2 & := e' \v_1, &
 \x_3 & := e \v_2, & \x_4 & := e' \v_2, \\
 \y_1 & := e' \v_1^*, & \y_2 & := -e \v_1^*, &
 \y_3 & := e' \v_2^*, & \y_4 & := -e \v_2^*.
\end{align*}
Using a basis
\[
 e \otimes \x_1, f \otimes \x_1, \ldots, e \otimes \x_4, f \otimes \x_4,
 f \otimes \y_1, -e \otimes \y_1, \ldots, f \otimes \y_4, -e \otimes \y_4
\]
of $\V^{\square}$, we identify $\Sp(\V^{\square})$ with $\Sp_{16}(F)$.
We define $\h \in \GL_2(F)$ by
\[
 \begin{bmatrix}
  he \\
  hf
 \end{bmatrix}
 = \h \cdot 
 \begin{bmatrix}
  e \\
  f
 \end{bmatrix}.
\]
Then we have $\h \cdot \a \cdot {}^t \h = \a$, where
\[
 \a = \begin{pmatrix} & -1 \\ 1 & \end{pmatrix} \in \GL_2(F).
\]
Moreover, we have $\iota(h^{-1}, 1) = \d^{-1} \cdot \h' \cdot \d$ and
\[
 \tau = \d^{-1} \cdot
 \begin{pmatrix}
  \1_2 & & & & & & & \\
  & \1_2 & & & & & & \\
  & & \0_2 & & & & & \1_2 \\
  & & & \0_2 & & & -\1_2 & \\
  & & & & \1_2 & & & \\
  & & & & & \1_2 & & \\
  & & & \1_2 & & & \0_2 & \\
  & & -\1_2 & & & & & \0_2
 \end{pmatrix}
 \cdot \d
 = \tau_4 \cdot 
 \begin{pmatrix}
  \1_2 & & & & & & & \\
  & \1_2 & & & & & & \\
  & & & -\a & & & & \\
  & & \a & & & & & \\
  & & & & \1_2 & & & \\
  & & & & & \1_2 & & \\
  & & & & & & & -\a \\
  & & & & & & \a & 
 \end{pmatrix},
\]
where
\[
 \d =
 \begin{pmatrix}
  \1_2 & & & & & & & \\
  & \1_2 & & & & & & \\
  & & \1_2 & & & & & \\
  & & & \1_2 & & & & \\
  & & & & \a & & & \\
  & & & & & \a & & \\
  & & & & & & \a & \\
  & & & & & & & \a
 \end{pmatrix}, \qquad
 \h' = 
 \begin{pmatrix}
  \dot{\h} & & & & & & & \frac{1}{2} \ddot{\h} \\
  & \dot{\h} & & & & & -\frac{1}{2} \ddot{\h} & \\
  & & \dot{\h} & & & -\frac{1}{2} \ddot{\h} & & \\
  & & & \dot{\h} & \frac{1}{2} \ddot{\h} & & & \\
  & & & 2 \ddot{\h} & \dot{\h} & & & \\
  & & -2 \ddot{\h} & & & \dot{\h} & & \\
  & -2 \ddot{\h} & & & & & \dot{\h} & \\
  2 \ddot{\h} & & & & & & & \dot{\h}
 \end{pmatrix},
\]
$\dot{\h} = \frac{1}{2} (\h + \1_2)$,
$\ddot{\h} = \frac{1}{2} (\h - \1_2)$, and 
\[
 \tau_4 = 
 \begin{pmatrix}
  \1_4 & & & \\
  & & & -\1_4 \\
  & & \1_4 & \\
  & \1_4 & & \\
 \end{pmatrix}.
\]
If $h=1$, then we have $z_{\V^{\triangle}}(\iota(h^{-1},1), \tau) = z_{\V^{\triangle}}(1, \tau) = 1$.
Assume that $h \ne 1$ and $\det \ddot{\h} = 0$.
Since $\det \h = 1$, we have $\tr \ddot{\h} = 0$.
Hence we may take $\a_1 \in \SL_2(F)$ such that $\a_1 \cdot \ddot{\h} \cdot \a_1^{-1} = \x$, where
\[
 \x = \begin{pmatrix} 0 & x \\ 0 & 0 \end{pmatrix}
\]
with some $x \ne 0$.
Put
\[
 \m_1 = 
 \begin{pmatrix}
  \a_1 & & & & & & & \\
  & \a_1 & & & & & & \\
  & & \a_1 & & & & & \\
  & & & \a_1 & & & & \\
  & & & & {}^t \a_1^{-1} & & & \\
  & & & & & {}^t \a_1^{-1} & & \\
  & & & & & & {}^t \a_1^{-1} & \\
  & & & & & & & {}^t \a_1^{-1}
 \end{pmatrix}.
\]
Noting that $\a_1 \cdot \dot{\h} \cdot \a_1^{-1} = \x + \1_2$ and $\a \cdot {}^t \a_1^{-1} \cdot \a^{-1} = \a_1$, we have $\m_1 \cdot \iota(h^{-1}, 1) \cdot \m_1^{-1} = \d^{-1} \cdot \h'' \cdot \d$, where
\[
 \h'' =
 \begin{pmatrix}
  \x + \1_2 & & & & & & & \frac{1}{2} \x \\
  & \x + \1_2 & & & & & -\frac{1}{2} \x & \\
  & & \x + \1_2 & & & -\frac{1}{2} \x & & \\
  & & & \x + \1_2 & \frac{1}{2} \x & & & \\
  & & & 2 \x & \x + \1_2 & & & \\
  & & -2 \x & & & \x + \1_2 & & \\
  & -2 \x & & & & & \x + \1_2 & \\
  2 \x & & & & & & & \x + \1_2
 \end{pmatrix}.
\]
We have
\begin{align*}
 z_{\V^{\triangle}}(\iota(h^{-1},1), \tau)
 & = z_{\V^{\triangle}}(\iota(h^{-1},1), \tau_4) \\
 & = z_{\V^{\triangle}}(\m_1 \cdot \iota(h^{-1},1),
 \tau_4 \cdot \tau_4^{-1} \cdot \m_1^{-1} \cdot \tau_4) \\
 & = z_{\V^{\triangle}}(\m_1 \cdot \iota(h^{-1},1) \cdot \m_1^{-1}, \tau_4) \\
 & = z_{\V^{\triangle}}(\d^{-1} \cdot \h'' \cdot \d, \tau_4).
\end{align*}
Put
\[
 \e = \begin{pmatrix} 1 & \\ & 0 \end{pmatrix}, \qquad
 \e^* = \begin{pmatrix} 0 & \\ & 1 \end{pmatrix}, \qquad
 \a_2 = \begin{pmatrix} -x^{-1} & \\ & -x \end{pmatrix},
\]
and
\[
 \m_2 = 
 \begin{pmatrix}
  & & & 2 \a_2 & & & & \\
  & & -2 \a_2 & & & & & \\
  & -2 \a_2 & & & & & & \\
  2 \a_2 & & & & & & & \\
  & & & & & & & \frac{1}{2} {}^t \a_2^{-1} \\
  & & & & & & -\frac{1}{2} {}^t \a_2^{-1} & \\
  & & & & & -\frac{1}{2} {}^t \a_2^{-1} & & \\
  & & & & \frac{1}{2} {}^t \a_2^{-1} & & &
 \end{pmatrix}.
\]
Then we have $\a_2 \cdot \x \cdot \a = -\e$, ${}^t \a_2^{-1} \cdot \a^{-1} \cdot \x = \e^*$, and hence $\m_2 \cdot \d^{-1} \cdot \h'' \cdot \d = \h'''$, where
\[
 \h''' = 
 \begin{pmatrix}
  & & & 2 \x' & -\e & & & \\
  & & -2 \x' & & & -\e & & \\
  & -2 \x' & & & & & -\e & \\
  2 \x' & & & & & & & -\e \\
  \e^* & & & & & & & \frac{1}{2} \x'' \\
  & \e^* & & & & & -\frac{1}{2} \x'' & \\
  & & \e^* & & & -\frac{1}{2} \x'' & & \\
  & & & \e^* & \frac{1}{2} \x'' & & & 
 \end{pmatrix},
\]
$\x' = - \e \a^{-1} + \a_2$, and $\x'' = \e^* \a + {}^t \a_2^{-1}$.
We have 
\[
 z_{\V^{\triangle}}(\d^{-1} \cdot \h'' \cdot \d, \tau_4)
 = z_{\V^{\triangle}}(\m_2 \cdot \d^{-1} \cdot \h'' \cdot \d, \tau_4)
 = z_{\V^{\triangle}}(\h''', \tau_4).
\]
Put
\[
 \b = \begin{pmatrix} 0 & \\ & x^{-1} \end{pmatrix}
\]
and
\[
 \n = 
 \begin{pmatrix}
  \1_2 & & & & & & & \frac{1}{2} \b \\
  & \1_2 & & & & & - \frac{1}{2} \b & \\
  & & \1_2 & & & - \frac{1}{2} \b & & \\
  & & & \1_2 & \frac{1}{2} \b & & & \\
  & & & & \1_2 & & & \\
  & & & & & \1_2 & & \\
  & & & & & & \1_2 & \\
  & & & & & & & \1_2  
 \end{pmatrix}.
\]
We have
\[
 \h''' \cdot \n =
 \begin{pmatrix}
  & & & 2 \x' & \x' \b - \e & & & \\
  & & -2 \x' & & & \x' \b - \e & & \\
  & -2 \x' & & & & & \x' \b - \e & \\
  2 \x' & & & & & & & \x' \b - \e \\
  \e^* & & & & & & & \frac{1}{2} (\e^* \b + \x'') \\
  & \e^* & & & & & -\frac{1}{2} (\e^* \b + \x'') & \\
  & & \e^* & & & -\frac{1}{2} (\e^* \b + \x'') & & \\
  & & & \e^* & \frac{1}{2} (\e^* \b + \x'') & & &
 \end{pmatrix}.
\]
Since
\[
 \e^* \b + \x'' = 
 \begin{pmatrix}
  -x & 0 \\
  1 & 0
 \end{pmatrix},
\]
we have $\V^{\triangle} \cdot \h''' \cdot \n \cdot \boldsymbol{\tau}^{-1} = \V^{\triangle}$, where
\[
 \boldsymbol{\tau} = 
 \begin{pmatrix}
  \e & & & & -\e^* & & & \\
  & \e & & & & -\e^* & & \\
  & & \e & & & & -\e^* & \\
  & & & \e & & & & -\e^* \\
  \e^* & & & & \e & & & \\
  & \e^* & & & & \e & & \\
  & & \e^* & & & & \e & \\
  & & & \e^* & & & & \e
 \end{pmatrix}.
\]
Hence we have
\[
 z_{\V^{\triangle}}(\h''', \tau_4)
 = z_{\V^{\triangle}}(\h''', \tau_4 \cdot \tau_4^{-1} \cdot \n \cdot \tau_4)
 = z_{\V^{\triangle}}(\h''' \cdot \n, \tau_4)
 = z_{\V^{\triangle}}(\boldsymbol{\tau}, \tau_4)
 = 1.
\]
Assume that $h \ne 1$ and $\det \ddot{\h} \ne 0$.
We have
\begin{align*}
 & \iota(h^{-1}, 1) \\ & = 
 \begin{pmatrix}
  & & & \frac{1}{2} {}^t \a {}^t \ddot{\h}^{-1} & \dot{\h} & & & \\
  & & - \frac{1}{2} {}^t \a {}^t \ddot{\h}^{-1} & & & \dot{\h} & & \\
  & - \frac{1}{2} {}^t \a {}^t \ddot{\h}^{-1} & & & & & \dot{\h} & \\
  \frac{1}{2} {}^t \a {}^t \ddot{\h}^{-1} & & & & & & & \dot{\h} \\
  & & & & & & & 2 \a^{-1} \ddot{\h} \\
  & & & & & & -2 \a^{-1} \ddot{\h} & \\
  & & & & & -2 \a^{-1} \ddot{\h} & & \\
  & & & & 2 \a^{-1} \ddot{\h} & & & 
 \end{pmatrix}
 \cdot \tau_8 \cdot \n,
\end{align*}
where
\[
 \tau_8 =
 \begin{pmatrix}
  & -\1_8 \\
  \1_8 &
 \end{pmatrix},
 \qquad \n =
 \begin{pmatrix}
  \1_2 & & & & & & & \frac{1}{2} \ddot{\h}^{-1} \dot{\h} \a \\
  & \1_2 & & & & & -\frac{1}{2} \ddot{\h}^{-1} \dot{\h} \a & \\
  & & \1_2 & & & -\frac{1}{2} \ddot{\h}^{-1} \dot{\h} \a & & \\
  & & & \1_2 & \frac{1}{2} \ddot{\h}^{-1} \dot{\h} \a & & & \\
  & & & & \1_2 & & & \\
  & & & & & \1_2 & & \\
  & & & & & & \1_2 & \\
  & & & & & & & \1_2
 \end{pmatrix}.
\]
Hence we have
\[
 z_{\V^{\triangle}}(\iota(h^{-1},1), \tau)
 = z_{\V^{\triangle}}(\tau_8 \cdot \n, \tau_4)
 = z_{\V^{\triangle}}(\tau_8, \n \cdot \tau_4).
\]
Since $\V^{\triangle} \cdot \tau_4^{-1} \n \tau_4= \V^{\triangle}$, we have
\[
 z_{\V^{\triangle}}(\tau_8, \n \cdot \tau_4) = z_{\V^{\triangle}}(\tau_8, \tau_4) = 1.
\]
Thus we obtain
\[
 z_{\V^{\triangle}}(\g, \h_0'^{-1}) = 1.
\]

Next assume that $B$ is ramified.
Choose a basis $e_1, e_2, e_3, e_4$ of $W$ over $F$.
We may assume that $\frac{1}{2} \tr_{B/F}( \langle e_i, e_j \rangle ) = a_i \cdot \delta_{ij}$ with some $a_i \in F^{\times}$.
Put $e'_i := e_i \cdot a_i^{-1}$.
Using a basis
\[
  e_1 \otimes \v_1, \ldots, e_4 \otimes \v_1,
  e_1 \otimes \v_2, \ldots, e_4 \otimes \v_2,
  e_1' \otimes \v_1^*, \ldots, e_4' \otimes \v_1^*,
  e_1' \otimes \v_2^*, \ldots, e_4' \otimes \v_2^*
\]
of $\V^{\square}$, we identify $\Sp(\V^{\square})$ with $\Sp_{16}(F)$.
We define $\h \in \GL_4(F)$ by 
\[
 \begin{bmatrix}
  h e_1 \\
  h e_2 \\
  h e_3 \\
  h e_4
 \end{bmatrix}
 = \h \cdot 
 \begin{bmatrix}
  e_1 \\
  e_2 \\
  e_3 \\
  e_4
 \end{bmatrix}.
\]
Then we have $\h \cdot \a \cdot {}^t \h = \a$, where
\[
 \a :=
 \begin{pmatrix}
  a_1 & & & \\
  & a_2 & & \\
  & & a_3 & \\
  & & & a_4 
 \end{pmatrix}
 \in \GL_4(F).
\]
Moreover, we have $\iota(h^{-1}, 1) = \d^{-1} \cdot \h' \cdot \d$ and
\[
 \tau = \d^{-1} \cdot \tau_4 \cdot \d = \tau_4 \cdot
 \begin{pmatrix}
  \1_4 & & & \\
  & \a^{-1} & & \\
  & & \1_4 & \\
  & & & \a
 \end{pmatrix},
\]
where
\[
 \d =
 \begin{pmatrix}
  \1_4 & & & \\
  & \1_4 & & \\
  & & \a & \\
  & & & \a
 \end{pmatrix}, \qquad
 \h' = 
 \begin{pmatrix}
  \dot{\h} & & & -\frac{1}{2} \ddot{\h} \\
  & \dot{\h} & \frac{1}{2} \ddot{\h} & \\
  & 2 \ddot{\h} & \dot{\h} & \\
  -2 \ddot{\h} & & & \dot{\h} \\
 \end{pmatrix},
\]
$\dot{\h} = \frac{1}{2} (\h + \1_4)$,
$\ddot{\h} = \frac{1}{2} (\h - \1_4)$, and
\[
 \tau_4 = 
 \begin{pmatrix}
  \1_4 & & & \\
  & & & -\1_4 \\
  & & \1_4 & \\
  & \1_4 & & \\
 \end{pmatrix}.
\]
If $h=1$, then we have $z_{\V^{\triangle}}(\iota(h^{-1},1), \tau) = z_{\V^{\triangle}}(1, \tau) = 1$.
Assume that $h \ne 1$.
Since $B$ is ramified, $h-1$ is given by the automorphism $x \mapsto \ba \cdot x$ of $B$ with some $\ba \in B^{\times}$.
In particular, we have $\ddot{\h} \in \GL_4(F)$.
We have
\[
 \iota(h^{-1}, 1) = 
 \begin{pmatrix}
  & \frac{1}{2} \a {}^t \ddot{\h}^{-1} & \dot{\h} & \\
  - \frac{1}{2} \a {}^t \ddot{\h}^{-1} & & & \dot{\h} \\ 
  & & & 2 \a^{-1} \ddot{\h} \\
  & & -2 \a^{-1} \ddot{\h} &
 \end{pmatrix}
 \cdot \tau_8 \cdot \n,
\]
where
\[
 \tau_8 =
 \begin{pmatrix}
  & -\1_8 \\
  \1_8 &
 \end{pmatrix}, \qquad
 \n = 
 \begin{pmatrix}
  \1_4 & & & -\frac{1}{2} \ddot{\h}^{-1} \dot{\h} \a \\ 
  & \1_4 & \frac{1}{2} \ddot{\h}^{-1} \dot{\h} \a & \\
  & & \1_4 & \\
  & & & \1_4
 \end{pmatrix}.
\]
Hence we have
\[
 z_{\V^{\triangle}}(\iota(h^{-1},1), \tau)
 = z_{\V^{\triangle}}(\tau_8 \cdot \n, \tau_4)
 = z_{\V^{\triangle}}(\tau_8, \n \cdot \tau_4).
\]
Since $\V^{\triangle} \cdot \tau_4^{-1} \n \tau_4= \V^{\triangle}$, we have
\[
 z_{\V^{\triangle}}(\tau_8, \n \cdot \tau_4) = z_{\V^{\triangle}}(\tau_8, \tau_4) = 1.
\]
Thus we obtain
\[
 z_{\V^{\triangle}}(\g, \h_0'^{-1}) = 1.
\]

Now we compute $z_{\V^{\triangle}}(\h_0'^{-1}, \h_0' \g \h_0'^{-1})$.
Put $\g'' = (d_Y(\nu), \iota(h_1, h_2)) \in \G(\U(V)^0 \times \U(W^{\square}))$ and $g' = g \cdot d_Y(\nu)^{-1} \in \U(V)^0$.
Then we have $\g = g' \cdot \g''$.
Via the identification $\V^{\square} = V \otimes_B W^{\square} = W \otimes_B V^{\square}$, we identify $g'$ with $\iota(g',g') \in \U(V^{\square})$.
Since $\V^{\triangle} \cdot \h_0' = \Y'^{\square} = Y \otimes_B W^{\square}$,
we have $\V^{\triangle} \cdot \h_0' \g'' \h_0'^{-1} = \V^{\triangle}$ and hence
\begin{align*}
 z_{\V^{\triangle}}(\h_0'^{-1}, \h_0' \g \h_0'^{-1})
 & = z_{\V^{\triangle}}(g_0^{-1}, g_0 \cdot \iota(g',g') \cdot g_0^{-1}) \\
 & = \hat{s}'(\iota(g',g') \cdot g_0^{-1}) \cdot \hat{s}'(g_0^{-1})^{-1} \cdot
 \hat{s}'(g_0 \cdot \iota(g',g') \cdot g_0^{-1})^{-1}.
\end{align*}
Hence, if $B$ is split, then we have
\[
 z_{\V^{\triangle}}(\h_0'^{-1}, \h_0' \g \h_0'^{-1}) = 1.
\]
Assume that $B$ is ramified.
We write $g' = \left( \begin{smallmatrix} a & b \\ c & d \end{smallmatrix} \right)$.
Since 
\[
 g_0^{-1} =
 \begin{pmatrix}
  \frac{1}{2} & -\frac{1}{2} & & \\
  & & \frac{1}{2} & \frac{1}{2} \\
  & & 1 & -1 \\
  -1 & -1 & &
 \end{pmatrix}, \qquad
 g_0 \cdot \iota(g',g') \cdot g_0^{-1} = \iota^{\natural}(g', \jj_Y'(g')) =
 \begin{pmatrix}
  a & & b & \\
  & a & & -b \\
  c & & d & \\
  & -c & & d
 \end{pmatrix},
\]
and
\[
 \iota(g',g') \cdot g_0^{-1} =
 \begin{pmatrix}
  \frac{a}{2} & -\frac{a}{2} & \frac{b}{2} & \frac{b}{2} \\
  \frac{c}{2} & -\frac{c}{2} & \frac{d}{2} & \frac{d}{2} \\
  c & c & d & -d \\
  -a & -a & -b & b
 \end{pmatrix},
\]
we have
\[
 \hat{s}'(g_0^{-1}) = -1, \qquad
 \hat{s}'(g_0 \cdot \iota(g',g') \cdot g_0^{-1}) = 1, \qquad
 \hat{s}'(\iota(g',g') \cdot g_0^{-1}) = -1.
\]
Hence we have
\[
 z_{\V^{\triangle}}(\h_0'^{-1}, \h_0' \g \h_0'^{-1}) = 1.
\]

Thus we obtain
\[
 \mu'(\g)
 = z_{\V^{\triangle}}(\g, \h_0'^{-1})
 \cdot z_{\V^{\triangle}}(\h_0'^{-1}, \h_0' \g \h_0'^{-1})^{-1}
 = 1.
\]
This completes the proof.
\end{proof}

\subsection{Splitting $z_{\Y^{\square}}$}
\label{ss:spl-double-B-Y}

Let $\V = \X \oplus \Y$ be the complete polarization given in \S \ref{ss:spl-setup}.
Put
\[
 \X^{\square} := \X \oplus \X, \qquad \Y^{\square} := \Y \oplus \Y. 
\]
Then $\V^{\square} = \X^{\square} \oplus \Y^{\square}$ is a complete polarization.
As in \S \ref{ss:spl-double-B-Y'}, we have
\[
 z_{\Y^{\square}}(\iota(g_1, g_2), \iota(g_1',g_2'))
 = z_{\Y}(g_1, g_1') \cdot z_{\Y}(g_2, g_2')^{-1}
\]
for $g_i, g_i' \in \Sp(\V)$.
The Weil representation $\omega_{\psi}^{\square}$ of $\Mp(\V^{\square})$ can be realized on the Schwartz space
\[
 \SS(\X^{\square}) = \SS(\X) \otimes \SS(\X).
\]
As representations of $\Mp(\V)_{\Y} \times \Mp(\V)_{\Y}$, we have
\[
 \omega_{\psi}^{\square} \circ \tilde{\iota} = \omega_{\psi} \otimes (\omega_{\psi} \circ \tilde{\jj}_{\Y}),
\]
where $\tilde{\jj}_{\Y}$ is the automorphism of $\Mp(\V)_{\Y} = \Sp(\V) \times \C^1$ defined by
\[
 \tilde{\jj}_{\Y}(g, z) = (\jj_{\Y}(g), z^{-1}), \qquad
 \jj_{\Y}(g) = d_{\Y}(-1) \cdot g \cdot d_{\Y}(-1).
\]

Put $\J := ((\j_1, \j_2), \j)$.
Here we view $(\j_1, \j_2) \in \GU(V)^0$ and $\j \in \GU(W)$.

\begin{lem}
\label{lem:j_Y}
We have
\[
 \jj_{\Y}(\g) = \J \cdot \g \cdot \J^{-1}
\]
for $\g \in \GU(V)^0 \times \GU(W)$.
In particular, we have
\[
 \jj_{\Y}(\G(\U(V)^0 \times \U(W))) = \G(\U(V)^0 \times \U(W)).
\]
\end{lem}

\begin{proof}
Let $\g = ((\ba_1^{-1}, \ba_2^{-1}), \ba) \in \GU(V)^0 \times \GU(W)$ with $\ba_i = a_i + b_i \i + c_i \j_i + d_i \i \j_i \in B_i^{\times}$ and $\ba = a + b \i + c \j + d \i \j \in B^{\times}$.
By \S \ref{ss:spl-setup}, we see that $\jj_{\Y}(\g) = ((\bb_1^{-1}, \bb_2^{-1}), \bb)$, where
\[
 \bb_i = a_i - b_i \i + c_i \j_i - d_i \i \j_i, \qquad
 \bb = a - b \i + c \j - d \i \j.
\]
On the other hand, since $\j_i \i = - \i \j_i$ and $\j \i = - \i \j$,
we have $\j_i \cdot \ba_i \cdot \j_i^{-1} = \bb_i$ and $\j \cdot \ba \cdot \j^{-1} = \bb$.
This yields the lemma.
\end{proof}

As in \S \ref{ss:B-spl}, \ref{ss:B1-spl}, fix $\h_0 \in \Sp(\V)$ such that $\X' = \X \h_0$ and $\Y' = \Y \h_0$, and define a map $s: \GU(V)^0 \times \GU(W) \rightarrow \C^1$ by $s := s' \cdot \mu$, where
\[
 \mu(g) = z_{\Y}(\h_0 g \h_0^{-1}, \h_0) \cdot z_{\Y}(\h_0, g)^{-1}.
\]
Put $\hat{\h}_0 := \h_0' \cdot \iota(\h_0, \h_0)^{-1} \in \Sp(\V^{\square})$.
Then we have
\[
 \V^{\bigtriangledown} \cdot \hat{\h}_0 
 = \X'^{\square} \cdot \iota(\h_0, \h_0)^{-1} = \X^{\square}, \qquad
 \V^{\triangle} \cdot \hat{\h}_0
 = \Y'^{\square} \cdot \iota(\h_0, \h_0)^{-1} = \Y^{\square}.
\]
Put
\[
 \hat{\mu}(g)
 = z_{\V^{\triangle}}(g, \hat{\h}_0^{-1})
 \cdot z_{\V^{\triangle}}(\hat{\h}_0^{-1}, \hat{\h}_0 g \hat{\h}_0^{-1})^{-1}
\]
for $g \in \Sp(\V^{\square})$.
Then we have
\[
 z_{\Y^{\square}}(g, g') = z_{\V^{\triangle}}(g,g') \cdot
 \hat{\mu}(g g') \cdot \hat{\mu}(g)^{-1} \cdot \hat{\mu}(g')^{-1}
\]
for $g, g' \in \Sp(\V^{\square})$.

\begin{lem}
\label{lem:spl-double-B-compare}
We have
\[
 \hat{s} \cdot \hat{\mu} = s \otimes (s \circ \jj_{\Y})
\]
on $\Gc$.
\end{lem}

\begin{proof}
For $\g = (g, h_1, h_2) \in \Gc$, we identify $\g$ with $\iota(\g_1, \g_2) \in \Sp(\V^{\square})$, where $\g_i = (g, h_i) \in \G(\U(V)^0 \times \U(W)) \subset \Sp(\V)$.
Then, by a direct calculation, one can see that
\begin{align*}
 \hat{\mu}(\g) \cdot \mu'(\g)^{-1}
 & = z_{\Y^{\square}}(\iota(\h_0, \h_0) \cdot \g \cdot \iota(\h_0, \h_0)^{-1}, \iota(\h_0, \h_0)) \cdot z_{\Y^{\square}}(\iota(\h_0, \h_0), \g)^{-1} \\
 & = z_{\Y}(\h_0 \g_1 \h_0^{-1}, \h_0)
 \cdot z_{\Y}(\h_0 \g_2 \h_0^{-1}, \h_0)^{-1}
 \cdot z_{\Y}(\h_0, \g_1)^{-1} \cdot z_{\Y}(\h_0, \g_2) \\
 & = \mu(\g_1) \cdot \mu(\g_2)^{-1}.
\end{align*}
Hence, by Lemma \ref{lem:spl-double-B-compare'}, we have
\[
 \hat{s} \cdot \hat{\mu} = \hat{s} \cdot \mu' \cdot (\mu \otimes \mu^{-1})
 = (s' \cdot \mu) \otimes ((s' \circ \jj_{\Y'}) \cdot \mu^{-1})
\]
on $\Gc$.
Since $s' \circ \jj_{\Y'} = s' = s'^{-1}$, we have
\[
 \hat{s} \cdot \hat{\mu} = s \otimes s^{-1}
\]
on $\Gc$.
By Proposition \ref{prop:spl-B} and Lemma \ref{lem:j_Y}, we have $s(\jj_{\Y}(\g)) = s(\g)^{-1}$ for $\g = \ba_i^{-1} \in \GU(V)^0$ with $\ba_i \in B_i^{\times}$ and $\g = \ba \in \GU(W)$ with $\ba \in B^{\times}$.
Also, we have $z_{\Y}(\jj_{\Y}(g_1), \jj_{\Y}(g_2)) = z_{\Y}(g_1, g_2)^{-1}$ for $g_1, g_2 \in \GSp(\V)$.
Since $s(\g_1 \g_2) = s(\g_1) \cdot s(\g_2) \cdot z_{\Y}(\g_1, \g_2)$ for $\g_1, \g_2 \in\GU(V)^0 \times \GU(W)$, we have
\[
 s \circ \jj_{\Y} = s^{-1}
\]
on $\GU(V)^0 \times \GU(W)$.
This completes the proof.
\end{proof}